\documentclass[11pt]{amsart}
\usepackage{amsmath}
\usepackage{amssymb}
\usepackage{amsthm}
\usepackage{latexsym}
\usepackage{graphicx}
\usepackage{hyperref}
\usepackage{enumerate}
\usepackage[all]{xy}
\usepackage{chngcntr}

\newtheorem{thm}{Theorem}[section]
\newtheorem{cor}[thm]{Corollary}
\newtheorem{lem}[thm]{Lemma}
\newtheorem{prop}[thm]{Proposition}

\theoremstyle{definition}
\newtheorem{defn}[thm]{Definition}

\newtheorem{ex}[thm]{Example}

\newcommand{\N}{\mathbb N}
\newcommand{\Z}{\mathbb Z}
\newcommand{\Q}{\mathbb Q}
\newcommand{\R}{\mathbb R}
\newcommand{\C}{\mathbb C}
\newcommand{\F}{\mathbb F}
\newcommand{\mf}{\mathfrak}
\newcommand{\mc}{\mathcal}
\newcommand{\mb}{\mathbf}
\newcommand{\mh}{\mathbb}

\def\Irr{{\rm Irr}}
\newcommand{\mr}{\mathrm}
\newcommand{\ind}{\mathrm{ind}}
\newcommand{\enuma}[1]{\begin{enumerate}[\textup{(}a\textup{)}] {#1} \end{enumerate}}

\newcommand{\cusp}{\mathrm{cusp}}
\newcommand{\nr}{\mathrm{nr}}

\newcommand{\cpt}{\mathrm{cpt}}
\newcommand{\Rep}{\mathrm{Rep}}
\newcommand{\Res}{\mathrm{Res}}

\newcommand{\matje}[4]{\left(\begin{smallmatrix} #1 & #2 \\ 
#3 & #4 \end{smallmatrix}\right)}

\newcommand{\Mod}{\mathrm{Mod}}
\newcommand{\Hom}{\mathrm{Hom}}
\newcommand{\End}{\mathrm{End}}

\newcommand{\isom}{\xrightarrow{\sim}}

\newcommand{\CO}{\texttt{CO}}
\newcommand{\lr}[1]{\langle #1 \rangle}
\newcommand{\temp}{\mathrm{temp}}
\newcommand{\disc}{\mathrm{disc}}
\newcommand{\fl}{\mathrm{fl}}
\newcommand{\Tor}{\mathrm{Tor}}

\begin{document}

\title[p-adic groups, representations and noncommutative geometry]
{Group algebras of reductive $p$-adic groups,\\ their representations and 
their noncommutative geometry}
\date{\today}
\maketitle

\begin{center}
{\Large Maarten Solleveld} \\[1mm]
IMAPP, Radboud Universiteit Nijmegen\\
Heyendaalseweg 135, 6525AJ Nijmegen, the Netherlands \\
email: m.solleveld@science.ru.nl
\end{center}

\begin{abstract}
This is a survey paper about representation theory and noncommutative geometry
of reductive $p$-adic groups $G$. The main focus points are:

1. The structure of the Hecke algebra $\mc H (G)$, the Harish-Chandra--Schwartz
algebra $\mc S (G)$ and the reduced $C^*$-algebra $C_r^* (G)$.

2. The classification of irreducible $G$-representations in terms of
supercuspidal representations.

3. The Hochschild homology and topological K-theory of these algebras.

In the final part we prove one new result, namely we compute $K_* (C_r^* (G))$
including torsion elements, in terms of equivariant K-theory of compact tori.
\end{abstract}

\tableofcontents

\section*{Introduction}

This survey paper is based on a series of lectures delivered by the author in
February--March 2025, for the thematic trimester "Representation theory and
noncommutative geometry". The lectures focused on the themes of the trimester, 
for reductive $p$-adic groups $G$. These notes serve the same general goals 
as the lectures (with less time pressure):
\begin{itemize}
\item Describe the structure of various (topological) versions of the group
algebra of $G$.
\item Explain the classification of irreducible $G$-representations in terms
of supercuspidal representations of Levi subgroups.
\item Discuss the computation of the Hochschild homology and the topological
K-theory of these group algebras.
\end{itemize}
The Hecke algebra $\mc H (G)$ of a reductive $p$-adic group $G$ has been a
popular object of study. A lot is known about this algebra and its
representations, and the more abstract part of that theory has been 
consolidated in Renard's textbook \cite{Ren}. The Harish-Chandra--Schwartz
algebra $\mc S (G)$ is important in harmonic analysis and for tempered
$G$-representations. Further, the reduced $C^*$-algebra $C_r^* (G)$ is
crucial in the noncommutative geometry of $G$. Although both topological
algebras $\mc S (G)$ and $C_r^* (G)$ have been known for more than half
a century, the results about them are scattered in the literature. We bring
some of those results together in these notes

Roughly speaking, in Section \ref{sec:1} we will discuss the following 
relations between the three group algebras of $G$:\\[1mm]

\begin{tabular}{ccc}
\multicolumn{3}{c}{{\bf Group algebras}}  \\
$\mc H (G)$ & $\mc S (G)$ & $C_r^* (G)$ \\[2mm]
\multicolumn{3}{c}{{\bf Locally constant functions on group}} \\
compactly supported & rapidly decaying & "$C^*$-norm bounded" \\[2mm]
\multicolumn{3}{c}{{\bf Representations}} \\
smooth & tempered & tempered unitary \\[2mm]
\multicolumn{3}{c}{{\bf Fourier transform}} \\
algebraic sections & differentiable sections & continuous sections \\[2mm]
\end{tabular}

The classification of irreducible (smooth, complex) $G$-representations has 
two main parts:
\begin{enumerate}[(i)]
\item supercuspidal $G$-representations,
\item irreducible representations in the parabolic inductions of a given
set of supercuspidal representations of a Levi subgroup of $G$.
\end{enumerate}
Part (i) is rather arithmetic, and outside the scope of these notes. We refer
the reader to the surveys \cite{Afg,Fin}. 
Part (ii) is more geometric and highly relevant for the noncommutative geometry of $G$. 
Moreover there is a beautiful solution to (ii), which was already conjectured around 
2012 by ABPS (Aubert, Baum, Plymen and the author). 
We recall that every Bernstein block $\Rep (G)^{\mf s}$ 
in the category of smooth $G$-representations comes with:
\begin{itemize}
\item a complex torus $T_{\mf s}$ which parametrizes the underlying supercuspidal
representations,
\item a finite group $W_{\mf s}$ such that the centre of $\Rep (G)^{\mf s}$ is
$\mc O (T_{\mf s})^{W_{\mf s}} = \mc O (T_{\mf s}/W_{\mf s})$.
\end{itemize}
A simplified version of the ABPS conjecture says that there is a canonical bijection
between 
\begin{itemize}
\item the set $\Irr (G)^{\mf s}$ of irreducible representations in $\Rep (G)^{\mf s}$,
\item the set of irreducible representations of the crossed product 
$\mc O (T_{\mf s}) \rtimes W_{\mf s}$.
\end{itemize}
Actually, in general one has to extend both $T_{\mf s}$ and $W_{\mf s}$ by a finite
group, and one has to twist the group algebra of $W_{\mf s}$ by a 2-cocycle. For the
sake of presentation, we suppress those details in the introduction.

The proof of the ABPS conjecture \cite{SolEnd} entails several steps:
\begin{itemize}
\item $\Rep (G)^{\mf s}$ is made equivalent to the module category of the 
$G$-endomorphism algebra of a progenerator $\Pi_{\mf s}$.
\item General analysis of the structure of $\End_G (\Pi_{\mf s})$.
\item Relate localized versions of $\End_G (\Pi_{\mf s})$ to graded Hecke algebras.
\item Study the representation theory of graded Hecke algebras.
\item Classify the irreducible representations of a graded Hecke algebra in 
terms of those of a crossed product algebra.
\item Classify the irreducible representations of a crossed product algebra.
\end{itemize}
In Sections \ref{sec:2} and \ref{sec:3} we treat all these topics, in survey style
(and in a different order).\\

After that we come to the noncommutative geometric part of the paper. Section \ref{sec:4}
starts with an introduction to Hochschild homology for algebras, based on \cite{Lod}
and focussing on algebras that are close to commutative. Then we discuss the Hochschild
homology of $\mc H (G)$ and of $\mc H (G)^{\mf s}$, the summand of $\mc H (G)$ 
associated to $\Rep (G)^{\mf s}$. The latter is computed via the Hochschild 
homology of graded Hecke algebras. It turns out that $\mc H (G)^{\mf s}$ has the\ 
same Hochschild homology as the crossed product $\mc O (T_{\mf s}) \rtimes W_{\mf s}$, which is known from \cite{Nis}.

Section \ref{sec:5} is dedicated to the Hochschild homology of topological algebras
related to $\mc S (G)$. There is a lot technique behind this, because it involves a
mix of homological algebra and functional analysis. We briefly discuss the setup and 
some examples of such topological Hochschild homology groups. 
Let $\mc S (G)^{\mf s}$ be the summand of $\mc S (G)$ associated to $\Rep (G)^{\mf s}$
and let $T_{\mf s}^u$ be the compact real subtorus of $T_{\mf s}$ that parametrizes
the unitary supercuspidal representations underlying $\Rep (G)^{\mf s}$. Then
\begin{equation}\label{eq:1}
HH_n (\mc S (G)^{\mf s}) \cong HH_n ( C^\infty (T_{\mf s}^u) \rtimes W_{\mf s}) ,
\end{equation}
and the right hand side had already been computed earlier \cite{Bry}.
We point out that $HH_* (C_r^* (G))$ is less interesting: it recovers the
cocenter $C_r^* (G) / [C_r^* (G),C_r^* (G)]$ but nothing more. 

The isomorphism \eqref{eq:1} and its analogue for $\mc H (G)^{\mf s}$ can be regarded
as a version of the ABPS conjecture in Hochschild homology. From there it is only a
small step to compute the periodic cyclic homology $HP_*$ of $\mc H (G)$ and $\mc S (G)$,
again in terms of crossed product algebras.

Finally we come to topological K-theory in Section \ref{sec:6}, the most truly 
noncommutative geometric part of the paper. We need K-theory both for Banach algebras
and for Fr\'echet algebras, and we recall some general results in the latter setting.
Via Chern characters, K-theory is related to periodic cyclic homology, which leads
to isomorphisms
\begin{equation}\label{eq:2}
K_* (C_r^* (G)) \otimes_\Z \C \cong K_* (\mc S (G))_\Z \C \cong HP_* (\mc S (G)) .
\end{equation}
These results can be compared with the Baum--Connes conjecture, which was proven 
for $G$ in \cite{Laf}. We formulate the comparison with a completely algebraic
counterpart to the Baum--Connes conjecture from \cite{HiNi}.

To determine $K_* (C_r^* (G))$ including torsion elements, we abandon the survey
style and add some new results in the last three paragraphs. Our idea is to transfer
the setup used for $\mc H (G)^{\mf s}$ to the summand $C_r^* (G)^{\mf s}$ of 
$C_r^* (G)$. We construct a progenerator $\Pi_{\mf s}^c$ for the module category
of $C_r^* (G)^{\mf s}$, and an action of $C(T_{\mf s}^u) \rtimes W_{\mf s}$ on
$\Pi_{\mf s}^c$. We deduce isomorphisms
\begin{equation}\label{eq:3}
K_* (C_r^* (G)^{\mf s}) \cong K_* (C(T_{\mf s}^u) \rtimes W_{\mf s}) \cong
K_{W_{\mf s}}^* (T_{\mf s}^u) , 
\end{equation}
where the rightmost term denotes the $W_{\mf s}$-equivariant K-theory of the
topological space $T_{\mf s}^u$. In Theorem \ref{thm:6.22}, the precise version of
\eqref{eq:3}, $T_{\mf s}^u$ and $W_{\mf s}$ are replaced by finite covers, and a
twist by a 2-cocycle of the extension of $W_{\mf s}$ is involved. In the end, 
that computes $K_* (C_r^* (G)^{\mf s})$ in terms of "twisted" 
$W_{\mf s}$-equivariant vector bundles on $T_{\mf s}^u$. That settles the ABPS
conjecture in topological K-theory \cite{ABPS2}.\\

\textbf{Acknowledgements} \

This paper was conceived when the author visited the Institut Henri Poincar\'e
in Paris (France), for the trimester program "Representation theory and noncommutative
geometry". It is a pleasure to thank all the organizers of this program for their
efforts, and the Institut Henri Poincar\'e for the hospitality.

\renewcommand{\theequation}{\arabic{section}.\arabic{equation}}
\counterwithin*{equation}{section}
\renewcommand{\thepart}{\Alph{part}}

\part{Representation theory of reductive $p$-adic groups}

\section{Group algebras} \label{sec:1}

\subsection{Example: $GL_1 (F)$} \

For a nice introduction to $p$-adic numbers we refer to \cite{Gou}. In this paragraph we 
mention several
well-known aspects of $p$-adic numbers without further comments.

Let $F$ be a non-archimedean local field. If $F$ has characteristic zero, then it is a finite
extension of the field of $p$-adic numbers $\Q_p$, for some prime $p$. On the other hand, 
if $F$ has characteristic
$p$, then it is a finite extension of the local function field $\F_p ((T))$.

Let $v_F : F \to \Z \cup \{\infty\}$ be the discrete valuation and fix an element $\varpi_F$ 
with $v_F (\varpi_F) = 1$. Let $\mf o_F = v_F^{-1}(\Z_{\geq 0} \cup \{\infty\})$ be the ring 
of integers of $F$ and let $\varpi_F \mf o_F = v_F^{-1}(\Z_{> 0} \cup \{\infty\})$ be its 
unique maximal ideal. The residue field $k_F = \mf o_F / \varpi_F \mf o_F$ is a finite field 
of characteristic $p$, whose cardinality we denote by $q_F$.

\begin{ex}
In $F = \Q_p$ we have $v_F (p^n a / b) = n$ for $a,b \in \Z$ not divisible by $p$. Further 
$\mf o_F = \Z_p$ and we can take $\varpi_F = p$. Then $k_F = \Z_p / p \Z_p = \Z / p \Z = \F_p$.

In $F = \F_p ((T))$ we have $v_F \big( \sum_{n=N}^\infty a_n T^n \big) = N$ when 
$a_n \in \F_p$ and $a_N \neq 0$. Here $\mf o_F = \F_p [[T]]$ and one takes $\varpi_F = T$. 
Then $k_F = \F_p [[T]] / T \F_p [[T]] = \F_p$.
\end{ex}

On $F$ one defines the absolute value $|x|_F = q_F^{-v_F (x)}$. We note that $v_F (0) = \infty$ 
and that $|0|_F = 0$. The absolute value defines a metric $d_F (x,y) = |x-y|_F$ on $F$, 
with respect to which $F$ is complete. Since the metric takes values in $q_F^\Z \cup \{0\}$, 
the image of $d_F$ does
not contain any interval in $\R$. Hence $F$ with the metric topology is a totally disconnected
Hausdorff space. It is not discrete though, for instance the sequence $(\varpi_F^n)_{n=1}^\infty$
converges to 0 in $F$. By the completeness of $F$ with respect to $d_F$ and the finiteness 
of $k_F$, $F$ is locally compact. The ring $\mf o_F$ is compact.

We shall discuss various group algebras of algebraic groups over $F$. Let us start with the
simplest example of such a group: $F^\times = GL_1 (F)$. Here and below, we denote the group of
invertible elements in a unital ring $R$ by $R^\times$. For $\mf o_F$ that gives
\[
\mf o_F^\times = \mf o_F \setminus \varpi_F \mf o_F = v_F^{-1}(0).
\] 
Notice that this is strictly smaller than $\mf o_F \cap F^\times = \mf o_F \setminus \{0\}$. 
In fact $\mf o_F^\times$ is the unique maximal compact subgroup of $F^\times$. We note that the 
valuation $v_F$ induces a group isomorphism $F^\times / \mf o_F^\times \cong \Z$.

For simplicity we look only at functions on $F^\times$ that are invariant under $\mf o_F^\times$.
As $F^\times$ is abelian, those are the same as $\mf o_F^\times$-biinvariant functions on
$F^\times$. Three interesting algebras of such functions are:
\begin{itemize}
\item The Hecke algebra 
\begin{equation}\label{eq:1.1}
\mc H (F^\times )^{\mf o_F^\times} = \C [F^\times / \mf o_F^\times] 
\cong \C [\Z] \cong \mc O (\C^\times) ,
\end{equation}
where $\mc O$ means regular functions on an affine variety.
\item The Harish-Chandra--Schwartz algebra
\begin{equation}\label{eq:1.2}
\mc S (F^\times)^{\mf o_F^\times} = \mc S (F^\times / \mf o_F^\times) \cong \mc S (\Z) =
\{ \text{Schwartz functions } f : \Z \to \C \} \cong C^\infty (S^1) ,
\end{equation}
where $S^1$ is unit circle in $\C$.
\item The reduced $C^*$-algebra
\begin{equation}\label{eq:1.3}
C_r^* (F^\times)^{\mf o_F^\times} = 
C_r^* (F^\times / \mf o_F^\times) \cong C_r^* (\Z) \cong C(S^1) .
\end{equation}
\end{itemize}
Remarkable here is that, although $F^\times$ is rather complicated and strange as a topological
space (totally disconnected but not discrete), these three algebras of functions on $F^\times$ 
are very nice and well-behaved. They are among the standard examples of algebras from a course 
on (respectively) algebraic geometry, differential geometry and topology.

\subsection{Definitions and first properties} \

From now on $\mc G$ is a connected reductive algebraic group defined over $F$. We call 
$G = \mc G (F)$ a reductive $p$-adic group (even though $F$ is not necessarily a $p$-adic 
field, it may have positive characteristic). Of course $G$ can be endowed with the Zariski
topology, but for the purposes of representation theory it is more useful to consider 
a topology on $F$ which comes from the metric topology on $F$. 

We spell out this topology. Since $\mc G$ is a linear algebraic group, it can be embedded in $GL_n$
for some $n \in \N$. On the matrix ring $M_n (F)$ we define a norm by
\[
\| (a_{ij})_{i,j=1}^n \|_F = \max_{i,j} |a_{ij}|_F .
\]
That in turn yields a metric on $GL_n (F)$ by 
\begin{equation}\label{eq:1.4}
d (A,B) = \max \{ \| A - B \|_F, \|A^{-1} - B^{-1}\|_F \} ,
\end{equation}
where the term with the inverses is needed to make inversion on $GL_n (F)$ continuous.
We restrict this to a metric $d$ on $G \subset GL_n (F)$. While $d$ depends on the choice of an
embedding $\mc G \to GL_n$, the resulting topology on $G$ does not. 
This makes $G$ into a totally disconnected locally compact group. (By convention locally compact
groups are Hausdorff.) 

Since the metric on $G$ only takes values in $q_F^\Z \cup \{0\}$, every open ball in $G$ is also
a closed ball. By the local compactness every closed ball is compact, so $G$ has lots of compact
open subsets. Even better, $G$ has many compact open subgroups. 

Here is one construction of small compact open subgroups. The group
\[
GL_n (\mf o_F) = \{ A \in M_n (\mf o_F) : \det (A) \in \mf o_F^\times \}
\]
is compact and open in $GL_n (F)$. Namely, $M_n (\mf o_F)$ and 
\[
M_n (\mf o_F)^\sharp := \{A \in M_n (\mf o_F) : \det (A) \in \varpi_F \mf o_F \}
\] 
are compact open subsets of $M_n (F)$, so their difference is compact. Further
\[
GL_n (\mf o_F) = M_n (\mf o_F) \setminus M_n (\mf o_F)^\sharp =
(M_n (\mf o_F) \cap GL_n (F)) \setminus (M_n (\mf o_F)^\sharp \cap GL_n (F)) ,
\]
which is open in $GL_n (F)$. For any $m \in \Z_{\geq 0}$ we have the ring homomorphism
$\mf o_F \to \mf o_F / \varpi_F^m \mf o_F$. This induces a group homomorphism
\[
\mr{mod}_{\varpi_F^m} : GL_n (\mf o_F) \to GL_n (\mf o_F / \varpi_F^m \mf o_F) ,
\]
whose image is finite because
\[
| \mf o_F / \varpi_F^m \mf o_F| = q_F^m < \infty .
\]
Then $GL_n (F)_m := \ker (\mr{mod}_{\varpi_F^m})$ is an open subgroup of the compact 
group $GL_n (\mf o_F)$, so automatically closed and compact. In fact $GL_n (F)_m$
equals the closed ball of radius $q_F^{-m}$ around $I_n$ in $GL_n (F)$. It follows that
\begin{equation}\label{eq:1.5}
G_m := G \cap GL_n (F)_m
\end{equation}
is a compact open subgroup of $G$. This is known as a congruence subgroup of $G$, because
it consists of matrices that are equal to the identity modulo $\varpi_F^m \mf o_F$. With
respect to the metric $d$ from \eqref{eq:1.4}, the group $G_m$ is the closed ball of radius 
$q_F^{-m}$ and the open ball of radius $q_F^{1/2-m}$, both around the unit element $I_n$. 
In particular the decreasing sequence of compact open subgroup $\{G_m : m \in \Z_{\geq 0}\}$ 
is a neighborhood basis of $I_n$ in $G$.

We denote the set of compact open subgroups of $G$ by $\CO (G)$. A large supply of such 
subgroups comes from Bruhat--Tits theory, for which we refer to \cite{KaPr,Tit}. For each
$K \in \CO (G)$, the spaces $G/K$ and $K \backslash G$ are discrete and countable. For the
latter, notice that $G$ is a countable union of compact balls, and that each compact 
subset of $G$ is covered by finitely many cosets of the open group $K$.

The above provides, among others, ways to partition $G$ are as a disjoint union of compact open
subsets, just take the left cosets of one of the subgroups $G_m$. The abundance of compact open
subsets means that $G$ admits many locally constant functions. For instance, any function on
the discrete space $G / G_m$ can be inflated to a locally constant function on $G$. We let
$C^\infty (G)$ be the vector space of locally constant functions $f : G \to \C$. It is an
algebra with respect to pointwise multiplication. The notation $C^\infty$ comes from manifolds,
even though on a totally disconnected space like $G$ there is no good notion of differentiability
for general functions. The reasoning is that locally constant functions are the only functions
on $G$ that one can differentiate for sure: all their partial derivatives are 0.

We fix a left Haar measure $\mu$ on $G$. The group $G$ is unimodular (because it is reductive,
see \cite[\S V.5.4]{Ren}), so $\mu$ is a also a right Haar measure. The convolution product 
of two integrable functions $f_1, f_2 : G \to \C$ is defined as
\[
(f_1 * f_2)(x) = \int_G f_1 (x g^{-1}) f_2 (g) \textup{d} g =
\int_G f_1 (g) f_2 (g^{-1} x) \textup{d} g .
\]
Here and later we suppress the Haar measures from the notations of integrals. We note that the
convolution product generalizes the multiplication in the full group algebra $\C[G]$, with
a modification for the measures of sets.
For instance, let $f_i = 1_{g_i K_i}$ be the indicator function of a left coset of 
$K_i \in \CO (G)$. For $x = g_1 k_1 g_2 k_2$ with $k_i \in K_i$ one computes 
\begin{align*}
(1_{g_1 K_1} * 1_{g_2 K_2})(x) & = \int_{g K_1} 1_{g_2 K_2} (g^{-1} x) \textup{d} g =
\int_{K_1} 1_{g_2 K_2}(h^{-1} g_1^{-1} x) \textup{d} h \\
& = \int_{K_1} 1_{g_2 K_2}(h^{-1} k_1 g_2 k_2) \textup{d} h = 
\int_{K_1} 1_{g_2 K_2}(k g_2 k_2) \textup{d}k \\
& = \mu (K_1 g_2 k_2 \cap g_2 K_2) = \mu (K_1 \cap g_2 K_2 g_2^{-1}) .
\end{align*}
The support of $f_1 * f_2$ is always contained in $\mr{supp}(f_1) \cdot \mr{supp}(f_2)$, so
\begin{equation}\label{eq:1.6}
1_{g_1 K_1} * 1_{g_2 K_2} = \mu (K_1 \cap g_2 K_2 g_2^{-1}) \, 1_{g_1 K_1 g_2 K_2} .
\end{equation}
This equation shows that the convolution product generalizes the multiplication
in the full group algebra $\C[G]$, with a modification for the measures of sets.
If we would replace $\mu$ by the counting measure on $G$ and we would only
use functions with finite support, then \eqref{eq:1.6}
would recover the multiplication in $\C [G]$.

\begin{defn}
The Hecke algebra $\mc H (G)$ is the vector space $C_c^\infty (G)$ of locally constant compactly
supported functions $f : G \to \C$, endowed with the convolution product.
\end{defn}

The $\C$-algebra $\mc H (G)$ is associative but not unital. Namely, from \eqref{eq:1.6} we
see that a unit element of $\mc H (G)$ would have to be supported only at the identity of $G$.
But $\{e\}$ is open in $G$ if and only if $G$ is finite, which happens only when $\mc G = \{e\}$.

For any $K \in \CO (G)$, the $K$-biinvariant functions form a subalgebra
\[
\mc H (G,K) = C_c (K \backslash G / K)
\]
of $\mc H (G)$. This subalgebra has
\[
\lr{K} := \mu (K)^{-1} 1_K 
\]
as unit, as one can check like in \eqref{eq:1.6}.

\begin{prop}\label{prop:1.1}
\enuma{
\item $\mc H (G) = \bigcup_{K \in \CO (G)} \mc H (G,K) = 
\bigcup\nolimits_{m=1}^\infty \mc H (G,G_m)$.
\item The algebra $\mc H (G)$ has local units: for every finite subset 
$S \subset \mc H (G)$ there exists an idempotent $e_S \in \mc H (G)$ such that
$e_S \, s = s = s \, e_S$ for all $s \in S$.
\item $\mc H (G)$ has countable dimension.
}
\end{prop}
\begin{proof}
(a) Every $f \in \mc H (G)$ is compactly supported and locally constant, so takes only finitely
many values in $\C$. For each nonzero value $z$, $f^{-1}(z)$ is a compact open subset of $G$.
Since the $G_m$ form a neighorhood basis of $e$ in $G$, there exists $m \in \N$ such that
$f^{-1}(z)$ is a union of $G_m$-double cosets. As $f$ takes only finitely many values,
there exists an $m$ that works for all values $z \neq 0$. Then $f \in \mc H (G,G_m)$. \\
(b) By part (a) there exist $K_s \in \CO (G)$ such that $s \in \mc H (G,K_s)$. Define 
$K = \bigcap_{s \in S} K_s$, this is a compact open subgroup because $S$ is finite.
Now $s \in \mc H (G,K)$ for all $s \in S$, so $e_S = \lr{K}$ has the required property.\\
(c) The space $G_m \backslash G / G_m$ is countable, so $\mc H (G,G_m)$ has countable
dimension. Combine that with part (a).
\end{proof}

On $C_c (G)$ we have the standard norms
\[
\| f \|_r = \big( \int_G |f(g)|^r \textup{d}g \big)^{1/r} \; \text{with } r \in \R_{\geq 1} 
\qquad \text{and} \qquad \| f \|_\infty = \sup_{g \in G} |f(g)| .
\]
\begin{lem}\label{lem:1.3}
$\mc H (G)$ is dense in $C_c (G)$ for the norms $\| \cdot \|_r$ with 
$r \in \R_{\geq 1} \cup \{\infty\}$.
\end{lem}
\begin{proof}
Let $f \in C_c (G)$. For $m \in \Z_{\geq 1}$ we define $f_m \in \mc H (G,G_m)$ as follows:
pick a set of representatives $\{ g_{m,i} \}_i$ for $G_m \backslash G / G_m$ and put
$f_m (g) = f (g_{m,i})$ for $g \in G_m g_{m,i} G_m$. By the continuity of $f$ and the 
compactness of its support, the sequence $(f_m )_{m=1}^\infty$ converges uniformly to $f$.
Hence it also converges to $f$ with respect to the norms $\| \cdot \|_r$.
\end{proof}

The Banach algebra $L^1 (G)$ acts continuously on the Hilbert space $L^2 (G)$, by the 
convolution product. That yields an injective homomorphism from $L^1 (G)$ to
$\mc B (L^2 (G))$, the $C^*$-algebra of bounded linear operators on $L^2 (G)$. We make
$\mc H (G)$ and $L^1 (G)$ into *-algebras by
\[
f^* (g) = \overline{f(g^{-1})} .
\]

\begin{defn}
The reduced $C^*$-algebra $C_r^* (G)$ is the closure of $L^1 (G)$ in $\mc B (L^2 (G))$,
with respect to the operator norm.
\end{defn}

By Lemma \ref{lem:1.3} $\mc H (G)$ is dense in $L^1 (G)$, hence $\mc H (G)$ is also
dense in $C_r^* (G)$. In other words, $C_r^* (G)$ can be regarded as the completion of
$\mc H (G)$ for the operator norm of $\mc B (L^2 (G))$. Lemma \ref{lem:1.3} also says
that $\mc H (G)$ is dense in $L^2 (G)$. Therefore the operator norm of $\mc H (G)$ 
acting on $L^2 (G)$ equals the operator norm of $\mc H (G)$ acting on itself by
left multiplication. Thus $C_r^* (G)$ can be constructed entirely in terms of $\mc H (G)$,
as a $C^*$-completion of that *-algebra.

For $K \in \CO (G)$, we let 
\[
C_r^* (G,K) = \lr{K} C_r^* (G) \lr{K}
\] 
be the sub-$C^*$-algebra of $K$-biinvariant functions in $C_r^* (G)$. In contrast with
$C_r^* (G)$, the algebra $C_r^* (G,K)$ has a unit, namely $\lr{K}$. One may identify
$C_r^* (G,K)$ with the closure of $\mc H (G,K) = \lr{K} \mc H (G) \lr{K}$ in $C_r^* (G)$
or in $\mc B (L^2 (G))$. It follows from Proposition \ref{prop:1.1} that
\begin{equation}\label{eq:1.9}
C_r^* (G) = \lim_{m \to \infty} C_r^* (G,G_m) = \lim_{K \in \CO (G)} C_r^* (G,K) ,
\end{equation}
where the limit is meant in the category of Banach algebras.

The original definition of the Harish-Chandra--Schwartz algebra of $G$ \cite{HC}
is rather complicated, we prefer the simpler construction in \cite{Vig}. 
Using the embedding $G \to GL_n (F)$, we define a length function on $G$ by
\[
g \mapsto \log \big( \max \big\{ \|g\|_F, \|g^{-1}\|_F \big\} \big) .
\]
Then the function 
\[
\sigma : G \to \R_{\geq 1} ,\; 
\sigma (g) = 1 + \log \big( \max \big\{ \|g\|_F, \|g^{-1}\|_F \big\} \big) 
\]
is a scale, which means that it satisfies $\sigma (g^{-1}) = \sigma (g)$ and 
$\sigma (g g') \leq \sigma (g) \sigma (g')$. For $m \in \Z_{>0}$ we define a norm $\nu_m$
on $C_c (G)$ by $\nu_m (f) = \| \sigma^m f \|_2$.

\begin{defn}\label{def:1.SG}
For $K \in \CO (G)$, $\mc S (G,K)$ is the completion of $\mc H (G,K)$ with respect to
the family of norms $\nu_m \; (m \in \Z_{>0})$. The Harish-Chandra--Schwartz algebra 
of $G$ is $\mc S (G) = \bigcup_{K \in \CO (G)} \mc S (G,K)$, endowed with the inductive
limit topology.
\end{defn}

Thus $\mc S (G)$ consists of locally constant functions on $G$ that decay rapidly.
These functions need not have compact support, but every one of them is 
biinvariant under some $K \in \CO (G)$.
Some important properties of the algebras $\mc S (G,K)$ where proven by Vign\'eras:

\begin{thm}\label{thm:1.4} \textup{\cite[Propositions 10, 13, 28]{Vig}}
\enuma{
\item $\mc S (G,K)$ is a nuclear Fr\'echet *-algebra with unit $\lr{K}$.
\item $\mc S (G,K)$ is a dense subalgebra of $C_r^* (G,K)$, with a finer topology.
\item $\mc S (G,K) \cap C_r^* (G,K)^\times = \mc S (G,K)^\times$, and this set is
open in $\mc S (G,K)$.
}
\end{thm}

By Theorem \ref{thm:1.4}.b and \eqref{eq:1.9}, $\mc S (G)$ is contained in $C_r^* (G)$.
Like $\mc H (G)$ and $C_r^* (G)$, the Harish-Chandra--Schwartz algebra of $G$ is not 
unital. The same argument as for Proposition \ref{prop:1.1}.b shows that $\mc S (G)$ 
does have local units, for instance the idempotents $\lr{K}$ with $K \in \CO (G)$.
 
The multiplication in $\mc S (G)$ is separately continuous, that is, for any fixed 
$a \in \mc S (G)$ the maps $f \mapsto f a$ and $f \mapsto a f$ are continuous 
\cite[Lemme III.6.1]{Wal}. However, $\mc S (G)$ is not a Fr\'echet algebra, because
the topological vector space $\mc S (G)$ is not Fr\'echet. It is a strict inductive 
limit of Fr\'echet spaces, but such spaces are not metrizable \cite[Corollaire 4.2]{DiSc}.

\subsection{Classes of representations of reductive $p$-adic groups} \ 

In these notes, all representations will by default be on complex vector spaces. The best
notion of continuity for a representations of a reductive $p$-adic group is smoothness.

\begin{defn}\label{def:smooth}
A $G$-representation $(\pi,V)$ is smooth if for all $v \in V$ the stabilizer group
$G_v = \{g \in G : \pi (g) v = v\}$ is open in $G$. Equivalently, $\pi$ is smooth
if the map $\pi : G \times V \to V$ is continuous with respect to the discrete
topology on $V$.
\end{defn}

This is a rather crude inpretation of smooth, like $C^\infty (G)$. It says that for
any fixed $v \in V$ the map $G \to V : g \mapsto \pi (g) v$ is locally constant.

\begin{ex}
Let $G$ act on $\mc H (G)$ by left translations: 
\[
(\lambda (g) f)(x) = f (g^{-1} x) \qquad \text{for } g,x \in G, f \in \mc H (G) .
\]
If $f \in \mc H (G,K)$, then $\lambda (k) f = f$ for all $k \in K$. Hence
$(\lambda, \mc H (G))$ is a smooth $G$-representation.
\end{ex}

The motivation for considering the class of smooth representations comes from
profinite groups. Consider a projective limit of finite groups $H = \varprojlim H_i$.
Then each finite group $H_i$ is a quotient of $H$ and $\ker (H \to H_i)$ is an open 
subgroup of $H$. It is natural to impose that every irreducible $H$-representation 
factors through $H_i$ for some $i$. Smoothness of $H$-representations enforces that
(at least under the small extra condition that $H$ is its own profinite completion), 
and at the same time is sufficiently flexible to enable direct limits of 
$H$-representations.

Recall that every compact totally disconnected Hausdorff group is a profinite 
group, and conversely. In particular every compact subgroup of a reductive
$p$-adic group is profinite. Thus smoothness of a $G$-representation $(\pi,V)$
means that for every $K \in \CO (G)$, the restriction $\pi |_K$ belongs to the
natural class of $K$-representations.

Every smooth $G$-representation extends to a representation of $\mc H (G)$ on $V$, by
\begin{equation}\label{eq:1.10}
\pi (f) v = \int_G f(g) \pi (g) v \textup{d}g \qquad f \in \mc H (G), v \in V. 
\end{equation}
Since $f \in C_c^\infty (G)$ and $\pi$ is smooth, this integral
boils down to a finite sum and there are no convergence issues. 

It is not quite true every $\mc H (G)$-module gives
rise to a (smooth) $G$-representation, because $\mc H (G)$ is not unital. For 
instance, one could have a vector space $W$ on which $\mc H (G)$ acts by
$f \cdot w = 0$ for all $f \in \mc H (G), w \in W$. That does not correspond to
any $G$-representation, because $G$ would have to act on $W$ by invertible linear
operators.

\begin{defn}
We say that an $\mc H (G)$-module is nondegenerate if for each $v \in V$ there
exists a $K_v \in \CO (G)$ such that $\lr{K_v} v = v$.
Let $\Rep (G)$ be the category of smooth $G$-representations and let $\Mod(\mc H (G))$
be the category of nondegenerate $\mc H (G)$-modules. The morphisms in these 
categories are the $\C$-linear maps that intertwine the action of $G$ or $\mc H (G)$.
\end{defn}

\begin{lem}\label{lem:1.5}
\enuma{
\item The categories $\Rep (G)$ and $\Mod(\mc H (G))$ are naturally equivalent, 
via \eqref{eq:1.10}.
\item Every finitely generated smooth $G$-representation has countable dimension.
}
\end{lem}
\begin{proof}
(a) Let $V \in \Mod (\mc H (G))$ and let $v \in V$. For every compact open subgroup
$K \subset K_v$ we have
\[
\lr{K} v = \lr{K} \lr{K_v} v = \lr{K_v} v = v .
\] 
This enables us to define an action $\pi$ of $G$ on $V$ by 
\begin{equation}\label{eq:1.11}
\pi (g) v = \lim_{K \in \CO (G)} \mu (K)^{-1} 1_{gK} \cdot v .
\end{equation}
As $\pi (g') v = \pi (g) v$ for all $g' \in g K_v$, $\pi$ is smooth. The functor 
\[
\Mod(\mc H (G)) \to \Rep (G) : V \mapsto \pi
\] 
is inverse to \eqref{eq:1.10}. \\
(b) Let $(\pi,V) \in \Rep (G)$ be generated by $m$ elements. By part (a) there exists a
surjection of $\mc H (G)$-modules $\mc H (G)^m \to V$. Combine that with Proposition
\ref{prop:1.1}.c.
\end{proof}

A deeper result in the direction of Lemma \ref{lem:1.5}.b is known as ``uniform 
admissibility" \cite{Ber}. It is usually stated for irreducible representations, and it
extends to finite length representations because the restriction of a smooth 
$G$-representation to a compact subgroup is always completely reducible.

\begin{thm}\label{thm:1.21} \textup{\cite{Ber}} \\
Let $K \in \CO (G)$. There exists $N(G,K) \in \N$ such that, for all $(\pi,V) \in \Rep (G)$
of finite length, $\dim V^K$ is at most $N(G,K)$ times the length of $\pi$. 
\end{thm}

For an arbitrary $G$-representation $(\rho,W)$, the space of smooth vectors is
\[
W^\infty = \{ w \in W : \lr{K} w = w \text{ for some } K \in \CO (G) \} .
\]
Then $\rho$ restricts to a smooth $G$-representation on $W^\infty$, 
called the smoothening of $(\rho,W)$. The group $G$ acts on the space $W^*$ of
linear functions $\lambda : W \to \C$ by 
\[
(\rho^\vee (g) \lambda)(w) = \lambda (\rho (g^{-1}) w) \qquad w \in W.
\]
We define $W^\vee = W^{*,\infty}$ to be the smooth part of the algebraic linear 
dual $W^*$, so the set of $\lambda : W \to \C$ that factor via
$\rho (\lr{K}) : W \to W^K$ for some $K \in \CO (G)$.

\begin{defn}
Let $(\pi,V)$ be a smooth $G$-representation. We call $(\pi^\vee,V^\vee)$ the
(smooth) contragredient of $\pi$. A matrix coefficient of $\pi$ is a function
of the form
\[
c_{\lambda,v} : G \to \C,\; c_{\lambda,v}(g) = \lambda (\pi (g)v) \qquad
\text{for some } v \in V, \lambda \in V^\vee . 
\]
\end{defn}






\begin{defn}\label{def:adm}
A finite length smooth $G$-representation $(\pi,V)$ is tempered if it extends to
a $\mc S (G)$-module by the formula
\[
\pi (f) v = \int_G f(g) \pi (g) v \textup{d} g \qquad f \in \mc S (G), v \in V.
\]
\end{defn}

By \cite[Appendix, Proposition 1]{SSZ}, Definition \ref{def:adm} is equivalent 
with the more common definition of temperedness in terms of growth 
of matrix coefficients \cite[\S III.2]{Wal}.
For representations of infinite length the matrix coefficients
do not say enough,  we need a more subtle version of temperedness.
Consider the category $\Mod (\mc S (G))$ of nondegenerate $\mc S (G)$-modules.
Since $\mc H (G) \subset \mc S (G)$, every nondegenerate $\mc S (G)$-module 
restricts to a nondegenerate $\mc H (G)$-module, which by Lemma \ref{lem:1.5}
can be regarded as a smooth $G$-representation.

\begin{defn}\label{def:tempsmooth}
The category of tempered smooth $G$-representations is the category 
$\Mod(\mc S (G))$ of nondegenerate $\mc S (G)$-modules.
\end{defn}

Thus any tempered smooth $G$-representation is by definition endowed with
an extension to a $\mc S (G)$-module. Notice that we do not put any continuity
condition on the action of $\mc S (G)$ on the module.

For a description of $C_r^* (G)$-modules, we need to look at unitary 
$G$-representations.

\begin{defn}
Let $V$ be a complex vector space with an inner product, and let $\pi$ be
a $G$-representation on $V$. We say that $\pi$ is pre-unitary if
$\lr{\pi (g)v,v'} = \lr{v, \pi (g^{-1})v'}$ for all $v,v' \in V$.
We say that $(\pi,V)$ is unitary if in addition $V$ is a Hilbert space
(so complete with respect to the norm from the inner product).
\end{defn}

For any (pre-)unitary $G$-representation $(\pi,V)$, the smoothening 
$(\pi,V^\infty)$ is a smooth pre-unitary $G$-representation. The pre-unitarity
on $V^\infty$ is equivalent to requiring that
\[
\lr{\pi (f) v, v'} = \lr{v,\pi (f^*) v'} \quad \text{for all } v,v' \in V^\infty.
\]
In other words, $\pi (f^*) = \int_G \overline{f(g^{-1})} \pi (g) \textup{d}g$
is the adjoint of $\pi (f) : V^\infty \to V^\infty$. We warn that a unitary
$G$-representation on an infinite dimensional Hilbert space is typically not smooth,
because smoothness and completeness of the Hilbert space fit badly together.

Every unitary $G$-representation $(\pi,V)$ extends to a representation of
the Banach *-algebra $L^1 (G)$. More precisely, the category of unitary
$L^1 (G)$-modules is naturally equivalent to the category of unitary 
$G$-representations. We say that $V$ is (topologically) irreducible if
$\{0\}$ and $V$ are the only $G$-invariant closed linear subspaces of $V$.
There are functors
\begin{equation}\label{eq:1.13}
\xymatrix{
\{ \text{smooth pre-unitary } G\text{-reps} \} 
\ar@/^/[rr]^{\hspace{8mm}\text{completion}} & &
\{ \text{unitary } G\text{-reps} \} \ar@/^/[ll]^{\hspace{8mm}\text{smoothening}}
} .
\end{equation}
These functors are bijective on irreducible representations, see 
\cite[\S 4.2]{SolQ} which is based on \cite{Ber}. In view of the complete 
reducibility of finite length unitary representations, it follows that \eqref{eq:1.13}
restricts to an equivalence between the subcategories of finite length objects
on both sides.

\begin{defn}\label{def:1.H} 
$\Mod_{C*}(C_r^* (G))$ is the category of those $C_r^* (G)$-modules that
are Hilbert spaces on which $C_r^* (G)$ acts unitarily.
\end{defn}
Via the natural homomorphism $L^1 (G) \to  C_r^* (G)$, $C_r^* (G)$-modules
can also be regarded as unitary $G$-representations. However, not all
unitary $G$-representations give rise to $C_r^* (G)$-modules, only those that
are weakly contained in $L^2 (G)$.

We call unitary $G$-representations that extend to $C_r^* (G)$-modules tempered.
In view of the continuity of the involved unitary operators, such an extension
is always given by the formula \eqref{eq:1.10}. 
Thus we have a natural identification
\[
\{ \text{tempered unitary } G\text{-representations} \} 
\longleftrightarrow \Mod_{C*} (C_r^* (G)) .
\]
By the next result this is compatible with our notion of temperedness for
smooth $G$-representations.

\begin{lem}\label{lem:1.6}
$\Mod_{C*} (C_r^* (G))$ consists precisely of the unitary $G$-representations 
that extend to $\mc S(G)$-modules.
\end{lem}
\begin{proof}
Since $\mc S (G)$ embeds in $C_r^* (G)$ (by Theorem \ref{thm:1.4}), every
$C_r^* (G)$-module $(\pi,V)$ is also a $\mc S (G)$-module. Hence $(\pi,V)$
is tempered in the sense of Definition \ref{def:tempsmooth}, except that
it need not be smooth as $G$-representation or nondegenerate as
$\mc S (G)$-module.

Consider a non-tempered unitary $G$-representation $\rho$. By Zorn's lemma it 
has at least one irreducible non-tempered subquotient $\rho'$. We will see
in Lemma \ref{lem:1.7} that $\rho'$ does not extend to a $\mc S (G)$-module.
Hence $\rho$ cannot extend to $\mc S (G)$ either.
\end{proof}

\subsection{Normalized parabolic induction and Jacquet restriction} \

Let $P$ be a parabolic subgroup of $G$, that is, the $F$-points of a parabolic
$F$-subgroup of $\mc G$. Let $U_P$ be its unipotent radical and let $L$ be a Levi 
factor of $P$. We briefly call $L$ a Levi subgroup of $G$. Any $(\pi,V) \in \Rep (L)$ 
may be regarded as a smooth $P$-representation via the quotient map $P \to L$. The 
smooth parabolic induction $\ind_P^G (\pi)$ is the vector space 
\[
\ind_P^G (V) = \big\{ f : G \to V \mid f(pg) = \pi (p) f(g) \; 
\text{for all } p \in p, g \in G, f \text{ is locally constant} \big\}
\]
endowed with the $G$-action by right translations. Then $\ind_P^G (\pi)$ is
a smooth $G$-representation. The functor $\ind_P^G : \Rep (L) \to \Rep (G)$
is exact, but does not preserve temperedness or pre-unitarity. To improve on
that, one involves the modular function $\delta_P$ of $P$. By \cite[\S 1.2.1]{Sil},
it can be computed as
\begin{equation}\label{eq:1.12}
\delta_P (lu) = \big| \det \big(\mr{Ad}(l) : \mr{Lie}(U_P) \to \mr{Lie}(U_P) 
\big) \big|_F \qquad  l \in L, u \in U_P . 
\end{equation}

\begin{defn}
The normalized parabolic induction of $(\pi,V) \in \Rep (L)$ is 
$I_P^G (\pi) = \ind_P^G \big( \pi \otimes \delta_P^{1/2} \big)$, on the vector space
\begin{multline*}
I_P^G (V) = \{ f : G \to V \mid f(u l g) = \delta_P^{1/2}(l) \pi (l) f(g) \; 
\text{for all } u \in U_P, l \in L, g \in G, \\
f \text{ is locally constant} \} .
\end{multline*}
\end{defn}

Let $X_\nr (L)$ be the group of unramified characters of $L$, ie. characters
whose kernel contains every compact subgroup of $L$. For instance, $\delta_P$
and $\delta_P^{1/2}$ are unramified. All the representations
$I_P^G (\pi \otimes \chi)$ with $\chi \in X_\nr (L)$ can be realized on the same
vector space, as follows. We pick a good maximal compact subgroup $K_0$ of $G$.
The Iwasawa decomposition \cite[\S 3.3.2]{Tit} says that
\begin{equation}\label{eq:1.18}
G = P K_0 = K_0 P .
\end{equation}
This implies that restriction of functions to $K_0$ defines a linear bijection
\begin{equation}\label{eq:1.19}
\ind_P^G (V_{\pi \otimes \chi}) \to \ind_{K_0 \cap P}^{K_0} (V_{\pi \otimes \chi}) .
\end{equation}
As $K_0$ is compact, its action on $\ind_{K_0 \cap P}^{K_0} (V_{\pi \otimes \chi})$
does not depend on $\chi$, so that we can identify it with $\ind_{K_0 \cap P}^{K_0} (V_\pi)$
as $K_0$-representations. We call this vector space $I_{P_0 \cap P}^{K_0}(V_\pi)$, and
we will always think of $I_P^G (\pi \otimes \chi)$ as defined on $I_{P_0 \cap P}^{K_0}(V_\pi)$.

The following properties can be found in \cite[\S IV.2.3, \S VI.1.1, \S VI.6.2, \S VII.5]{Ren}
and \cite[Lemme III.2.3]{Wal}.

\begin{thm}\label{thm:1.10}
The functor $I_P^G : \Rep (L) \to \Rep (G)$ is exact and preserves pre-unitarity,
finite length, finite generation and temperedness. 
\end{thm}

Suppose that $\pi$ is pre-unitary with respect to an inner product $\lr{\,,\,}_V$.
In the proof of Theorem \ref{thm:1.10} it is shown that $I_P^G (\pi)$ is pre-unitary
with respect to the inner product
\begin{equation}\label{eq:1.14}
\lr{f_1,f_2} = \int_{P \backslash G} \lr{f_1 (g),f_2 (g)}_V \textup{d} g 
\qquad f_1,f_2 \in I_P^G (V).
\end{equation}
For any $g \in G$ there are canonical isomorphisms of $G$-representations
\begin{equation}\label{eq:1.26}
\begin{array}{ccc}
\ind_P^G (\pi) & \to & \ind_{g P g^{-1}}^G (g \cdot \pi) \\
I_P^G (\pi) & \to & I_{g P g^{-1}}^G (g \cdot \pi) \\
 f & \mapsto & [x \mapsto f (g^{-1}x)] 
\end{array}.
\end{equation}
The version of \eqref{eq:1.26} for $I_P^G (\pi)$ follows from that for $\ind_P^G (\pi)$ 
and the equality of modular functions 
$\delta_{g P g^{-1}} = g \cdot \delta_P : x \mapsto \delta_P (g^{-1} x g)$.

The parabolic induction of an $L$-representation depends only a little on the choice of
a parabolic subgroup with Levi factor $L$:

\begin{lem}\label{lem:1.20} \textup{\cite[Lemma 1.1]{ABPS1}} \\
Let $P$ and $P'$ be parabolic subgroups with Levi factor $L$ and let $\pi \in \Rep (L)$
have finite length. Then the finite length $G$-representations $I_P^G (\pi)$ and
$I_{P'}^G (\pi)$ have the same irreducible subquotients, counted with multiplicity.
\end{lem}

Let $(\rho,W) \in \Rep (G)$. The Jacquet module of $\rho$ with respect to $P$ is
\[
W_{U_P} = W / W(U_P) = W / \mr{span} \{ \rho (u) w - w : u \in U_P, w \in W \} .
\]
This is the largest quotient of $W$ on which $U_P$ acts trivially. It is 
naturally a smooth $L$-representation, denoted $\rho_{U_P}$ and called the Jacquet
restriction (or parabolic restriction) of $\rho$. Like $\ind_P^G$, it can be
improved by a normalization.

\begin{defn}
The normalized Jacquet (or parabolic) restriction $J_P^G (\rho)$ is the vector space
$W_{U_P}$ with the $L$-action 
\[
l \cdot (w + W(U_P)) = \delta_P^{-1/2}(l) \rho (l) w + W(U_P) .
\]
\end{defn}

The functor $J_P^G$ has nice many properties, for instance exactness and preservation 
of finite length \cite[\S VI.6.4]{Ren}. However, it does not preserve temperedness
or pre-unitarity. Frobenius reciprocity provides adjointness relations
\begin{equation}\label{eq:1.15}
\begin{array}{lll}
\Hom_G (\rho, \ind_P^G (\pi)) & \cong & \Hom_M (\rho_{U_P},\pi), \\
\Hom_G (\rho, I_P^G (\pi)) & \cong & \Hom_M (J_P^G (\rho), \pi ).
\end{array}
\end{equation}
There is a much deeper second adjointness relation, due to Bernstein: 

\begin{thm}\label{thm:1.8} \textup{\cite[\S VI.9.6]{Ren}} \\
Let $\bar P$ be the parabolic subgroup with Levi factor $L$ that is opposite to
$P$, ie. $P \cap \bar P = L$. There are natural bijections
\[
\begin{array}{lll}
\Hom_G (\ind_P^G (\pi),\rho) & \cong & \Hom_M (\pi, \delta_P \otimes \rho_{U_{\bar P}}) ,\\
\Hom_G (I_P^G (\pi),\rho) & \cong & \Hom_M (\pi, J^G_{\bar P} (\rho)) .
\end{array}
\]
\end{thm}

\subsection{Description of the tempered dual} \

\noindent For any algebra $A$, we denote by $\Irr (A)$ the set of irreducible 
$A$-representations up to isomorphism. We endow it with the Jacobson topology,
whose closed subsets are
\[
V(S) =  \{ \pi \in \Irr (A) : S \subset \ker (\pi) \} 
\qquad \text{for } S \subset A.
\]
In particular we have the space $\Irr (\mc S (G))$ of irreducible nondegenerate 
$\mc S (G)$-modules. We call that the tempered dual $\Irr_\temp (G)$.

More analytically, $\Irr (L^1 (G))$ is the unitary dual $\Irr_{\mr{unit}}(G)$ 
of $G$. It contains $\Irr (C_r^* (G))$, which is sometimes called the reduced
(unitary) dual of $G$.

\begin{lem}\label{lem:1.7}
Smoothening of $G$-representations defines a bijection 
\[
\Irr (C_r^* (G)) \longrightarrow \Irr (\mc S (G)).
\]
In particular every element of $\Irr_\temp (G)$ is a pre-unitary smooth 
$G$-representation.
\end{lem}
\begin{proof}
The Plancherel formula for $G$ \cite{Wal} implies that every irreducible 
tempered $G$-representation $\pi$ belongs to the support of the Plancherel
measure for $G$. That is another way of saying that $\pi$ is weakly contained
in $L^2 (G)$ \cite[\S 18.8]{Dix}. Hence $\pi$ is pre-unitary and its Hilbert
space completion is an element $\tilde \pi$ of the reduced dual. By
\eqref{eq:1.13}, $\pi$ is the smoothening of $\tilde \pi$.
\end{proof}

We warn that not every pre-unitary irreducible smooth $G$-representation
belongs to $\Irr_\temp (G)$. For instance, the trivial $G$-representation
is unitary but not tempered (unless $\mc G = \{e\}$). 
From now on we will be a little sloppy and call a pre-unitary smooth 
$G$-representation unitary, as is customary. 

By Schur's lemma, the center $Z(G)$ acts by a character on any irreducible 
$G$-representation. If the representation is unitary, then its central
character is unitary as well. In that case $|c_{\lambda,v}| : G \to \R_{\geq 0}$,
the absolute value of the matrix coefficient $c_{\lambda,v}$, descends to a
map $G/Z(G) \to \R_{\geq 0}$. This applies more generally to any $G$-representation
on which $Z(G)$ acts by a unitary character.

\begin{defn}
Let $\pi$ be a smooth $G$-representation that admits a unitary $Z(G)$-character.
We say that $\pi$ is square-integrable modulo center if $|c_{\lambda,v}| \in 
L^2 (G/Z(G))$ for every matrix coefficient $c_{\lambda,v}$ of $\pi$. 

If $\pi$ is moreover irreducible, then it is called a discrete series representation. 
We denote the set of discrete series $G$-representations (up to isomorphism)
by $\Irr_\disc (G)$.
\end{defn}

Every discrete series $G$-representation $\pi$ is weakly contained in $L^2 (G)$, so
its completion extends to a $C_r^* (G)$-module. By Lemma \ref{lem:1.7} $\pi$ is
tempered and unitary. If $Z(G)$ is compact, then every discrete series $G$-representation
can be embedded in $L^2 (G)$, which implies that it is an isolated point in
$\Irr (C_r^* (G)) \cong \Irr_\temp (G)$ \cite[\S 18.4]{Dix}. Sometimes
discrete series are defined such that they can exist only when $Z(G)$ is compact.

\begin{ex}
The most important example of a discrete series representation is the Steinberg
representation. It exists for any $G$ and is defined as
\[
\mr{St}_G = \ind_B^G (\mr{triv}) \big/ 
\sum\nolimits_{B \subsetneq P \subset G} \ind_P^G (\mr{triv}),
\]
where $B$ is a minimal parabolic subgroup of $G$. This representation is unitary
because $\mr{triv} \in \Rep (B)$ and $\ind_B^G (\mr{triv})$ are so. See \cite{Cas} 
for a proof that $\mr{St}_G$ is irreducible and discrete series. 
\end{ex}

Recall that a character $\chi : G \to \C^\times$ is unramified if it is trivial
on every compact subgroup of $G$. We will use the notations
\begin{align*}
& X_\nr (G) = \{ \text{unramified characters of } G\}, \\
& X_\nr^u (G) = \{ \text{unitary unramified characters of } G\}.
\end{align*}
Let $G^1 \subset G$ be the subgroup generated by all compact subgroups of $G$.
It is an open normal subgroup that contains the derived group of $G$. The
quotient $G / G^1$ is a free abelian group of the same rank as the $F$-split
part of $Z(\mc G)^\circ$, say rank $d$. Although the natural map $Z(G) \to G/G^1$
need not be surjective, its image is sublattice of the same rank $d$, so that
image has finite index in $G/G^1$.

There are isomorphisms of topological groups
\begin{equation}\label{eq:1.16}
\begin{array}{lllllll}
\Hom (G/G^1, \C^\times) & = & X_\nr (G) & \cong & \Hom_\Z (\Z^d ,\C^\times) & 
\cong & (\C^\times)^d ,\\
\Hom (G/G^1, S^1) & = & X_\nr^u (G) & \cong & \Hom_\Z (\Z^d ,S^1) & \cong & (S^1)^d .
\end{array}
\end{equation}
In this way $X_\nr (G)$ acquires the structure of a complex algebraic torus 
and $X_\nr^u (G)$ is the maximal compact real subtorus of $X_\nr (G)$.

For every discrete series $G$-representation $\pi$ and every $\chi \in X_\nr^u (G)$,
$\pi \otimes \chi$ is again discrete series. The group
\[
X_\nr (G,\pi) := \{ \chi \in X_\nr (G) : \pi \otimes \chi \cong \pi \}
\]
consists of characters that are trivial on $Z(G)$, because $\pi \otimes \chi$
needs to have the same central character as $\pi$. Since $G/G^1 Z(G)$ is finite,
the group $X_\nr (G,\pi)$ is finite and contained in $X_\nr^u (G)$. The bijection
\begin{equation}\label{eq:1.17}
X_\nr^u (G) / X_\nr (G,\pi) \to X_\nr^u (G) \pi = 
\{ \pi \otimes \chi \in \Irr_\temp (G) : \chi \in X_\nr^u (G) \} 
\end{equation}
endows $X_\nr^u (G)$ we the topological structure of a compact real torus (but
$X_\nr^u (G) \pi$ does not have a multiplication). We will soon see that every set of 
the form $X_\nr^u (G) \pi$ with $\pi \in \Irr_\disc (G)$ is a connected component of 
$\Irr_\temp (G)$. These are the components of $\Irr_\temp (G)$ of minimal dimension, and 
when $Z(G)$ is compact, they constitute precisely the discrete part of $\Irr_\temp (G)$.

To describe the connected components of $\Irr_\temp (G)$ outside the discrete series,
we need the following result of Harish-Chandra.

\begin{thm}\label{thm:1.9} \textup{\cite[Proposition III.4.1]{Wal}} \\
Let $\pi \in \Irr_\temp (G)$.
\enuma{
\item There exists a parabolic subgroup $P$ with Levi factor $L$ and
a discrete series $L$-representation $\delta$, such that $\pi$ is a direct summand of 
$I_P^G (\delta)$. 
\item The pair $(L,\delta)$ is uniquely determined by $\pi$, up to
$G$-conjugation and isomorphism of $L$-representations.
}
\end{thm}

By Theorem \ref{thm:1.10} the $G$-representation $I_P^G (\delta)$ is tempered, unitary and
of finite length, so completely reducible. The following result is known as Harish-Chandra's
disjointness theorem. It says in particular that $I_P^G (\delta)$ does not depend on the choice
of a parabolic subgroup $P \subset G$ with given Levi factor $L$.

\begin{thm}\label{thm:1.2}
Let $P = L U_P$ and $Q = M U_P$ be parabolic subgroups of $G$ and let $\delta \in \Irr_\disc (L)$
and $\sigma \in \Irr_\disc (M)$. The following are equivalent:
\begin{itemize}
\item[(i)] There exists a $g \in G$ such that $M = g L g^{-1}$ and $\sigma \cong g \cdot \delta$.
\item[(ii)] $I_P^G (\delta)$ and $I_Q^G (\sigma)$ are isomorphic.
\item[(iii)] $I_P^G (\delta)$ and $I_Q^G (\sigma)$ have an irreducible subquotient in common.
\end{itemize}
\end{thm}
\begin{proof}
Clearly (ii) is stronger than (iii), which by Theorem \ref{thm:1.9}.b implies (i). By Lemma 
\ref{lem:1.20} and the complete reducibility of $I_P^G (\delta)$ and $I_Q^G (\sigma)$, 
(i) implies (ii).
\end{proof}

On the set $\{ (L,\delta) : L \text{ Levi subgroup of } G, \delta \in \Irr_\disc (L) \}$
we put the equivalence relation $\sim$ generated by $G$-conjugation, isomorphism of 
$L$-representations and tensoring $L$-representations by elements of $X_\nr^u (L)$. We write
\[
\Delta (G) = \{ (L,\delta) : L \text{ Levi subgroup of } G, \delta \in \Irr_\disc (L) \} / \sim
\]
and we denote its elements by $[L,\delta]_G$. To any $\mf d \in \Delta (G)$ we associate the set
\[
\Irr_\temp (G)_{\mf d} = \{ \pi \in \Irr_\temp (G) : \pi \text{ is a summand of }
I_P^G (\delta) \text{ for some } (L,\delta) \in \mf d \} .
\]
From Theorems \ref{thm:1.9} and \ref{thm:1.2} one deduces:

\begin{cor}\label{cor:1.11}
The sets $\Irr_\temp (G)_{\mf d}$ with $\mf d \in \Delta (G)$ are precisely the connected 
components of $\Irr_\temp (G)$, so 
\[
\Irr_\temp (G) = \bigsqcup\nolimits_{\mf d \in \Delta (G)} \Irr_\temp (G)_{\mf d} .
\]
\end{cor}

We call this the Harish-Chandra decomposition of $\Irr_\temp (G)$ and we call the sets
$\Irr_\temp (G)^{\mf d}$ Harish-Chandra components for $G$.

By the generic irreducibility of parabolically induced representations 
\cite[Th\'eo\-r\`eme 3.2]{Sau},
in the family $\{ I_P^G (\delta \otimes \chi) : \chi \in X_\nr^u (L) \}$ the irreducible
representations form an open dense subset. Hence $\Irr_\temp (G)_{[L,\delta]}$ looks like
the compact torus $X_\nr^u (L) \delta$, with some points on submanifolds of smaller 
dimensions replaced by finite packets of inseparable points.

\begin{ex}\label{ex:1.A} $G = SL_2 (F)$ \\
The set $\Irr_\disc (G)$ is countable, apart from $\mr{St}_G$ are its elements are 
supercuspidal representations.

Up to conjugation, the only proper Levi subgroup of $G$ is the diagonal torus
$T \cong F^\times$. The set $\Irr_\disc (T) / X_\nr^u (T)$ is naturally in bijection with
$\Irr (\mf o_F^\times)$, which is countable because $\mf o_F^\times$ is profinite.
The only further equivalences in $\Delta (G)$ come from conjugation by $N_G (T) =
T \cup s_\alpha T$, where $s_\alpha$ acts on $T$ by inversion. Hence the 
non-discrete part of $\Delta (G)$ is $\Irr (\mf o_F^\times)$ modulo inversion.

Let $B$ be the Borel subgroup of upper triangular matrices in $G$. 
A representation $I_B^G (\chi)$ with $\chi \in \Irr_\disc (T) = \Irr_{\mr{unit}}(T)$
is reducible if and only if $\chi$ has order two. There are three such characters (if
$p>2$): the quadratic unramified character $\chi_-$ and two ramified characters 
$\chi_r, \chi_r \chi_-$ (these fail when $p=2$). Two representations $I_B^G (\chi)$ and
$I_B^G (\chi')$ have common irreducible subquotients if and only if $\chi' \in 
\{\chi,s_\alpha (\chi) = \chi^{-1}\}$, and in that case $I_B^G (\chi) \cong I_B^G (\chi')$.

A class $[T,\chi]_G \in \Delta (G)$ with $\chi_{\mf o} := \chi |_{\mf o_F^\times}$ 
of order bigger than two gives rise to a circle of irreducible $G$-representations, 
which is equivalent to 
the analogous circle $\Irr_\temp (G)_{[T,\chi^{-1}]}$. When $\chi_{\mf o}$ has
order at most two, the isomorphisms from conjugation with $s_\alpha$ mean that 
$\Irr (G)_{[T,\chi]}$ looks a half-circle, with one or two double points. 

Altogether, the topological space $\Irr_\temp (G)$ looks like

\includegraphics[width=12cm]{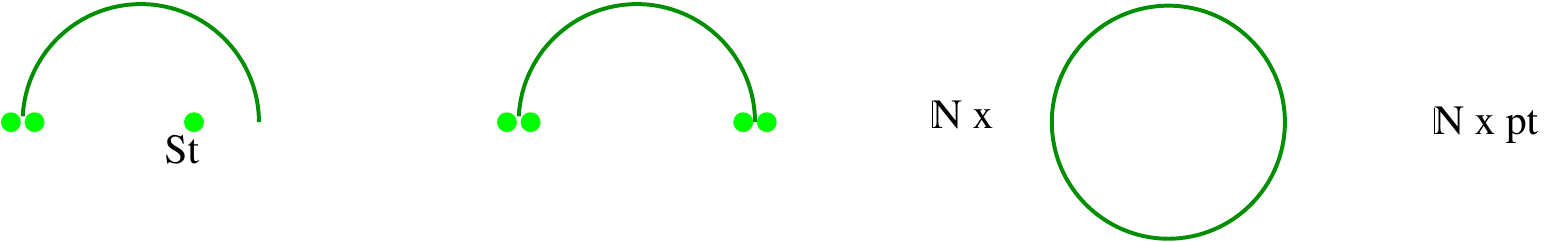}

Here "$\N$ x" just means countably many copies of something. 
We draw $\mr{St}_G$ close to $\Irr_\temp (G)_{[T,1]}$ because it is a subquotient of
$I_B^G (\chi)$ for some $\chi \in X_\nr (T)$.
\end{ex}

\subsection{Structure of $\mc S (G)$ and $C_r^* (G)$} \
\label{par:1.6}

The algebra $\mc S (G)$ can be described by its Fourier transform, which relates closely
to the Plancherel formula for $G$ \cite{Wal}. By Theorem \ref{thm:1.9}.a that means: 
$\mc S(G)$ is determined by how it acts on the representations $I_P^G (\delta)$ with 
$P = L U_P$ a parabolic subgroup of $G$ and $\delta \in \Irr_\disc (L)$. 

In view of Theorem \ref{thm:1.2}, we need only one $L$ and one $P$ for each 
conjugacy class of Levi subgroups of $G$. We fix a set of Levi subgroups $\mf{Lev}(G)$
representing the conjugacy classes of Levi subgroups, and for each $L \in \mf{Lev}(G)$
we fix one parabolic subgroup $P$ with Levi factor $L$. We have to analyse how the images
of $\mc S (G)$ in $\End_\C (I_P^G (\delta))$ and in $\End_\C (I_P^G (\delta'))$ are 
related when $I_P^G (\delta) \cong I_P^G (\delta')$.

First we look at $\chi_1 \in X_\nr^u (L,\delta)$. By definition there exists an
isomorphism of $L$-representations $\phi_{\chi_1} : \delta \to \delta \otimes \chi_1$,
unique up to scalars. Since $\delta$ and $\delta \otimes \chi_1$ are unitary representations
(on the same inner product space), we may assume that $\phi_{\chi_1}$ is unitary.
Then $\phi_{\chi_1}$ defines a unitary $L$-isomorphism $\delta \otimes \chi \to
\delta \otimes \chi_1 \otimes \chi$ for all $\chi \in X_\nr^u (L)$. By the functoriality
of $I_P^G$, we obtain a family of $G$-isomorphisms
\begin{equation}\label{eq:1.20} 
I(\chi_1,P,\delta,\chi) : I_P^G (\delta \otimes \chi) \to 
I_P^G (\delta \otimes \chi_1 \otimes \chi) \qquad \chi \in X_\nr^u (L).
\end{equation}
The formula \eqref{eq:1.14} shows that $I(\chi_1,P,\delta,\chi)$ is unitary. Via \eqref{eq:1.19},
we can consider $I(\chi_1,P,\delta,\chi)$ as an operator on the vector space 
$I_{K_0 \cap P}^P (V_\delta)$, then it does not depend on $\chi$.

Next we look at the other source of equivalences between parabolically induced representations in 
Theorem \ref{thm:1.9}: conjugation by elements $w \in N_G (L)$. Harish-Chandra initiated the study
of the integral operators
\begin{equation}\label{eq:1.21}
\begin{array}{cccc}
J(w,P,\delta,\chi) : & I_P^G (\delta \otimes \chi) & \to & 
I_P^G (w \cdot \delta \otimes w \cdot \chi) \\
 & f & \mapsto & [g \mapsto \int_{w^{-1} U_P w \cap U_{\bar P}} f(u w^{-1} g) \textup{d}u]
\end{array}.
\end{equation}
As a map from $I_{K_0 \cap P}^{K_0}(V_\delta)$ to itself, this integral depends rationally on
$\chi \in X_\nr (L)$ and converges when $|\chi|$ is large enough in a certain direction 
\cite[Th\'eor\`eme IV.1.1]{Wal}. Using \cite[Lemma 1.8]{Hei2} one can normalize 
$J(w,P,\delta,\chi)$ with $\chi \in X_\nr^u (L)$ to a unitary $G$-isomorphism
\begin{equation}\label{eq:1.22}
J' (w,P,\delta,\chi) : I_P^G (\delta \otimes \chi) \to 
I_P^G (w \cdot \delta \otimes w \cdot \chi) .
\end{equation}
(This can be done so that $J' (w,P,\delta,\chi)$ depends continuously on $\chi$, but it is only
canonical up to functions from $X_\nr^u (L)$ to $S^1$.) The definition of $\Delta (G)$ entails
that it suffices to look at $\delta$ in set of representatives for $\Irr_\disc (L) / (X_\nr^u (L) 
\rtimes N_G (L))$. In combination with the union over $L \in \mf{Lev}(L)$, that yields a set of
representatives for $\Delta (G)$. In this setting we only need the isomorphisms \eqref{eq:1.22} 
for $w \in N_G (L)$ such that $w \cdot \delta \in X_\nr^u (L) \delta$ (as subsets of $\Irr (L)$,
so up to isomorphism of representations). Moreover $w$ and $wl$ with $l \in L$ give essentially
the same operators \eqref{eq:1.21} and \eqref{eq:1.22}, so we may let $w$ run over a set of 
representatives for the finite group 
\[
W_{\mf d} = \{w \in N_G (L)/L : w \cdot \delta \in X_\nr^u (L) \delta \}.
\]
For each $w$ in this set we choose a unitary $L$-isomorphism 
$\phi_{w,\delta} : w \cdot \delta \to \delta \otimes \chi_{w,\delta}$, for some 
$\chi_{w,\delta} \in X_\nr^u (L)$. Then we can compose \eqref{eq:1.22} with 
$I_P^G (\phi_{w,\delta})$ to obtain a family of unitary $G$-isomorphisms
\begin{equation}\label{eq:1.23}
I (w,P,\delta,\chi) : I_P^G (\delta \otimes \chi) \to 
I_P^G (\delta \otimes \chi_{w,\delta} \otimes w \cdot \chi) \qquad \chi \in X_\nr^u (L) .
\end{equation}

\begin{defn}\label{def:1.E}
Write $\mf d = [L,\delta]_G$. We define $W^e_{\mf d} = W^e_{[L,\delta]}$ as the finite group of 
diffeomorphisms of $\{\delta\} \times X_\nr^u (L) \cong X_\nr^u (L)$ generated by 
$\chi \mapsto \chi_1 \otimes \chi$ with $\chi_1 \in X_\nr (L,\delta)$ and 
$\chi \mapsto \chi_{w,\delta} \otimes w\cdot \chi$ with $w \in W_{\mf d}$.
\end{defn}
Thus $W^e_{\mf d}$ is an extension of $W_{\mf d}$ by $X_\nr (L,\delta)$. 
We note that $X_\nr (L,\delta)$ is really needed here, in general the action of $W_{\mf d}$ on
$X_\nr^u (L) \delta$ does not lift to a group action on $X_\nr^u (L)$.
Combining \eqref{eq:1.20} and \eqref{eq:1.23}, we obtain operators as in \eqref{eq:1.23} for all 
$w \in W^e_{\mf d}$, canonical up to functions $X_\nr^u (L) \to S^1$. It follows that, for
all $w,w' \in W^e_{\mf d}$, the operators
\begin{equation}\label{eq:1.24}
I (w,P,\delta, w'(\chi)) \circ I(w',P,\delta,\chi) \quad \text{and} \quad
I (w w', P, \delta, \chi)
\end{equation}
differ only by a function $X_\nr^u (L) \to S^1$. This enables us to define a group action of
$W^e_{\mf d}$ on the algebra of functions from $X_\nr^u (L)$ to 
$\End_\C (I_{K_0 \cap P}^P (V_\delta))$ by
\begin{multline}\label{eq:1.25}
(w \cdot \mc F)(w(\chi)) = I(w,P,\delta,\chi) \circ \mc F (\chi) \circ I(w,P,\delta,\chi)^{-1} \\
\chi \in X_\nr^u (L), \mc F : X_\nr^u (L) \to \End_\C \big( I_{K_0 \cap P}^{K_0} (V_\delta) \big).
\end{multline}
This action stabilizes various subalgebras, for instance $C(X_\nr^u (L)) \otimes 
\End_\C^\infty (I_{K_0 \cap P}^{K_0} V_\delta)$. Here $\End_\C^\infty (W)$, for a 
$G$-representation $W$, means the smooth vectors in $\End_\C (W)$ as $G \times G$-representation. The algebra $\End_\C^\infty (W)$ has $W$ as unique irreducible
module, and it is Morita equivalent to $\C$ via the bimodules $W$ and $W^\vee$.

For each $f \in \mc S (G)$, the intertwining property of $I(w,P,\delta,\chi)$ says that
\begin{equation}\label{eq:1.27}
I(w,P,\delta,\chi) \circ I_P^G (\delta \otimes \chi)(f) = 
I_P^G (\delta \otimes w(\chi))(f) \circ I (w,P,\delta,\chi) .
\end{equation}
Hence the operators $I_P^G (\delta \otimes \chi)(f)$ are invariant for the $W^e_{\mf d}$-action
from \eqref{eq:1.25}. Further, since $\mc S (G)$ is a smooth $G \times G$-representation,
$I_P^G (\delta \otimes \chi)(f) \in \End_\C^\infty (I_{K_0 \cap P}^{K_0} (V_\delta))$. 
The example \eqref{eq:1.2} shows that we can expect that $I_P^G (\delta \otimes \chi)(f)$ 
is a smooth function of $\chi \in X_\nr^u (L)$.
We are ready to state the Plancherel isomorphism for $G$, which is due to Harish-Chandra.

\begin{thm}\label{thm:1.16} \textup{\cite{Wal}} 
\enuma{
\item The Fourier transform defines an isomorphism of topological *-algebras
\[
\mc S (G) \cong \bigoplus\nolimits_{[L,\delta]_G \in \Delta (G)} \Big( C^\infty (X_\nr^u (L)) 
\otimes \End_\C^\infty \big( I_{K_0 \cap P}^{K_0} (V_\delta) \big) \Big)^{W^e_{[L,\delta]}} . 
\]
\item For any $K \in \CO (G)$ with $K \subset K_0$, part (a) restricts to an isomorphism of
unital Fr\'echet *-algebras
\[
\mc S (G,K) \cong \bigoplus\nolimits_{[L,\delta]_G \in \Delta (G)} \Big( C^\infty (X_\nr^u (L)) 
\otimes \End_\C \big( I_{K_0 \cap P}^{K_0} (V_\delta)^K \big) \Big)^{W^e_{[L,\delta]}} . 
\]
Here each space $I_{K_0 \cap P}^{K_0} (V_\delta)^K$ has finite dimension, 
and it is nonzero for only finitely many $[L,\delta]_G \in \Delta (G)$.
}
\end{thm}

Theorem \ref{thm:1.16}.b applies arbitarily small $K \in \CO (G)$, but not to all 
$K \in \CO (G)$. We may also replace $K_0$ by another good maximal compact subgroup
of $G$, so that Theorem \ref{thm:1.16}.b applies to other $K \in \CO (G)$. We call
\[
\mc S (G)_{[L,\delta]} \cong \Big( C^\infty (X_\nr^u (L)) \otimes
\End_\C^\infty \big( I_{K_0 \cap P}^{K_0} (V_\delta) \big) \Big)^{W^e_{[L,\delta]}}
\]
a Harish-Chandra block of $\mc S (G)$. Then we can think of Theorem \ref{thm:1.16}.a as an
explicit ``Harish-Chandra decomposition"
\begin{equation}\label{eq:1.28}
\mc S (G) = \bigoplus\nolimits_{[L,\delta]_G \in \Delta (G)} \mc S (G)_{[L,\delta]} .
\end{equation}

\begin{ex}\label{ex:1.B} $G = SL_2 (G)$.\\
We extend the description of $\Irr_\temp (G)$ in Example \ref{ex:1.B} to $\mc S (G)$.
We take $K_0 = SL_2 (\mf o_F)$ and we represent $W(G,T) \cong S_2$ in $K_0$.
For $[G,\delta]_G \in \Delta (G)$ we have $W_{G,\delta} = \{1\}$ because $X_\nr (G) = \{1\}$.
Hence $\mc S (G)_{[G,\delta]} \cong \End_\C^\infty (V_\delta)$.

For every $[T,\chi] \in \Delta (G)$, the group $X_\nr (T,\chi)$ is trivial because $T = Z(T)$.
When the order of $\chi_{\mf o} = \chi |_{\mf o_F^\times}$ is larger than two, the group 
$W^e_{[T,\chi_0]}$ is trivial and
\[
\mc S (G)_{[T,\chi_0]} \cong 
C^\infty (S^1) \otimes \End_\C^\infty \big( I_{K_0 \cap B}^{K_0} \C \big) .
\]
For $\chi_{\mf o}$ of order at most two we get a summand 
\[
\mc S (G)_{[T,\chi]} \cong \big( C^\infty (S^1) \otimes 
\End_\C^\infty \big( I_{K_0 \cap B}^{K_0} \C \big) \big)^{S_2} ,
\] 
but the $S_2$-actions differ for $\chi_{\mf o} = 1$ and $\chi_{\mf o}$ ramified quadratic (this case 
does not occur when $p=2$). It becomes easier if we are satisfied with a description up to Morita 
equivalence. Namely, $\mc S (G)$ is Morita equivalent to
\[
\big( C^\infty (S^1) \otimes M_2 (\C) \big)^{S_2} \oplus \C_{St} \, \oplus \, C^\infty (S^1) 
\rtimes S_2 \, \oplus \, \bigoplus\nolimits_{j=1}^\infty C^\infty (S^1) \, \oplus \,
\bigoplus\nolimits_{i=1}^\infty \C ,
\]
where $\C_{St}$ comes from the Steinberg representation. 
This compares well with the picture of $\Irr_\temp (G)$ in Example \ref{ex:1.A}.
See \eqref{eq:2.23} for the definition of crossed products like $C^\infty (S^1) 
\rtimes S_2$.
\end{ex}

The reduced $C^*$-algebra $C_r^* (G)$ admits a description similar to Theorem \ref{thm:1.16}. By 
Lemma \ref{lem:1.6} and \eqref{eq:1.13} we have to replace all pre-unitary tempered representations
by their Hilbert space completions, then they become $C_r^* (G)$-modules. We note that, unlike
Theorem \ref{thm:1.16}.b, $C_r^* (G)$ does not have to act on the Hilbert space completion of 
$I_P^G (V_\delta)$ via one of the finite dimensional subspaces $I_P^G (V_\delta)^K$. Instead it acts
by compact operators, because the compact operators form the closure of the algebra of finite rank
operators on a Hilbert space (with respect to the operator norm). We denote the $C^*$-algebra of 
compact operators on the Hilbert space completion of an inner product space $V$ by $\mf K (V)$.
A study of the $C^*$-norm on both sides of Theorem \ref{thm:1.16} leads to:

\begin{thm}\label{thm:1.17} \textup{\cite{Ply}} \\
The Fourier transform (or equivalently Theorem \ref{thm:1.16}) induces an isomorphism of $C^*$-algebras
\[
C_r^* (G) \cong \bigoplus\nolimits_{[L,\delta]_G \in \Delta (G)} \Big( C (X_\nr^u (L)) 
\otimes \mf K \big( I_{K_0 \cap P}^{K_0} (V_\delta) \big) \Big)^{W^e_{[L,\delta]}} . 
\]
Here the tensor products and the direct sum are taken in the sense of $C^*$-algebras. For any
open subgroup $K \subset K_0$, this restricts to an isomorphism of unital $C^*$-algebras
\[
C_r^* (G,K) \cong \bigoplus\nolimits_{[L,\delta]_G \in \Delta (G)} \Big( C (X_\nr^u (L)) 
\otimes \End_\C \big( I_{K_0 \cap P}^{K_0} (V_\delta)^K \big) \Big)^{W^e_{[L,\delta]}} . 
\]
\end{thm}

Notice that $C (X_\nr^u (L))$ is the $C^*$-completion of $C^\infty (X_\nr^u (L))$, and that it
genera\-lizes the example \eqref{eq:1.13} for $GL_1 (F)$. 
For $[L,\delta ]_G \in \Delta (G)$ we write
\[
C_r^* (G)_{[L,\delta]} = \Big( C (X_\nr^u (L)) 
\otimes \mf K \big( I_{K_0 \cap P}^{K_0} (V_\delta) \big) \Big)^{W^e_{[L,\delta]}} .
\]
Like \eqref{eq:1.28}, Theorem \ref{thm:1.17} gives a Harish-Chandra decomposition
\begin{equation}\label{eq:1.33}
C_r^* (G) \cong \bigoplus\nolimits_{[L,\delta]_G \in \Delta (G)} C_r^* (G)_{[L,\delta]} .
\end{equation}
In many cases $C_r^* (G)_{[L,\delta]}$ is Morita equivalent to the crossed product
$C(X_\nr^u (L)) \rtimes W^e_{[L,\delta]}$, see \cite[Theorem 1.4]{AfAu}.

\subsection{Description of the smooth dual} \

By Lemma \ref{lem:1.5}, $\Irr (\mc H (G))$ is the space of irreducible smooth
$G$-representations, which we write simply as $\Irr (G)$. We endow it with the Jacobson
topology from $\mc H (G)$. The space $\Irr (\mc S(G)) = \Irr_\temp (G)$ of irreducible 
nondegenerate $\mc S (G)$-modules injects in $\Irr (\mc H (G))$ \cite[Appendix, Proposition 
3]{SSZ}. However, the Jacobson topology from $\mc S (G)$ on $\Irr_\temp (G)$ is finer than
the subspace topology from $\Irr (\mc H (G))$. The typical example is $S^1 \subset \C^\times$,
where $S^1$ carries the Euclidean topology and $\C^\times$ is endowed with the Jacobson
topology. 

\begin{defn}\label{def:cusp}
Let $(\pi,V) \in \Rep (G)$. We say that $\pi$ is
\begin{itemize}
\item compact modulo center if all its matrix coefficients are compactly supported 
modulo $Z(G)$,
\item cuspidal if $J_P^G (\pi) = 0$ for all proper parabolic subgroups $P$ of $G$,
\item supercuspidal if $\pi$ is irreducible and not isomorphic to any subquotient of
$I_P^G (\rho)$, for any proper parabolic subgroup $P = L U_P \subset G$ and any 
$\rho \in \Rep (L)$.
\end{itemize}
\end{defn}

From Frobenius reciprocity \eqref{eq:1.15} one sees that every super\-cus\-pidal 
$G$-repre\-sen\-tation $\pi$ is cuspidal:
\[
\Hom_L \big( J_P^G (\pi),J_P^G (\pi) \big) \cong 
\Hom_G \big( \pi, I_P^G (J_P^G (\pi)) \big) = 0 ,
\]
so $J_P^G (\pi) =  0$. Jacquet and Bernstein proved that the three notions in 
Definition \ref{def:cusp} coincide for irreducible representations: 

\begin{thm}\label{thm:1.12} \textup{\cite{Jac}, \cite[Th\'eor\`eme VI.2.1]{Ren}} \\
A smooth $G$-representation is compact modulo center if and only if it is cuspidal. 

Moreover every irreducible cuspidal $G$-representation is supercuspidal. 
\end{thm}
\begin{proof}
We provide an argument for the second claim. 
Suppose that $\pi \in \Irr (G)$ is cuspidal but not supercuspidal.
Then it is isomorphic to a subquotient of $I_P^G (\rho)$, for some smooth representation
$\rho$ of a proper Levi subgroup $L \subset G$. By \cite[Lemme VI.3.6]{Ren},
$\pi$ is also a subrepresentation of $I_P^G (\rho)$.\!
\footnote{``Supercuspidale" in \cite{Ren} is cuspidal in our terminology.} 
 Frobenius reciprocity shows that
\[
0 \neq \Hom_L (\pi, I_P^G (\rho)) \cong \Hom_G ( J_P^G (\pi), \rho) ,
\]
so $J_P^G (\pi) \neq 0$, contradicting the cuspidality of $\pi$.
\end{proof}

Motivated by Theorem \ref{thm:1.12}, we denote the set of supercuspidal $G$-representations
(up to isomorphism) by $\Irr_\cusp (G)$. 

If $Z(G)$ is compact, then every $\pi \in \Irr_\cusp (G)$ is an isolated point of
$\Irr (G)$. In general, if $\pi \in \Irr_\cusp (G)$ and $\chi \in X_\nr (G)$, then 
$\pi \otimes \chi$ is again supercuspidal. Let $X_\nr (G,\pi)$ be the stabilizer of $\pi$ 
for the action of $X_\nr (G)$ by tensoring. Like in \eqref{eq:1.17}, we have a bijection
\begin{equation}
X_\nr (G) / X_\nr (G,\pi) \to X_\nr (G) \pi = 
\{ \pi \otimes \chi \in \Irr (G) : \chi \in X_\nr (G)\}.
\end{equation}
This endows $X_\nr (G) \pi$ with the structure of an algebraic variety, namely a complex 
algebraic torus. It will turn out that $X_\nr (G) \pi$ is a connected component of 
$\Irr (G)$, of minimal dimension.

A supercuspidal representation need not be tempered or unitary, but in a sense it is not 
far off. By Lemma \ref{lem:1.7} every tempered supercuspidal representation is unitary, 
and conversely every unitary supercuspidal representation is discrete series so in particular 
tempered. In particular
\[
\Irr_{\cusp,\temp}(G) := \Irr_\cusp (G) \cap \Irr_\temp (G) \quad 
\text{is a subset of } \Irr_\disc (G).
\]
The group of smooth characters $\Hom (G, \R_{>0}^\times)$ consists of unramified characters, 
because $\R_{>0}^\times$ has no compact subgroups apart from $\{1\}$. We write
\[
X_\nr^+ (G) = \Hom (G, \R_{>0}^\times) ,
\]
a group isomorphic to $\Hom (\Z^d, \R_{>0}^\times) \cong (\R_{>0}^\times)^d$. Then $X_\nr (G)$
admits the polar decomposition
\begin{equation}\label{eq:1.29}
X_\nr (G) = X_\nr^u (G) \times X_\nr^+ (G) .
\end{equation}

\begin{lem}\label{lem:1.13} \textup{\cite[Lemma 4.2]{FlSo}}\\
Tensoring provides a bijection
\[
\Irr_{\cusp,\temp} (G) \times X_\nr^+ (G) \to \Irr_\cusp (G) .
\]
\end{lem}

By Lemma \ref{lem:1.13} and \eqref{eq:1.29}, we may identify
\begin{equation}\label{eq:1.30}
\Irr_\cusp (G) / X_\nr (G) = \Irr_{\cusp,\temp}(G) / X_\nr^u (G) .
\end{equation}
Like all irreducible tempered $G$-representations arise from discrete series representations via
parabolic induction, all irreducible smooth $G$-representations arise from supercuspidal 
representations via parabolic induction. That was shown by Jacquet \cite{Jac}, while Bernstein
proved the uniqueness in the next result.

\begin{thm}\label{thm:1.14} \textup{\cite[Th\'eor\`eme VI.5.4]{Ren}} \\
Let $\pi \in \Irr (G)$. There exists a parabolic subgroup $P = L U_P$ of $G$ and a $\sigma \in 
\Irr_\cusp (L)$ such that $\pi$ is isomorphic to a subquotient of $I_P^G (\sigma)$. Moreover
the pair $(L,\sigma)$ is uniquely determined up to $G$-conjugation. 
\end{thm}

By \eqref{eq:1.26} and Lemma \ref{lem:1.20}, every pair $(M,\tau)$ which is $G$-conjugate to 
$(L,\sigma)$ yields a  parabolically induced representation $I_Q^G (\tau)$ with exactly the
same irreducible constituents as $I_P^G (\sigma)$. Nevertheless $I_Q^G (\tau)$ need not be
isomorphic to $I_P^G (\sigma)$.

The $G$-conjugacy class of $(L,\sigma)$ in Theorem \ref{thm:1.14} is called the supercuspidal
support of $\pi$, denoted Sc$(\pi)$. This can be regarded as a map
\[
\mr{Sc} : \Irr (G) \longrightarrow 
\{ (L,\sigma) : L \subset G \text{ Levi subgroup}, \sigma \in \Irr_\cusp (L) \} / G .
\]
On the set of pairs $(L,\sigma)$ as above we put the equivalence relation generated by 
$G$-conjugation, by isomorphism of $L$-representations by and 
$(L,\sigma \otimes \chi) \sim (L,\sigma)$ for $\chi \in X_\nr (G)$. We write
\[
\mf B (G) = \{ (L,\sigma) : L \subset G \text{ Levi subgroup}, 
\sigma \in \Irr_\cusp (L) \} / \sim .
\]
By \eqref{eq:1.30}, $\mf B (G)$ is a subset of $\Delta (G)$. The elements of $\mf B (G)$, 
denoted $[L,\sigma]_G$, are called inertial equivalence classes for $G$. To any 
$\mf s = [L,\sigma]_G$ one associates the set 
\begin{align*}
\Irr (G)^{\mf s} & = \{ \pi \in \Irr (G) : \mr{Sc}(\pi) / \sim \; \in \mf s \} \\
& = \{ \pi \in \Irr (G) : \pi \text{ is a subquotient of } I_P^G (\sigma \otimes \chi)
\text{ for some } \chi \in X_\nr (L) \} .
\end{align*}
Theorem \ref{thm:1.14} is the main step towards the Bernstein decomposition of the smooth dual:

\begin{thm}\label{thm:1.15} \textup{\cite[Th\'eor\`eme VI.7.1]{Ren}} \\
The sets $\Irr (G)^{\mf s}$ with $\mf s \in \mf B (G)$ are precisely the connected components of
$\Irr (G)$, so $\Irr (G) = \bigsqcup_{\mf s \in \mf B (G)} \Irr (G)^{\mf s}$.
\end{thm}

The set $\Irr (G)^{\mf s}$ endowed with the Jacobson topology from $\Irr (\mc H (G))$ is called 
a Bernstein component of $\Irr (G)$. By \cite[Th\'eor\`eme 3.2]{Sau}, the representations 
$I_P^G (\sigma \otimes \chi)$ are irreducible for $\chi$ in a Zariski-open dense subset of
$X_\nr (G)$. Further, by Theorem \ref{thm:1.14} two representations $I_P^G (\sigma \otimes \chi)$
and $I_{P'}^G (\sigma \otimes \chi')$ have common irreducible subquotients if and only if
$\sigma \otimes \chi' \cong n \cdot (\sigma \otimes \chi)$ for some $n \in N_G (L)$. Hence
$\Irr (G)^{[L,\sigma]}$ is a possibly nonseparated algebraic variety with maximal separable
quotient $X_\nr (L)\sigma / N_G (L)$, and such that the inseparable points live only over some
lower dimensional subvarieties of $X_\nr (L)\sigma / N_G (L)$.

\begin{ex}\label{ex:1.C} $G = SL_2 (F)$\\
Every supercuspidal representation gives an isolated point in $\Irr (G)$.

Every inertial equivalence class $[T,\chi]_G$ is determined by 
$\chi_{\mf o} = \chi |_{\mf o_F^\times}$. If ord$(\chi_{\mf o}) > 2$, then all the 
representations $I_B^G (\chi')$ with $\chi' \in X_\nr (T) \chi$
are irreducible and mutually inequivalent. There are still isomorphisms 
\begin{equation}\label{eq:1.31}
I_B^G (\chi') \cong I_B^G (s \cdot \chi') = I_B^G ({\chi'}^{-1}) 
\text{ for } s \in N_G (T) \setminus T.
\end{equation}
This gives countably many Bernstein components homeomorphic to $\C^\times$, indexed by
\[
\{ \chi_{\mf o} \in \Irr (\mf o_F^\times) : \mr{ord}(\chi_{\mf o}) > 2 \} / W(G,T).
\]
When $\chi_{\mf o}$ has order two, the representations $I_B^G (\chi')$ with 
$\chi' \in X_\nr (T) \chi$ are reducible if $\chi'$ is quadratic and irreducible otherwise, 
while the only equivalences among them are \eqref{eq:1.31}. When $\chi_{\mf 0} = 1$, 
characters $\chi'$ in $X_\nr (T) \chi$ are precisely the unramified characters of $T$. 
The representation $I_B^G (\chi_-)$ with ord$(\chi_-) = 2$ is a direct sum of two 
irreducibles. If $\chi \in X_\nr (T)$ sends a generator of $T / T^1 \cong \Z$ to
$q_F$ or to $q_F^{-1}$, then $I_B^G (\chi)$ has two irreducible constituents: $\mr{St}_G$ 
and the trivial $G$-representation. All other representations $I_B^G (\chi)$ with 
$\chi \in X_\nr (T)$ are irreducible
and satisfy \eqref{eq:1.31}. Altogether, we find that $\Irr (SL_2 (F))$ is homeomorphic to

\includegraphics[width=12cm]{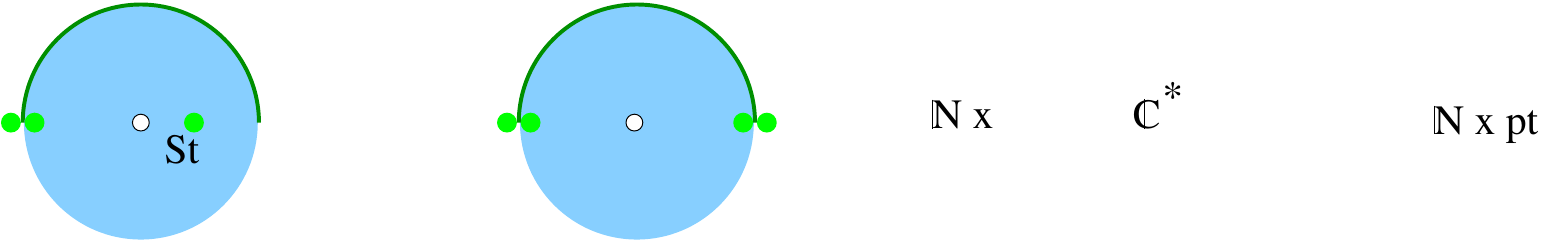}

The two discs with a hole in the 
middle represent $\C^\times / S^2$ with $S^2$ acting by inversion. It is interesting to 
compare this with the picture of $\Irr_\temp (SL_2 (F))$ from Example \ref{ex:1.A}: 

\includegraphics[width=12cm]{IrrtGa.pdf}

We see that $\Irr (SL_2 (F))$ is some sort of complexification of $\Irr_\temp (SL_2 (F))$. 
More precisely $\Irr_\temp (SL_2 (F))$ is built from circles and points, and if we replace 
each circle by $\C^\times$, we obtain a space $\Irr_\temp (SL_2 (F))_\C$ with a natural 
bijection to $\Irr (SL_2 (F))$. However, the topology of $\Irr_\temp (SL_2 (F))_\C$ is finer 
than that of $\Irr (SL_2 (F))$. For instance, $I_B^G (\mr{triv})$ and $\mr{St}_G$ are in 
different connected components of $\Irr_\temp (SL_2 (F))_\C$ but in the same Bernstein component.
\end{ex}

The relation between Examples \ref{ex:1.A} and \ref{ex:1.C} generalizes to arbitrary reductive $p$-adic
groups \cite{ABPS1}. Namely, $\Irr_\temp (G)$ is built from copies of $X_\nr^u (L)$ for Levi
subgroups $L \subset G$. If we replace each occurrence of $X_\nr^u (L)$ by its complexification
$X_\nr (L)$, then we obtain a space $\Irr_\temp (G)_\C$. By an extension of the Langlands 
classification, $\Irr_\temp (G)_\C$ maps bijectively to $\Irr (G)$. This can be stated (rather
imprecisely) as
\begin{equation}\label{eq:1.32}
\Irr (G) \text{ is canonically in bijection with a complexification of } \Irr_\temp (G).
\end{equation}
Hence, for any reductive $p$-adic group the smooth dual and the tempered dual are equally
difficult to determine.

\subsection{Rough structure of $\mc H (G)$} \
\label{par:1.8}

There are two stronger versions of Theorem \ref{thm:1.15}, which are very useful to understand
$\Rep (G)$ and $\mc H (G)$. For $\mf s \in \mf B (G)$, we define 
\[
\Rep (G)^{\mf s} = \{ \pi \in \Rep (G) : \text{ every irreducible subquotient of } \pi 
\text{ lies in } \Irr (G)^{\mf s} \} . 
\]
In other words, $\Rep (G)^{\mf s}$ is the full subcategory of $\Rep (G)$ generated by 
$\Irr (G)^{\mf s}$. Further, we define a two-sided ideal of $\mc H (G)$ by
\[
\mc H (G)^{\mf s} = \{ f \in \mc H (G) : 
\pi (h) = 0 \;\text{for all}\;´ \pi \in \Irr (G) \setminus \Irr (G)^{\mf s} \} .
\]
The next theorem is known as the Bernstein decomposition.

\begin{thm}\label{thm:1.18} \textup{\cite{BeDe,Ren}} \\
Each $\Rep (G)^{\mf s}$ is a block of $\Rep (G)$, that is, an indecomposable direct
summand of $\Rep (G)$. The category $\Rep (G)$ admits an orthogonal decomposition
\[
\Rep (G) = \prod\nolimits_{\mf s \in \mf B (G)} \Rep (G)^{\mf s}.
\]
Each two-sided ideal $\mc H (G)^{\mf s}$ of $\mc H (G)$ is indecomposable and 
$\mc H (G) = \bigoplus_{\mf s \in \mf B (G)} \mc H (G)^{\mf s}$.
\end{thm}

By Lemma \ref{lem:1.5}.a, $\Rep (G)^{\mf s}$ can be identified with $\Mod (\mc H (G)^{\mf s})$.
These subcategories are called Bernstein blocks of $\Rep (G)$ or for $G$.

To understand $\mc H (G)$, it suffices to classify $\mf B (G)$ and to understand each 
$\mc H (G)^{\mf s}$. However,
in spite of Theorem \ref{thm:1.18} and \eqref{eq:1.32}, the structure of $\Rep (G)$ is
substantially more complicated than that of $\Mod (\mc S (G))$. The main reason is that
the building blocks $I_P^G (\delta)$ for $\Mod(\mc S (G))$ are unitary, while for $\mc H (G)$
the building blocks $I_P^G (\sigma)$ with $\sigma \in \Irr_\cusp (L)$ need not be 
completely reducible.

Like the Plancherel isomorphism for $\mc S (G)$, one would like to understand $\mc H (G)$
by the Fourier transform, which in this case means its action on the representations
$I_P^G (\sigma)$ with $\sigma \in \Irr_\cusp (L)$. We proceed as in Paragraph \ref{par:1.6}.
By Lemma \ref{lem:1.13} we may assume that $\sigma \in \Irr_{\cusp,\temp}(L)$, so that it is
a discrete series representation. 

For $\chi_1 \in X_\nr (L,\sigma)$ we have the operator
\[
I(\chi_1,P,\sigma,\chi) : I_P^G (\sigma \otimes \chi) \to I_P^G (\sigma \otimes \chi_1 \otimes \chi) .
\]
For $w \in N_G (L)$ there is the intertwining operator
\[
J (w,P,\sigma,\chi) : I_P^G (\sigma \otimes \chi) \to I_P^G (w \cdot \sigma \otimes w \cdot \chi)
\]
from \eqref{eq:1.21}, which depends rationally $\chi \in X_\nr (L)$ once we identify the
underlying vector spaces with $I_{K_0 \cap P}^{K_0} (V_\sigma)$  as in \eqref{eq:1.19}. 
The operator $J(w,P,\sigma,\chi)$ may have poles and zeros at some nonunitary 
$\chi \in X_\nr (L)$, and that complicates things. Even when we 
normalize it to $J'(w,P,\sigma,\chi)$ as in \eqref{eq:1.22}, it need not be injective 
for some $\chi \in X_\nr (L) \sigma \setminus X_\nr^u (L) \sigma$. This means that the 
action \eqref{eq:1.25} of $W_{[L,\sigma]}$ on $C(X_\nr^u (L)) \otimes \End_\C^\infty 
\big( I_{K_0 \cap P}^{K_0} V_\sigma \big)$ does not stabilize the subalgebra 
$\mc O (X_\nr (L)) \otimes \End_\C^\infty \big( I_{K_0 \cap P}^{K_0} V_\sigma \big)$.
The (weaker) analogue of Theorem \ref{thm:1.16} for $\mc H (G)$ is:

\begin{thm}\label{thm:1.19} \textup{\cite{Hei1}} 
\enuma{
\item The Fourier transform (the action of $\mc H (G)$ on the representations
$I_P^G (\sigma')$ with $\sigma' \in \Irr_\cusp (L)$) determines an isomorphism of
*-algebras $\mc H (G) \cong$
\[
\bigoplus_{[L,\sigma]_G \in \mf B (G)} \!\!\! \Big( C^\infty (X_\nr^u (L)) 
\otimes \End_\C^\infty \big( I_{K_0 \cap P}^{K_0} (V_\sigma) \big) \Big)^{W^e_{[L,\sigma]}}
\cap \mc O (X_\nr (L)) \otimes \End_\C^\infty \big( I_{K_0 \cap P}^{K_0} (V_\sigma) \big) .
\]
\item For every $K \in \CO (G)$ with $K \subset K_0$, part (a) restricts to an
isomorphism
\begin{align*}
\mc H (G,K) \cong \bigoplus\nolimits_{[L,\sigma]_G \in \mf B (G)} & \Big( C^\infty 
(X_\nr^u (L)) \otimes \End_\C \big( I_{K_0 \cap P}^{K_0} (V_\sigma)^K \big) 
\Big)^{W^e_{[L,\sigma]}} \\ & \cap \; \mc O (X_\nr (L)) \otimes \End_\C 
\big( I_{K_0 \cap P}^{K_0} (V_\sigma)^K \big) .
\end{align*}
Here $I_{K_0 \cap P}^{K_0} (V_\sigma)^K$ has finite dimension, and it is nonzero
for only finitely many $[L,\sigma]_G \in \mf B (G)$.
}
\end{thm} 

Notice that in Theorem \ref{thm:1.19} the regular functions on $X_\nr (L)$ appear
in the same way as $\mc H (F^\times)^{\mf o_F^\times} \cong \mc O (\C^\times)$ 
in \eqref{eq:1.1}.

\begin{ex}\label{ex:1.D}
Every $\sigma \in \Irr_\cusp (G)$ gives a summand 
\[
\mc H (G)^{[G,\sigma]} = \End_\C^\infty (V_\sigma) 
\cong \bigcup\nolimits_{n \in \N} M_n (\C)
\]
of $\mc H (G)$. For an inertial equivalence class $\mf s = [T,\chi]_G$ with $\chi_{\mf o} =
\chi |_{\mf o_F^\times}$ of order bigger than two, $W_{[T,\chi]}^e$ is trivial and 
\[
\mc H (G)^{\mf s} \cong \mc O (X_\nr (T)) \otimes \End_\C^\infty 
\big( I_{B \cap K_0}^{K_0} (\C_\chi) \big) ,
\]
where $K_0 = SL_2 (\mf o_F)$. Let $\chi_r \in \Irr (T)$ be ramified quadratic,
and represent $S_2 \cong W(G,T)$ in $K_0$, then
\[
\mc H (G)^{[T,\chi_r]} \cong \big( \mc O (X_\nr (T) \otimes \End_\C \big( \C[S_2] 
\big) \big)^{S_2} \otimes \End_\C^\infty \big( I_{B \cap K_0}^{K_0} (\C_{\chi_r})^{S_2} \big) .
\]
Finally, for the trivial character of $T$ there is an affine Hecke algebra 
$\mc H_{\mr{aff}}$ such that
\[
\mc H (G)^{[T,1]} \cong \mc H_{\mr{aff}} \otimes  
\End_\C^\infty \big( I_{B \cap K_0}^{K_0} (\mr{triv}_T)^{S_2} \big) .
\]
We find that $\mc H (G)$ is Morita equivalent to
\[
\hspace{4cm} \mc H_{\mr{aff}} \, \oplus \, \mc O (X_\nr (T)) \rtimes S_2 \, \oplus \, 
\bigoplus\nolimits_{j=1}^\infty \mc O (\C^\times) \, \oplus \, 
\bigoplus\nolimits_{i=1}^\infty \C .
\]
This can be compared with the description of $\mc S (G)$ up to Morita 
equivalence, from Example \ref{ex:1.B}:
\[
\Big( \big( C^\infty (S^1) \otimes M_2 (\C) \big)^{S_2} \oplus \C_{St} \Big) \oplus \, 
C^\infty (S^1) \rtimes S_2 \, \oplus \, \bigoplus\nolimits_{j=1}^\infty C^\infty (S^1) 
\, \oplus \, \bigoplus\nolimits_{i=1}^\infty \C .
\]
\end{ex}

\subsection{Decompositions of $\mc S (G)$ and $C_r^* (G)$} \

The Bernstein decomposition of $\mc H (G)$ induces decompositions of $\mc S (G)$ and 
$C_r^* (G)$. Namely, let $\mc S (G)^{\mf s}$ be the closure of $\mc H (G)^{\mf s}$ in 
$\mc S (G)$ and let $C_r^* (G)^{\mf s}$ be the closure of $\mc H (G)^{\mf s}$ in
$C_r^* (G)$. Theorem \ref{thm:1.18} and the density of $\mc H (G)$ in $\mc S (G)$ and
in $C_r^* (G)$ imply the Bernstein decompositions
\begin{equation}\label{eq:1.34}
\begin{array}{lll}
\mc S (G) & = & \bigoplus\nolimits_{\mf s \in \mf B (G)} \mc S (G)^{\mf s}, \\
C_r^* (G) & = & \bigoplus\nolimits_{\mf s \in \mf B (G)} C_r^* (G)^{\mf s},
\end{array}
\end{equation}
where for $C_r^* (G)$ the direct sum is meant in the sense of Banach algebras.
The two-sided ideals $\mc S (G)^{\mf s} \subset \mc S (G)$ and 
$C_r^* (G)^{\mf s} \subset C_r^* (G)$ are often decomposable. Hence the decompositions
\eqref{eq:1.34} are coarser than the Harish-Chandra decompositions in \eqref{eq:1.28}
and \eqref{eq:1.33}, which are block decompositions. For $\mf d \in \Delta (G)$ and 
$\mf s \in \mf B (G)$, we write 
\[
\mf d \in \Delta (G,\mf s) \quad \text{when} \quad 
\Irr_\temp (G)_{\mf d} \subset \Irr (G)^{\mf s}. 
\]
We note that $\mf s$ itself is an element of $\Delta (G,\mf s)$, in fact the only
element that can be represented by a supercuspidal representation of a Levi subgroup
of $G$. From \eqref{eq:1.34} and Theorems \ref{thm:1.16} and \ref{thm:1.17} we obtain
\begin{equation}\label{eq:1.35}
\begin{array}{lll}
\mc S (G)^{\mf s} & = & \bigoplus\nolimits_{\mf d \in \Delta (G,\mf s)} 
\mc S (G)_{\mf d}, \\
C_r^* (G)^{\mf s} & = & \bigoplus\nolimits_{\mf d \in \Delta (G,\mf s)} C_r^* (G)_{\mf d}.
\end{array}
\end{equation}
The finite groups indexing the intertwining operators in Theorems \ref{thm:1.16} and
\ref{thm:1.19} for $\Rep (G)^{\mf s}$ are related:

\begin{lem}\label{lem:1.23}
For $\mf d \in \Delta (G,\mf s)$, the group $W_{\mf d}^e$ is a subquotient of $W_{\mf s}^e$.
\end{lem}
\begin{proof}
We write $\mf d = [M,\delta]_G$ and $\mf s = [L,\sigma]_G$, where $M \supset L$. Recall
from Definition \ref{def:1.E} that 
\[
W_{[M,\delta}^e / X_\nr (M,\delta) \cong W_{[M,\delta]} \quad \text{and} \quad
W_{[L,\sigma]}^e / X_\nr (L,\sigma) \cong W_{[L,\sigma]} . 
\] 
The group $W_{[M,\delta]} \subset N_G (M) / M$ stabilizes $\Irr_\temp (G)_{[M,\delta]} 
\subset \Irr (G)^{[L,\sigma]}$, so stabilizes $\Rep (G)^{[L,\sigma]}$. Hence $W_{[M,\delta]}$
can be represented in $\mr{Stab}_{N_G (M,L)/L} (X_\nr (L) \sigma)$, which is a 
subgroup of $W_{[L,\sigma]}$. 

Suppose that $(L,\sigma \otimes \chi_L)$ represents Sc$(\delta)$ and that $\chi_M \in 
X_\nr (M,\delta)$. Then $(L,\sigma \otimes \chi_L (\chi_M \chi)|_L)$ represents
Sc$(\delta \otimes \chi)$, for all $\chi \in X_\nr^u (M)$. By the uniqueness of
cuspidal supports up to $G$-conjugation, there exists $w \in W_{[L,\sigma]}^e$ such that
$w (\sigma \otimes \chi_L \chi|_L) = \sigma \otimes \chi_L (\chi_M \chi) |_L$  for all 
$\chi_M \in X_\nr^u (M)$. 

Hence the action of any element of $W_{[M,\delta]}^e$ on
$X_\nr (M) \delta$ arises from an element of $W_{[L,\sigma]}^e$. It follows that
$W_{[M,\delta]}^e$ is isomorphic to the quotient of $\mr{Stab}_{W_{[L,\sigma]}^e}
(X_\nr^u (M) \delta)$ by the elements that act trivially on $X_\nr (M) \delta$.
\end{proof}

For $K \in \CO (G)$ we put
\begin{equation}\label{eq:1.50}
\begin{array}{lllll}
\mc H (G,K)^{\mf s} & = & \langle K \rangle \mc H (G)^{\mf s} \langle K \rangle & = &
\mc H (G,K) \cap \mc H (G)^{\mf s},\\
\mc S (G,K)^{\mf s} & = & \langle K \rangle \mc S (G)^{\mf s} \langle K \rangle & = &
\mc S (G,K) \cap \mc S (G)^{\mf s},\\
C_r^* (G,K)^{\mf s} & = & \langle K \rangle C_r^* (G)^{\mf s} \langle K \rangle & = &
C_r^* (G,K) \cap C_r^* (G)^{\mf s}.
\end{array}
\end{equation}
Then $\mc S (G,K)^{\mf s}$ and $C_r^* (G)^{\mf s}$ are closures of $\mc H (G,K)^{\mf s}$,
and they are generated by  $\mc H (G,K)^{\mf s}$ as two-sided ideals in, respectively,
$\mc S (G,K)$ and $C_r^* (G,K)$.
To understand $\mc H (G)^{\mf s}$, it suffices to consider $\mc H (G,K)^{\mf s}$ for
$K \in \CO (G)$ in a countably family:

\begin{thm}\label{thm:1.20}
\textup{\cite[\S 2.2 and Corollaire 3.9]{BeDe}} \\
There exists a decreasing sequence $(K_n )_{n=1}^\infty$ of compact open subgroups of $G$
with the following properties:
\begin{itemize}
\item $\bigcap_{n=1}^\infty K_n = \{1\}$
\item Each $K_n$ is a normal subgroup of the good maximal compact subgroup $K_0$.
\item For each $\mf s \in \mf B (G)$ and each $n \in \Z_{>0}$, $\mc H (G,K_n)^{\mf s}$ 
is either 0 or Morita equivalent to $\mc H (G)^{\mf s}$. 
\item The bimodules for such a Morita equivalence are 
$\langle K_n \rangle \mc H (G)^{\mf s}$ and $\mc H (G)^{\mf s} \langle K_n \rangle$.
\end{itemize}
\end{thm}

From Theorems \ref{thm:1.19} and \ref{thm:1.20} one concludes:

\begin{cor}\label{cor:1.21}
There exists a unique finite subset $\mf B (G,K_n) \subset \mf B (G)$ such that
\[
\mc H (G,K_n) \quad \text{equals} \quad 
\bigoplus\nolimits_{\mf s \in \mf B (G,K_n)} \mc H (G,K_n )^{\mf s}
\]
and is Morita equivalent to $\bigoplus_{\mf s \in \mf B (G,K_n)} \mc H (G)^{\mf s}$.

The sequence of sets $(\mf B (G,K_n) )_{n=1}^\infty$ increases
and its union is $\mf B (G)$.
\end{cor}

There are versions of Theorem \ref{thm:1.20} and Corollary \ref{cor:1.21} for 
$\mc S (G)$ and $C_r^* (G)$:

\begin{prop}\label{prop:1.22}
Let $K_n$ be as in Theorem \ref{thm:1.20}.
\enuma{
\item For any $\mf s \in \mf B (G)$, $\mc S (G,K_n)^{\mf s}$ is either 0 or 
Morita equivalent to $\mc S (G)^{\mf s}$. In the latter case, the Morita
bimodules are $\langle K_n \rangle \mc S (G)^{\mf s}$ and $\mc S (G)^{\mf s} 
\langle K_n \rangle$.
\item The algebra $\mc S (G,K_n )$ equals $\bigoplus\nolimits_{\mf s \in \mf B (G,K_n)} 
\mc S (G,K_n )^{\mf s}$ and is Morita equivalent to 
$\bigoplus_{\mf s \in \mf B (G, K_n)} \mc S (G)^{\mf s}$.
\item Parts (a) and (b) also hold for $C_r^* (G,K_n)$.
}
\end{prop}
\begin{proof}
(a) If $\mc H (G,K_n )^{\mf s} = 0$, then also $\mc S (G,K_n)^{\mf s} = 0$. 
Therefore we may assume that $\mc S (G,K_n)^{\mf s}$ is nonzero. 
We consider $\langle K_n \rangle \mc S (G)^{\mf s}$ and $\mc S (G)^{\mf s} 
\langle K_n \rangle$, which are bimodules for $\mc S (G)^{\mf s}$ and
$\mc S (G,K_n )^{\mf s}$. Multiplication provides an isomorphism of 
$\mc S (G,K_n )^{\mf s}$-bimodules
\begin{equation}\label{eq:1.36}
\langle K_n \rangle \mc S (G)^{\mf s} \otimes_{\mc S (G)^{\mf s}} \mc S (G)^{\mf s} 
\langle K_n \rangle \to \mc S (G,K_n )^{\mf s} ,
\end{equation}
and an isomorphism of $\mc S (G)^{\mf s}$-bimodules
\begin{equation}\label{eq:1.37}
\mc S (G)^{\mf s} \langle K_n \rangle \otimes_{\mc S (G,K_n )^{\mf s}} \langle K_n 
\rangle \mc S (G)^{\mf s} \to \mc S (G)^{\mf s} \langle K_n \rangle \mc S (G)^{\mf s}.
\end{equation}
The right hand side of \eqref{eq:1.37} is an $\mc S (G)^{\mf s}$-sub-bimodule
of $\mc S (G)^{\mf s}$, and by Theorem \ref{thm:1.20} it contains $\mc H (G)^{\mf s}$.
Since $\mc H (G)^{\mf s}$ generates $\mc S (G)^{\mf s}$ as an ideal, we deduce that
\begin{equation}\label{eq:1.38}
\mc S (G)^{\mf s} \langle K_n \rangle \mc S (G)^{\mf s} = \mc S (G)^{\mf s} .
\end{equation} 
From \eqref{eq:1.36}, \eqref{eq:1.37} and \eqref{eq:1.38} we see that $\langle K_n 
\rangle \mc S (G)^{\mf s}$ and $\mc S (G)^{\mf s} \langle K_n \rangle$ implement
a Morita equivalence between $\mc S (G)^{\mf s}$ and $\mc S (G,K_n )^{\mf s}$.\\
(b) This follows from Corollary \ref{cor:1.21} and part (a).\\
(c) This can be shown in the same way as parts (a) and (b).
\end{proof}

\section{Twisted graded Hecke algebras} \label{sec:2}

In this section we survey some algebras which will play an important role
in the analysis of the Hecke algebra of a reductive $p$-adic group.

\subsection{Twisted crossed products}\label{par:tcp}  \

Let $\Gamma$ be a finite group. A 2-cocycle for $\Gamma$ is a map
$\natural : \Gamma \times \Gamma \to \C^\times$ such that
\begin{equation}\label{eq:2.1}
\natural (\gamma_1 \gamma_2, \gamma_3) \natural (\gamma_1, \gamma_2) =
\natural (\gamma_1, \gamma_2 \gamma_3) \natural (\gamma_2, \gamma_3)
\qquad \text{for all}\; \gamma_1, \gamma_2, \gamma_3 \in \Gamma.
\end{equation}
To these data one associates the twisted group algebra $\C[\Gamma,\natural]$,
which has a $\C$-basis $\{ T_\gamma : \gamma \in \Gamma \}$ and multiplication rules
\[
T_{\gamma_1} T_{\gamma_2} = \natural (\gamma_1, \gamma_2) T_{\gamma_1 \gamma_2} 
\qquad \text{for all}\; \gamma_1, \gamma_2 \in \Gamma.
\]
The condition \eqref{eq:2.1} means precisely that $\C[\Gamma,\natural]$ is an
associative algebra. Repla\-cing $T_e$ by $\natural (e,e)^{-1} T_e$, we can achieve
that $T_e \cdot T_e = T_e$ (at the cost of modifying $\natural$). Therefore we 
may and will always assume that $\natural (e,e) = 1$. Then \eqref{eq:2.1} for
$(\gamma_1,\gamma_2,\gamma_3) = (e,e,\gamma)$ shows that $\natural (e,\gamma) = 1$
while \eqref{eq:2.1} for $(\gamma_1,\gamma_2,\gamma_3) = (\gamma,e,e)$ shows that 
$\natural (\gamma,e) = 1$. In other words, the condition $\natural (e,e)= 1$ implies
that $T_e$ is the unit element of $\C [\Gamma, \natural]$. 

For any function $f : \Gamma \to \C^\times$, one may pass to new basis elements
$T'_\gamma = f(\gamma) T_\gamma$. In those terms, the multiplication rules read
\begin{equation}\label{eq:2.14}
T'_{\gamma_1} T'_{\gamma_2} = f(\gamma_1) f(\gamma_2) f(\gamma_1 \gamma_2)^{-1}
\natural (\gamma_1, \gamma_2) T'_{\gamma_1 \gamma_2} .
\end{equation}
The map 
\[
b(f) : (\gamma_1, \gamma_2) \mapsto f(\gamma_1) f(\gamma_2) f(\gamma_1 \gamma_2)^{-1}
\]
is called the coboundary of $f$. Hence the algebra $\C [\Gamma, \natural]$ depends, 
up to rescaling, only on $\natural$ modulo coboundaries, that is, on the image of
$\natural$ in $H^2 (\Gamma, \C^\times)$. This construction yields a map from the
second group cohomology $H^2 (\Gamma, \C^\times)$ to twisted versions of $\C [\Gamma]$
up to isomorphism.

There exists a finite central extension $\Gamma^*$ of $\Gamma$, such that the 
inflation to $\Gamma^* \times \Gamma^*$ of any 2-cocycle $\natural$ for $\Gamma$
represents the trivial class in $H^2 (\Gamma^* ,\C^\times)$. It is known as the
Schur extension or multiplier of $\Gamma$ \cite[\S 53]{CuRe}. Let $Z^*$ be the kernel 
of $\Gamma^* \to \Gamma$. Then $\natural$ determines a character $c_\natural$ of
$Z^*$, with central idempotent $e_\natural \in \C [Z^*] \subset \C [\Gamma^*]$,
such that
\begin{equation}\label{eq:2.2}
\C [\Gamma, \natural] \cong e_\natural \C [\Gamma^*] .
\end{equation}
From \eqref{eq:2.2} one can recover $\natural$ (up to some coboundary), as 
follows. Pick representatives $\{ \gamma^* : \gamma \in \Gamma \}$ for $\Gamma$
in $\Gamma^*$. Then $\{e_\natural \gamma^* : \gamma \in \Gamma \}$ is a $\C$-basis
of $e_\natural \C [\Gamma^*]$, and $\natural$ can be defined by the formula
\[
e_\natural \gamma_1^* \cdot e_\natural \gamma_2^* = \natural (\gamma_1, \gamma_2)
e_\natural (\gamma_1 \gamma_2)^* .
\]
Moreover $\C [\Gamma^*] = \bigoplus_{c_\natural \in \Irr (Z^*)} 
e_\natural \C [\Gamma^*]$, so each $\C [\Gamma, \natural]$ is isomorphic to
a direct summand of $\C [\Gamma^*]$. As $\C [\Gamma^*]$ is semisimple, so is
$\C [\Gamma, \natural]$.

\begin{ex}
For $\Gamma = S_2 \times S_2$, the Schur multiplier is $Q_8$, the quaternion
group of order eight. The group $Q_8$ has four irreducible representations of
dimension one, and one of dimension two, so $\C [Q_8] \cong \C^4 \oplus M_2 (\C)$.

We have $Z^* = Z(Q_8) = \{ \pm 1\}$, so there are precisely two inequivalent
twisted group algebras of $\Gamma$. The first comes from $c_\natural = 1
\in \Irr (Z^*)$, it is just $\C [\Gamma]$. As $\C[\Gamma] \cong \C^4$, this
corresponds to the direct summand $\C^4$ of $\C[Q_8]$.  

The second comes from $c_\natural = \mr{sign} \in \Irr (Z^*)$. It corresponds to
the remaining direct summand of $\C [Q_8]$, so $\C [S_2 \times S_2, \natural]
\cong M_2 (\C)$ for any 2-cocycle $\natural$ whose image in $H^2 (S_2 \times S_2,
\C^\times)$ is nontrivial.
\end{ex}

For $g \in \Gamma$ we introduce the map 
\begin{equation}\label{eq:2.C}
\natural^g : \Gamma \to \C^\times, \qquad \natural^g (\gamma) =
T_\gamma T_g T_\gamma^{-1} T_{\gamma g^{-1} \gamma^{-1}}^{-1} .
\end{equation}
One checks that the restriction $\natural^g |_{Z_\Gamma (g)}$ is a character of
$Z_\Gamma (g)$. These characters measure the difference between $\C [\Gamma,\natural]$
and $\C [\Gamma]$. For instance, they can be used to count the number of
irreducible representations of $\C [\Gamma,\natural]$:

\begin{lem}\label{lem:2.1} \textup{\cite[Lemma 1.1]{SolTwist}} \\
The cardinality of $\{ g \in \Gamma : \natural^g |_{Z_\Gamma (g)} = 1 \} / 
\Gamma \text{-conjugacy} \}$ equals $|\Irr (\C [\Gamma,\natural])|$.

More precisely, for every conjugacy class $C$ with the above property, define 
a trace $\nu_C$ on $\C [\Gamma,\natural]$ by $\nu_C (T_g) = 1$ if $g \in C$ and
$\nu_C (T_g) = 0$ otherwise. Then the $\nu_C$ form a basis of the space of all
traces on the semisimple algebra $\C [\Gamma,\natural]$.
\end{lem}

Notice that Lemma \ref{lem:2.1} generalizes the well-known equality between the 
number of conjugacy classes and the number irreducible representations (over $\C$)
of a finite group.

Suppose now that $\Gamma$ acts on a $\C$-algebra $A$, by automorphisms. Then we can 
form the crossed product algebra $A \rtimes \Gamma$, which is
$A \otimes_\C \C [\Gamma]$ which multiplication rules
\begin{equation}\label{eq:2.23}
(a_1 \otimes \gamma_1) (a_2 \otimes \gamma_2) = a_1 \gamma_1 (a_2) 
\otimes \gamma_1 \gamma_2 \qquad \text{for all}\; a_i \in A, \gamma_i \in \Gamma .
\end{equation}
More generally, for any 2-cocycle $\natural$ of $\Gamma$ there is a twisted
crossed product $A \rtimes \C [\Gamma,\natural]$. It it the same vector space
$A \otimes_\C \C [\Gamma]$, but now the multiplication is given by
\[
(a_1 \otimes T_{\gamma_1}) (a_2 \otimes T_{\gamma_2}) = a_1 \gamma_1 (a_2) 
\otimes \natural ( \gamma_1, \gamma_2) T_{\gamma_1 \gamma_2} 
\qquad \text{for all}\; a_i \in A, \gamma_i \in \Gamma .
\]
This is an associative algebra that contains $A$ and $\C [\Gamma,\natural]$
as subalgebras (for the latter we need to assume that $A$ is unital). 
There is a close relation between $\Irr (A)$ and $\Irr (A \rtimes \C [\Gamma,
\natural])$, known as Clifford theory. We refer to \cite[p. 24]{RaRa} for
the cases with $\natural = 1$, and to \cite[\S 1]{AMS1} for how to handle
nontrivial $\natural$.\\

In the remainder of this paragraph we focus on the cases where $A$ is commutative 
and unital. More concretely, consider the algebra $A = \mc O (X)$ of regular 
functions on some complex affine variety $X$. The upcoming arguments also work 
for smooth functions on a closed manifold and for continuous functions on a 
compact Hausdorff space, the main point is that $A$ can be regarded as an algebra
of $\C$-valued functions on $\Irr (A)$. 

So, we assume that $\Gamma$ acts on $X$ by homeomorphisms, and we form the
twisted crossed product
\begin{equation}\label{eq:2.B}
B = A \rtimes \C[\Gamma,\natural] = \mc O (X) \rtimes \C [\Gamma,\natural] .
\end{equation}
The crossed product $A \rtimes \Gamma$ is self-opposite, so its left and right
module categories are equivalent. The analogue for the twisted crossed product $B$ is
more subtle. The opposite algebra $B^{op}$ can be studied via the isomorphism
\begin{equation}\label{eq:2.17}
\begin{array}{cccc}
B^{op} = (A \rtimes \C [\Gamma,\natural])^{op} & \isom & 
A \rtimes \C [\Gamma,\natural^{-1}] \\
a T_\gamma & \mapsto & T_\gamma^{-1} a & a \in A, \gamma \in \Gamma .
\end{array}
\end{equation}
This enables us to identify right $B$-modules with left modules over 
$A \rtimes \C [\Gamma,\natural^{-1}]$.
We work out the classification of the irreducible left modules of the
algebra $B$, a simple case of Clifford theory. 

\begin{lem}\label{lem:2.2}
\enuma{
\item $B$ has finite rank as a module over its centre $Z(B)$. 
\item All irreducible $B$-representations have finite dimension.
}
\end{lem}
\begin{proof}
(a) Firstly, $\mc O (X)^\Gamma = \mc O (X/\Gamma)$ is a central subalgebra of $B$.
Since $\mc O (X)$ is finitely generated and integral over $\mc O (X)^\Gamma$,
$\mc O (X)$ has finite rank as $\mc O (X)^\Gamma$-module. As $Z(B) \supset 
\mc O(X)^\Gamma$, $B$ also has finite rank over $Z(B)$.\\
(b) Let $V$ be an irreducible $B$-module. As $B$ has countable dimension, so
has $V$. Hence Schur's lemma applies, and says that $Z(B)$ acts by scalars on $V$.
With part (a) it follows that the image of $B$ in $\End_\C (V)$ has finite dimension.
Therefore $B \cdot v$ has finite dimension for any $v \in V$. But by the 
irreducibility $B \cdot v = V$ whenever $v \neq 0$. 
\end{proof}

Lemma \ref{lem:2.2} implies that the restriction of any irreducible $B$-module
$(\pi,V)$ to $\mc O (X)$ has an $\mc O (X)$-eigenvector, say $v_x$ with weight
$x \in X$. By Frobenius reciprocity
\begin{equation}\label{eq:2.3}
\Hom_B (\ind_{\mc O (X)}^B (\C_x), \pi) \cong \Hom_{\mc O (X)}(\C_x,\pi) \neq 0,
\end{equation}
so $\pi$ is a quotient of $\ind_{\mc O (X)}^B (\C_x)$.
For any $\gamma \in \Gamma$ and $f \in \mc O (X)$, $\pi (T_\gamma f) v_x \in V$ 
is an $\mc O (X)$-eigenvector with weight $\gamma x$. If $\pi$ is irreducible,
then it cannot have more weights than these $\gamma x$. It follows that the
irreducible quotients of $\ind_{\mc O (X)}^B (\C_x)$ are precisely the 
irreducible $B$-modules with $\mc O (X)^\Gamma$-character $\Gamma x$.

Let $\Gamma_x$ be the stabilizer of $x$ in $\Gamma$. For $\rho \in 
\Irr (\C [\Gamma_x,\natural])$, we form the irreducible $\mc O (X) \rtimes
\C [\Gamma_x,\natural]$-module $x \otimes \rho$, on which $\mc O (X)$ acts by
evaluation at $x$ and $\C [\Gamma_x,\natural]$ acts by $\rho$. We write
\[
\pi (x,\rho) = \ind_{\mc O (X) \rtimes \C [\Gamma_x,\natural]}^{
\mc O (X) \rtimes \C [\Gamma,\natural]} (x \otimes \rho) .
\]
We let $\Gamma$ act on 
\[
\big\{ (x,\rho) : x \in X, \rho \in \Irr (\C[\Gamma_x,\natural]) \big\}
\]
by $\gamma (x,\rho) = (\gamma x, \gamma \rho)$, where $\gamma \rho (h) = 
\rho (T_\gamma^{-1} h T_\gamma)$ for $h \in \C [\Gamma_{\gamma x},\natural]$.
The space
\begin{equation}\label{eq:2.16}
(X /\!/ \Gamma )_\natural = \big\{ (x,\rho) : x \in X, \rho \in 
\Irr (\C[\Gamma_x,\natural]) \big\} \, / \Gamma
\end{equation}
is called a twisted extended quotient.

\begin{thm}\label{thm:2.3}
Recall that $B =  \mc O (X) \rtimes \C [\Gamma,\natural]$.
\enuma{
\item The $B$-module $\pi (x,\rho)$ is irreducible.
\item There is canonical a bijection
\[
\begin{array}{ccc}
(X /\!/ \Gamma )_\natural & \to & \Irr (B) \\
(x,\rho) / \sim & \mapsto & \pi (x,\rho) 
\end{array}.
\]
}
\end{thm}
\begin{proof}
(a) Consider the ideal $I_{\Gamma x} = \{ f \in \mc O (X)^\Gamma : f(\Gamma x) = 0\}$
of $\mc O (X)^\Gamma$. By the Chinese remainder theorem 
\[
\mc O (X) / I_{\Gamma x} \mc O (X) = \mc O(X) / \{ f \in \mc O (X) : f |_{\Gamma x} = 0 \}   
\cong \bigoplus\nolimits_{\gamma \in \Gamma / \Gamma_x} \C_{\gamma x} .
\]
The irreducible $B$-modules with $\mc O (X)^\Gamma$-character $\Gamma x$ descend
to mo\-dules of 
\[
B / I_{\Gamma x} B \cong \Big( \bigoplus\nolimits_{\gamma \in \Gamma / \Gamma_x} 
\C_{\gamma x} \Big) \rtimes \C [\Gamma,\natural] .
\]
Let $p \in \bigoplus\nolimits_{\gamma \in \Gamma / \Gamma_x} \C_{\gamma x}$ be the
idempotent which is 1 in $\C_x$ and 0 in all other summands, then as bimodules:
\begin{equation}\label{eq:2.4}
\begin{array}{lll}
p (B / I_{\Gamma x} B ) & = & \C_x \otimes_\C \C [\Gamma, \natural] ,\\
(B / I_{\Gamma x} B ) p & = & \C [\Gamma, \natural] \otimes_\C \C_x ,\\
p (B / I_{\Gamma x} B ) p & = & \C_x \otimes_\C \C [\Gamma_x, \natural] 
\cong \C [\Gamma_x, \natural], \\
(B / I_{\Gamma x} B ) p (B / I_{\Gamma x} B ) & = & B / I_{\Gamma x} B . 
\end{array}
\end{equation}
The first three of these equalities are straightforward, the last follows because
$(B / I_{\Gamma x} B ) p (B / I_{\Gamma x} B )$ contains $T_\gamma p T_\gamma^{-1} = \gamma (p)$
and $\sum_{\gamma \in \Gamma / \Gamma_x} \gamma (p) = 1$. From \eqref{eq:2.4} we 
see that the bimodules $p (B / I_{\Gamma x} B )$ and $(B / I_{\Gamma x} B ) p$ provide a Morita
equivalence between $B / I_{\Gamma x} B$ and $\C [\Gamma_x,\natural]$. This
equivalence sends $x \otimes \rho \in \Irr (\C_x \otimes_\C \C [\Gamma_x, \natural])$ to 
\[
(B / I_{\Gamma x} B) p \otimes_{\C_x \otimes_\C \C [\Gamma_x, \natural]} (x \otimes \rho) 
= \ind^{B / I_{\Gamma x}B}_{\bigoplus \nolimits_{\gamma \in \Gamma / \Gamma_x} 
\C_{\gamma x} \rtimes \C [\Gamma_x,\natural]} (x \otimes \rho) , 
\]
which is therefore an irreducible $B / I_{\Gamma x} B$-module. Via inflation to $B$,
we find that $\ind_{\mc O (X) \rtimes \C [\Gamma_x,\natural]}^{\mc O (X) \rtimes 
\C [\Gamma,\natural]} (x \otimes \rho) = \pi (x,\rho)$ is irreducible.\\
(b) By the remarks before the theorem, every irreducible $B$-module has a unique
$\mc O (X)^\Gamma$-character $\Gamma x$. The set of irreducible $B$-modules with that
$\mc O (X)^\Gamma$-character is naturally in bijection with $\Irr (B / I_{\Gamma x}B)$. 
The proof of part (a) yields a bijection
\[
\begin{array}{ccc}
\Irr ( \C_x \otimes \C [\Gamma_x,\natural]) & \to & \Irr ( B / I_{\Gamma x} B) \\
x \otimes \rho & \mapsto & \pi (x,\rho) 
\end{array}.
\]
Hence every element of $\Irr (B)$ has the form $\pi (x,\rho)$. The only freedom in the
above construction is the choice of $x$ in $\Gamma x$. Suppose that instead we pick
$\gamma x \in \Gamma x$. Then conjugation by $T_\gamma$ can be pulled through the
entire construction, and we end up with $\gamma \rho$ instead of $\rho$. Thus
$\pi (x,\rho) \cong \pi (\gamma x,\gamma \rho)$ and these are the only possible
equivalences between modules of this form. 
\end{proof}

Twisted crossed products for free group actions are considerably easier.

\begin{prop}\label{prop:2.10}
Suppose that the action of $\Gamma$ on $X$ is free.
\enuma{
\item The categories of finite length modules of $\mc O (X/\Gamma)$ and of $B = \mc O (X) \rtimes 
\C[\Gamma ,\natural]$ are naturally equivalent.
\item If $\natural$ is trivial in $H^2 (\Gamma,\C^\times)$, then
$\mc O (X) \rtimes \C[\Gamma ,\natural]$ is Morita equivalent to $\mc O (X/\Gamma)$.
}
\end{prop}
\begin{proof}
(a) The $\mc O (X/\Gamma)$-weights provide a decomposition of the category of finite length modules
\[
\Mod_\fl (\mc O (X) \rtimes \C[\Gamma ,\natural] ) = \bigoplus\nolimits_{\Gamma x \in X / \Gamma}
\Mod_\fl (\mc O (X) \rtimes \C[\Gamma ,\natural] )_{\Gamma x} ,
\]
and similarly for $\mc O (X/\Gamma)$. Therefore it suffices to fix a $\Gamma$-orbit 
$\Gamma x$ in $X$ and to consider finite length $\mc O (X) \rtimes \C[\Gamma ,\natural]$-modules 
$V$ such that all irreducible
subquotients of $V |_{\mc O (X/\Gamma)}$ are isomorphic to $\C_{\Gamma x}$. 

With the notations from the proof of Theorem \ref{thm:2.3}, let
\begin{align*}
& \widehat{\mc O (X/\Gamma )}_{\Gamma x} = \varprojlim_n \mc O (X/\Gamma) / I_{\Gamma x}^n \\
& \widehat{\mc O (X )}_{\Gamma x} = \varprojlim_n \mc O (X) / I_{\Gamma x}^n \mc O (X) \cong
\bigoplus\nolimits_{x' \in \Gamma x} \widehat{\mc O (X)}_{x'} \\
& \hat B_{\Gamma x} = B \otimes_{\mc O (X/\Gamma)} \widehat{\mc O (X/\Gamma)}_{\Gamma x} =
\varprojlim_n B / I_{\Gamma x}^n B
\end{align*}
be the formal completions of $\mc O (X/\Gamma), \mc O (X)$ and $B$ with respect to the ideal 
ge\-ne\-rated by $I_{\Gamma x}$. This completion operation does not change the category of 
finite length modules whose only $\mc O (X/\Gamma)$-weight is $\Gamma x$. As in \eqref{eq:2.4}, 
one checks that the bimodules
$p \hat B_{\Gamma x}$ and $\hat B_{\Gamma x} p$ provide a Morita equivalence between 
\[
p \hat B_{\Gamma x} p \cong \widehat{\mc O (X/\Gamma)}_{\Gamma x} \quad \text{and} \quad
\hat B_{\Gamma x} p \hat B_{\Gamma x} = \hat B_{\Gamma x} . 
\]
(The freedom of the action is used in the last equality.) That yields equivalences of categories
\[
\Mod_\fl (B)_{\Gamma x} \cong \Mod_\fl \big( \hat B_{\Gamma x} \big) \cong 
\Mod_\fl \big( \widehat{\mc O (X/\Gamma )}_{\Gamma x} \big) \cong 
\Mod_\fl (\mc O (X/\Gamma))_{\Gamma x} .
\]
(b) Suppose that $\natural \in H^2 (\Gamma, \C^\times)$ is trivial. As explained after 
\eqref{eq:2.14}, $B = \mc O (X) \rtimes \C [\Gamma, \natural]$
is isomorphic to $\mc O (X) \rtimes \Gamma$. This enables us to construct the idempotent 
\[
e_\Gamma := |\Gamma|^{-1} \sum\nolimits_{\gamma \in \Gamma} \gamma \in \C[\Gamma] \subset B .
\]
We claim that $B e B = B$. For this, it suffices to check that 
$\widehat{B e B}_{\Gamma x} = \hat B_{\Gamma x}$ for all $x \in X$. Since $\Gamma$ acts freely 
on $X$, $1_x \, |\Gamma| \, e_\Gamma 1_x \in \widehat{B e B}_{\Gamma x}$
equals $1_x \in \widehat{\mc O (X)}_x$. As in the proof of Theorem \ref{thm:2.3}, it follows that
$\widehat{B e B}_{\Gamma x} = \hat B_{\Gamma x}$, proving our claim.

Now the isomorphisms of bimodules
\[
e B \cong \mc O (X) ,\; B e \cong \mc O (X) ,\; e B e \cong \mc O (X/\Gamma) ,\; B e B = B 
\]
show that $eB$ and $Be$ provide a Morita equivalence between $B$ and $\mc O (X / \Gamma)$.
\end{proof}

For nontrivial $\natural \in H^2 (\Gamma,\C^\times)$ one cannot say directly that Proposition 
\ref{prop:2.10}.b fails, it depends on $X$. In the cases where $X$ is irreducible as an
algebraic variety, one can improve on Proposition \ref{prop:2.10}, as follows. 
 
Suppose that there is a Morita equivalence between $B$ and $\mc O (X/\Gamma)$. Let 
$\C (X)$ be the quotient field of $\mc O (X)$ and tensor the Morita bimodules with $\C (X/\Gamma)$
over $\mc O (X/\Gamma)$. That yields a Morita equivalence between
$\C (X)^\Gamma = \C (X/\Gamma)$ and
\begin{equation}\label{eq:2.15}
\C (X/\Gamma) \otimes_{\mc O (X/\Gamma)} B \cong  \C (X/\Gamma) \otimes_{\mc O (X/\Gamma)}  
\mc O (X) \rtimes \C [\Gamma,\natural] = \C (X) \rtimes \C[\Gamma,\natural] .
\end{equation}
By the next result and Proposition \ref{prop:2.10}.b, \eqref{eq:2.15} is only possible 
if $\natural \in H^2 (\Gamma,\C^\times)$ is trivial.

\begin{thm}\label{thm:2.11}
Let $X$ be an irreducible algebraic variety, endowed with a faithful action of a finite
group $\Gamma$. Then $\C (X) \rtimes \C [\Gamma, \natural]$ is a central simple 
$\C (X)^\Gamma$-algebra. It is Morita equivalent to $\C (X) \rtimes \Gamma$ if and only
if $\natural$ is trivial in $H^2 (\Gamma, \C^\times)$.
\end{thm}
\begin{proof}
By the general theory of central simple algebras, $\C (X) \rtimes \C [\Gamma,\natural]$
is one, see for instance \cite[\S 7.5]{Ker}. By a result of Noether \cite[\S 7.7]{Ker},
$\C (X) \rtimes \C [\Gamma,\natural]$ is Morita equivalent to $\C (X) \rtimes \Gamma$
if and only if $\natural$ is trivial in $H^2 (\Gamma, \C(X / \Gamma)^\times)$.

We claim that the natural map 
\[
H^2 (\Gamma, \C^\times) \to H^2 (\Gamma, \C(X/\Gamma)^\times) \quad \text{is injective.}
\]
Suppose that $c \in Z^2 (\Gamma, \C^\times)$ is a 2-cocycle, which equals a 2-coboundary
$b(f)$ for some $f : \Gamma \to \C (X/\Gamma)^\times$. Write $f = f_1 / f_2$ for some
$f_i \in \mc O (X/\Gamma) \setminus \{0\}$, and pick $x \in X$ such that $f_1 (\Gamma x) 
f_2 (\Gamma x) \neq 0$. Then $f$ and $b(f)$ can be evaluated at $\Gamma x$, which
gives $b(f(x)) = b(f)(x) = c(x) = c$. Hence $c$ is trivial in $H^2 (\Gamma,\C^\times)$.
\end{proof}

\subsection{Definitions of graded Hecke algebras} \

We will analyse the modules over the Hecke algebra of a reductive $p$-adic group
in terms of modules of graded Hecke algebras. Here we provide a short introduction
to those algebras, which were discovered by Lusztig \cite{Lus-Gr}. 

We need the following data:
\begin{itemize}
\item a finite dimensional real vector space $\mf t_\R$,
\item the linear dual space $\mf t_\R^\vee$,
\item a reduced integral root system $\Phi$ in $\mf t_\R^\vee$, with a basis $\Delta$,
\item the Weyl group $W = W (\Phi)$, which acts on $\mf t_\R$ and on $\mf t_\R^\vee$,
\item a $W$-invariant parameter function $k : \Phi \to \C$,
\item a formal variable $\mb r$,
\item the complexifcations $\mf t$ of $\mf t_\R$ and $\mf t^\vee$ of $\mf t_\R^\vee$.
\end{itemize}

\begin{defn}
The graded Hecke algebra $\mh H (\mf t, W, k, \mb r)$ is the vector space
$\mc O (\mf t) \otimes \C [\mb r] \otimes \C [W]$ with multiplication rules
\begin{itemize}
\item $\C [W]$ and $\mc O (\mf t) \otimes \C[\mb r] = \mc O (\mf t \oplus \C)$ 
are embedded as subalgebras,
\item $\C[\mb r]$ is central,
\item for $\alpha \in \Delta$ and $f \in \mc O (\mf t)$:
\begin{equation}\label{eq:2.5}
f \cdot s_\alpha - s_\alpha \cdot s_\alpha (f) = k(\alpha) \mb r (f - s_\alpha (f)) / \alpha .
\end{equation}
\end{itemize}
The grading on $\mh H (\mf t, W, k, \mb r)$ is twice the usual grading on the polynomial
algebra $\mc O (\mf t \oplus \C)$, while all nonzero elements of $\C[W]$ have degree 0.
\end{defn}

It is easy to check that $f - s_\alpha (f)$ is divisible by $\alpha$ in $\mc O (\mf t)$, so
that \eqref{eq:2.5} is really a relation in $\mh H (\mf t,W,k,\mb r)$. For 
$f = x \in \mf t^\vee$, \eqref{eq:2.5} simplifies to
\begin{equation}\label{eq:2.6}
x \cdot s_\alpha - s_\alpha \cdot s_\alpha (x) = k(\alpha) \mb r  \, x(\alpha^\vee) ,
\end{equation}
where $\alpha^\vee \in \mf t_\R$ denotes the coroot of $\alpha$. The elements 
$\alpha \in \mf t^\vee$ and $\mb r$ have degree two, so the relation \eqref{eq:2.5} 
is homogeneous. It follows that $\mh H (\mf t,W,k,\mb r)$ is a graded algebra:
\[
\mr{deg}(x y) = \mr{deg}(x) + \mr{deg}(y) \text{ when } x, y \in \mh H (\mf t,W,k,\mb r)
\text{ are homogeneous.}
\]
\begin{ex}
If $\Phi$ is empty, then $\mh H (\mf t,W,k,\mb r)$ reduces to $\mc O (\mf t) \otimes \C[\mb r]$.
For $k = 0$ we have
\[
\mh H (\mf t,W,0,\mb r) = (\mc O (\mf t) \rtimes W) \otimes \C [\mb r] .
\]
\end{ex}

In practice the central element $\mb r$ is often specialized to a complex number. In view of the identity $\mh H (\mf t ,W, k, z\mb r) = \mh H (\mf t, W, z k ,\mb r)$ 
for $z \in \C$, it suffices to consider the specialization of $\mb r$ to 1. We define
\begin{equation}\label{eq:2.D}
\mh H (\mf t, W,k) = \mh H (\mf t,W, k,\mb r) / (\mb r - 1) .
\end{equation}
The vector space $\mh H (\mf t,W,k) = \mc O (\mf t) \otimes \C [W]$ is graded in the same way
as $\mh H (\mb t,W,k,\mb r)$. However, the multiplication relation \eqref{eq:2.5} becomes
\[
f \cdot s_\alpha - s_\alpha \cdot s_\alpha (f) = k(\alpha) (f - s_\alpha (f)) / \alpha ,
\]
in $\mh H (\mf t,W,k)$, which is usually not homogeneous. Therefore $\mh H (\mf t,W,k)$ is not
graded as an algebra. Instead it is a filtered algebra, which means that for homogeneous 
elements $x, y \in \mh H (\mf t,W,k)$, the product $xy$ is a sum of terms of degrees at most
$\mr{deg}(x) + \mr{deg}(y)$.

Multiplication by $z \in \C^\times$ on $\mf t$ induces an algebra isomorphism
\begin{equation}\label{eq:2.19}
\begin{array}{ccc}
\mh H (\mf t, W, k) & \to & \mh H (\mf t, W, zk) \\
f w & \mapsto & (f \circ z) w 
\end{array}.
\end{equation}
For $z = 0$ this becomes a projection
\[
\mh H (\mf t, W, k) \to \C [W] \subset \mh H (\mf t, W, 0) .
\]
In addition to the data at the start of this paragraph, let $\Gamma$ be a finite group acting
linearly on $\mf t_\R^\vee$, such that it stabilizes $\Phi, \Delta$ and $k$. Then $\Gamma$
acts on $W$ by conjugation in $GL (\mf t_\R^\vee)$ and on $\mh H (\mf t,W,k)$ by 
\[
\gamma (fw) = (f \circ \gamma^{-1}) (\gamma w \gamma^{-1}) \qquad f \in \mc O (\mf t), w \in W,
\gamma \in \Gamma .
\]
\begin{defn}\label{def:2.B}
Let $\natural : \Gamma \times \Gamma \to \C^\times$ be a 2-cocycle. We call
\[
\mh H (\mf t,W\Gamma,k,\natural) = \mh H (\mf t,W,k) \rtimes \C [\Gamma,\natural]
\]
a twisted graded Hecke algebra. 
\end{defn}

Notice that we allow $\Gamma = \{e\}$, in which case $\mh H (t,W\Gamma,k,\natural)$ reduces to
$\mh H (\mf t,W,k)$. By \eqref{eq:2.2}, $\mh H (\mf t,W,k,\natural)$ is a direct summand of 
the algebra $\mh H (\mf t,W,k) \rtimes \Gamma^*$. Therefore most results that have been proven
for $\mh H (\mf t,W,k) \rtimes \Gamma^*$ apply automatically to $\mh H (\mf t,W,k,\natural)$.
We will tacitly use that several times.

We shall also want to consider the opposite algebras of twisted graded Hecke algebras.
These are of the same kind, because there is an algebra isomorphism
\begin{equation}\label{eq:2.E}
\begin{array}{cccc}
\mh H (\mf t, W\Gamma, k, \natural )^{op} & \isom & \mh H (\mf t, W\Gamma, k, \natural^{-1} ) \\
f w N_\gamma & \mapsto & N_\gamma^{-1} w^{-1} f & f \in \mc O (\mf t), w \in W, \gamma \in \Gamma .
\end{array}
\end{equation}
The centre of $\mh H (\mf t,W,k)$ is known from \cite[Proposition 3.11]{Lus-Gr}:
\[
Z(\mh H (\mf t,W,k)) = \mc O (\mf t)^W = \mc O (\mf t / W) .
\]  
We have $Z(\mh H (\mf t,W\Gamma,k,\natural)) \supset \mc O (\mf t)^{W\Gamma}$, with equality
if $\Gamma$ acts faithfully on $\mf t_\R^\vee$. By the same arguments as in Lemma \ref{lem:2.2}.a,
\begin{equation}\label{eq:2.7}
\mh H (\mf t,W\Gamma,k,\natural) \text{ has finite rank as module over its centre and over }
\mc O (\mf t)^{W\Gamma} .
\end{equation}
For another parameter function $k'$, the algebras $\mh H (\mf t,W\Gamma,k',\natural)$ and
$\mh H (\mf t,W\Gamma,k,\natural)$ are usually not isomorphic. Nevertheless they are always
very similar, in several ways. The clearest relation between these algebras can be seen when
we include the quotient field $\C (\mf t / W\Gamma) = \C (\mf t)^{W \Gamma}$ of 
$\mc O (\mf t / W\Gamma)$. There are field isomorphisms
\[
\begin{array}{ccc}
\mc O (\mf t) \otimes_{\mc O (\mf t / W \Gamma)} \C (\mf t / W\Gamma) & \cong & \C (\mf t) \\
f_1 \otimes f_2 / f_3 & \mapsto & f_1 f_2 / f_3 \\
g_1 \prod_{w \in W \Gamma \setminus \{e\}} w (g_2) \otimes \prod_{w \in W \Gamma} w (g_2)^{-1} &
\text{\reflectbox{$\mapsto$}} & g_1 / g_2 
\end{array}.
\]
In particular $\C (\mf t)$ is naturally a subalgebra of 
$\C (\mf t / W\Gamma) \otimes_{\mc O (\mf t / W \Gamma)} \mh H (\mf t,W\Gamma,k,\natural)$.
For $\alpha \in \Delta$ we define
\begin{equation}\label{eq:2.20}
\tau_{s_\alpha} = (1 + s_\alpha) \frac{\alpha}{\alpha + k (\alpha)} - 1 \in 
\C (\mf t / W\Gamma) \otimes_{\mc O (\mf t / W \Gamma)} \mh H (\mf t,W\Gamma,k,\natural) .
\end{equation}
A direct calculation shows that $\tau_{s_\alpha}^2 = 1$. For $\gamma \in \Gamma$ we put
$\tau_\gamma = T_\gamma \in \C [\Gamma,\natural]$.

\begin{thm}\label{thm:2.4} 
\textup{ \cite[\S 5]{Lus-Gr} and \cite[Proposition 1.5.1]{SolAHA}} 
\enuma{
\item The map $s_\alpha \mapsto \tau_{s_\alpha}$ extends to a group homomorphism
\[
\tau : W \to \big( \C (\mf t / W\Gamma) \otimes_{\mc O (\mf t / W \Gamma)} 
\mh H (\mf t,W\Gamma,k,\natural) \big)^\times .
\]
\item There is an isomorphism of $\C (\mf t / W\Gamma)$-algebras
\[
\begin{array}{ccc}
\C (\mf t) \rtimes \C [W\Gamma,\natural] & \to & 
\C (\mf t / W\Gamma) \otimes_{\mc O (\mf t / W \Gamma)} \mh H (\mf t,W\Gamma,k,\natural)  \\
f \otimes w & \mapsto & f \tau_w 
\end{array}.
\]
}
\end{thm}

We note that 
\[
\C (\mf t) \rtimes \C [W\Gamma,\natural] = \C (\mf t / W\Gamma) 
\underset{\mc O (\mf t / W \Gamma)}{\otimes} \mc O (\mf t) \rtimes \C[W\Gamma,\natural] = 
\C (\mf t / W\Gamma) 
\underset{\mc O (\mf t / W \Gamma)}{\otimes} \mh H (\mf t,W\Gamma,0,\natural) .
\]
Theorem \ref{thm:2.4} shows that, upon including $\C (\mf t / W\Gamma)$, the dependence of
$\mh H (\mf t,W\Gamma,k,\natural)$ on $k$ disappears.

\subsection{Representation theory} \

In this paragraph we survey representations of a (twisted) graded Hecke algebra 
$\mh H = \mh H (\mf t,W,k,\natural)$, in analogy with representations of reductive groups over
local fields. For more background on $\mh H$-representations we refer to \cite{KrRa,SolGHA}.

The starting point is always representations of the maximal commutative subalgebra 
$\mc O (\mf t)$. 

\begin{defn}
Let $(\pi,V)$ be a $\mh H$-representation and let $\lambda \in \mf t$. We write
\[
V_\lambda = \big\{ v \in V : \exists N \in \N \text{ such that } (\pi (x) - 
x (\lambda))^N v = 0 \text{ for all}\: x \in \mf t^\vee \big\} .
\]
We say that $\lambda$ is an $\mc O (\mf t)$-weight (or simply a weight) of $\pi$ if $V_\lambda$
is nonzero. The set of $\mc O (\mf t)$-weights of $\pi$ is denoted Wt$(\pi)$.
\end{defn}

With simple linear algebra one checks that:

\begin{lem}\label{lem:2.5}
\enuma{
\item Every nonzero weight space $V_\lambda$ contains an $\mc O (\mf t)$-eigenvector, that is,  
a $v \neq 0$ such that $\pi (x) v = x(\lambda)v$ for all $x \in \mf t^\vee$.
\item If $\dim V$ is finite, then $V = \bigoplus_{\lambda \in \mr{Wt}(\pi)} V_\lambda$.
}
\end{lem}

Let $\C_\lambda$ be the onedimensional $\mc O (\mf t)$-module with weight $\lambda$. With 
induction one constructs the $\mh H$-module
\[
I (\lambda) = \ind_{\mc O (\mf t)}^{\mh H} (\C_\lambda) .
\]
By \cite[Theorem 6.4]{BaMo2}, Wt$(I(\lambda)) = W \Gamma \lambda$.

\begin{lem}\label{lem:2.6}
Let $(\pi,V)$ be an irreducible $\mh H$-representation.
\enuma{
\item $\dim V$ is finite.
\item $\pi$ is a quotient of $I(\lambda)$, for some $\lambda \in \mf t$.
}
\end{lem}
\begin{proof}
(a) This is shown in the same way as Lemma \ref{lem:2.2}.\\
(b) Lemma \ref{lem:2.5} ensures that $\pi$ has an $\mc O (\mf t)$-weight, say $\lambda$,
with an eigenvector. Then Frobenius reciprocity shows that
\[
\Hom_{\mh H} (I(\lambda),\pi) \cong \Hom_{\mc O (\mf t)}(\C_\lambda, \pi) \neq 0 .
\]
Hence there exists a nonzero $\mh H$-homomorphism from $I(\lambda)$ to $\pi$, which by
the irreducibility of $\pi$ must be surjective.
\end{proof}

In simplest noncommutative case, one can classify the irreducible $\mh H$-representa\-tions by hand.

\begin{ex}\label{ex:2.A}
We take $\mf t_\R = \mf t_\R^\vee = \R$, $\Phi = \{ \pm 1\}$ and $W = S_2$. 
We write $k(\alpha) = k \in \C$ and we consider $\mh H = \mh H (\mf t,S_2,k)$. 
There are two onedimensional $\mh H$-representations:
\begin{itemize}
\item The trivial representation $\mr{triv}$ defined by $\mr{triv}|_{\mc O (\mf t)} = \C_k$
and $\mr{triv}|_{\C[S_2]} = \mr{triv}$.
\item The Steinberg representation $\mr{St}$ defined by $\mr{St}|_{\mc O (\mf t)} = \C_{-k}$
and $\mr{St}|_{\C[S_2]} = \mr{sign}$.
\end{itemize}
All other irreducible representations are of the form $I(\lambda)$ with $\lambda \in \C$. As
vector spaces $\mh H = \C[S_2] \otimes \mc O (\mf t)$, so $\Res^{\mh H}_{\C[S_2]} I(\lambda)$
is the regular representation $\C[S_2]$. The only proper $S_2$-invariant subspaces are
\begin{equation*}
\C (1 + s_\alpha) \quad \text{and} \quad \C (1 - s_\alpha) ,
\end{equation*}
so if $I(\lambda)$ is reducible, then at least one of these two is an $\mh H$-subrepresentation.
For $x \in \mf t^\vee$ (corresponding to $x(1) \in \C$) one computes in $I(\lambda)$, using 
\eqref{eq:2.6} with $\mb r = 1$ and $\alpha^\vee = 2$:
\begin{align*}
& x \cdot (1 + s_\alpha) = x + s_\alpha \, s_\alpha (x) + k(\alpha) x(\alpha^\vee) =
(\lambda + 2 k - \lambda s_\alpha) x(1) \\
& x \cdot (1 - s_\alpha) = x - s_\alpha \, s_\alpha (x) - k(\alpha) x(\alpha^\vee) =
(\lambda - 2 k + \lambda s_\alpha) x(1) .
\end{align*}
Hence $\C (1 + s_\alpha)$ is an $\mh H$-submodule of $I(\lambda)$ if and only if $\lambda = -k$, 
while $\C (1-s_\alpha)$ is an $\mh H$-submodule of $I(\lambda)$ if and only if $\lambda = k$.
One can check that $I(k)$ and $I(-k)$ both have length two, with irreducible subquotients
triv and St. All the representations $I(\lambda)$ with $\lambda \in \C \setminus \{k,-k\}$ are
irreducible. By Lemma \ref{lem:2.6}.b, this exhausts $\Irr (\mh H)$. 

Frobenius reciprocity tells us that
\[
\Hom_{\mh H} (I(\lambda) ,I(\lambda')) \cong \Hom_{\mc O (\mf t)} (\C_\lambda, I(\lambda')) .
\]
This is nonzero if and only if $\lambda' = \pm \lambda$, because Wt$(I(\lambda')) = S_2 \lambda' =
\{ \lambda', -\lambda'\}$. Therefore the irreducible representations $I(\lambda)$ and $I(\lambda')$
are isomorphic if and only if $\lambda' = \pm \lambda$.  
\end{ex}

Temperedness of $\mh H$-representations is defined via their weights. Consider the positive Weyl
chamber
\[
(\mf t_\R^\vee)^+ = \{ x \in \mf t_\R^\vee : x(\alpha^\vee) \geq 0 \; 
\text{for all}\; \alpha \in \Delta \} 
\]
and the obtuse negative cone
\[
\mf t_\R^- = \{ \lambda \in \mf t_\R : x (\lambda) \leq 0 \; 
\text{for all}\; x \in (\mf t_\R^\vee)^+ \} .
\]
\begin{ex}
Suppose that $\mf t_\R = \mf t_\R^\vee = \R^2$ with the standard inner product. The positive Weyl
chamber and the obtuse negative cone for $R = A_1$ (from $GL_2$) and $R = B_2$ look like 
\hspace{1cm} \includegraphics[width=8cm]{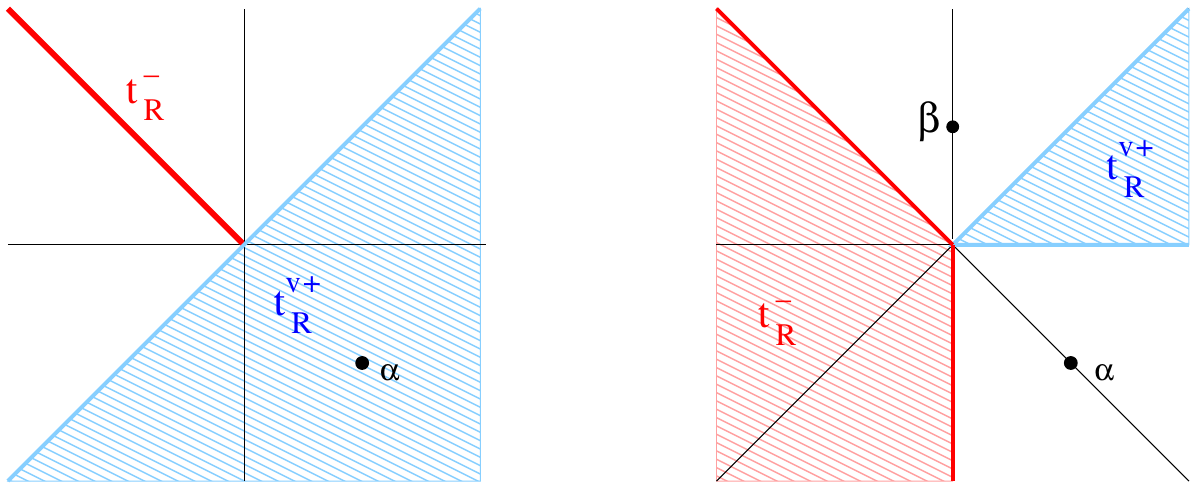}
\end{ex}

\begin{defn}\label{def:2.F}
Let $(\pi,V)$ be a finite dimensional $\mh H$-representation. We call $\pi$ tempered if 
Wt$(\pi) \subset \mf t_\R^- + i \mf t_\R$. More restrictively, we say that $\pi$ is discrete 
series if it is irreducible and Wt$(\pi) \subset \mr{int}(\mf t_\R^-) + i \mf t_\R$. Here
int$(\mf t_\R^-)$ denotes the interior of $\mf t_\R^-$ as a subset of the $\R$-span of 
the root system in $\mf t_\R$ dual to $\Phi$.

In terms of the canonical real part map $\Re : \mf t \to \mf t_\R$, we can reformulate these 
conditions as, respectively, $\Re \mr{Wt} (\pi) \subset \mf t_\R^-$ and
$\Re \mr{Wt} (\pi) \subset \mr{int} (\mf t_\R^-)$.
\end{defn}

It is known from \cite[\S 9]{SolEnd}\footnote{Our discrete series representations are called 
tempered essentially discrete series in \cite{SolEnd}.} that the notions of temperedness and 
discrete series for representations of twisted graded Hecke algebras correspond to the 
synonymous notions for representations of reductive $p$-adic groups.

\begin{ex}
The representation $I(\lambda)$ is tempered if and only if $\lambda \in i \mf t_\R$, because
Wt$(I(\lambda)) = W\lambda$. It is never discrete series, because 
$W \Re (\lambda) \subset \mf t_\R$ cannot be contained in the interior of $\mf t_\R^-$. 

Let $\mh H (\C,S_2,k)$ be as in Example \ref{ex:2.A}. The Steinberg representation is tempered if
and only if $\Re (k(\alpha)) \geq 0$. It is discrete series if and only if $\Re (k(\alpha)) > 0$.
\end{ex}

\begin{defn}
$\Irr_\temp (\mh H)$ is the set of irreducible tempered $\mh H$-representations. For a subset
$E \subset \mf t$, we let $\Irr (\mh H)_E$ be the set of irreducible $\mh H$-representations $\pi$
such that Wt$(\pi) \subset E$. The same condition determines the subset $\Irr_\temp (\mh H)_E$.
We define $\Mod (\mh H)_E$ as the category of finite length $\mh H$-representations all whose
weights lie in $E$. We say that an $\mh H$-representation $(\pi,V)$ has real weights if it 
lies in $\Mod (\mh H)_{\mf t_\R}$.
\end{defn}

\begin{ex}
Assume that $k = 0$. Then $\mh H (\mf t,W,0) = \mc O (\mf t) \rtimes W$, which we discussed in
Paragraph \ref{par:tcp}. From Theorem \ref{thm:2.3} we see that the weights of any irreducible
$\mc O (\mf t) \rtimes W$-representation always form one full $W$-orbit $W \lambda \subset \mf t$.
As $W \Re (\lambda)$ cannot be contained in int$(\mf t_\R^-)$, $\mh H (\mf t,W,0)$ does not have
discrete series representations. 

The conditions Wt$(\pi) \subset i \mf t_\R$ and Wt$(\pi) \subset \mf t_\R$ together imply
Wt$(\pi) = \{0\}$, so $\Irr_\temp (\mc O (\mf t) \rtimes W)_{\mf t_\R} = \Irr (\mc O (\mf t) 
\rtimes W)_{\{0\}}$. This set is naturally identified with $\Irr (W)$.
\end{ex}

In general the representations in $\Irr_\temp (\mh H)_{\mf t_\R}$ can have more weights than
just 0, but not many, the conditions tempered and real weights are very restrictive. For an
algebra or group $A$, let $R(A)$ be the Grothendieck group of the category of finite length 
$A$-representations. When $A$ is a group or a group algebra, $R(A)$ is called the 
representation ring of $A$ (with the tensor product as multiplication).

\begin{thm}\label{thm:2.7} 
\textup{\cite[Theorem 6.5.c]{SolHomGHA} and \cite[Theorem 6.2.a]{SolSurv}} \\
Suppose that $k$ is real-valued. The set $\{ \pi |_{\C[W\Gamma,\natural]} : \pi \in \Irr_\temp
(\mh H)_{\mf t_\R} \}$ is $\Z$-basis of $R(\C[W\Gamma,\natural])$.
\end{thm}

\begin{ex}
Let $\mh H = \mh H (\C,S_2,k)$ be as in Example \ref{ex:2.A}, with $k > 0$. Then
$\Irr_\temp (\mh H) = \{ \mr{St}, I(0) \}$. We recall that $\mr{St}|_{\C[S_2]} = \mr{sign}$ and
$I(0) |_{\C[S_2]} = \C [S_2]$. These form a basis of the representation ring of $S_2$, for
instance $\mr{triv} = \C[S_2] - \mr{sign}$ in $R(S_2)$.
\end{ex}

\subsection{Parametrization of irreducible representations} \label{par:2.4} \

The best method to produce new $\mh H$-representations is parabolic induction. For the moment
we do this without $\Gamma$, so for $\mh H (\mf t,W,k)$.

\begin{defn}
Let $P \subset \Delta$ and let $\Phi_P \subset \Phi$ be the associated parabolic root subsystem.
Then $\mh H^P = \mh H (t,W(\Phi_P),k)$ is called a parabolic subalgebra of $\mh H (\mf t,W,k)$.
Parabolic induction is the functor $\ind_{\mh H^P}^{\mh H (\mf t,W,k)}$.
\end{defn}

\begin{ex}
If $P$ is empty, then $\mh H^P = \mc O(\mf t)$. 
If $P = \Delta$, then $\mh H^P = \mh H (\mf t,W,k)$.
\end{ex}

It is known \cite[Corollary 6.5]{BaMo1} that parabolic induction
preserves temperedness.
The vector space $\mf t$ decomposes as $P^\perp \oplus \C P^\vee$, where 
$P^\vee = \{ \alpha^\vee : \alpha \in P \}$. Accordingly $\mh H^P$ decomposes as
a tensor product of algebras
\begin{equation}\label{eq:2.8}
\mh H^P = \mh H (\C P^\vee,W(\Phi_P),k) \otimes \mc O (P^\perp) = 
\mh H_P \otimes \mc O (P^\perp) .
\end{equation}
When we think of $\mh H^P$ as corresponding to a Levi subgroup $M$ of a reductive group $G$,
$\mh H_P$ corresponds to the derived group of $M$.

For any $\mh H_P$-representation $(\pi_P,V_P)$ and any $\lambda^P \in P^\perp$, we can form 
the $\mh H^P$-representation $(\pi_P \otimes \lambda^P, V_P)$. By varying $\lambda^P$, this
construction yields continuous or algebraic families of $\mh H^P$-representations.

We sketch how the classification of $\Irr (\mh H)$ is obtained. A crucial role is played by 
Theorem \ref{thm:2.7}, and therefore we assume throughout this paragraph that
\begin{equation}\label{eq:2.9}
\text{the parameter function } k \text{ is real-valued.}
\end{equation}
The condition \ref{eq:2.9} implies \cite[Lemma 2.13]{Slo}:
\begin{equation}\label{eq:2.18}
\text{every discrete series representation } \pi \text{ of } \mh H^P \text{ has Wt}
(\pi) \subset \R P^\vee + i (\mf t_\R \cap P^\perp) .  
\end{equation}
By completion with respect to central characters, as in \cite[\S 7--8]{Lus-Gr} and 
\cite[\S 2.1]{SolAHA}, one shows that for any $\lambda, \mu \in \mf t_\R$ 
there is an equivalence of categories
\begin{equation}\label{eq:2.10}
\Mod (\mh H (\mf t, W, k)))_{W (i \lambda + \mu)} \cong 
\Mod (\mh H (\mf t,W_{i\lambda},k))_{W_{i \lambda} \cdot \mu} .
\end{equation}
Notice that $W_{i \lambda} = W_\lambda$ is again a Weyl group, from the root system
$\{ \alpha \in \Phi : \alpha (\lambda) = 0\}$. With \eqref{eq:2.10} one can reduce the issues 
in this paragraph from $\Mod (\mh H)$ to $\mh H$-representations with real weights.
By \cite[Proposition 2.7]{AMS3}, \eqref{eq:2.10} restricts to a bijection
\[
\Irr_\temp (\mh H (\mf t, W, k))_{W (i \lambda + \mu)} \longleftrightarrow 
\Irr_\temp (\mh H (\mf t,W_{i\lambda},k))_{W_{i \lambda} \cdot \mu} .
\]
Taking the union over all $\mu \in \mf t_\R$ yields a bijection
\[
\Irr_\temp (\mh H (\mf t,W,k))_{W i \lambda + \mf t_\R} \longleftrightarrow
\Irr_\temp (\mh H (\mf t,W_{i \lambda},k))_{\mf t_\R} .
\]
By Theorem \ref{thm:2.7}, restriction to $\C[W_{i\lambda}]$ sends the last set to a basis
of $R(\C [W_{i \lambda}])$. Theorem \ref{thm:2.3} provides a bijection
\[
R(\C [W_{i \lambda}]) \longrightarrow R(\mc O (\mf t) \rtimes W)_{W i \lambda} =
R(\mh H (\mf t,W,0))_{W i \lambda} ,
\]
and we note that this consists entirely of tempered representations.
The composition of all these maps is 
\begin{equation}\label{eq:2.11}
\begin{array}{ccc}
\Irr_\temp (\mh H (\mf t,W,k))_{W_{i\lambda} + \mf t_\R} & \to &
R (\mc O (\mf t) \rtimes W )_{W_{i \lambda}} \\
\ind_{\mh H (\mf t,W_{i \lambda}, k)}^{\mh H (\mf t,W,k)} (\pi) & \mapsto &
\ind_{\mc O (\mf t) \rtimes W_{i \lambda}}^{\mc O (\mf t) \rtimes W} (\pi |_{\C [W_{i \lambda}]}
\otimes i \lambda) 
\end{array}.
\end{equation}
Now we let $\lambda$ run over $\mf t_\R$, and we conclude that $\Irr_\temp (\mh H)$ maps
canonically to a $\Z$-basis of $R(\mc O (\mf t) \rtimes W)_{i \mf t_\R} =
R (\mc O (\mf t) \rtimes W)_\temp$.

Next we involve the Langlands classification for graded Hecke algebras, from \cite{Eve}. We write
\[
P^{\perp +} = \big\{ \lambda^P \in P^\perp : \alpha (\Re (\lambda^P)) > 0 \quad 
\text{for all}\; \alpha \in \Delta \setminus P \big\}.
\]
For each $\lambda^P \in P^{\perp +}$ and each $\pi_P \in \Irr_\temp (\mh H_P)$, the 
$\mh H$-representation $\ind_{\mh H^P}^{\mh H} (\pi_P \otimes \lambda_P)$
has a unique irreducible quotient. That sets up a bijection \cite{Eve}
\begin{equation}\label{eq:2.12}
\bigcup\nolimits_{P \subset \Delta} \Irr_\temp (\mh H_P) \times P^{\perp +} 
\longrightarrow \Irr (\mh H (\mf t,W,k)) = \Irr (\mh H).
\end{equation}
This also applies to $\mc O (\mf t) \rtimes W = \mh H (\mf t,W,0)$. The Langlands classification
and \eqref{eq:2.11} lead to a version of \cite[Theorem 2.3.1]{SolAHA} for graded instead of
affine Hecke algebras:

\begin{thm}\label{thm:2.8}
Suppose that $k$ is real-valued. There exists a natural bijection
\[
\zeta_0 : R(\mh H (\mf t,W,k)) \to R(\mh H (\mf t,W,0)) = R(\mc O (\mf t) \rtimes W)
\]
satisfying the following properties.
\begin{enumerate}[(i)]
\item For $\rho \in \Irr_\temp (\mh H (\mf t,W,k))_{\mf t_\R}$ we have
$\zeta_0 (\rho) = \C_0 \otimes \rho|_{\C [W]}$.
\item $\zeta_0$ commutes with parabolic induction and with character twists:
\[
\zeta_0 \big( \ind_{\mh H^P}^{\mh H (\mf t,W,k)} (\pi_P \otimes \lambda^P) \big) =
\ind_{\mc O (\mf t) \rtimes W(\Phi_P)}^{\mc O (\mf t) \rtimes W} (\zeta_0 (\pi_P) \otimes
\lambda^P) \text{ for } \pi_P \in \Irr (\mh H_P), \lambda^P \in P^\perp .
\]
\item $\zeta_0$ preserves the underlying (virtual) $\C[W]$-representations.
\item For any $\lambda \in \mf t_\R$, $\zeta_0$ sends virtual representations with 
$\mc O (\mf t)$-weights in $i \lambda + \mf t_\R$ to virtual representations with 
$\mc O (\mf t)$-weights in $i \lambda + \mf t_\R$.
\item $\zeta_0$ sends tempered representations to tempered representations, and
restricts to a bijection between the tempered parts on both sides.
\end{enumerate}
\end{thm}

In Theorem \ref{thm:2.8}.(iv) $\zeta_0$ may adjust the weights by elements of $\mf t_\R$,
and these changes are always weights of a discrete series representation. 

Further, with \cite[\S 6]{SolSurv} $\zeta_0$ can be refined to a bijection
\begin{equation}\label{eq:2.13}
\zeta_\Irr : \Irr (\mh H (\mf t,W,k)) \to \Irr (\mc O (\mf t) \rtimes W) .
\end{equation}
Both Theorem \ref{thm:2.8} and \eqref{eq:2.13} confirm what we already saw in Theorem 
\ref{thm:2.4}: that $\mh H (\mf t,W,k)$ and $\mh H (\mf t,W,0)$ are very similar.

\begin{ex} 
We consider $\mh H = \mh H (\C,S_2,k)$ with $k = k(\alpha) > 0$. With the classification of 
$\Irr (\mh H)$ from Example \ref{ex:2.A} at hand, the maps $\zeta_0$ and $\zeta_\Irr$ can be
tabulated.
\[
\begin{array}{cccc}
\Irr (\mh H ) & R(\mc O (\mf t) \rtimes W) & \Irr (\mc O (\mf t) \rtimes W) \\
\hline
I(\lambda) & \pi (\lambda) := \ind_{\mc O (\mf t)}^{\mc O (\mf t) \rtimes W}(\lambda) & 
\pi (\lambda) & \lambda \in \C \setminus \{k,-k\} \\
I(0) & \pi (0) & \C_0 \otimes \mr{triv} \\
\mr{St} & \C_0 \otimes \mr{sign} & \C_0 \otimes \mr{sign} \\
\mr{triv} & \pi (k) - \C_0 \otimes \mr{sign} & \pi (k)
\end{array}
\]
Notice that for St and triv, $\zeta_0$ changes the $\mc O (\mf t)$-weights, while $\zeta_\Irr$ for 
$I(0) = \ind_{\mc O (\mf t)}^{\mh H}(\C_0)$ is not compatible with parabolic induction. Further 
$\zeta_0 (\mr{triv})$ is not an actual representation, and it is certainly not irreducible.
\end{ex}

Now we generalize to the twisted graded Hecke algebras 
\begin{equation}\label{eq:2.21}
\mh H (\mf t,W\Gamma ,k,\natural) = \mh H (\mf t,W,k) \rtimes \C[\Gamma,\natural]. 
\end{equation}
The equivalence of categories \eqref{eq:2.10} and the
arguments leading up to \eqref{eq:2.11} also work for these algebras, so there is a canonical map
\[
\Irr_\temp (\mh H (\mf t,W\Gamma,k,\natural)) \to R(\mc O (\mf t) \rtimes \C[W\Gamma,\natural])_\temp
\]
whose image is a $\Z$-basis of its range. 

A parabolic subalgebra of $\mh H (\mf t,W\Gamma ,k,\natural)$ has the form $\mh H (\mf t,W(\Phi_P)
\Gamma_P, k,\natural)$, where $\Gamma_P \subset \Gamma$ is a subgroup stabilizing $P$. (In general
there are several choices for $\Gamma_P$, and in principle they are all feasible. Sometimes specific
circumstances determine $\Gamma_P$.) The tensor product decomposition \eqref{eq:2.8} generalizes 
only in the weaker form
\[
\mh H (\mf t ,W(\Phi_P) \Gamma_P ,k,\natural) = \big( \mh H (\mf t,W(\Phi_P),k) \otimes 
\mc O (P^\perp) \big) \rtimes \C[\Gamma_P,\natural] .
\]
This creates complications for the Langlands correspondence, the version for\\ 
$\mh H (\mf t,W\Gamma ,k,\natural)$ is more cumbersome, see \cite[Corollary 6.8]{SolSurv}.

\begin{thm}\label{thm:2.9}
Theorem \ref{thm:2.8} and \eqref{eq:2.13} hold also for $\mh H (\mf t,W\Gamma,k,\natural)$. 

To retain property (iii), we can either start with representations of the parabolic subalgebra 
$\mh H^P = \mh H (\mf t,W(\Phi_P),k)$ (so with $\Gamma_P = \{e\}$), or we have to put an extra 
condition on $\lambda^P$.
\end{thm}
\begin{proof}
For representations with real weights, this is proven in \cite[Proposition 6.10.a]{SolSurv}.
The general case follows from that with \eqref{eq:2.10}.
\end{proof}

The property \eqref{eq:2.18} can be used to obtain some information about families of
discrete series representations of parabolic subalgebras of \eqref{eq:2.21}.

\begin{lem}\label{lem:2.10}
For $P \subset \Delta$ and $\lambda \in i \mf t_\R$, let $\Delta (P,\lambda) \subset
\Irr \big( \mh H (\mf t ,W(\Phi_P) \Gamma_P ,k,\natural) \big)$ be the set of discrete series
representations occurring in $\ind_{\mc O (\mf t)}^{\mh H (\mf t ,W(\Phi_P) \Gamma_P ,k,\natural)}
\C_{\lambda + \nu}$ for some $\nu \in \mf t_\R$.
Let $C$ be a contractible subset of $i \mf t_\R$, such that $(W(\Phi_P) \Gamma_P )_\lambda$ is 
the same for all $\lambda \in C$. Then all the sets $\Delta (P,\lambda)$ with $\lambda \in C$
are canonically in bijection with each other.
\end{lem}
\begin{proof}
First we look at $\mh H (\mf t,W,k)$ and its parabolic subalgebra $\mh H^P$. By \eqref{eq:2.18},
$\Delta (P,\lambda)$ is empty unless $\lambda \in i (\mf t_\R \cap P^\vee)$. From \eqref{eq:2.8} 
we see that 
\begin{equation}\label{eq:2.22}
\Delta (P,\lambda) = \Delta (P,0) \otimes \C_\lambda \text{ for } 
\lambda \in i (\mf t_\R \cap P^\vee). 
\end{equation}
The set $i (\mf t_\R \cap P^\vee)$ is precisely the fixed point set of $W(\Phi_P)$ in 
$i \mf t_\R$. Therefore either $C \subset i (\mf t_\R \cap P^\vee)$ or 
$C \cap i (\mf t_\R \cap P^\vee)$ is empty. In the latter case $\Delta (P,\lambda)$
is empty, while in the former case \eqref{eq:2.22} provides a canonical bijection from
$\Delta (P,\lambda)$ to $\Delta (P,\lambda')$ for any $\lambda,\lambda' \in C$.

Now we consider $\mh H_+^P := \mh H (\mf t,W(\Phi_P) \Gamma_P, k,\natural)$, and we denote
$\Delta (P,\lambda)$ for that algebra by $\Delta (P,\lambda)_+$. Notice that the discrete 
series condition from Definition \ref{def:2.F} is stable under the action of $\Gamma_P$ on 
$\mf t$. Therefore the restriction to $\mh H^P$ of any $\delta_+ \in \Delta (P,\lambda)_+$
has all irreducible subquotients in $\Delta (P,\lambda)$. It follows that we can exhaust
$\Delta (P,\lambda)_+$ with the irreducible subquotients of $\ind_{\mh H^P}^{\mh H^P_+}
(\delta)$ for $\delta \in \Delta (P,\lambda)$. 

By Clifford theory, as for instance in
\cite[p. 24]{RaRa} or \cite[\S 1]{AMS1}, $\ind_{\mh H^P}^{\mh H^P_+}(\delta)$ is 
completely reducible and its decomposition into irreducible representations is go\-ver\-ned
by its algebra of self-intertwiners, which is a twisted group algebra $\C [\Gamma_{P,\delta},
\natural_\delta]$. Write $\delta = \delta_0 \otimes \lambda$ as in \eqref{eq:2.22}.
For $\gamma \in \Gamma_P$ we have 
\[
\gamma (\delta_0 \otimes \lambda) = \gamma (\delta_0) \otimes \gamma (\lambda) .
\]
Consequently $\Gamma_{P,\delta_0 \otimes \lambda}$ is the same for all $\lambda \in C$.
The 2-cocycles $\natural_{\delta_0 \otimes \lambda}$ come a choice of intertwining operators
\[
I^\gamma_{\delta_0 \otimes \lambda} \in \Hom_{\mh H^P} (\gamma (\delta_0) \otimes \gamma 
(\lambda), \delta_0 \otimes \lambda ), 
\] 
see \cite[(4)--(5)]{AMS1}. We can choose these independently of $\lambda \in C$, so 
$\natural_{\delta_0 \otimes \lambda}$ does not depend on $\lambda \in C$ either. Hence
the entire decomposition of $\ind_{\mh H^P}^{\mh H^P_+}(\delta_0 \otimes \lambda)$ 
depends continuously on $\lambda \in C$. Together with the contractibility of $C$,
that yields a canonical bijections between the sets of irreducible subquotients of
$\ind_{\mh H^P}^{\mh H^P_+}(\delta_0 \otimes \lambda)$ and of 
$\ind_{\mh H^P}^{\mh H^P_+}(\delta_0 \otimes \lambda')$, for any $\lambda, \lambda' \in C$.
When we carry this out for all $\delta \in \Delta (P,\lambda)$, we obtain the 
desired bijection $\Delta (P,\lambda)_+ \to \Delta (P,\lambda')_+$.  
\end{proof}

\section{Progenerators and their endomorphism algebras} \label{sec:3}

We look in more detail at the structure of the Hecke algebra of a reductive $p$-adic group $G$.
In Paragraph \ref{par:1.8} we surveyed it in terms of harmonic analysis, but that description does
not suffice to say much about the representation theory of $\mc H (G)$. 

The Bernstein decomposition (Theorem \ref{thm:1.18}) reduces the issue to understanding each
Bernstein block $\Rep (G)^{\mf s}$, so we will focus on one such block. We will study it by means
of a finitely generated projective generator, a progenerator for short. This strategy is very 
general, it can be employed whenever one has a ring $A$ and a progenerator $P$ for Mod$(A)$.
The main point is the following result from category theory:

\begin{prop} \label{prop:3.1} 
\textup{\cite[Theorem 1.8.2.1]{Roc}} \\
There are equivalences of categories
\[
\begin{array}{cccc}
\Mod (A) & \longleftrightarrow & \End_A (P) - \Mod & \cong \Mod (\End_A (P)^{op}) \\
M & \mapsto & \Hom_A (P,M) \\
V \otimes_{\End_A (P)} P & \text{\reflectbox{$\mapsto$}} & V
\end{array} .
\]
Here $\End_A (P) - \Mod$ denotes the category of right $ \End_A (P)$-modules. The bimodules 
for the Morita equivalence between $A$ and $\End_A (P)^{op}$ are $P$ and $\Hom_A (P,A)$.
\end{prop}

Progenerators are quite common in the representation theory of reductive $p$-adic groups, 
although sometimes implicitly. Namely, suppose that $(K,\rho)$ is a type for $\Rep (G)^{\mf s}$, 
in the sense of Bushnell and Kutzko \cite{BuKu}. Then $\ind_K^G (\rho)$ is a progenerator of 
$\Rep (G)^{\mf s}$. Moreover the Hecke algera $\mc H (G,K,\rho)$, as defined in \cite[\S 2]{BuKu}, 
is the opposite algebra of $\End_G (\ind_K^G (\rho))$, and the equivalence of categories
$\Rep (G)^{\mf s} \cong \Mod (\mc H (G,K,\rho))$ from \cite[Theorem 4.3]{BuKu} 
is just an instance of Proposition \ref{prop:3.1}. 

In this paper we will not use types, because we want to treat all Bernstein blocks, whereas types are
not always available. Whenever one has a type $(K,\rho)$, the algebra $\mc H (G,K,\rho)$ is Morita
equivalent to the $G$-endomorphism algebra of any other progenerator for $\Rep (G)^{\mf s}$, so in
that sense the choice of a progenerator does not really matter.

\subsection{The cuspidal case} \

Let $L = \mc L (F)$ be a reductive $p$-adic group (which in the next paragraphs will be a Levi
subgroup of $G$). Let $\sigma \in \Irr (L)$ be supercuspidal and consider the Bernstein block
$\Rep (L)^{\mf s_L}$ with $\mf s_L = [L,\sigma ]_L$. Recall from Theorem \ref{thm:1.12} that
$\Res^L_{L^1} \sigma$ is a compact $L^1$-representation. Compact representations are always 
projective \cite[Proposition IV.1.6]{Ren}, because they behave like representations 
of compact groups. 

We note that $\Res^L_{L^1} \sigma$ has finite length because $[L : Z(L) L^1]$ is finite. Let 
\[
\Pi_{\mf s_L} := \ind_{L^1}^L (\Res^L_{L^1} \sigma)
\]
be the smooth compact induction, from $L^1$ to $L$, of $\sigma$. Here compact means that the 
underlying vector space is
\[
\{ f : G \to V_\sigma \mid f (k l) = \sigma (k) f(l) \;\text{for all}\; k \in L^1, l \in L,\;
\mr{supp}(f) \; \text{is compact in} \; L / L^1 \} . 
\]

\begin{prop} \label{prop:3.2}
\textup{(Bernstein, see \cite[Proposition VI.4.1]{Ren})} \\
The $L$-representation $\Pi_{\mf s_L}$ is a progenerator of $\Rep (L)^{\mf s_L}$.
\end{prop}

We note that $\Pi_{\mf s_L}$ is canonical, in the sense that it depends only on $[L,\sigma]_L$, or
equivalently only on $X_\nr (L) \sigma$. Thus Propositions \ref{prop:3.1} and \ref{prop:3.2} give a
canonical equivalence of categories
\begin{equation}\label{eq:3.4}
\Rep (L)^{[L,\sigma]} \cong \End_L (\Pi_{[L,\sigma]}) - \Mod.
\end{equation}
There are isomorphisms of $L$-representations
\[
\Pi_{\mf s_L} = \ind_{L^1}^L (\Res^L_{L^1} \sigma) \cong \C [L/L^1] \otimes \sigma \cong
\mc O (X_\nr (L)) \otimes \sigma ,
\]
where $L$ acts diagonally on the tensor products. Since $L / L^1$ is commutative, the 
multiplication action of $\C [L/L^1] \cong \mc O (X_\nr (L))$ on $\C [L/L^1]$ is by 
$L$-intertwiners. This gives an embedding
\begin{equation}\label{eq:3.1}
\mc O (X_\nr (L)) \to \End_L \big( \mc O (X_\nr (L)) \otimes \sigma \big) 
\cong \End_L (\Pi_{\mf s_L}) .
\end{equation}
Recall the finite group $X_\nr (L,\sigma) = \{ \chi \in X_\nr (L) : \sigma \otimes \chi 
\cong \sigma \}$. It acts on $X_\nr (L)$ by translations, and there is a homeomorphism
\begin{equation}\label{eq:3.A}
X_\nr (L) / X_\nr (L,\sigma) \longrightarrow \Irr (L)^{\mf s_L} . 
\end{equation}
For $\chi \in X_\nr (L,\sigma)$, the choice of a nonzero element of $\Hom_L (\sigma, \sigma 
\otimes \chi)$ gives rise to an element $T_\chi \in \End_L (\Pi_{\mf s_L})$, which lifts the
multiplication action of $\chi$ on $X_\nr (L)$. We define a 2-cocycle $\natural_{\mf s_L}$ of 
$X_\nr (L,\sigma)$ by
\begin{equation}\label{eq:3.2}
T_{\chi_1} T_{\chi_2} = \natural_{\mf s_L} (\chi_1, \chi_2) T_{\chi_1 \chi_2} .
\end{equation}

\begin{thm} \label{thm:3.3}
\textup{\cite[Proposition 2.2 and (2.25)]{SolEnd}} \\
The elements $T_\chi$ with $\chi \in X_\nr (L,\sigma)$ determine an algebra isomorphism
\[
\End_L (\Pi_{\mf s_L}) \cong \mc O (X_\nr (L)) \rtimes \C [X_\nr (L,\sigma), \natural_{\mf s_L}] .
\]
\end{thm}

From Theorem \ref{thm:3.3} one sees that the centre is
\begin{equation}\label{eq:3.9}
Z(\End_L (\Pi_{\mf s_L}) \cong \mc O (X_\nr (L))^{X_\nr (L,\sigma)} =
\mc O (X_\nr (L) / X_\nr (L,\sigma)) \cong \mc O (\Irr (L)^{\mf s_L}) .
\end{equation}
This is also the centre of the category $\Rep (L)^{\mf s_L}$. By Proposition \ref{prop:2.10} 
there are equivalences of categories of finite length representations
\begin{equation}\label{eq:3.3}
\Rep_\fl (L)^{\mf s_L} \cong \End_L (\Pi_{\mf s_L})-\Mod_\fl \cong 
\mc O ( \Irr (L)^{\mf s_L} )-\Mod_\fl .
\end{equation}
By Proposition \ref{prop:2.10}.b the restriction to finite length can be omitted if 
$\natural_{\mf s_L}$ is trivial, but by \eqref{eq:2.15} and Theorem \ref{thm:2.11} it is necessary if 
$\natural_{\mf s_L}$ is nontrivial in $H^2 (X_\nr (L,\sigma), \C^\times)$. 
An example of the latter situation is \cite[Example 2.G]{SolEnd}.  

\subsection{The non-cuspidal case} \label{par:3.2} \

Let $P = L U_P$ be a parabolic subgroup of $G$ with Levi factor $L$. Recall that 
$\mf s = [L,\sigma]_G$ and $\mf s_L = [L,\sigma]_L$. The following result of Bernstein
is quite deep, in particular it uses the second adjointness relation (Theorem \ref{thm:1.8}).

\begin{prop}\label{prop:3.4}
The $G$-representation 
\[
\Pi_{\mf s} := I_P^G (\Pi_{\mf s_L}) = I_P^G ( \ind_{L^1}^L (\Res^L_{L^1} \sigma))
\]
is a progenerator of $\Rep (G)^{\mf s}$. It is canonical, in the sense 
that up to isomorphism it depends only $\mf s$.
\end{prop}

From Propositions \ref{prop:3.1} and \ref{prop:3.2} we obtain a canonical equivalence of categories
\begin{equation}\label{eq:3.5}
\Rep (G)^{\mf s} \cong \End_G (\Pi_{\mf s}) - \Mod.
\end{equation}

\begin{ex}\label{ex:3.B}
Suppose that $G$ is quasi-split and that $\mf s = [T,1]_G$ for a maximal torus $T$ of $G$. 
Then, by \cite[Th\'eor\`eme 2]{Blo}
\[
\Pi_{\mf s} = I_B^G (\ind_{T^1}^T (\mr{triv})) = I_B^G (\C [T/T^1]) \cong \C [ G/I]
\]
for an Iwahori subgroup $I$ of $G$. In this case
\[
\End_G (\Pi_{\mf s}) \cong \End_G (\C [G/I]) \cong \mc H (G,I) ,
\]
so in words $\End_G (\Pi_{\mf s})$ is isomorphic to the Iwahori--Hecke algebra of $G$.
It is known from \cite{IwMa} that $\mc H (G,I)$ has the structure of an affine Hecke algebra.
The equivalence between the module category of $\mc H (G,I)$ and the category of Iwahori-spherical
$G$-representations from \cite{Bor} is a special case of the equivalence obtained from 
Propositions \ref{prop:3.1} and \ref{prop:3.2}, combined with an isomorphism between $\mc H (G,I)$
and its opposite algebra. 
\end{ex}

However, in general the structure of $\End_G (\Pi_{\mf s})$ is considerably more involved, 
and we will only approach it in several steps. The first observation in this direction is that
\eqref{eq:3.1} and the functor $I_P^G$ provide an embedding
\begin{equation}\label{eq:3.7}
\mc O (X_\nr (L)) \to \End_G (\Pi_{\mf s}) .
\end{equation}
Recall that 
\[
W_{\mf s} = \mr{Stab}_{N_G (L)/L} ([L,\sigma]_L) = 
\{ w \in N_G (L) : w \cdot \sigma \in X_\nr (L) \sigma \} / L .
\]
The action of $N_G (L)$ on $\Irr (L)$ induces an action of $W_{\mf s}$ on $\Irr (L)^{\mf s_L}$.

\begin{ex}\label{ex:3.A}
We consider the special case $W_{\mf s} = \{e\}$, which is very common. Let $\pi,\rho \in 
\Rep (L)^{\mf s_L}$. Recall that Bernstein's geometric lemma \cite[Th\'eor\`eme VI.5.1]{Ren}
provides a filtration of the $L$-representation $J_P^G I_P^G (\pi)$. The condition 
$W_{\mf s} = \{e\}$ implies that from the irreducible subquotients of this filtration only 
$\pi$ itself belongs to $\Rep (L)^{\mf s_L}$. From that and Frobenius reciprocity we obtain
\[
\Hom_G (I_P^G (\pi), I_P^G (\rho)) \cong \Hom_L (J_P^G I_P^G (\pi) ,\rho) \cong \Hom_L (\pi,\rho) .
\]
Therefore the functor $I_P^G : \Rep (L)^{\mf s_L} \to \Rep (G)^{\mf s}$ is an equivalence of
categories. In particular $I_P^G$ induces an algebra isomorphism
\[
\End_L (\Pi_{\mf s}) \to \End_G (I_P^G (\Pi_{\mf s_L})) = \End_G (\Pi_{\mf s}).
\]
\end{ex}

This is a very satisfactory outcome, but of course things are more complicated (and more 
interesting) when $W_{\mf s} \neq \{e\}$. 

Recall that the Bernstein centre of $G$ \cite{BeDe} is the centre of the category $\Rep (G)$. 
It can also be expressed in terms of distributions on $G$ \cite[Proposition 3.2g]{MoTa}.
Some aspects of the Bernstein decomposition involve the Bernstein centre: 

\begin{thm}\label{thm:3.5} \textup{\cite{BeDe}} \\
There are natural isomorphisms 
\[
Z(\Rep (G)^{\mf s}) \cong Z(\End_G (\Pi_{\mf s})) \cong \mc O (\Irr (L)^{\mf s_L})^{W_{\mf s}} =
\mc O (\Irr (L)^{\mf s_L} / W_{\mf s} ) .
\]
\end{thm}

Recall from Definition \ref{def:1.E} that $W_{\mf s}^e$ is an extension of $W_{\mf s}$ by
$X_\nr (L,\sigma)$, which acts on $X_\nr (L)$. By construction, the quotient map
\[
X_\nr (L) \to \Irr (L)^{\mf s_L} : \chi \mapsto \sigma \otimes \chi
\]
induces a homeomorphism
\[
X_\nr (L) / W_{\mf s}^e \to \Irr (L)^{\mf s_L} / W_{\mf s} .
\]
Knowing this, Theorem \ref{thm:3.5} says that 
\begin{equation}\label{eq:3.8}
Z(\End_G (\Pi_{\mf s})) \cong \mc O (X_\nr (L))^{W_{\mf s}^e} = \mc O (X_\nr (L) / W_{\mf s}^e) .
\end{equation}
As we saw in \eqref{eq:1.23} and Theorem \ref{thm:1.19}, the group $W_{\mf s}^e$ acts on
the family of representations $I_P^G (\sigma \otimes \chi)$ with $\chi \in X_\nr (L)$, but by
operators that depend rationally on $\chi$ and may have poles. Moreover, from \eqref{eq:1.24}
and \eqref{eq:3.2} we see that in general this is only a projective action of $W_{\mf s}^e$.
Still, from these one can construct, as done in \cite[\S 4]{SolEnd}, an embedding
\begin{equation}\label{eq:3.6}
\C [W_{\mf s}^e, \natural_{\mf s}] \to 
\End_G (\Pi_{\mf s}) \otimes_{\mc O (X_\nr (L) / W_{\mf s}^e} \C (X_\nr (L) / W_{\mf s}^e) ,
\end{equation}
for a suitable 2-cocycle $\natural_{\mf s}$ generalizing \eqref{eq:1.24}.
This and the next result can be compared with Theorem \ref{thm:2.4}.

\begin{thm}\label{thm:3.6} \textup{\cite[Corollary 5.8]{SolEnd}} \\
The embeddings \eqref{eq:3.7} and \eqref{eq:3.6} combine to an algebra isomorphism
\[
\End_G (\Pi_{\mf s}) \otimes_{\mc O (X_\nr (L) / W_{\mf s}^e)}
\C (X_\nr (L) / W_{\mf s}^e) \cong \C (X_\nr (L)) \rtimes \C [W_{\mf s}^e, \natural_{\mf s}] .
\] 
\end{thm}

In Theorem \ref{thm:3.6} it is necessary to include the quotient field 
$\C (X_\nr (L) / W_{\mf s}^e)$, unlike in Theorem \ref{thm:3.3}, without that Theorem 
\ref{thm:3.6} would only hold in special cases. Since $W_{\mf s}^e$ acts faithfully on 
$X_\nr (L)$, $\C (X_\nr (L))$ has dimension $|W_{\mf s}^e|$ over 
$\C ( X_\nr (L) / W_{\mf s}^e)$. Theorem \ref{thm:3.6} shows that 
\[
\dim_{\C ( X_\nr (L) / W_{\mf s}^e)} \End_G (\Pi_{\mf s}) \otimes_{\mc O (X_\nr (L) / W_{\mf s}^e)}
\C (X_\nr (L) / W_{\mf s}^e) = |W_{\mf s}^e |^2 .
\]
This can be stated more precisely in terms of central simple algebras. Namely, it follows from
Theorems \ref{thm:2.11} and \ref{thm:3.6} that:

\begin{cor}\label{cor:3.7}
$\End_G (\Pi_{\mf s}) \otimes_{\mc O (X_\nr (L) / W_{\mf s}^e)} \C (X_\nr (L) / W_{\mf s}^e)$ is
a central simple algebra over $\C ( X_\nr (L) / W_{\mf s}^e)$. It is Morita equivalent to
$\C (X_\nr (L)) \rtimes W_{\mf s}^e$ if and only if $\natural_{\mf s}$ is trivial in 
$H^2 (W_{\mf s}^e, \C^\times)$.
\end{cor}

\subsection{Localization on the Bernstein centre}\label{par:3.3} \

For any $W_{\mf s}$-stable subset $U \subset \Irr (L)^{\mf s_L}$, one can consider
\[
\Rep (G)^{\mf s}_U = \{ \pi \in \Rep (G)^{\mf s} : \mr{Sc}(\pi') \in (L,U) 
\text{ for all irreducible subquotients } \pi' \text{ of } \pi \} .
\]
In view of Theorem \ref{thm:3.5}, this category can be obtained from $\Rep (G)^{\mf s}$ by
imposing conditions on how the Bernstein centre $Z(\Rep (G)^{\mf s})$ may act on the representations.

Often it is more prudent to restrict to finite length representations. That will be indicated by a subscript fl, so $\Rep_\fl (G)^{\mf s}_U$. When $U_i$, for $i$ in some (possibly
infinite) index set, are disjoint $W_{\mf s}$-invariant subsets of $\Irr (L)^{\mf s_L}$,
there is a decomposition
\begin{equation}\label{eq:3.13}
\Rep_\fl (G)^{\mf s}_{\cup_i U_i} = \bigoplus\nolimits_i \Rep_\fl (G)^{\mf s}_{U_i} ,
\end{equation}
This does not work with representations of arbitrary length, and it is an important 
reason why it is easier to work with representations of finite length.

To proceed, we make the relation between $\End_G (\Pi_{\mf s})$ and supercuspidal supports explicit.
Let $Q = M U_Q$ be a parabolic subgroup of $G$ containing $P = L U_P$, so that 
$\mf s_M = [L,\sigma]_M$ is defined. Then the functor $I_Q^G$ provides an embedding 
\[
\End_M (\Pi_{\mf s_M}) \to \End_G (\Pi_{\mf s}).
\]

\begin{lem}\label{lem:3.8} \textup{\cite[Lemma 5.1]{SolComp}} 
\enuma{
\item The equivalences of categories \eqref{eq:3.5} are compatible with parabolic induction: 
they form a commutative diagram
\[
\begin{array}{ccc}
\Rep (G)^{\mf s} & \cong & \End_G (\Pi_{\mf s}) - \Mod\\
\uparrow I_Q^G & & \uparrow \mr{ind} \\
\Rep (M)^{\mf s_M} & \cong & \End_M (\Pi_{\mf s_M}) - \Mod
\end{array}.
\]
\item The equivalences of categories \eqref{eq:3.5} are compatible with parabolic restriction: 
they form a commutative diagram
\[
\begin{array}{ccc}
\Rep (G)^{\mf s} & \cong & \End_G (\Pi_{\mf s}) - \Mod \\
\downarrow \mr{pr}_{\mf s_M} \circ J_{\overline Q}^G & & \downarrow \mr{Res} \\
\Rep (M)^{\mf s_M} & \cong & \End_M (\Pi_{\mf s_M}) - \Mod
\end{array}.
\]
Here $\mr{pr}_{\mf s_M} : \Rep (M) \to \Rep (M)^{\mf s_M}$ is the projection from the Bernstein
decomposition.
}
\end{lem} 

From Lemma \ref{lem:3.8}, \eqref{eq:3.9} and Theorem \ref{thm:3.5}, we see that the equivalence
of categories \eqref{eq:3.5} is compatible with supercuspidal supports, in the following sense.

\begin{cor}\label{cor:3.9}
Suppose that $\pi \in \Irr (G)^{\mf s}$ has supercuspidal support $(L,\sigma \otimes \chi)$. Then
$\Hom_G (\Pi_{\mf s},\pi) \in \Irr (\End_G (\Pi_{\mf s}))$ has central character 
$W_{\mf s} (\sigma \otimes \chi) \in \Irr (L)^{\mf s_L} / W_{\mf s}$, or equivalently
$W_{\mf s}^e \chi \in X_\nr (L) / W_{\mf s}^e$.
\end{cor}

Corollary \ref{cor:3.9} enables us to analyse the representation theory of $\End_G (\Pi_{\mf s})$ 
by putting conditions on the supercuspidal supports. For a $W_{\mf s}^e$-stable 
subset $U' \subset X_\nr (L)$ we define
\[
\End_G (\Pi_{\mf s}) - \Mod_{U'} = \{ V \in \End_G (\Pi_{\mf s}) - \Mod :
\text{ all } \mc O (X_\nr (L))\text{-weights of } V \text{ lie in } U' \} ,
\]
and its subcategory $\End_G (\Pi_{\mf s})- \Mod_{\fl,U'}$.

If $U$ equals $\{\sigma \otimes \chi : \chi \in U'\}$, then \eqref{eq:3.5} and Corollary 
\ref{cor:3.9} provide equivalences of categories
\begin{equation}\label{eq:3.10}
\Rep (G)^{\mf s}_U \cong \End_G (\Pi_{\mf s})- \Mod_{U'} \quad \text{and} \quad
\Rep_\fl (G)^{\mf s}_U \cong \End_G (\Pi_{\mf s})- \Mod_{\fl,U'} .
\end{equation}
When $U' \subset X_\nr (L)$ is open (with respect to the Zariski topology or with respect to the
analytic topology), we can analyse $\End_G (\Pi_{\mf s})-\Mod_{\fl,U'}$ by localizing 
$\End_G (\Pi_{\mf s})$ with respect to an ideal of 
$Z(\End_G (\Pi_{\mf s})) \cong \mc O (X_\nr (L) / W_{\mf s}^e)$
or by involving complex analytic functions on $U'$. For specific $U'$, this localization may 
be Morita equivalent to a localization of a simpler algebra.

From now we assume
\begin{equation}
\sigma \in \Irr_\cusp (L) \text{ is unitary (or equivalently tempered).}
\end{equation}
By Lemma \ref{lem:1.13}, that is no restriction on $\mf s$ or $\mf s_L$. We are interested
in the category $\Rep (G)^{\mf s}_{X_\nr^+ (L) W_{\mf s} \sigma}$, which will be related
to a twisted graded Hecke algebra.

Let $X^* (Z^\circ (L))$ be the lattice of $F$-rational characters $Z^\circ (\mc L) = Z (\mc L)^\circ \to GL_1$
and let $\Phi (G,Z^\circ (L))$ be the set of $\alpha \in X^* (Z^\circ (L))$ that appear in 
the adjoint action of $Z^\circ (L)$ on the Lie algebra of $G$. This is not necessarily a root 
system, but it is always a generalized root system in the sense of \cite{DiFi}. 

For a reduced root $\alpha \in \Phi (G,Z^\circ (L))$, let $L_\alpha$ be the reductive
group generated by $L$ and the root subgroups of $G$ associated to the multiples of $\alpha$.
Then $L$ is a maximal proper Levi subgroup of $L_\alpha$. We say that $\alpha$ belongs to 
$\Phi (G,Z^\circ (L))_\sigma$ if, for some $\chi_\sigma \in \Hom (L / Z(L_\alpha), 
\R_{>0}) \setminus \{1\}$, the representation $I^{L_\alpha}_{L (P \cap L_\alpha)} (\sigma 
\otimes \chi_\sigma)$ is reducible. By \cite[Corollary 1.3 and (1.8)]{SolParam} and
\cite[Proposition 2.13]{Hei2}, 
\[
\Phi (G,Z^\circ (L))_\sigma \subset X^* (Z^\circ (L))
\text{ is a reduced root system.} 
\]
For $\alpha \in \Phi (G,Z^\circ (L))_\sigma \subset X^* (Z^\circ (L))$, the above 
unramified character $\chi_\sigma$ is unique up
to inversion. It can be captured with one real number $q_{\sigma,\alpha}$, as follows. 
Put $L_\sigma^2 = \bigcap_{\chi \in X_\nr (L,\sigma)} \ker \chi$ and let $h_\alpha^\vee$ 
(a version of the coroot $\alpha^\vee$) be the generator of $(L_\sigma^2 \cap L_\alpha^1) 
/ L^1 \cong \Z$ from \cite[(A.2)]{SolEnd}. Then $q_{\sigma,\alpha} \in \R_{>1}$ is the 
unique number such that $\chi_\sigma (h_\alpha^\vee) \in \{q_{\sigma,\alpha}, 
q_{\sigma,\alpha}^{-1} \}$. For explicit computations of the numbers $q_{\sigma,\alpha}$ 
we refer to \cite{SolParam,Oha}.

We introduce the data for our twisted graded Hecke algebra.
\begin{itemize}
\item $\mf t = \mr{Lie}(X_\nr (L)) = \Hom (L,\C) \cong \Hom (Z(L),\C) 
\cong X^* (Z(L)) \otimes_\Z \C$.
\item $R_\sigma = \{ h_\alpha^\vee : \alpha \in \Phi (G,Z^\circ (L))_\sigma \} \subset
L_\sigma^2 / L^1$ is a root system by \cite[Proposition 3.1]{SolEnd}. The set
$R_\sigma^+$ of $h_\alpha^\vee$ for which $\alpha$ appears in Lie$\,U_P$ is a positive
system in $R_\sigma$.
\item The parameter function $k_\sigma : R_\sigma \to \R_{>0}$ given by
$k_\sigma (h_\alpha^\vee) = \log (q_{\sigma,\alpha})$.
\item The finite group $W_{\mf s, \sigma}$ acts on $R_\sigma$. It can be written as
\begin{equation}\label{eq:3.21}
W_{\mf s,\sigma} = W(R_\sigma) \rtimes \Gamma_\sigma ,\quad \text{where } \Gamma_\sigma \text{ is the stabilizer of } R_\sigma^+ .
\end{equation}
\item The 2-cocycle $\natural_\sigma$ given by the multiplication rules for the 
intertwining ope\-ra\-tors in \eqref{eq:1.24}. It records how much $w \mapsto I(w,L (P \cap
L_\alpha), \sigma, \chi = 1)$ deviates from a group homomorphism. We note that 
$\natural_\sigma$ depends on a normalization of these intertwining ope\-ra\-tors.
\end{itemize}
Let $\mh H (\mf t,W_{\mf s,\sigma},k_\sigma, \natural_\sigma)$ be the algebra
as in Definition \ref{def:2.B}, determined by the above data.

\begin{ex}\label{ex:3.C}
Consider $G = SL_2 (F)$, the diagonal torus $T = L$ and $\sigma = \mr{triv}_T$.
In this setting $\mf t = \Hom (T,\C^\times) \cong \C$, $W_{\mf s} = N_G (T)/T
\cong S_2$ and the group $\Gamma_{\sigma'}$ is trivial for all $\sigma' \in 
X_\nr (T)$. For various $\sigma'$, the algebras $\mh H (\mf t, W_{\mf s,\sigma'},
k_{\sigma'})$ are:
\begin{itemize}
\item $\mh H (\C , S_2, \log (q_F))$ for $\sigma' = \mr{triv}_T$,
\item $\mh H (\C, S_2, 0) = \mc O (\C) \rtimes S_2$ for $\sigma' \in X_\nr (T)$ quadratic,
\item $\mh H (\C, \{e\},0) = \mc O (\C)$ for other $\sigma' \in X_\nr (T)$.
\end{itemize}
\end{ex}

For open $U \subset X_\nr (L)$, let $C^{an}(U)$ be the algebra of complex
analytic functions on $U$. We define the analytic localization of 
$\End_G (\Pi_{\mf s})$ on $U$ as
\begin{equation}\label{eq:3.22}
\End_G (\Pi_{\mf s}) \otimes_{\mc O (X_\nr (L))} C^{an}(U) .
\end{equation}
This is an algebra if $U$ is $W_{\mf s}^e$-stable.

We note that the map
\[
\exp_\sigma : \mf t \to \Irr (L)^{\mf s_L} ,\quad x \mapsto \exp (x) \otimes \sigma
\]
restricts to a diffeomorphism from $\mf t_\R := X^* (Z^\circ (L))$ to $X_\nr^+ (L) \sigma$.
Analytic localization of $\End_G (\Pi_{\mf s})$ on a small tabular open neighborhood of
$X_\nr^+ (L) W_{\mf s}$ in $\Irr (L)^{\mf s_L}$ can be compared with analytic localization
of $\mh H (\mf t, W_{\mf s,\sigma}, k_\sigma, \natural_\sigma)$ on a small tubular open
neighborhood $U_\sigma$ of $\mf t_\R$ in $\mf t$, an algebra of the form
\begin{equation}\label{eq:3.15}
\mh H (\mf t, W_{\mf s,\sigma}, k_\sigma, \natural_\sigma) 
\otimes_{\mc O (\mf t / W_{\mf s,\sigma})} C^{an}(U_\sigma)^{W_{\mf s,\sigma}} .
\end{equation}
Arguments involving these localizations lead to:

\begin{thm}\label{thm:3.10} \textup{\cite[Corollary 8.1]{SolEnd}} \\
There are equivalences of categories
\[
\Rep_\fl (G)^{\mf s}_{X_\nr^+ (L) W_{\mf s} \sigma} \cong 
\End_G (\Pi_{\mf s})-\Mod_{\fl, X_\nr^+ (L) W_{\mf s} \sigma} \cong
\mh H (\mf t, W_{\mf s,\sigma}, k_\sigma, \natural_\sigma) - \Mod_{\fl,\mf t_\R} 
\]
such that
\enuma{
\item Once a 2-cocycle $\natural_\sigma$ has been fixed (by a normalization of the 
involved inter\-twi\-ning operators), the equivalences are canonical.
\item The equivalences preserve temperedness.
\item $\pi \in \Irr (G)^{\mf s}_{X_\nr^+ (L) W_{\mf s} \sigma}$ is discrete series
if and only if the image of $\pi$ in\\
$\Irr \big( \mh H (\mf t, W_{\mf s,\sigma}, k_\sigma, \natural_\sigma) \big)$ 
is discrete series and $\mr{rk}(R_\sigma)$ equals $\dim_F (Z(L) / Z(G))$.
\item The equivalences are compatible with parabolic induction and restriction,
in the same sense as Lemma \ref{lem:3.8}.
\item For $\pi \in \Irr (G)^{\mf s}_{X_\nr^+ (L) W_{\mf s} \sigma}$ with $\mr{Sc}(\pi)$
represented by $\sigma \otimes \chi \in X_\nr^+ (L) \sigma$, the associated
$\mh H (\mf t, W_{\mf s,\sigma}, k_\sigma, \natural_\sigma)$-module has central
character $W_{\mf s,\sigma} \log \chi \in \mf t_\R / W_{\mf s,\sigma}$. Equivalently,
$\exp_\sigma$ translates central characters into supercuspidal supports.
}
\end{thm}

Via Lemma \ref{lem:1.5}, $\Rep_\fl (G)^{\mf s}_{X_\nr^+ (L) W_{\mf s} \sigma}$ is
equivalent to the category of finite length $\mc H (G)^{\mf s}$-modules all whose
supercuspidal supports lie in $X_\nr^+ (L) W_{\mf s} \sigma / W_{\mf s}$. Theorem
\ref{thm:3.10} is our final answer to the question about the structure of 
$\mc H (G)^{\mf s}$ and $\End_G (\Pi_{\mf s})$, in terms of their module categories.

\subsection{The ABPS conjecture} \

The big advantage of Theorem \ref{thm:3.10} is that it allows us to study
$G$-representations via graded Hecke algebras. In particular all the results from
Section \ref{sec:2} can now be applied to $\Rep_\fl (G)^{\mf s}_{X_\nr^+ (L) W_{\mf s}
\sigma}$. Let $R (G)^{\mf s}_{X_\nr^+ (L) W_{\mf s} \sigma}$ be the Grothendieck
group of $\Rep_\fl (G)^{\mf s}_{X_\nr^+ (L) W_{\mf s} \sigma}$. From \eqref{eq:2.E} and 
Theorems \ref{thm:2.8}, \ref{thm:2.9} and \ref{thm:3.10} we conclude:

\begin{thm}\label{thm:3.11}
Fix a 2-cocycle $\natural_{\mf \sigma}$ as in Theorem \ref{thm:3.10}. 
\enuma{
\item There exist canonical group isomorphisms
\[
R (G)^{\mf s}_{X_\nr^+ (L) W_{\mf s} \sigma} \longleftrightarrow 
R (\End_G (\Pi_{\mf s})^{op} )_{X_\nr^+ (L) \sigma} \longleftrightarrow
R \big( \mh H (\mf t, W_{\mf s,\sigma}, k_\sigma, \natural_\sigma^{-1} ) \big)_{\mf t_\R} .
\]
These isomorphisms are compatible with parabolic induction and they 
preserve temperedness. 
\item Part (a) can be refined canonically to a bijection
\[
\Irr (G)^{\mf s}_{X_\nr^+ (L) W_{\mf s} \sigma} \longleftrightarrow 
\Irr \big( \mc O (\mf t) \rtimes \C[W_{\mf s,\sigma},\natural_\sigma^{-1}] \big)_{\mf t_\R} .
\]
It preserves temperedness, but it need not respect parabolic induction or
supercuspidal supports/central characters.
}
\end{thm}

Next we want to combine instances of Theorem \ref{thm:3.11} for all $\sigma' = \chi_u \otimes 
\sigma \in \Irr_\temp (L)^{\mf s_L}$. Recall from Theorem \ref{thm:3.6} and \eqref{eq:2.E} that
\[
\End_G (\Pi_{\mf s})^{op} \otimes_{\mc O (X_\nr (L) / W_{\mf s}^e)}
\C (X_\nr (L) / W_{\mf s}^e) \cong \C (X_\nr (L)) \rtimes \C [W_{\mf s}^e, \natural_{\mf s}^{-1}] .
\] 
By \cite[Lemma 7.1]{SolEnd} and Clifford theory, there are equivalences of categories
\begin{equation}\label{eq:3.12}
\begin{aligned}
\mc O (\mf t) \rtimes \C [W_{\mf s,\sigma'},\natural_{\sigma'}] - \Mod_{\fl, \mf t_\R} & \cong
\mc O (X_\nr (L) \rtimes \C [W_{\mf s,\sigma'}, \natural_{\sigma'}] - \Mod_{\fl,\chi_u X_\nr^+ (L)} \\
& \cong \mc O (X_\nr (L) \rtimes \C [W_{\mf s}^e, \natural_{\mf s}] 
- \Mod_{\fl,W_{\mf s}^e \chi_u X_\nr^+ (L)} .
\end{aligned}
\end{equation}

\begin{thm}\label{thm:3.13} 
\textup{\cite[Theorem 2.5]{SolHH}} \\
Theorem \ref{thm:3.11}.a and \eqref{eq:3.12} induce a canonical group isomorphism
\[
\zeta^\vee : R(G)^{\mf s} \to 
R \big( \mc O (X_\nr (L)) \rtimes \C [W_{\mf s}^e, \natural_{\mf s}^{-1}] \big)
\]
with the following properties:
\enuma{
\item $\zeta^\vee$ and its inverse preserve temperedness. Moreover $\zeta^\vee$ sends 
tempered representations to tempered representations (so not to virtual representations).
\item If $\chi_u \in X_\nr^u (L)$ and all irreducible subquotients of $\pi \in R(G)^{\mf s}$ have cuspidal
support in $\sigma \otimes W_{\mf s}^e \chi_u X_\nr^+ (L)$, then all $\mc O (X_\nr (L))$-weights of 
$\zeta^\vee (\pi)$ lie in $W_{\mf s}^e \chi_u X_\nr^+ (L)$.
\item In the setting of (b), suppose that $\pi \in \Irr_\temp (G)^{\mf s}$. Then 
\[
\zeta^\vee (\pi) = \ind_{\mc O (X_\nr (L)) \rtimes \C [(W_{\mf s}^e)_{\chi_u}, \natural_{\mf s}^{-1}]}^{
\mc O (X_\nr (L)) \rtimes \C [W_{\mf s}^e, \natural_{\mf s}^{-1}]} (\chi_u \otimes \pi_{\chi_u}), 
\]
where $\pi_{\chi_u} \in \Mod - \C [(W_{\mf s}^e)_{\chi_u}, \natural_{\mf s}^{-1}]$ is obtained 
like in Theorem \ref{thm:2.8}.a. 
\item $\zeta^\vee$ commutes with parabolic induction and with unramified twists, in the sense that for a parabolic subgroup $Q = M U_Q \supset P$ and $\chi \in X_\nr (M)$:
\[
\zeta^\vee (I_Q^G (\tau \otimes \chi)) = \ind_{\mc O (X_\nr (L)) \rtimes \C [W_{\mf s_M}^e, 
\natural_{\mf s}^{-1}]}^{\mc O (X_\nr (L)) \rtimes \C [W_{\mf s}^e, \natural_{\mf s}^{-1}]}
(\zeta^\vee_M (\tau) \otimes \chi |_L ) .
\]
}
\end{thm}
We point out that in Theorem \ref{thm:3.13} the canonicity holds after $\natural_{\mf s}$ has 
been fixed, the choice of the 2-cocycle $\natural_{\mf s}$ is (in general) not canonical.

Recall from \eqref{eq:2.16} that the irreducible representations of twisted
crossed product algebras can be parametrized by twisted extended quotients.
Combining Theorems \ref{thm:3.13}, \ref{thm:3.11}.b and Theorem \ref{thm:2.3}.b leads to:

\begin{thm}\label{thm:3.12} \textup{\cite[Theorem 9.9]{SolEnd}} 
\enuma{
\item There exists a family $\natural$ of 2-cocycles $\natural_\sigma^{-1} \;
(\sigma \in \Irr_\temp (L)^{\mf s_L})$ and bijections 
\[
\begin{array}{lll}
\Irr (G)^{\mf s}_{X_\nr^+ (L) W_{\mf s} \sigma} & \longleftrightarrow &
(\mf t_\R /\!/ W_{\mf s,\sigma} )_{\natural_\sigma^{-1}} \\
\Irr (G)^{\mf s} & \longleftrightarrow & ( \Irr (L)^{\mf s_L} /\!/ W_{\mf s})_\natural \\
\Irr_\temp (G)^{\mf s} & \longleftrightarrow & 
( \Irr_\temp (L)^{\mf s_L} /\!/ W_{\mf s})_\natural 
\end{array}
\]
\item For $\sigma' \in X_\nr^+ (L) \sigma$ we write $\natural_{\sigma'} = 
\natural_\sigma$. The bijections from part (a) combine to a bijection from $\Irr (G)$ 
to the set of $G$-orbits in
\[
\big\{ (L,\sigma',\rho) : L \subset G \text{ Levi subgroup, } \sigma' \in \Irr_\cusp (L),
\rho \in \Irr (\C [\mr{Stab}_G (L,\sigma')/L, \natural_{\sigma'}^{-1}]) \big\} .
\]
}
\end{thm}

Theorem \ref{thm:3.12} proves a version of the ABPS conjecture, as formulated in
\cite[\S 2.3]{ABPS2}. This was built upon several earlier versions of the conjecture,
starting with \cite[Conjecture 1]{ABP}.

Notice that the last space in Theorem \ref{thm:3.12}.b projects naturally onto the 
variety of supercuspidal supports for $G$, by forgetting the $\rho$'s. The fibers 
of that map are finite, and parametrized by 
$\Irr (\C [\mr{Stab}_G (L,\sigma')/L, \natural_{\sigma'}^{-1}])$. However, like in 
Theorem \ref{thm:3.11}.b, the bijections in Theorem \ref{thm:3.12} do not always 
translate the supercuspidal support map for $G$-representations into the natural 
projection $(L,\sigma',\rho) \mapsto (L,\sigma')$. In general it has to be corrected 
by some element of $X_\nr^+ (L)$, the absolute value of a $Z^\circ (L)$-weight of some 
discrete series representation of $L$.

We may regard $\mr{Stab}_G (L,\sigma')/L$ as $\pi_0 (\mr{Stab}_G (L,\sigma') )$,
where $\pi_0$ is meant as algebraic varieties over $F$. We can modify the notion
of a twisted extended quotient by replacing stabilizers by component groups of
stabilizers. (That does not make a difference if we divide by finite groups like 
we did so far.) Let us denote such twisted extended quotients by $/\!/_\natural$.
Then Theorem \ref{thm:3.12}.b can be reformulated as a bijection
\begin{equation}\label{eq:3.11}
\Irr (G) \longleftrightarrow \Big( \bigsqcup\nolimits_{L \subset G \text{ Levi subgroup}}
\Irr_\cusp (L) \Big) /\!/_\natural \, G .
\end{equation}
This is a remarkably simple way to parametrize all irreducible smooth 
$G$-repre\-sen\-tations in terms of the supercuspidal representations of its Levi subgroups. 

\newpage

\part{Noncommutative geometry of reductive $p$-adic groups}

\section{Hochschild homology for algebras}\label{sec:4}

Hochschild (co)homology, which appeared first in \cite{Hoc}, is a (co)homology theory
for associative algebras over commutative rings. (Nowadays there are notions of
Hochschild homology in categorical settings, but they are outside the scope of this paper.)
We focus on algebras over fields, to make the definitions easier.

Let $A$ be a unital algebra over a field $k$ and let $A^{op}$ be the opposite algebra.
Let $M$ be an $A$-bimodule, or equivalently a $A \otimes_k A^{op}$-module.

\begin{defn}\label{def:4.1}
Let $n \in \Z_{\geq 0}$. The $n$th homology of $A$ with coefficients in $M$ is
\[
H_n (A,M) = \Tor_n^{A \otimes_k A^{op}} (A,M) .
\]
The $n$-th Hochschild homology of $A$ is
\[
HH_n (A) = H_n (A,A) = \Tor_n^{A \otimes_k A^{op}} (A,A) .
\]
\end{defn}

One can compute $HH_n (A)$ as the homology of an explicit differential complex
$(A^{\otimes m}, d_m )_{m=1}^\infty$ \cite[\S 1.1]{Lod}. This shows that $HH_n$ is a
functor from unital $k$-algebras to $k$-vector spaces. We abbreviate
\[
HH_* (A) = \bigoplus\nolimits_{n=0}^\infty HH_n (A) .
\]
There are two good ways to generalize Definition \ref{def:4.1} to non-unital algebras.
\begin{itemize}
\item Assume that for every finite set $S \subset A$ there exists an idempotent
$e_S \in A$ such that $e_S a = a = s e_S$ for all $a \in S$. Such an algebra is
called locally unital. For locally unital algebras we can still use Definition \ref{def:4.1},
without significant changes. This will be useful for Hecke algebras of reductive
$p$-adic groups.
\item Let $A_+$ be the vector space $A \oplus k$ with multiplication
\begin{equation}\label{eq:4.20}
(a_1,k_1) (a_2,k_2) = (a_1 a_2 + k_1 a_2 + k_2 a_1, k_1 k_2).
\end{equation}
This algebra has unit $(0,1)$, contains $A$ as an ideal and is called the unitization
of $A$. We note that $A \mapsto A_+$ is a functor. We define
\[
HH_n (A) = \mathrm{coker} (HH_n (k) \to HH_n (A_+)) ,
\]
where the map is induced by the inclusion $k \to A_+$.
\end{itemize}

\subsection{Basic properties of $HH_*$} \label{par:4.1} \

In this paragraph $A$ is a locally unital algebra and $M$ is an $A$-bimodule.\\
 
\textbf{1. Degree zero.} 
One can compute the zeroth Hochschild homology groups directly from the definition:
\[
H_0 (A,M) = A \otimes_{A \otimes_k A^{op}} M = M / [M,A] ,
\]
where $[M,A] \subset M$ is the $k$-span of $\{ m a - a m : m \in M, a \in A \}$. In particular
\[
HH_0 (A) = A / [A,A] ,
\]  
a vector space known as the cocenter of $A$.

Let $R_f (A)$ be the Grothendieck group of the category of finite dimensional $A$-representations.
There is a natural pairing
\[
\begin{array}{ccc}
HH_0 (A) \times R_f (A) & \to & k \\
(a, \pi) & \mapsto & \mr{tr} \, \pi (a) 
\end{array}.
\]
This is the main use of Hochschild homology in representation theory.\\

\textbf{2. $HH_*$ as derived functor.} 
Let $A \leftarrow P_*$ be a projective $A \otimes_k A^{op}$-module resolution. By Definition \ref{def:4.1}
\begin{equation}\label{eq:4.1}
H_n (A,M) = H_n (P_* \otimes_{A \otimes_k A^{op}} M) .
\end{equation}
For example, if $A = k$, then we can take the resolution $k \leftarrow P_0 = k \leftarrow P_1 = 0$.
With that we find
\begin{equation}\label{eq:4.2}
HH_n (k) = H_n (P_0 \leftarrow 0) = \begin{cases} k & n=0 \\ 0 & n > 0 \end{cases} .
\end{equation}
By 1. above, the Hochschild homology functors $HH_n$ can be regarded as derived functors of the cocenter
functor. Loosely speaking, that means that $HH_*$ is the universal derived functor with the tracial
property that it kills all commutators. Thus $HH_* (A)$ is a rather subtle invariant of $A$, which
contains a lot of information. On the other hand, it is often quite difficult to determine the
Hochschild homology of an algebra. \\

\textbf{3. Additivity.} 
For another locally unital algebra $B$, there is a natural isomorphism \cite[Theorem 1.2.15]{Lod}
\[
HH_n (A \oplus B) \cong HH_n (A) \oplus HH_n (B) .
\] 

\textbf{4. Continuity.}
Consider a direct limit of locally unital algebras $\lim_i A_i$. The properties of Tor
yield a natural isomorphism
\[
HH_n (\lim\nolimits_i A_i) \cong \lim\nolimits_i HH_n (A_i) .
\]

\textbf{5. Module structure over the centre.}
When $A$ is unital, $H_* (A,M)$ and $HH_n (A)$ are naturally modules over the centre $Z(A)$.
This follows for instance from the $Z(A)$-module structure of $P_* \otimes_{A \otimes_k A^{op}} M$
in \eqref{eq:4.1}. \\

\textbf{6. Morita invariance.}
The Hochschild homology of $A$ depends only on the category of left $A$-modules. More precisely,
suppose we have an equivalence of categories $\Mod (A) \isom \Mod (B)$. Then $A$ and $B$ are 
Morita equivalent and there exist projective bimodules $P$ and $Q$ implementing the equivalence of
categories. These also yield an equivalence of categories
\[
\begin{array}{ccc}
\Mod (A \otimes_k A^{op}) & \isom & \Mod (B \otimes_k B^{op}) \\
M & \mapsto & P \otimes_A M \otimes_A Q
\end{array} ,
\]
which sends $A$ to $P \otimes_A Q \cong B$. This induces an isomorphism 
\begin{equation}\label{eq:4.3}
HH_n (A) \isom HH_n (B) ,
\end{equation}
see \cite[Theorem 1.2.7]{Lod}. If $A$ is unital, then $B$ is naturally a $Z(A)$-algebra
and \eqref{eq:4.3} is $Z(A)$-linear.

\begin{ex}\label{ex:4.A}
Let $\Gamma$ be a finite group and consider $A = \C [\Gamma]$. With additivity, Morita invariance and 
\eqref{eq:4.2} we compute
\begin{align*}
HH_n (\C [\Gamma]) & \cong HH_n \big( \bigoplus\nolimits_{\pi \in \Irr (\Gamma)} \End_\C (V_\pi) \big) \cong
\bigoplus\nolimits_{\pi \in \Irr (\Gamma)} HH_n \big( \End_\C (V_\pi) \big) \\
& \cong \bigoplus\nolimits_{\pi \in \Irr (\Gamma)} HH_n (\C) \cong 
\begin{cases} \bigoplus_{\pi \in \Irr (\Gamma)} \C & n = 0 \\ 0 & n > 0 \end{cases}.
\end{align*}
More concisely: $HH_0 (\C [\Gamma]) \cong Z (\C [\Gamma])$ and $HH_n (\C [\Gamma]) = 0$ for $n > 0$.
The $Z(A)$-module structure is just multiplication.
\end{ex}

\textbf{7. Localization.}
We suppose that $A$ is unital. Let $S \subset Z(A)$ be subset which is closed under multiplication,
and contains 1 but not 0. Recall that the localization of $A$ with respect to $S$ is
\[
S^{-1} A = \{ s^{-1} a : s \in S, a \in A \},
\]
where $s^{-1} a$ means $(s^{-1},a)$ modulo the equivalence relation $(s_1^{-1} s_2 ,b) \sim
(s_1^{-1}, s_2 b)$ for $s_1, s_2 \in S$ and $b \in A$. The operations in $S^{-1} A$ are like in the
quotient field of a domain. Then $S^{-1} A$ is a $k$-algebra and $S^{-1} M$ is an $S^{-1} A$-bimodule.
Since localization is an exact functor, there are natural isomorphisms \cite[Proposition 1.1.17]{Lod}
\begin{equation}\label{eq:4.4}
\begin{aligned}
& S^{-1} H_n (A,M) \cong H_n (S^{-1} A, S^{-1}M),\\
& S^{-1} HH_n (A) \cong HH_n (S^{-1} A) .
\end{aligned}
\end{equation}
This can be used to reduce the determination of $HH_* (A)$ to a local problem on the spectrum
of $Z(A)$.

\subsection{Hochschild homology of some commutative algebras} \

We start with an example that shows very regular behaviour.

\begin{ex}\label{ex:4.B}
Take $A = k [x]$, so that $A \otimes_k A^{op} \cong k[x,y]$. There is a projective 
bimodule resolution
\[
k[x,y] \xleftarrow{\mod (x-y)} P_0 = k[x,y] \xleftarrow{\text{mult } x-y} P_1 = k[x,y] \leftarrow P_2 = 0 .
\]
With \eqref{eq:4.1} we find
\[
HH_n (k[x]) = H_n \big( P_* \otimes_{k[x,y]} k[x] \big) = H_n \big( k[x] \xleftarrow{0} k[x] \big) =
\begin{cases} k[x] & n = 0,1 \\ 0 & n > 1 \end{cases}
\]
\end{ex}

However, in general the Hochschild homology of $A$ may be nonzero in degrees far above the Krull dimension of $A$.

\begin{ex}\label{ex:4.C}
Consider $B = k[x]/(x^2)$ and the $B$-bimodules $P_n = k[x,y]/(x^2,y^2) \cong B \otimes_k B^{op}$ for all $n \geq 0$.
They form a projective resolution
\[
B \xleftarrow{\mod (x-y)} P_0 \xleftarrow{\text{mult } x-y} P_1 \xleftarrow{\text{mult } x+y} P_2
\xleftarrow{\text{mult } x-y} P_3 \xleftarrow{\text{mult } x+y} \cdots
\]
We compute, assuming that the characteristic of $k$ is not 2:
\[
HH_n (B) = H_n \big( B \xleftarrow{0} B \xleftarrow{2x} B \xleftarrow{0} B \xleftarrow{2x} \cdots \big) =
\begin{cases} k[x]/(x^2) & n = 0\\
k[x] / (x) = k & n \text{ odd} \\
x k[x] / (x^2) = k x & n > 0 \text{ even}
\end{cases}.
\]
\end{ex}
The difference between Examples \eqref{ex:4.B} and \eqref{ex:4.C} is that $A = k[x]$ is the
coordinate ring of a non-singular algebraic variety, while $B = k[x]/(x^2)$ is not reduced. 
Example \eqref{ex:4.B} is a
simple case of the Hochschild--Kostant--Rosenberg theorem:

\begin{thm}\label{thm:4.2} \textup{\cite{HKR}} \\
Let $V$ be a nonsingular affine variety over an algebraically closed field $k$.
Let $\mc O (V)$ be the $k$-algebra of regular functions on $V$ and let $\Omega^n (V)$
be the $\mc O (V)$-module of differential $n$-forms on $V$.

There is a natural isomorphism of $\mc O (V)$-modules $HH_n (\mc O (V)) \cong \Omega^n (V)$.
\end{thm}
 
Theorem \ref{thm:4.2} provides an interpretation of $HH_* (A)$: it is a kind of 
differential forms on $\Irr (A)$. For commutative algebras that can be made quite precise
\cite[\S 1.3]{Lod}, while for non-commutative algebras it means that $HH_* (A)$ can be
regarded as some "non-commutative differential forms". However, this interpretation only 
applies well to nice algebras, for Example \ref{ex:4.B}
shows that $HH_* (A)$ is very sensitive to changes like non-reducedness.

\subsection{Hochschild homology of finite type algebras} \

Besides commutative algebras, there are some classes of algebras which are close to
commutative and whose Hochschild homology can be computed reasonably well. 

Let $V$ be a complex affine variety. Recall that an $\mc O (V)$-algebra is a $\C$-algebra
$A$ with an algebra homomorphism from $\mc O (V)$ to the centre of the multiplier algebra
of $A$. In other words, $\mc O (V)$ acts on $A$ and
\[
f (a_1 a_2) = f(a_1) a_2 = a_1 f(a_2) \quad \text{for all } f \in \mc O (V), a_i \in A.  
\]
For any unital $\mc O (V)$-algebra $A$, the action of $\mc O (V)$ comes from an algebra
homomorphism $\mc O (V) \to Z(A)$. If $B$ is a nonunital $\mc O (V)$-algebra, then $B \oplus 
\mc O (V)$ has the structure of a unital $\mc O (V)$-algebra, like $B_+$ in \eqref{eq:4.20}.

\begin{defn}\label{def:4.F}
We say that an $\mc O (V)$-algebra $A$ has finite type if $A$ has finite rank as 
$\mc O (V)$-module. An arbitrary $\C$-algebra has finite type if it is a finite type
$\mc O (V)$-algebra for some complex affine variety $V$.
\end{defn}

\begin{ex}
$M_n (\mc O (V))$ is a finite type $\mc O (V)$-algebra.

Let $K \in \CO (G)$ and let $\mf s = [L,\sigma]_G$.
By Theorem \ref{thm:1.19}, the unital algebra $\mc H (G,K)^{\mf s}$ from \eqref{eq:1.50}
has finite rank over 
\[
\mc O (X_\nr (L) / W_{\mf s}^e) \cong \mc O (\Irr (L)^{\mf s_L}) \cong Z (\Rep (G)^{\mf s}).
\]
The algebra $\mc H (G,K)$ is a finite type $\mc O (V)$-algebra for
$V = \sqcup_{\mf s} \, \Irr (L)^{\mf s_L} / W_{\mf s}$, where $\mf s$ runs over the finite
subset of $\mf B (G)$ such that $\mc H (G,K)^{\mf s}$ is nonzero.   
\end{ex}

Standard algebraic techniques like those described in Paragraph \ref{par:4.1} can be used
to study Hochschild homology of finite type algebras, and to some extent reduce it to the
case of commutative algebras \cite[\S 2--3]{KNS}. If $A$ is a unital finite type algebra 
and $n \in \Z_{\geq 0}$, then $HH_n (A)$ is an $\mc O (V)$-module of finite rank.

Consider a non-singular affine $\C$-variety $V$, and let $\Gamma$ be a finite group acting 
on $V$ by automorphisms. 
Brylinski \cite{Bry} and Nistor \cite{Nis} generalized the Hochschild--Kostant--Rosenberg
theorem to the crossed product $\mc O (V) \rtimes \Gamma = \mc O (V) \rtimes \C [\Gamma]$
from Paragraph \ref{par:tcp}. 

More generally we may involve a 2-cocycle $\natural : \Gamma \times \Gamma \to \C^\times$. By Lemma \ref{lem:2.2}, the twisted crossed product $\mc O (V) \rtimes \C [\Gamma, \natural]$, 
is a finite type $\mc O (V / \Gamma)$-algebra. We recall the map
$\natural^\gamma : \Gamma \to \C^\times$ from \eqref{eq:2.C}. 

\begin{thm}\label{thm:4.3} \textup{\cite[Theorem 1.2 and (1.17)]{SolTwist}} \\
There exists an isomorphism of $\mc O (V)^\Gamma$-modules
\[
HH_n \big( \mc O (V) \rtimes \C [\Gamma, \natural] \big) \cong  \bigoplus_{\gamma \in \Gamma / 
\mathrm{conjugacy}} (\Omega^n (V^\gamma) \otimes \natural^\gamma)^{Z_\Gamma (\gamma)}
\cong \Big( \bigoplus_{\gamma \in \Gamma} \Omega^n (V^\gamma) \otimes \natural^\gamma \Big)^\Gamma .
\]
\end{thm}

This generalizes (and relies on) the aforementioned result of Brylinski--Nistor, which can be
recovered by setting $\natural = 1$. In view of Theorems \ref{thm:2.3}, \ref{thm:4.2} and 
\ref{thm:4.3}, we may regard $\big( \bigoplus_{\gamma \in \Gamma} \Omega^n (V^\gamma) \otimes 
\natural^\gamma \big)^\Gamma$ as a kind of differential forms on the twisted extended quotient
$(V /\!/ \Gamma )_\natural$.

\subsection{Hochschild homology of graded Hecke algebras} \

Let $\mh H (\mf t, W ,k)$ be a graded Hecke algebra, as in \eqref{eq:2.D}, and let $\mh H
(\mf t, W \Gamma, k, \natural) = \mh H (\mf t, W, k) \rtimes \C [\Gamma, \natural]$ be 
a twisted graded Hecke algebra, as in Definition \ref{def:2.B}. Recall from Theorem \ref{thm:2.4}
that the algebras $\mh H (\mf t, W\Gamma, k', \natural)$ with varying parameters $k' : \Phi \to
\C$ are all very similar, and that for $k' = 0$ we recover the simpler algebra $\mc O (\mf t) \rtimes
\C [W\Gamma,\natural]$. That can also be seen in Hochschild homology:

\begin{thm}\label{thm:4.4} \textup{\cite[Theorem 3.4]{SolHomGHA} and \cite[(2.5)--(2.6)]{SolTwist}} \\
There exist isomorphisms of vector spaces
\[
\begin{array}{lclcl}
HH_* (\mh H (\mf t, W, k)) & \cong & HH_* (\mh H (\mf t, W, 0)) &
\cong & \big( \bigoplus_{w \in W} \Omega^* (\mf t^w) \big)^W \\
HH_* (\mh H (\mf t, W\Gamma, k, \natural)) & \cong & HH_* (\mh H (\mf t, W\Gamma, 0, \natural)) & 
\cong & \big( \bigoplus_{w \in W \Gamma} \Omega^* (\mf t^w) \otimes \natural^w \big)^{W\Gamma}
\end{array}
\]
\end{thm}

Theorem \ref{thm:4.4} comes from the filtration of $\mh H (\mf t, W\Gamma, k, \natural)$ by degrees,
and from an associated spectral sequence that converges to $HH_* (\mh H (\mf t, W\Gamma, k, \natural))$.
We point out that the isomorphisms in Theorem \ref{thm:4.4} are usually not linear over $\mc O (\mf t)^W$
or $\mc O (\mf t)^{W \Gamma}$. 

From now on we assume that $k$ is real-valued, like in Paragraph \ref{par:2.4}. Then we have the natural
bijection
\[
\zeta_0 : R ( \mh H (\mf t, W\Gamma, k, \natural) ) \to R (\mc O (\mf t) \rtimes \C [W\Gamma, \natural])
\]
from Theorems \ref{thm:2.8} and \ref{thm:2.9}.

\begin{thm}\label{thm:4.5} \textup{\cite[Corollary 2.10 and Proposition 2.11]{SolTwist}} \\
$\zeta_0$ induces a natural $\C$-linear bijection
\[
HH_* (\zeta_0) : HH_* (\mc O (\mf t) \rtimes \C [W\Gamma,\natural]) \to 
HH_* (\mh H (\mf t,W\Gamma,k,\natural)) .
\]
\end{thm}

The map $HH_* (\zeta_0)$ can be characterized as follows. For any tempered $\sigma \in \Irr (\mh H_P)$
there is an algebraic family of $\mh H$-representations
\[
\mf F_{P,\sigma} = \big\{ \ind_{\mh H^P}^{\mh H} (\sigma \otimes \C_\lambda) : 
\lambda \in P^{\vee \perp} \subset \mf t \big\} .
\]
This family gives rise to an algebra homomorphism
\[
\mc F_{P,\sigma} : \mh H = \mh H (\mf t,W\Gamma,k,\natural) \to
\mc O (P^{\vee \perp}) \otimes \End_\C \big( \ind_{\mh H^P}^{\mh H} (\sigma \otimes \C_0) \big) .
\]
By Morita invariance and Theorem \ref{thm:4.2} we have 
\[
HH_* \big( \mc O (P^{\vee \perp}) \otimes \End_\C \big( \ind_{\mh H^P}^{\mh H} (\sigma \otimes \C_0) \big) \big)
\cong HH_* \big( \mc O (P^{\vee \perp}) \big) \cong \Omega^* (P^{\vee \perp}) .
\]
The same can be done with $\mc F_{P,\zeta_0 (\sigma)}$, and in Theorem \ref{thm:4.5} there is an equality
\begin{equation}\label{eq:4.5}
HH_* (\mc F_{P,\sigma}) \circ HH_* (\zeta_0) = HH_* (\mc F_{P,\zeta_0 (\sigma)}) :
HH_* (\mc O (\mf t) \rtimes \C [W\Gamma,\natural]) \to \Omega^* (P^{\vee \perp}) .
\end{equation}
Conversely, imposing \eqref{eq:4.5} for all families of the form $\mf F_{P,\sigma}$ determines $HH_* (\zeta_0)$.

\begin{ex}\label{ex:4.D}
Consider $\mf t = \C$ and $R = \{ \pm 1\}$. By Theorem \ref{thm:4.4} 
\[
HH_* (\mh H(\mf t,W,k)) \cong HH_* (\mc O (\C) \rtimes S_2)) \cong \Omega^* (\C)^{S_2} \oplus \Omega^* (\{0\})
= \begin{cases} \C [z^2] \oplus \C & n = 0 \\ 
\C [z^2] z \, \textup{d} z & n = 1 \\
0 & n > 1 \end{cases} .
\]
With Theorem \ref{thm:4.5} and \eqref{eq:4.5} we can make this more explicit. The first family of representations
to consider comes from $P = \emptyset$ and the module $\C_0$ of $\mh H_\emptyset = \mc O (\mf t)$. Then 
$\zeta_0 (\C_0) = \C_0$ and 
\[
\mc F_{\emptyset, \C_0} : HH_* (\mc O (\C) \rtimes S_2) \to \Omega^* (\C)
\] 
can be identified with the projection $\Omega^* (\C)^{S_2} \oplus \Omega^* (\{0\}) \to \Omega^* (\C)^{S_2}$.
The Steinberg representation of $\mh H (\mf t,W,k)$ forms another family of representations. It satisfies
$\zeta_0 (\mr{St}) = \mr{sign}_{S_2} \otimes \C_0$ and $HH_* (\zeta_0 (\mr{St}))$ identifies with the projection
$\Omega^* (\C)^{S_2} \oplus \Omega^* (\{0\}) \to \Omega^* (\{0\})$. Hence
\begin{equation}\label{eq:4.6}
HH_* (\mc F_{\emptyset,\C_0}) \oplus HH_* (\zeta_0 (\mr{St})) : HH_n (\mc O (\C) \rtimes S_2) \to
\Omega^* (\C)^{S_2} \oplus \Omega^* (\{0\})
\end{equation}
is a $\C$-linear bijection. With \eqref{eq:4.5} it follows that
\begin{equation}\label{eq:4.7}
HH_* (\mc F_{\emptyset,\C_0}) \oplus HH_* (\mr{St}) : HH_n (\mh H (\mf t,W,k)) \to
\Omega^* (\C)^{S_2} \oplus \Omega^* (\{0\})
\end{equation}
is also a $\C$-linear bijection. Via \eqref{eq:4.6} and \eqref{eq:4.7}, $HH_n (\zeta_0)$ corresponds to 
the identity on $\Omega^* (\C)^{S_2} \oplus \Omega^* (\{0\})$. The summand $\C \cong \Omega^* (\{0\})$ in
$HH_* (\mh H (\mf t,W,k))$ has $\mc O (\C)^{S_2}$-weight $\pm k$ because it comes from St, while 
$\Omega^* (\{0\})$ has $\mc O (\C)^{S_2}$-weight 0 in $HH_* (\mc O (\C) \rtimes S_2)$. Therefore 
$HH_n (\zeta_0)$ is not $\mc O (\C)^{S_2}$-linear.
\end{ex}

\subsection{$HH_* (\mc H (G))$: cuspidal Bernstein blocks} \
\label{par:4.5}

Let $G$ be a reductive $p$-adic group and let $\mc H (G)$ be its Hecke algebra, as in Section \ref{sec:1}.
Our goal is to compute $HH_* (\mc H (G))$ in terms of the representation theory of $G$. By Proposition
\ref{prop:1.1}.b $\mc H (G)$ has local units, so in relation to Hochschild homology we may treat it as 
a unital algebra. Recall that by the Bernstein decomposition (Theorem \ref{thm:1.18}) 
$\mc H (G) = \bigoplus_{\mf s \in \mf B (G)} \mc H (G)^{\mf s}$. By the additivity and the continuity
of $HH_*$:
\begin{align}\nonumber
HH_* ( \mc H (G)) & = HH_* \Big( \lim_{S \subset \mf B (G), S \text{ finite}} \bigoplus_{\mf s \in S} 
\mc H (G)^{\mf s} \Big) \cong \lim_{S \subset \mf B (G), S \text{ finite}} HH_* \Big( \bigoplus_{\mf s \in S} 
\mc H (G)^{\mf s} \Big) \\
\label{eq:4.8} & \cong \lim_{S \subset \mf B (G), S \text{ finite}}  \bigoplus\nolimits_{\mf s \in S} 
HH_* ( \mc H (G)^{\mf s}) = \bigoplus\nolimits_{\mf s \in \mf B (G)}  HH_* ( \mc H (G)^{\mf s}) .
\end{align}
This means that to compute $HH_n (\mc H(G))$, it suffices to classify $\mf B (G)$ and to determine
$HH_* (\mc H (G)^{\mf s})$ for each Bernstein block $\Rep (G)^{\mf s}$ of $\Rep (G)$. We will treat
$\mf B (G)$ as a black box and we will focus on $HH_* (\mc H (G)^{\mf s})$ for one arbitrary $\mf s \in 
\mf B (G)$. We write $\mf s = [L,\sigma ]_G$ and let $\Pi_{\mf s} = I_P^G  (\ind_{L^1}^L (\Res^L_{L^1} \sigma))$
be the projective generator of $\Rep (G)^{\mf s} = \Mod (\mc H (G)^{\mf s})$ from Paragraph \ref{par:3.2}.
By Morita invariance, the equivalence of categories \eqref{eq:3.5} induces a natural isomorphism
\begin{equation}\label{eq:4.9}
HH_* (\mc H(G)^{\mf s}) \cong HH_* ( \End_G (\Pi_{\mf s})^{op}) .
\end{equation}
In the cuspidal cases this quickly leads to a nice description. 

\begin{prop}\label{prop:4.6}
Suppose that $\mf s$ is cuspidal, that is, $L = G$. There is a natural isomorphism of 
$\mc O (\Irr (G)^{\mf s})$-modules
\[
HH_n (\mc H (G)^{\mf s}) \cong \Omega^n (\Irr (G)^{\mf s}) .
\]
\end{prop}
\begin{proof}
From Theorem \ref{thm:3.3} and \eqref{eq:2.E} we know that
\[
\End_G (\Pi_{\mf s})^{op} \cong \mc O (X_\nr (G)) \rtimes \C [X_\nr (G,\sigma),\natural_{\mf s}^{-1}] .
\]
Then Theorem \ref{thm:4.3} says that 
\begin{equation}\label{eq:4.10}
HH_n \big( \End_G (\Pi_{\mf s})^{op} \big) \cong \Big( \bigoplus\nolimits_{\chi \in X_\nr (G,\sigma)}
\Omega^n (X_\nr (G)^\chi) \otimes (\natural_{\mf s}^{-1})^\chi \Big)^{X_\nr (G,\sigma)} .
\end{equation}
Here $X_\nr (G,\sigma)$ acts on $X_\nr (G)$ by multiplication, so $X_\nr (G)^\chi$ is empty unless
$\chi = 1$. As $(\natural_{\mf s}^{-1})^1 = 1$, the right hand side of \eqref{eq:4.10} reduces to 
$\Omega^n (X_\nr (G))^{X_\nr (G,\sigma)}$. Then \eqref{eq:4.10} (from right to left) can be written as
\begin{equation}\label{eq:4.11}
\Omega^n (X_\nr (G))^{X_\nr (G,\sigma)} \to HH_n (X_\nr (G)) \to HH_n (\End_G (\Pi_{\mf s}^{op}) ,
\end{equation}
where the first arrow comes from Theorem \ref{thm:4.2} and is natural, while the second arrow is
induced by the inclusion $\mc O (X_\nr (G)) \to \End_G (\Pi_{\mf s}^{op})$. By \eqref{eq:3.1}
that inclusion is natural once $\sigma \in \Irr (G)^{\mf s}$ has been chosen. 

The action of $X_\nr (G,\sigma)$ on $X_\nr (G)$ is free, and $X_\nr (G) / X_\nr (G,\sigma) \cong
\Irr (G)^{\mf s}$ by \eqref{eq:3.A}, so
\begin{equation}\label{eq:4.12} 
\Omega^n (X_\nr (G))^{X_\nr (G,\sigma)} = \Omega^n ( X_\nr (G) / X_\nr (G,\sigma) ) \cong
\Omega^n (\Irr (G)^{\mf s}) .
\end{equation}
The composition of \eqref{eq:4.9}, \eqref{eq:4.10} and \eqref{eq:4.12} is the required isomorphism
\begin{equation}\label{eq:4.13}
HH_n (\mc H (G)^{\mf s}) \isom \Omega^n (\Irr (G)^{\mf s}) .
\end{equation}
The isomorphism \eqref{eq:4.12} depends only on the choice of $\sigma$ in $\Irr (G)^{\mf s}$.
That cancels out with the same choice in \eqref{eq:4.11}, so the composition of \eqref{eq:4.10} 
and \eqref{eq:4.12} is natural. As \eqref{eq:4.9} is also natural, so is \eqref{eq:4.13}.

Recall from \eqref{eq:3.4} that 
\begin{equation}\label{eq:4.14}
Z \big( \End_G (\Pi_{\mf s})^{op} \big) = Z(\End_G (\Pi_{\mf s})) \cong
\mc O (X_\nr (G))^{X_\nr (G,\sigma)} \cong \mc O (\Irr (G)^{\mf s}) .
\end{equation}
We saw in \eqref{eq:4.3} that \eqref{eq:4.9} intertwines the actions of \eqref{eq:4.14}, and by
Theorem \ref{thm:4.3} the same holds for \eqref{eq:4.10}. Further \eqref{eq:4.12} is by 
definition $\mc O (\Irr (G)^{\mf s})$-linear, and we conclude that \eqref{eq:4.13} is
$\mc O (\Irr (G)^{\mf s})$-linear as well.
\end{proof}

\subsection{$HH_* (\mc H (G))$: non-cuspidal Bernstein blocks} \
\label{par:4.6}

In this paragraph $\mf s = [L,\sigma]_G$ with $L \neq G$. We want to determine
$HH_* (\mc H (G)^{\mf s})$, where $\Mod (\mc H (G)^{\mf s})$ is a non-cuspidal
Bernstein block of $\Rep (G)$. It is isomorphic to $HH_* (\End_G (\Pi_{\mf s})^{op})$,
but we do not understand $\End_G (\Pi_{\mf s})$ well enough to handle this directly.
Instead, we will approach it via localization on the Bernstein centre, as in
Paragraph \ref{par:3.3}. Recall from Theorem \ref{thm:3.5} that the Bernstein centre
for $\mc H (G)^{\mf s}$ is
\[
Z( \Rep (G)^{\mf s}) \cong \mc O (\Irr (L)^{\mf s_L})^{W_{\mf s}} \cong
\mc O (X_\nr (L) / W_{\mf s}^e) .
\]
We may and will assume that $\sigma \in \Irr_\cusp (L)^{\mf s_L}$ is tempered. It is
known from \cite[\S 7]{SolEnd} that the ``analytic localization" of $\End_G (\Pi_{\mf s})$
at $X_\nr^+ (L) W_{\mf s} \sigma$ is isomorphic to the ``analytic localization" at 
$\mf t_\R$ of a twisted graded Hecke algebra, denoted $\mh H (\mf t, W_{\mf s,\sigma}, 
k_\sigma, \natural_\sigma)$ in Theorem \ref{thm:3.10}. However, since $X_\nr^+ (L) 
W_{\mf s} \sigma$ is not Zariski-closed in $\Irr (L)^{\mf s_L} / W_{\mf s}$, we cannot
localize there by means of subsets of $Z(\Rep (G)^{\mf s})$.

The best we can achieve in that way is: for any $\chi \in X_\nr^+ (L)$, take
\[
S_\chi = \{ s \in Z (\Rep (G)^{\mf s}) : s (\sigma \otimes \chi) \neq 0 \}
\]
and consider $S_\chi^{-1} HH_* (\mc H (G)^{\mf s})$. This can be related to the
localization of\\ $\mh H (\mf t,W_{\mf s,\sigma}, k_\sigma, \natural_\sigma)$ with 
respect to a maximal ideal $I_{\log \chi}$ of its centre. By \eqref{eq:4.4} and Theorem 
\ref{thm:3.10}, that relates
\[
S_\chi^{-1} HH_* (\mc H (G)^{\mf s}) \cong HH_* \big( S_\chi^{-1} \End_G (\Pi_{\mf s})^{op} \big)
\]
to a localization of $HH_* \big( \mh H (\mf t,W_{\mf s,\sigma}, k_\sigma, \natural_\sigma) \big)$
at $I_{\log \chi}$. That provides a description of the localization of $HH_* (\mc H (G)^{\mf s})$ 
at one arbitrary point of $\Irr (L)^{\mf s_L} / W_{\mf s}$. 

Thus $HH_* (\mc H (G)^{\mf s})$ gives a sheaf over $\Irr \big( Z( \Rep (G)^{\mf s}) \big)$, whose
stalks we understand. In \cite{SolHH}, this entire sheaf is reconstructed by a
convoluted glueing procedure. The main tools for that are algebraic families of
$G$-representations like in \eqref{eq:4.5}. 

Let $M \subset G$ be a Levi subgroup containing $L$, and let $(\tau ,V_\tau) \in 
\Irr_\temp (M)^{\mf s_M}$. This gives a family of representations
\[
\mf F_{M,\tau} = \big\{ I_{PM}^G (\tau \otimes \chi_M ) : \chi_M \in X_\nr (M) \big\} 
\]
and a homomorphism of $\mc O (\Irr (L)^{\mf s})^{W_{\mf s}}$-algebras
\begin{equation}\label{eq:4.19}
\mc F_{M,\tau} : \mc H (G)^{\mf s} \to 
\mc O (X_\nr (M)) \otimes \End_\C^\infty (I_{PM}^G (V_\tau)) .
\end{equation}
By the Morita equivalence of $\C$ with $\End_\C^\infty (I_{PM}^G (V_\tau))$, the smooth part 
of the $G \times G$-representation $\End_\C (I_{PM}^G (V_\tau))$, $\mc F_{M,\tau}$ induces
a $\mc O (\Irr (L)^{\mf s_L})^{W_{\mf s}}$-linear map 
\begin{equation}\label{eq:4.15}
HH_* (\mc F_{M,\tau}) : HH_* (\mc H (G)^{\mf s}) \to HH_* \big( \mc O (X_\nr (M)) \big) 
\cong \Omega^* (X_\nr (M)) .
\end{equation}
Applying Theorem \ref{thm:3.13} to $\mf F_{M,\tau}$ yields a family of
$\mc O (X_\nr (L)) \rtimes \C [W_{\mf s}^e, \natural_{\mf s}^{-1}]$-modules
\[
\mf F_{M, \zeta^\vee_M (\tau)} = \Big\{ \ind_{\mc O (X_\nr (L)) \rtimes \C [W_{\mf s_M}^e, 
\natural_{\mf s}^{-1}]}^{\mc O (X_\nr (L)) \rtimes \C [W_{\mf s}^e, \natural_{\mf s}^{-1}]} 
(\zeta^\vee_M (\tau) \otimes \chi_M ) : \chi_M \in X_\nr (M)  \Big\} .
\]
As in \eqref{eq:4.15}, this induces a $\mc O (X_\nr (L))^{W_{\mf s}^e}$-linear map
\[
HH_* (\mc F_{M,\zeta^\vee_M (\tau)}) : HH_* \big( \mc O (X_\nr (L)) \rtimes \C [W_{\mf s}^e, 
\natural_{\mf s}^{-1}] \big) \to HH_* \big( \mc O (X_\nr (M)) \big) .
\]

\begin{thm}\label{thm:4.7} \textup{\cite[Theorem 2.14]{SolHH}} \\
There exists a unique $\C$-linear bijection 
\[
HH_n (\zeta^\vee) : HH_* \big( \mc O (X_\nr (L)) \rtimes \C [W_{\mf s}^e, 
\natural_{\mf s}^{-1}] \big) \to HH_n (\mc H (G)^{\mf s})
\]
such that, for all families $\mc F_{M,\tau}$ as above,
\[
HH_n (\mc F_{M,\tau}) \circ HH_n (\zeta^\vee) = HH_n \big( \mc F_{M,\zeta^\vee_M (\tau)} \big) . 
\]
\end{thm}

In general $HH_n (\zeta^\vee)$ does not respect the actions of $\mc O (X_\nr (L))^{W_{\mf s}^e}
\cong \mc O (\Irr (L)^{\mf s_L})^{W_{\mf s}}$, but it can be described precisely how
it deviates from $\mc O (\Irr (L)^{\mf s_L})^{W_{\mf s}}$-linearity \cite[Theorem B]{SolHH}.
On the part of $HH_n (\mc H (G)^{\mf s})$ that comes from $\Irr_\temp (G)_{[M,\delta]}$ with
$[M,\delta]_G \in \Delta (G,\mf s)$, the deviation is given by an element 
$r_\delta \in X^+_\nr (M)$ such that $\delta$ is a constituent of 
$I_P^G (\sigma' \otimes r_\delta)$ for a unitary $\sigma' \in X_\nr (L) \sigma$.

From Theorems \ref{thm:4.3} and \ref{thm:4.7} we obtain a $\C$-linear bijection
\begin{equation}\label{eq:4.18}
HH_n (\mc H (G)^{\mf s}) \to \Big( \bigoplus\nolimits_{w \in W_{\mf s}^e} 
\Omega^n (X_\nr (L)^w \otimes (\natural_{\mf s}^{-1})^w \Big)^{W_{\mf s}^e} .
\end{equation}
This map can also be constructed via families of (virtual) $G$-representations, see
\cite[Theorem 2.13.b]{SolHH}.

\begin{ex}\label{ex:4.E}
We consider $G = SL_2 (F)$ and $\mf s = [T,\mr{triv}]_G$. Then $W_{\mf s}^e = S_2, 
\natural_{\mf s} = 1$ and Theorem \ref{thm:4.7} says that
\[
HH_n (\mc H (G)^{\mf s}) \cong HH_* ( \mc O (X_\nr (T)) \rtimes S_2 ) .
\]
We use three families of representations to construct this isomorphism:
\[
\mf F_{T,1} = \{ I_B^G (\chi) : \chi \in X_\nr (T) \},
\]
$\mf F_{G,\mr{St}} = \{ \mr{St} \}$ and $\mf F_{G,2} = \{ \pi_+ - \pi_- \}$, where 
$I_B^G (\chi_-) = \pi_+ \oplus \pi_-$ for the unique $\chi_- \in X_\nr (T)$ of order two.
These induce $\mc O (X_\nr (T))^{S_2}$-linear maps
\begin{align*}
& HH_n (\mc F_{T,1}) : HH_n (\mc H (G)^{\mf s}) \to HH_n \big( \mc O (X_\nr (T)) \otimes
\End_\C^\infty (I_B^G (\mr{triv})) \big) \cong \Omega^n (X_\nr (T)) ,\\
& HH_n (\mc F_{G,\mr{St}}) :  HH_n (\mc H (G)^{\mf s}) \to HH_n ( \End_\C^\infty (V_{\mr{St}}) )
\cong \Omega^n (\{ \mr{St} \}) ,\\
& HH_n (\mc F_{G,2}) = HH_n (\pi_+) - HH_n (\pi_-) : HH_n (\mc H (G)^{\mf s}) \to
HH_n (\mc O (\{\chi_-\})) .
\end{align*}
The sum of the three maps is an isomorphism of $\mc O (X_\nr (T))^{S_2}$-modules
\begin{equation}\label{eq:4.16}
HH_n (\mc H (G)^{\mf s}) \to \Omega^n (X_\nr (T))^{S_2} \oplus \Omega^n (\{ \mr{St} \})
\oplus \Omega^n ( \{\chi_-\} ) .
\end{equation}
The right hand side of \eqref{eq:4.16} is isomorphic to
\begin{equation}\label{eq:4.17}
HH_n (\mc O (T) \rtimes S_2) \cong \big( \Omega^n (X_\nr (T)) \oplus \Omega^n (\{\mr{triv}\} ) 
\oplus \Omega^n (\{ \chi_- \} )\big)^{S_2} ,
\end{equation}
and the canonical map from \eqref{eq:4.16} to \eqref{eq:4.17} is almost linear over 
$\mc O (X_\nr (T))^{S_2} \cong Z (\Rep (G)^{\mf s})$. The only deviation from 
$\mc O (X_\nr (T))^{S_2}$-linearity is that the $Z (\Rep (G)^{\mf s})$-character of St 
does not agree with $\mr{triv} \in X_\nr (T) / S_2$.
\end{ex}

\section{Hochschild homology for topological algebras}\label{sec:5}

We would like to compute the Hochschild homology of topological algebras appearing in
the representation theory of $p$-adic groups, like $C^\infty (X) \rtimes \Gamma$ or
$\mc S (G)$. The definition of Hochschild homology in Section \ref{sec:4} can be 
applied to any algebra, so in particular to $C^\infty (X)$ or $C(X)$ for a smooth
manifold $X$. However, that does not give interesting results, because the functor 
$HH_*$ from Definition \ref{def:4.1} does not take the topology of an algebra into
account. The best way to improve that is by using a topological tensor product.

As is common in noncommutative geometry, we will work mostly with Fr\'echet algebras.
For later use we define precisely which algebras we mean by that.

\begin{defn}\label{def:5.10}
A Fr\'echet algebra is a $\C$-algebra $A$ such that:
\begin{itemize}
\item $A$ is a Fr\'echet space,
\item the topology on $A$ can be defined by a countable family of seminorms
$p$ which are submultiplicative: $p(ab) \leq p(a) p(b)$ for all $a,b \in A$.
\end{itemize}
\end{defn}
The submultiplicativity implies that for any Fr\'echet algebra $A$ the 
multiplication map $A \times A \to A$ is continuous. The class of Fr\'echet algebras 
contains all Banach algebras and spaces of smooth functions $C^\infty (X)$.

For two Fr\'echet spaces $V$ and $W$, we denote their completed projective tensor
product by $V \hat \otimes W$. This is a completion of $V \otimes W$ and a Fr\'echet
space, with the following universal property: for any Fr\'echet space $Z$ there is
a natural bijection between
\begin{itemize}
\item the set of continuous $\C$-linear maps from $V \hat \otimes W$ to $Z$,
\item the set of continuous $\C$-bilinear maps from $V \times W$ to $Z$.
\end{itemize}
\begin{ex}
For two smooth manifolds $X$ and $Y$, there is an isomorphism of Fr\'echet spaces
\[
C^\infty (X) \hat \otimes C^\infty (Y) \cong C^\infty (X \times Y).
\]
\end{ex}

For a Fr\'echet algebra $A$, there is also a notion of the completed projective
tensor product of $A$-modules. Let $V$ be a right Fr\'echet $A$-module and let 
$W$ be a left Fr\'echet $A$-module. Then $V \hat \otimes_A W$ is a completion of
$V \otimes_A W$ which, like $V \hat \otimes W$, has a universal property with 
respect to $A$-balanced $\C$-bilinear continuous maps from $V \times W$ to a
Fr\'echet space $Z$. Concretely, this works out to
\[
V \hat \otimes_A W = V \hat \otimes W \, \big/ \, 
\overline{\mr{span} \{ v a \otimes w - v \otimes a w : v \in V, a \in A, w \in W\} } ,
\] 
where the bar means closure.

To define Hochschild homology for Fr\'echet algebras one needs homological algebra
in a topological setting, for which we refer to \cite{Tay}. For Fr\'echet algebras
we will always do that with respect to $\hat \otimes$, and we often suppress that 
from the notations. 

\begin{defn}\label{def:5.1}
Let $A$ be a unital Fr\'echet algebra. For $n \in \Z_{\geq 0}$, $\Tor_n^A$ denotes
the $n$-th derived functor of $\hat \otimes_A$. The $n$-th Hochschild homology of
$A$ is
\[
HH_n (A) = HH_n (A, \hat \otimes) = \Tor_n^{A \hat \otimes A^{op}} (A,A) .
\]
For a possibly non-unital Fr\'echet algebra $B$ we put
\[
HH_n (B) = \mr{coker} \big( HH_n (\C) \to HH_n (B_+) \big) .
\]
\end{defn}

Each $HH_n$ is a functor from Fr\'echet spaces to topological vector spaces. 
These functors share several properties with their purely algebraic counterparts:

\begin{itemize}
\item One can compute $HH_* (A)$ as the homology of an explicit differential 
complex $(A^{\hat \otimes m}, d_m)$.
\item In degree zero the definition shows that $HH_0 (A) = A / \overline{[A,A]}$.
\item Additivity holds for unital Fr\'echet algebras.
\item $HH_n (A)$ has the structure of a $Z(A)$-module.
\item Morita invariance holds for unital Fr\'echet algebras.
\end{itemize}
Continuity of $HH_n (?,\hat \otimes)$ is problematic, because a direct limit
of Fr\'echet algebras is often not a Fr\'echet space.

We point out that Hochschild homology works badly for Banach algebras. Consider
a commutative Banach algebra $B$, for instance $C(Y)$ for a compact Hausdorff space
$Y$. Then $HH_0 (B, \hat \otimes) = B$ and
\begin{equation}\label{eq:5.1}
HH_n (B, \hat \otimes) = 0 \qquad \text{for } n \in \Z_{>0}. 
\end{equation}
In fact \eqref{eq:5.1} also holds for large classes of noncommutative Banach
algebras \cite{Joh}. Roughly speaking Hochschild homology detects some 
differentiable structure, and Banach algebras are too complete for that.

\subsection{Comparison between algebraic and topological settings} \ 

For a smooth manifold $X$, we write $\Omega_{sm}^n (X)$ for the space of 
smooth differential $n$-forms on $X$. The following version of the 
Hochschild--Kostant--Rosenberg theorem (Theorem \ref{thm:4.2}) was discovered 
by Connes, in the case of compact manifolds \cite{Con}.

\begin{thm}\label{thm:5.3} \textup{\cite{Tel}} \\
There is a natural isomorphism of Fr\'echet $C^\infty (X)$-modules
\[
HH_n (C^\infty (X)) \cong \Omega_{sm}^n (X).
\]
\end{thm}

Let $\Gamma$ be a finite group acting on $X$ by diffeomorphisms, and let
$\natural : \Gamma \times \Gamma \to \C^\times$ by a 2-cocycle. Then we can
form the crossed product $C^\infty (X) \rtimes \Gamma$ and the twisted crossed
product $C^\infty (X) \rtimes \C [\Gamma,\natural]$. Like in the algebraic
setting, $HH_n (C^\infty (X) \rtimes \Gamma)$ was computed by Brylinski \cite{Bry}.
That can be generalized to a topological version of Theorem \ref{thm:4.3}:

\begin{thm}\label{thm:5.4} \textup{\cite[Proposition 3.12]{SolHH}} \\
There exists an isomorphism of Fr\'echet $C^\infty (X)^\Gamma$-modules
\[
HH_n (C^\infty (X) \rtimes \C [\Gamma, \natural]) \cong 
\big( \bigoplus\nolimits_{\gamma \in \Gamma} \Omega_{sm}^n (X^\gamma) 
\otimes \natural^\gamma \big)^\Gamma .
\]
\end{thm}

From on the one hand Theorems \ref{thm:4.2} and \ref{thm:4.3} and on the other
hand Theorems \ref{thm:5.2} and \ref{thm:5.3}, we see that there is a clear
analogy between $HH_n$ in an algebro-geometric setting and $HH_n$ in a 
differential geometric setting. This can be made precise for larger classes of
(noncommutative) algebras. 

Let $V$ be a complex affine variety and let $X$ be a real analytic manifold 
which is contained and Zariski-dense in $V$. Then $\mc O (V)$ embeds in
$C^\infty (X)$ and $C^\infty (X)$ is a $\mc O (V)$-module (usually of 
infinite rank). Assume that a finite group $\Gamma$ acts on $V$ by automorphisms
of algebraic varieties, and that $\Gamma$ stabilizes $X$. Then $X / \Gamma$ is 
an orbifold and $C^\infty (X)^\Gamma$ is (by definition) the ring of smooth 
functions on $X / \Gamma$.

Let $A$ be a finite type $\mc O (V)^\Gamma$-algebra, as in Definition 
\ref{def:4.F}. The algebra 
\[
C^\infty (X)^\Gamma \otimes_{\mc O (V)^\Gamma} A
\]
has finite rank over $C^\infty (X)^\Gamma$, and
by \cite[Lemma 1.3]{KaSo} it is a Fr\'echet algebra.

\begin{ex}
The crossed product $\mc O (V) \rtimes \Gamma$ is a finite type 
$\mc O (V / \Gamma)$-algebra and
\[
C^\infty (X / \Gamma) \otimes_{\mc O (V / \Gamma)} \mc O (V) \rtimes \Gamma
= C^\infty (X) \rtimes \Gamma .
\] 
\end{ex}

\begin{thm}\label{thm:5.2} \textup{\cite[Theorems C and D]{KaSo}} \\
Let $V,X,\Gamma,A$ be as above and assume that $T_x (X) \otimes_\R \C = T_x (V)$ 
for all $x \in X$. 
\enuma{
\item $C^\infty (X)^\Gamma$ is flat over $\mc O (V)^\Gamma$.
\item There is an isomorphism of Fr\'echet $C^\infty (X)^\Gamma$-modules
\[
HH_n \big( C^\infty (X)^\Gamma \otimes_{\mc O (V)^\Gamma} A \big) \cong
C^\infty (X)^\Gamma \otimes_{\mc O (V)^\Gamma} HH_n (A).
\]
}
\end{thm}

One can recover Theorems \ref{thm:5.3} and \ref{thm:5.4} from Theorems \ref{thm:4.2} 
and \ref{thm:4.3} by applying Theorem \ref{thm:5.2}, see \cite[\S 4]{KaSo}.
In a similar way, Theorem \ref{thm:5.2} will help us to determine 
$HH_n (\mc S(G,K)_{\mf d})$ for some $\mf d \in \Delta (G)$.

\subsection{Hochschild homology of $\mc S (G)$}\label{par:5.2} \

Let $G$ be a reductive $p$-adic group and let $\mc S (G)$ be its 
Harish-Chandra--Schwartz algebra, as in Definition \ref{def:1.SG}. We want to
describe its Hochschild homology in terms of representation theory, like we
did for $HH_* (\mc H (G))$ in Paragraphs \ref{par:4.5} and \ref{par:4.6}. 

However, since $\mc S (G)$ is not a Fr\'echet space, Definition \ref{def:5.1}
will probably produce suboptimal results. The best solution for that technical
problem is to consider $\mc S (G)$ as a bornological algebra \cite{Mey1} and to 
use Hochschild homology based on the completed bornological tensor product 
$\hat \otimes_b$ \cite[Chapter I]{Mey}. This fits with Definition \ref{def:5.1},
because for Fr\'echet spaces $\hat \otimes_b$ and $\hat \otimes$ agree 
\cite[Theorem I.87]{Mey}. The functor $HH_* (?,\hat \otimes_b)$ enjoys 
continuity properties for strict inductive limits of Fr\'echet algebras, 
which we state only in two concrete cases.

\begin{thm}\label{thm:5.5} 
\textup{\cite[Theorem 2]{BrPl} and \cite[Theorem I.93]{Mey}} \\
Let $(K_m )_{m=1}^\infty$ be as in Theorem \ref{thm:1.20} and let $\mf s \in \mf B (G)$.
There are natural isomorphisms of topological vector spaces
\[
\begin{array}{lllll}
HH_n (\mc S (G) ,\hat \otimes_b ) & = & HH_n \big( \lim\limits_{m \to \infty} \mc S(G,K_m), 
\hat \otimes_b \big) & \cong & \lim\limits_{m \to \infty} HH_n (\mc S (G,K_m)) ,\\
HH_n (\mc S (G)^{\mf s} ,\hat \otimes_b ) & = & HH_n \big( \lim\limits_{m \to \infty} 
\mc S(G,K_m )^{\mf s}, \hat \otimes_b \big) & \cong & 
\lim\limits_{m \to \infty} HH_n (\mc S (G,K_m )^{\mf s}) .
\end{array}
\]
\end{thm}

For each $\mf s \in \mf B (G)$ we pick $m(\mf s) \in \Z_{>0}$ such that 
$\mf s \in \mf B (G,K_{m (\mf s)})$, and we write $K_{\mf s} = K_{m (\mf s)}$. 
By Proposition \ref{prop:1.22}.a all the spaces $HH_n (\mc S (G,K_m)^{\mf s})$ with 
$m \geq m(\mf s)$ are isomorphic. From Theorem \ref{thm:5.5}, Corollary \ref{cor:1.21} 
and Proposition \ref{prop:1.22}.b we obtain
\begin{align}
\nonumber HH_n (\mc S (G) ,\hat \otimes_b ) & \cong \lim_{m \to \infty} 
HH_n \big( \bigoplus\nolimits_{\mf s \in \mf B (G,K_m)} \mc S (G,K_m )^{\mf s} \big) \\
\label{eq:5.2} & \cong \lim_{m \to \infty} \bigoplus\nolimits_{\mf s \in \mf B (G,K_m)}  
HH_n \big( \mc S (G,K_m )^{\mf s} \big) \\
\nonumber & \cong \lim_{m \to \infty} \bigoplus_{\mf s \in \mf B (G,K_m)}  
HH_n \big( \mc S (G)^{\mf s}, \hat \otimes_b \big) = \bigoplus_{\mf s \in \mf B (G)}  
HH_n \big( \mc S (G)^{\mf s}, \hat \otimes_b \big) .
\end{align}
From \eqref{eq:5.2} we see that, in order to determine $HH_n (\mc S (G), \hat \otimes_b)$, 
it suffices to compute $HH_n (\mc S (G,K_{\mf s})^{\mf s})$ for each $\mf s \in \mf B (G)$. 
This brings us back to unital Fr\'echet algebras. 

We recall from Theorem \ref{thm:1.16} and \eqref{eq:1.36} that
\begin{align*}
\mc S (G,K_{\mf s})^{\mf s} & = \bigoplus\nolimits_{\mf d \in \Delta (G,\mf s)} 
\mc S (G,K_{\mf s}) \cap \mc S (G)_{\mf d} \\
& \cong \bigoplus\nolimits_{[L,\delta]_G \in \Delta (G,\mf s)} 
\Big( C^\infty (X_\nr^u (L)) \otimes \End_\C \big( I_{K_0 \cap P}^{K_0} 
(V_\delta)^{K_{\mf s}} \big) \Big)^{W^e_{[L,\delta]}} .
\end{align*}
Write $\mf s = [M,\sigma]_G$ with $\sigma$ unitary supercuspidal and recall that
$\mf s \in \Delta (G,\mf s)$. That brings us in a good position to apply the results from
the previous paragraph.

\begin{lem}\label{lem:5.6}
\enuma{
\item There is an isomorphism of Fr\'echet algebras
\[
\mc S (G,K_{\mf s}) \cap \mc S (G)_{[M,\sigma]_G} \cong 
C^\infty (X_\nr^u (M))^{W^e_{\mf s}} \otimes_{\mc O (X_\nr (M))^{W^e_{\mf s}}}
\mc H (G,K_{\mf s})^{\mf s} .
\]
\item There is an isomorphism of Fr\'echet $C^\infty (X_\nr^u (M))^{W^e_{\mf s}}
$-modules 
\[
HH_n \big( \mc S (G,K_{\mf s}) \cap \mc S (G)_{[M,\sigma]_G} \big) \cong  
C^\infty (X_\nr^u (M))^{W^e_{\mf s}} \otimes_{\mc O (X_\nr (M))^{W^e_{\mf s}}}
HH_n \big( \mc H (G,K_{\mf s})^{\mf s} \big) .
\]
}
\end{lem}
\begin{proof}
(a) In Theorem \ref{thm:1.19} the underlying intertwining operators do not have 
singularities on $X_\nr^u (M) \sigma$, so the Fourier transform induces an
algebra isomorphism
\[
\mc H (G,K_{\mf s} )^{\mf s} |_{X_\nr^u (M) \sigma} \cong 
\Big( \mc O (X_\nr (M)) |_{X_\nr^u (M)} \otimes \End_\C \big( I_{K_0 \cap P}^{K_0} 
(V_\sigma)^{K_{\mf s}} \big) \Big)^{W^e_{\mf s}} .
\]
It follows that
\[
C^\infty (X_\nr^u (M))^{W^e_{\mf s}} \otimes_{\mc O (X_\nr (M))^{W^e_{\mf s}}}
\mc H (G,K_{\mf s})^{\mf s} \cong \! \Big( C^\infty (X_\nr^u (M)) \otimes 
\End_\C \big( I_{K_0 \cap P}^{K_0} (V_\sigma)^{K_{\mf s}} \big) \! \Big)^{\! W^e_{\mf s}} 
\hspace{-2mm}.
\]
By Theorem \ref{thm:1.16}.b, the right hand side equals 
$\mc S (G,K_{\mf s}) \cap \mc S (G)_{[M,\sigma]_G}$.\\
(b) In view of part (a), this is an instance of Theorem \ref{thm:5.2}.
\end{proof}

Unfortunately it remains difficult to determine $HH_n \big( \mc S (G,K_{\mf s}) 
\cap \mc S (G)_{[M,\sigma]_G} \big)$ in Lemma \ref{lem:5.6}.b, because the
$\mc O (X_\nr (M))^{W^e_{\mf s}}$-module structure of 
$HH_n (\mc H (G,K_{\mf s})^{\mf s})$ is tricky.

\begin{ex}\label{ex:5.B}
Consider $G = SL_2 (F)$ and $\mf s = [T,\mr{triv}]_G$. By Lemma \ref{lem:5.6}.b 
and \eqref{eq:4.16} we have
\begin{multline*}
HH_n (\mc S (G)_{\mf s}, \hat \otimes_b) \cong C^\infty (X_\nr^u (T))^{S_2}
\otimes_{\mc O (X_\nr (T))^{S_2}} HH_n (\mc H (G)^{\mf s}) \\
\cong C^\infty (X_\nr^u (T))^{S_2} \otimes_{\mc O (X_\nr (T))^{S_2}} 
\big( \Omega^n (X_\nr (T))^{S_2} \oplus \Omega^n (\{\mr{St}\}) \oplus 
\Omega^n (\{\chi_-\}) \big) .
\end{multline*}
Here $\Omega^n (\{\mr{St}\})$ drops out, because its 
$\mc O (X_\nr (T))$-character does not lie in $X_\nr^u (T)$. The other two summands
of $HH_n (\mc H (G)^{\mf s})$ carry the $\mc O (X_\nr (T))$-action expected from the
notations. We find
\begin{equation}\label{eq:5.3}
HH_n (\mc S (G)_{\mf s}, \hat \otimes_b) \cong \Omega^n_{sm} (X_\nr^u (T))^{S_2} 
\oplus \Omega^n_{sm} (\{\chi_-\}) . 
\end{equation}
As $\mc S (G)^{\mf s} = \mc S (G)_{\mf s} \oplus \mc S (G)_{[G,\mr{St}]}$ and 
$\mc S (G)_{[G,\mr{St}]}$ is Morita equivalent to $\C$, \eqref{eq:5.3} yields
\begin{equation}\label{eq:5.4}
HH_n (\mc S (G)^{\mf s}, \hat \otimes_b) \cong \Omega^n_{sm} (X_\nr^u (T))^{S_2} 
\oplus \Omega^n_{sm} (\{\chi_-\}) \oplus \Omega^n (\{ \mr{St} \}) .
\end{equation}
Theorem \ref{thm:5.4} shows that \eqref{eq:5.4} is isomorphic 
to $HH_n \big( C^\infty (X_\nr^u (T)) \rtimes S_2 \big)$.
\end{ex}

There is a version of Lemma \ref{lem:5.6}.a for $\mf d \in \Delta (G,\mf s)$
which are not represented by a supercuspidal representation \cite[Lemma 3.3]{SolHH},
but it is more complicated and does not fit in the framework of Theorem 
\ref{thm:5.2}. To handle $HH_n (\mc S (G,K_{\mf s}) \cap \mc S (G)_{\mf d})$ for
such $\mf d$ we use families of representations, like in Paragraph \ref{par:4.6}.
For a parabolic subgroup $P = L U_P \subset G$ and $\tau \in \Irr_\temp (L)^{\mf s_L}$ 
we have a family of tempered representations 
\[
\mf F_{L,\tau}^u = \{ I_P^G (\tau \otimes \chi_L) : \chi_L \in X_\nr^u (L) \} 
\]
and a homomorphism of $C^\infty (X_\nr^u (M))^{W^e_{\mf s}}$-algebras
\[
\mc F_{L,\tau}^u : \mc S (G,K_{\mf s})^{\mf s} \to C^\infty (X_\nr^u (L)) 
\otimes \End_\C \big( I_{P \cap K_0}^{K_0} (\tau)^{K_{\mf s}} \big) .
\]
Recall $\zeta^\vee$ from Theorem \ref{thm:3.13}. It preserves temperedness and
the tempered part of $R\big( \mc O (X_\nr (M)) \big)$ can be identified with 
$R \big( C^\infty (X_\nr^u (M)) \big)$, so $\zeta^\vee$ restricts to a group isomorphism
\begin{equation}\label{eq:5.5}
\zeta_u^\vee : R(\mc S (G)^{\mf s}) \to
R \big( C^\infty (X_\nr^u (M)) \rtimes \C [W^e_{\mf s}, \natural_{\mf s}^{-1}] \big) .
\end{equation}
In particular $\zeta_u^\vee (\mf F_{L,\tau}^u) = \mf F^u_{L,\zeta^\vee_u (\tau)}$ is
a family of $C^\infty (X_\nr^u (M)) \rtimes \C [W^e_{\mf s}, \natural_{\mf s}^{-1}]
$-modules.

\begin{thm}\label{thm:5.7} \textup{\cite[Theorem 3.13]{SolHH}} \\
There exists a unique isomorphism of Fr\'echet spaces
\[
HH_n (\zeta_u^\vee) : HH_n \big( C^\infty (X_\nr^u (M)) \rtimes \C [W^e_{\mf s}, 
\natural_{\mf s}^{-1}] \big) \to HH_n ( \mc S (G,K_{\mf s})^{\mf s} ) 
\]
such that
\[
HH_n (\mc F_{L,\tau}^u) \circ HH_n (\zeta_u^\vee) = 
HH_n \big( \mc F^u_{L,\zeta^\vee_u (\tau)} \big)
\]
for all families of tempered representations $\mf F_{L,\tau}^u$.
\end{thm}

From Theorems \ref{thm:5.4} and \ref{thm:5.7} we obtain an isomorphism of
Fr\'echet spaces 
\begin{equation}\label{eq:5.6}
HH_n (\mc S (G,K_{\mf s})^{\mf s}) \cong \Big( \bigoplus\nolimits_{w \in W^e_{\mf s}}
\Omega^n_{sm} (X_\nr^u (M))^w \otimes (\natural_{\mf s}^{-1})^w \Big)^{W^e_{\mf s}}.
\end{equation}
This isomorphism can also be constructed more directly, with suitable families
of tempered representations.

\begin{ex}\label{ex:5.C}
We given an overview of $HH_n (\mc S(G)^{\mf s},\hat \otimes_b)$ for all Bernstein 
blocks of $G = SL_2 (F)$. For $\mf s = [G,\sigma]_G$ with $\sigma \in \Irr_\cusp (G)$, 
we have
\[
HH_* (\mc S (G,K_{\mf s})^{\mf s}) = HH_0 (\mc S (G,K_{\mf s})^{\mf s}) \cong \C . 
\]
For $\mf s = [T,\chi]_G$ with ord$(\chi |_{\mf o_F^\times}) > 2$ we have
\[
HH_n (\mc S (G,K_{\mf s})^{\mf s}) \cong HH_n \big( C^\infty (X_\nr^u (T)) \big)
\cong \Omega^n_{sm} (X_\nr^u (T)) \cong \Omega^n_{sm}(S^1) .
\]
For $\mf s = [T,\chi_2 ]_G$ with ord$(\chi_2 |_{\mf o_F^\times}) = 2$,
\eqref{eq:5.6} shows that
\[
HH_n (\mc S (G,K_{\mf s})^{\mf s}) \cong HH_n \big( C^\infty (X_\nr^u (T)) \rtimes 
S_2 \big) \cong \Omega^n_{sm} (X_\nr^u (T))^{S_2} \oplus \Omega^n_{sm} (\{1,-1\}) .
\]
Finally, for $\mf s = [T,\mr{triv}]_G$, we already determined
$HH_n (\mc S (G,K_{\mf s})^{\mf s})$ in Example \ref{ex:5.B}.
\end{ex}

\subsection{Trace Paley--Wiener theorems for $\mc H (G)$ and $\mc S (G)$} \

The structure of
\[
HH_0 (\mc H (G)) = \mc H (G) / [\mc H (G) ,\mc H (G)]
\]
can be described with the trace Paley--Wiener for reductive $p$-adic groups,
which we recall now. The main ingredient is the trace pairing
\[
\begin{array}{ccc}
HH_0 (\mc H (G)) \times R(G) & \to & \C \\
(f, \pi) & \mapsto & \mr{tr} \, \pi (f)  
\end{array}.
\]
To see that this pairing is well-defined, pick $K \in \CO (G)$ such that 
$f \in \mc H (G,K)$. For $(\pi,V) \in \Rep_\fl (G)$, $\mr{tr} \, \pi (f)$ equals
$\mr{tr} \, \pi (f)|_{V^K}$. As $\dim_\C V^K$ is finite (Theorem \ref{thm:1.21}),
$\mr{tr} \, \pi (f)|_{V^K} \in \C$ is defined.

We say that a linear form $\lambda \in \Hom_\Z (R(G),\C)$ is regular if
\begin{itemize}
\item $\lambda$ is supported on finitely many Bernstein blocks of $\Rep (G)$,
\item for every parabolic subgroup $P = L U_P$ and every 
$\pi \in \Irr (L)$, the function
\[
X_\nr (L) \to \C : \chi \mapsto \lambda (I_P^G (\pi \otimes \chi))
\qquad \text{is regular.}
\] 
\end{itemize}
We denote the $\C$-vector space of such $\lambda$ by $\Hom_\Z (R(G),\C)_{\mr{reg}}$.

\begin{thm}\label{thm:5.8} \textup{\cite{BDK}} \\
The trace pairing induces an isomorphism of $Z(\Rep (G))$-modules 
\[
HH_0 (\mc H (G)) \longrightarrow \Hom_\Z (R(G),\C )_{\mr{reg}}.
\]
\end{thm}

Next we consider a twisted crossed product $A \rtimes \C [\Gamma,\natural]$ as in
\eqref{eq:2.B}, where $\Gamma$ is a finite group and $A$ denotes either $C^\infty (X)$
for a smooth manifold or $\mc O (X)$ for a non-singular affine $\C$-variety.
From Theorems \ref{thm:4.3} and \ref{thm:5.4} we know that there are isomorphisms
of $A^\Gamma$-modules
\begin{equation}\label{eq:5.7}
\begin{aligned}
HH_0 (C^\infty (X) \rtimes \C [\Gamma,\natural]) & \cong \big( \bigoplus\nolimits_{
\gamma \in \Gamma} C^\infty (X^\gamma) \otimes \natural^\gamma \big)^\Gamma \\
& \cong \bigoplus\nolimits_{\gamma \in \Gamma / \text{conjugacy}} \big( 
C^\infty (X^\gamma) \otimes \natural^\gamma \big)^{Z_\Gamma (\gamma)} ,
\end{aligned}
\end{equation}
and similarly with $\mc O$ instead of $C^\infty$. The specialisation of $HH_0 (A)$ at 
an arbitrary character $\Gamma x \in X / \Gamma$ of $A^\Gamma$ can be identified with 
\begin{equation}\label{eq:5.8}
\C \big\{ T_\gamma : \gamma \in \Gamma_x ,\, \natural^\gamma 
|_{Z_{\Gamma_x} (\gamma)} = 1 ,\, \text{modulo } \Gamma_x \text{-conjugation} \big\}  .
\end{equation}
From Theorem \ref{thm:2.3} and Lemma \ref{lem:2.1} one deduces that \eqref{eq:5.8} 
pairs nondegenerately with the Grothendieck group of the category of finite 
dimensional $A \rtimes \C [\Gamma,\natural]$-modules with $A^\Gamma$-character 
$\Gamma x$. This implies that the trace pairing
\begin{equation}\label{eq:5.9}
HH_0 (A \rtimes \C [\Gamma,\natural]) \times R (A \rtimes \C [\Gamma, \natural])
\longrightarrow \C \qquad \text{is nondegenerate.}
\end{equation}
Reasoning in such ways, one can recover Theorem \ref{thm:5.8} from Theorems
\ref{thm:3.13} and \ref{thm:4.7}, see \cite[Proposition 2.9.b]{SolHH}.

As we saw in Paragraph \ref{par:5.2}, there are versions of Theorems
\ref{thm:3.13} and \ref{thm:4.7} for $\mc S (G)$. That makes the existence of a 
trace Paley--Wiener theorem for $\mc S (G)$ plausible, and \eqref{eq:5.7}
suggests that it should involve smooth rather than regular functions. We call
a linear form $\lambda \in \Hom_\Z (R(\mc S (G)),\C)$ smooth if:
\begin{itemize}
\item $\lambda$ is supported on finitely many Bernstein blocks of $\Rep (G)$,
\item for every parabolic subgroup $P = L U_P$ and every 
$\tau \in \Irr_\temp (L)$, the function
\[
X_\nr (L) \to \C : \chi \mapsto \lambda (I_P^G (\tau \otimes \chi))
\qquad \text{is smooth.}
\] 
\end{itemize}
We denote the $\C$-vector space of such linear forms by $\Hom_\Z (R(\mc S(G)),\C)_\infty$.

\begin{thm}\label{thm:5.9} 
\textup{\cite[Theorem D.ii and (3.4)]{SolHH}} \\
The trace pairing induces an isomorphism of $Z(\Rep (G))$-modules
\[
HH_0 (\mc S (G),\hat \otimes_b) = \mc S (G) \big/ \overline{[\mc S (G),\mc S (G)]}
\longrightarrow \Hom_\Z ( R(\mc S (G)), \C )_\infty .
\]
\end{thm}

\subsection{Periodic cyclic homology} \

Periodic cyclic homology of algebras consists of two functors $HP_0, HP_1$ from
algebras to vector spaces. They are periodic in the sense that $HP_{n+2m} = HP_n$
for all $n,m \in \Z$. The functor $HP_* = HP_0 \oplus HP_1$ plays an important
role in noncommutative geometry, because it is an analogue of DeRham cohomology
(which can be defined for commutative algebras). Its functorial properties are
analogous to those of topological K-theory (see Paragraph \ref{par:6.1}),
in particular it is Morita invariant \cite{Cun2} and there are six-term exact
sequences associated to short exact sequences of algebras \cite{CuQu}. 
In practice, these make it possible to determine $HP_* (A)$ 
for many algebras $A$.

The periodic cyclic homology of a $\C$-algebra $A$ is defined as the homology
of an explicit infinite double complex $CC_{per}(A)$ with spaces $A^{\otimes n}$
\cite[\S 5.1]{Lod}. A part of $CC_{per}(A)$ computes $HH_* (A)$, and as a result
$HP_* (A)$ and $HH_* (A)$ are closely related. In many cases $HP_* (A)$ can be
computed as the homology of $HH_* (A)$ with respect to a new differential
$B$ discovered by Connes. 

Periodic cyclic homology is also defined for topological algebras, when one 
fixes a topological tensor product. We focus on Fr\'echet algebras and 
$\hat \otimes$. Then $HP_* (A,\hat \otimes)$ is the homology of a double complex 
$CC_{per}(A, \hat \otimes)$ obtained from $CC_{per}(A)$ by replacing all terms 
$A^{\otimes n}$ by $A^{\hat \otimes n}$.

We can now formulate the Hochschild--Kostant--Rosenberg--Connes theorem.
In these cases $HP_* (A)$ is computed from $HH_* (A)$ (which is known from 
Theorems \ref{thm:4.2} and \ref{thm:5.3}) as the homology with respect to usual
$d$ for differential forms.

\begin{thm}\label{thm:5.10}
\textup{\cite[\S 5.1.12]{Lod} and \cite{Con,Tel}}\\
Let $i \in \{0,1\}$ and let $H_{dR}^*$ denote 
de Rham cohomology with complex coefficients.
\enuma{
\item Let $V$ be a nonsingular complex affine variety. There is a natural
isomorphism 
\[
HP_i (\mc O (V)) \cong \bigoplus\nolimits_{m \in \Z_{\geq 0}} H_{dR}^{i + 2m} (V).
\]
\item Let $X$ be a smooth manifold. There is a natural isomorphism
\[
HP_i (C^\infty (X), \hat \otimes) \cong \bigoplus\nolimits_{m \in \Z_{\geq 0}}
H_{dR}^{i + 2m} (X).
\]
}
\end{thm}

Suppose that in the setting of Theorem \ref{thm:5.10}, a finite group $\Gamma$ acts
on $V$ and on $X$, by automorphism in the appropriate category. Then we can combine
Theorem \ref{thm:5.10} with Theorems \ref{thm:4.3} and \ref{thm:5.4} to obtain:

\begin{cor}\label{cor:5.15}
There are isomorphisms 
\[
\begin{array}{lll}
HP_* ( \mc O (V) \rtimes \C [\Gamma, \natural]) & \cong & 
\big( \bigoplus\nolimits_{\gamma \in \Gamma} H_{dR}^* (V^\gamma) \otimes 
\natural^\gamma \big)^\Gamma , \\
HP_* ( C^\infty (X) \rtimes \C [\Gamma, \natural], \hat \otimes) & \cong & 
\big( \bigoplus\nolimits_{\gamma \in \Gamma} H_{dR}^* (X^\gamma) \otimes 
\natural^\gamma \big)^\Gamma .
\end{array}
\]
\end{cor}

Continuity for $HP_*$ is more subtle than for $HH_*$, because $CC_{per}(A)$
has infinitely many terms in negative degrees. 

\begin{thm}\label{thm:5.11}
\textup{\cite[Proposition 2.2]{Nis} and \cite[Theorem 3]{BrPl}} 
\enuma{
\item Suppose that $\lim_{i \to \infty} A_i$ is an inductive limit of algebras and 
that there exists $N \in \N$ such that $HH_n (A_i ) = 0$ for all $n \geq N$ and all 
$i$. Then $HP_* (\lim_{i \to \infty} A_i) \cong \lim_{i \to \infty} HP_* (A_i )$.
\item Suppose that $\lim_{i \to \infty} B_i$ is a strict inductive limit of nuclear
Fr\'echet algebras and that there exists $N \in N$ such that 
$HH_n (B_i ,\hat \otimes ) = 0$ for all $n \geq N$ and all $i$. 
Then $HP_* (\lim_{i \to \infty} B_i ,\hat \otimes_b ) \cong 
\lim_{i \to \infty} HP_* (B_i ,\hat \otimes)$.
}
\end{thm}

We return to our main players, $\mc H (G)$ and $\mc S (G)$. Theorem \ref{thm:5.11}
enables us to apply an argument like in \eqref{eq:4.8} and \eqref{eq:5.2}.
Recall that for $\mf s \in \mf B (G)$ we picked $K_{\mf s} \in \CO (G)$ in
Paragraph \ref{par:5.2}.

\begin{lem}\label{lem:5.12}
\textup{\cite[(3.3) and (3.4)]{SolHP}}\\
There are natural isomorphisms
\[
\begin{array}{lllll}
HP_* (\mc H (G)) & \cong & \bigoplus_{\mf s \in \mf B (G)} HP_* (\mc H (G)^{\mf s})
& \cong & \bigoplus_{\mf s \in \mf B (G)} HP_* (\mc H (G, K_{\mf s})^{\mf s}) ,\\
HP_* (\mc S (G), \hat \otimes_b) & \cong & \bigoplus_{\mf s \in \mf B (G)} 
HP_* (\mc S (G)^{\mf s}, \hat \otimes_b) & \cong & 
\bigoplus_{\mf s \in \mf B (G)} HP_* (\mc S (G, K_{\mf s})^{\mf s} ,\hat \otimes) . 
\end{array}
\]
\end{lem}

It turns out that the isomorphisms \eqref{eq:4.19} and \eqref{eq:5.6} (which come
from Theorems \ref{thm:4.7} and \ref{thm:5.7}) are very suitable to determine 
$HP_* (\mc H (G))$ and $HP_* (\mc S (G), \hat \otimes_b)$ in representation
theoretic terms. By \cite[(4.4) and (4.5)]{SolHH}, Connes' differential $B$
on $HH_* (\mc H (G))$ and on $HH_* (\mc S (G), \hat \otimes_b)$ corresponds to
the usual exterior differential $d$ on the differential forms in 
\eqref{eq:4.18} and \eqref{eq:5.6}. Taking homology with respect to $d$ yields:

\begin{thm}\label{thm:5.13}
\textup{\cite[Theorem 4.1 and (4.7)--(4.8)]{SolHH}} \\
Let $i \in \{0,1\}$. There are isomorphisms of vector spaces
\[
\begin{array}{lll}
HP_i (\mc H (G,K_{\mf s})^{\mf s}) & \cong & HP_i \big(
\mc O (X_\nr (L)) \rtimes \C [W_{\mf s}^e ,\natural_{\mf s}^{-1}] \big) \\
& \cong & \bigoplus_{m \in \Z_{\geq 0}} \Big( \bigoplus_{w \in W_{\mf s}^e}
H_{dR}^{i + 2m} (X_\nr (L)^w) \otimes (\natural_{\mf s}^{-1})^w \Big)^{W_{\mf s}^e} ,\\
HP_i (\mc S (G,K_{\mf s})^{\mf s}, \hat \otimes) & \cong & HP_i \big( C^\infty 
(X_\nr^u (L)) \rtimes \C [W_{\mf s}^e ,\natural_{\mf s}^{-1}], \hat \otimes \big) \\
& \cong & \bigoplus_{m \in \Z_{\geq 0}} \Big( \bigoplus_{w \in W_{\mf s}^e}
H_{dR}^{i + 2m} (X_\nr^u (L)^w) \otimes (\natural_{\mf s}^{-1})^w \Big)^{W_{\mf s}^e} .
\end{array}
\]
\end{thm}

As $X_\nr^u (L)$ is $W_{\mf s}^e$-equivariant deformation retract of $X_\nr (L)$,
Lemma \ref{lem:5.12} and Theorem \ref{thm:5.13} recover \cite[Theorem 3.3]{SolHP}
and \cite[Conjecture 8.9]{BHP}:

\begin{cor}\label{cor:5.14}
The inclusions $\mc H (G,K_{\mf s})^{\mf s} \to \mc S (G,K_{\mf s})^{\mf s}$
and $\mc H (G) \to \mc S (G)$ induce isomorphisms on periodic cyclic homology.
\end{cor}

Furthermore, by \cite[Lemma 4.4]{SolHH} $HP_* (\mc H (G,K_{\mf s})^{\mf s})$ can be 
realized as the subset of $HH_* (\mc H (G,K_{\mf s})^{\mf s})$ consisting of differential 
forms that are locally constant on the various varieties $X_\nr (L)^w$, in the picture 
from \eqref{eq:4.18}. 

Similarly $HP_* (\mc S (G,K_{\mf s})^{\mf s}, \hat \otimes)$ has 
a canonical realization as the set of locally constant differential forms in 
$HH_* (\mc S (G,K_{\mf s})^{\mf s}, \hat \otimes)$, in the picture \eqref{eq:5.6}.

These are instances of a more general phenomenon in geometric group theory, already
observed in \cite{HiNi,Schn}. For any totally disconnected locally compact group $G'$
there is a decomposition
\begin{equation}\label{eq:5.10}
HH_* (\mc H (G')) = HH_* (\mc H (G'))_\cpt \oplus HH_* (\mc H (G'))_{\mr{ncpt}} .
\end{equation}
The subscript cpt means "supported on compact elements": this part of 
$HH_* (\mc H (G'))$ only uses chains 
\[
f_0 \otimes f_1 \otimes \cdots \otimes f_n \in \mc H (G')^{\otimes (n+1)}
\]
such that $f_0 (g_0) f_1 (g_1) \cdots f_n (g_n)$ with $g_i \in G'$ is zero unless
$g_0 g_1 \cdots g_n$ lies in a compact subgroup of $G'$. The subscript ncpt stands
for "supported on noncompact elements", which means that it comes from chains such
that $f_0 (g_0) f_1 (g_1) \cdots f_n (g_n)$ is zero whenever $g_0 g_1 \cdots g_n$ lies
in a compact subgroup of $G'$. 

Suppose in addition that $G'$ acts properly on an affine building, like any reductive
$p$-adic group does \cite[\S 2]{Tit}. Then the noncompact part of $HH_* (\mc H (G'))$ disappears
in periodic cyclic homology, by \cite[Theorem 6.2]{HiNi} or \cite[Theorem II]{Schn}.

\begin{thm} \label{thm:5.15} 
\textup{\cite[Theorems 1.1 and 4.2]{HiNi} and \cite[Theorem I]{Schn}} \\
In the above setting there is an isomorphism
\[
HP_i (\mc H (G')) \cong \bigoplus\nolimits_{m \in \Z_{\geq 0}} HH_{i+2m} (\mc H (G'))_\cpt .
\]
\end{thm}

For our group $G$ acting on its Bruhat--Tits building, 
Theorem \ref{thm:5.13}, \eqref{eq:4.18} and the aforementioned \cite[Lemma 4.4]{SolEnd} 
recover Theorem \ref{thm:5.15} in a representation-theoretic way.

\section{Topological K-theory}\label{sec:6}

K-theory started \cite{Ati} as a way to classify vector bundles on a topological 
space $X$, up to stable isomorphism. That gives a contravariant functor $K^0$ from 
topological spaces to abelian groups, and another functor $K^1$ is obtained by
composing $K^0$ with the suspension functor for topological spaces.
There are also higher functors $K^n$, but by Bott periodicity
these reduce to $K^0$ or $K^1$. The $\Z/ 2\Z$-graded functor $K^* = K^0 \oplus K^1$ 
forms a generalized cohomology theory, so roughly speaking it behaves like
singular or \u{C}ech cohomology. K-theory for topological spaces can be extended 
naturally to pairs of spaces $X_1 \supset X_2$. 

In noncommutative geometry, topological K-theory is usually defined and studied 
for $C^*$-algebras or Banach algebras, see for instance \cite{Bla}. The (covariant)
functor $K_0$ classifies finitely generated projective modules up to stable isomorphism.
The functor $K_1$ can be obtained as $K_0$ composed with a suspension functor
for algebras. Recall that by the Gelfand--Naimark theorem \cite[Theorem 1.4]{FGV} 
every commutative $C^*$-algebra has the form
\[
C_0 (Y) = \{ f \in C (Y \cup \{\infty\}) : f(\infty) = 0\} ,
\]
where $Y$ is locally compact Hausdorff space with one-point compactification
$Y \cup \{\infty\}$. By the Serre--Swan theorem \cite[\S 2.3 and Corollary 3.2.1]{FGV}
there is a natural isomorphism
\begin{equation}\label{eq:6.1}
K_i (C_0 (Y)) \cong K^i (Y \cup \{ \infty\}, \{\infty\}) \qquad i = 0,1 .
\end{equation}
Most references state this only for $i = 0$, but the definitions of $K^1$ in terms
of $K^0$ and $K_1$ in terms of $K_0$ are analogous, and therefore \eqref{eq:6.1} 
also holds for $i = 1$. Largely due to the results of Gelfand--Naimark and 
Serre--Swan, K-theory of $C^*$-algebras is at the heart of noncommutative geometry.

\subsection{Versions and properties of K-theory} \label{par:6.1} \

Topological K-theory can also be defined for Fr\'echet algebras (in the sense of
Definition \ref{def:5.10}), see \cite{Cun,Phi2}. This generalizes K-theory for 
Banach algebras and enjoys the following properties:
\begin{itemize}
\item \textbf{Additivity.} 
For any Fr\'echet algebras $A_n \; (n \in \N)$ and $i \in \{0,1\}$:
\[
K_i \Big( \prod\nolimits_{n=1}^\infty A_n \Big) \cong \prod\nolimits_{n=1}^\infty
K_i (A_n) .
\]
\item \textbf{Stability.} 
$K_i (M_n (A)) \cong K_i (A)$ for any Fr\'echet algebra $A$, and the same
with $M_n (\C)$ replaced by the algebra of compact operators on a separable
Hilbert space.
\item \textbf{Continuity \cite[\S 5.2.4, \S 5.5.1 and \S 8.1.5]{Bla}.}
If $\lim_{n \to \infty} B_n$ is an inductive limit of Banach algebras, then
\[
K_i\big( \lim_{n \to \infty} B_n \big) \cong \lim_{n \to \infty} K_i (B_n) .
\]
\item \textbf{Excision.} Let $0 \to A \to B \to C$ be an exact sequence of
Fr\'echet algebras. Then there exists a natural six-term exact sequence
\begin{equation}\label{eq:6.22}
\begin{array}{ccccc}
K_0 (A) & \to & K_0 (B) & \to & K_0 (C) \\
\uparrow & & & & \downarrow \\
K_1 (C) & \leftarrow & K_1 (B) & \leftarrow & K_1 (A) 
\end{array}.
\end{equation}
\item \textbf{Homotopy invariance.} If $\phi_0, \phi_1 : A \to B$ are
homotopic morphisms of Fr\'echet algebras, then $K_i (\phi_0) = K_i (\phi_1)$.
\end{itemize}
The suspension of a Fr\'echet algebra $A$ is defined as 
\[
\Sigma A = \{ f \in C (S^1;A) : f(1) = 0 \} .
\]
There are natural isomorphisms
\begin{equation}\label{eq:6.23}
K_0 (\Sigma A) \cong K_1 (A) \quad \text{and} \quad K_1 (\Sigma A) \cong K_0 (A) .
\end{equation}
The first is either a definition or \cite[Theorem 3.14]{Phi2} and the second is
a reformulation of Bott periodicity \cite[Theorem 5.5]{Phi2}. These isomorphisms
are compatible with excision, in the sense that the exact hexagon \eqref{eq:6.22}
can be rewritten as 
\begin{equation}\label{eq:6.24}
\begin{array}{ccccc}
K_0 (A) & \to & K_0 (B) & \to & K_0 (C) \\
\uparrow & & & & \downarrow \\
K_0 (\Sigma C) & \leftarrow & K_0 (\Sigma B) & \leftarrow & K_0 (\Sigma A) 
\end{array}.
\end{equation} 
Consider the split exact sequence of Fr\'echet algebras
\begin{equation}\label{eq:6.25}
0 \to \Sigma A \to C (S^1,A) \xrightarrow{\mr{ev}_1} A \to 0 ,
\end{equation}
where the splitting sends $a \in A$ to the constant function on $S^1$ with 
value $a$. Topo\-lo\-gical K-theory respects split exact sequences, so
\eqref{eq:6.25} and \eqref{eq:6.23} induce isomorphisms
\begin{equation}\label{eq:6.26}
K_i (C(S^1,A)) \cong K_i (\Sigma A) \oplus K_i (A) \cong K_0 (A) \oplus 
K_1 (A) \qquad i = 0, 1 .
\end{equation}
The isomorphism $K_* (M_n (\C)) \cong K_* (A)$ can be regarded as an instance
of Morita invariance of topological K-theory. For unital Fr\'echet algebras,
this can be pushed further.

\begin{thm}\label{thm:6.19}
Let $A$ and $B$ be unital Fr\'echet algebras, such that $A^\times \subset A$
and $B^\times \subset B$ are open. Suppose that $A$ and $B$ are Morita equivalent.
Then the Morita bimodules induce an isomorphism $K_* (A) \cong K_* (B)$.
\end{thm}
\begin{proof}
Let $P$ and $Q$ be the bimodules implementing the Morita equivalence, so
$P \otimes_B Q \cong A$ and $Q \otimes_A P \cong B$. Since $A$ and $B$ are
unital, $P$ and $Q$ are finitely generated and projective. Then $C(S^1,P)$ and
$C(S^1,B)$ are finitely generated projective bimodules for $C(S^1,A)$ and 
$C(S^1,B)$, such that
\[
C(S^1,P) \otimes_{C(S^1,B)} C(S^1,Q) \cong C(S^1,A) \; \text{and} \; 
C(S^1,Q) \otimes_{C(S^1,A)} C(S^1,P) \cong C(S^1,B) .
\]
This provides a Morita equivalence between $C(S^1,A)$ and $C(S^1,B)$. 
The assumptions on $A$ and $B$ and the compactness of $S^1$ imply that
$C(S^1,A)^\times \subset C(S^1,A)$ and $C(S^1,B)^\times \subset C(S^1,B)$
are open. According to \cite[Theorem 7.7]{Phi2}, $K_0 (C (S^1,A))$ is 
naturally isomorphic to the K-group of the monoid of finitely generated
projective $C(S^1,A)$-modules. The same holds for $C(S^1,B)$. Therefore
the maps
\begin{equation}\label{eq:6.38}
K_0 (C (S^1,A)) \longleftrightarrow K_0 (C(S^1,B))
\end{equation} 
induced by $C(S^1,Q) \otimes_{C(S^1,A)}$ and $C(S^1,P) \otimes_{C(S^1,B)}$
are group isomorphisms. The decomposition
\[
K_0 (C (S^1,A)) \cong K_0 (A) \oplus K_1 (A)
\]
from \eqref{eq:6.25}--\eqref{eq:6.26} comes from $C_0 (S^1,\{1\}) \to
C (S^1) \to C(\{1\})$. It works in the same way for $C(S^1,B)$, so
\eqref{eq:6.38} respects those decompositions and yields isomorphisms
\[
K_0 (A) \cong K_0 (B) \quad \text{and} \quad K_1 (A) \cong K_1 (B) . \qedhere
\]
\end{proof}

The Serre--Swan theorem admits a generalization to Fr\'echet algebras, 
of which we state a simplified version:

\begin{thm}\label{thm:6.1} \textup{\cite[Theorem 7.15]{Phi2}} \\
Let $Y$ be a locally compact Hausdorff space. Let $A$ be a commutative
Fr\'echet algebra such that the  maximal ideal space of the unitization $A_+$ is
$Y \cup \{\infty\}$, where $\infty$ corresponds to the canonical projection 
$A_+ \to \C$. (For example $A$ could be $C^\infty (Y)$
if $Y$ is a compact smooth manifold.) Then there are natural isomorphisms
\[
K_i (A) \cong K_i (C_0 (Y)) \cong K^i (Y \cup \{\infty\} ,\{\infty \} ) 
\qquad i = 0, 1. 
\]
\end{thm}

In Theorem \ref{thm:6.1} the natural map $A \to C_0 (Y)$ induces an isomorphism
on $K$-theory. That phenomenon holds in much larger generality, and it is 
called the density theorem in K-theory:

\begin{thm}\label{thm:6.2} 
\textup{\cite[Th\'eor\`eme A.2.1]{Bos}} \\
Let $\phi : A \to B$ be a morphism of Fr\'echet algebras.
Let $A_+$ be the unitization of $A$ from \eqref{eq:4.20}, and extend
$\phi$ to $\phi_+ : A_+ \to B_+$. Assume that:
\begin{itemize}
\item $A_+^\times$ (the set of invertible elements in $A_+$) is open in $A_+$, 
\item $B_+^\times$ is open in $B_+$,
\item $\phi (A)$ is dense in $B$,
\item whenever $a \in A_+$ and $\phi_+ (a)$ is invertible in $B_+$, $a$ is
invertible in $A_+$.
\end{itemize}
Then $K_* (\phi) : K_* (A) \to K_* (B)$ is an isomorphism.
\end{thm}

Many versions of K-theory admit an extension to an equivariant setting.
We focus on equivariance with respect to actions of a finite group $\Gamma$.
Atiyah \cite{Ati} already defined the $\Gamma$-equivariant K-theory of 
a $\Gamma$-space $Y$. By design $K_\Gamma^0 (Y)$ classifies $\Gamma$-equivariant
vector bundles on $Y$, up to stable equivalence, and $K_\Gamma^1 (Y)$ is
$K_\Gamma^0$ of the suspension of $X$.

Recall that the classical Chern character assigns to a vector bundle over $Y$
a class in the \u{C}ech cohomology of $Y$. It gives rise to a natural transformation
\begin{equation}\label{eq:6.2}
\text{Ch} : K^* (Y) \to \breve{H}^* (Y;\Q) .  
\end{equation}
If one replaces $K^* (Y)$ by $K^* (Y) \otimes_\Z \Q$, then \eqref{eq:6.2} is often
an isomorphism. 

Suppose now that $X$ is a compact Hausdorff space with a $\Gamma$-action. In this
setting Baum and Connes \cite{BaCo} constructed a $\Gamma$-equivariant version of
\eqref{eq:6.2}. Write 
\[
\widetilde X = \bigsqcup\nolimits_{\gamma \in \Gamma} \{\gamma\} \times X^\gamma, 
\]
and let $\Gamma$ act on $\widetilde X$ by $\gamma_1 \cdot (\gamma_2,x) = 
(\gamma_1 \gamma_2 \gamma^{-1}, \gamma_1 x)$. 

\begin{thm}\label{thm:6.3}
\textup{\cite[Theorem 1.19]{BaCo}} \\
There are natural isomorphisms 
\[
K_\Gamma^* (X) \otimes_\Z \C \isom \big( K^* (\widetilde X) \otimes_\Z \C 
\big)^\Gamma \isom \breve{H}^* (\widetilde X; \C)^\Gamma \cong
\breve{H}^* (\widetilde X / \Gamma ;\C) .
\]
\end{thm}

The map $K_\Gamma^* (X) \to \breve{H}^* (\widetilde X; \C)^\Gamma$ in
Theorem \ref{thm:6.3} is called an equivariant Chern character.

For the equivariant K-theory $K_*^\Gamma$ of $\Gamma$-$C^*$-algebras we refer to 
\cite{Phi1}, here we restrict ourselves to a few remarks. The equivariant Serre--Swan
theorem \cite[Theorem 2.3.1]{Phi1} says that for any locally compact Hausdorff 
$\Gamma$-space $Y$ there is a natural isomorphism
\begin{equation}\label{eq:6.3}
K_*^\Gamma (C_0 (Y)) \cong K_\Gamma^* (Y \cup \{\infty\}, \{\infty\}) .
\end{equation}
Besides $K_\Gamma^*$, there is an equivalent way to introduce $\Gamma$-equivariance:

\begin{thm}\label{thm:6.4} \text{\cite{Jul}} \\
For any $\Gamma$-$C^*$-algebra $B$, there is a natural isomorphism
$K_*^\Gamma (B) \cong K_* (B \rtimes \Gamma)$.
\end{thm}

In view of Theorem \ref{thm:6.4}, it makes sense to define
\[
K_*^\Gamma (A) = K_* (A \rtimes \Gamma) 
\qquad \text{for any Fr\'echet } \Gamma\text{-algebra} .
\]
When $X$ is a compact Hausdorff $\Gamma$-space, Theorems \ref{thm:6.3}, \ref{thm:6.4}
and \eqref{eq:6.3} combine to
\begin{equation}\label{eq:6.4}
K_* (C(X) \rtimes \Gamma) \otimes_\Z \C \cong K_\Gamma^* (X) \otimes_\Z \C 
\cong \breve{H}^* (\widetilde{X};\C)^\Gamma .
\end{equation}
Let $\natural : \Gamma \times \Gamma \to \C^\times$ be a 2-cocycle. Recall from
\eqref{eq:2.2} that $\C [\Gamma,\natural] \cong e_\natural \C [\Gamma^*]$ for a 
central extension $Z^* \to \Gamma^* \to \Gamma$, a character $c_\natural$ of
$Z^*$ and the associated idempotent $e_\natural \in \C [Z^*]$. Since $e_\natural$
is a central idempotent in $C(X) \rtimes \Gamma^*$, we can write
\[
K_* (C(X) \rtimes \C [\Gamma,\natural]) = K_* (e_\natural (C(X) \rtimes \Gamma^*))
\cong e_\natural K_* ( C(X) \rtimes \Gamma^*) \cong e_\natural K_{\Gamma^*}^* (X) .
\]
\begin{defn}\label{def:6.6}
The $\natural$-twisted equivariant $K$-theory of a locally compact Hausdorff 
$\Gamma$-space $Y$ is
\[
K^*_{\Gamma,\natural} (Y) := e_\natural K_{\Gamma^*}^* (Y \cup \{\infty\}, 
\{\infty\}) \cong K_* (C_0 (Y) \rtimes \C [\Gamma,\natural]) .
\]
\end{defn}
This is the K-theory of $\Gamma^*$-equivariant vector bundles over $X$ on which
$Z^*$ acts as the character $c_\natural$.
From \eqref{eq:6.4} and \cite[proof of Theorem 1.2]{SolTwist} we deduce that the
$\Gamma^*$-equivariant Chern character induces an isomorphism
\begin{multline}\label{eq:6.6}
K^*_{\Gamma,\natural} (X) \otimes_\Z \C \cong 
e_\natural K_* ( C(X) \rtimes \Gamma^*) \isom \\ 
e_\natural \breve{H}^* \Big( \bigsqcup\nolimits_{\gamma \in \Gamma^*} \{\gamma\} \times 
X^\gamma \Big)^{\Gamma^*} \cong \Big( \bigoplus\nolimits_{\gamma \in \Gamma} \breve{H}^* 
(X^\gamma) \otimes \natural^\gamma \Big)^\Gamma .
\end{multline}
When $X$ is in addition a smooth manifold, \eqref{eq:6.4} and \eqref{eq:6.6}
can be composed with Theorem \ref{thm:6.2} (for $C^\infty (X) \to C (X)$ and
related inclusions) to obtain
\begin{equation}\label{eq:6.7}
\begin{aligned}
& K_*^\Gamma (C^\infty (X)) \otimes_\Z \C \cong K_* (C^\infty (X) \rtimes \Gamma) 
\otimes_\Z \C \cong H_{dR}^* (\widetilde X )^\Gamma ,\\
& K_* (C^\infty (X) \rtimes \C [\Gamma, \natural]) \otimes_\Z \C \cong
\Big( \bigoplus\nolimits_{\gamma \in \Gamma} H_{dR}^* (X^\gamma) \otimes 
\natural^\gamma \Big)^\Gamma .
\end{aligned}
\end{equation}
There is a natural transformation 
\begin{equation}\label{eq:6.40}
Ch : K_* \to HP_*
\end{equation}
from K-theory for Fr\'echet algebras to
periodic cyclic homology for Fr\'echet algebras, also called the Chern character
\cite{Cun}. It is constructed as the restriction of a bivariant Chern character,
and latter has a universal property which makes it unique \cite[Korollar 6.5]{Cun}.
Therefore, for any smooth manifold $Y$,
\[
\text{Ch} : K_* (C^\infty (Y)) \to HP_* (C^\infty (Y), \hat \otimes) \cong 
H_{dR}^* (Y) 
\]
agrees with the classical Chern character \eqref{eq:6.2}. This is also checked 
in a more concrete, purely algebraic setting in \cite[Proposition 8.3.9]{Lod}. 
That compatibility of Chern characters and Theorems \ref{thm:5.10}.b and 
\ref{thm:6.1} imply:

\begin{thm}\label{thm:6.5}
\textup{\cite{Con,SolCh}} \\
Let $X$ be a compact smooth manifold. 
\enuma{
\item $\mr{Ch} \otimes \mr{id} : K_* (C^\infty (X)) \otimes_\Z \C \to
HP_* (C^\infty (X), \hat \otimes )$ is an isomorphism.
\item Let $\Gamma$ be a finite group acting on $X$ and let $\natural$ 
be a 2-cocycle of $\Gamma$. Then
\[
\mr{Ch} \otimes \mr{id} : K_* (C^\infty (X) \rtimes \C [\Gamma,\natural] ) 
\otimes_\Z \C \to HP_* (C^\infty (X) \rtimes \C [\Gamma,\natural], \hat \otimes )
\] is an isomorphism.
}
\end{thm}

In fact Theorem \ref{thm:6.5} applies in much larger generality. Firstly, twisted 
crossed pro\-ducts can be replaced by $\Gamma$-invariants for certain actions of 
$\Gamma$ on matrix algebras over $C^\infty (X)$ \cite[Theorem 6]{SolCh}.
Secondly, Theorem \ref{thm:6.4} also holds for many noncompact manifolds $X$. 
Not for all though, because $\otimes_\Z \C$ does not commute with infinite direct 
products \cite[Appendix]{SolCh}.

\subsection{K-theory of $C_r^* (G)$ and $\mc S (G)$, modulo torsion} \
\label{par:6.2}

We survey the relations between the topological K-theory and the periodic cyclic
homology of $C_r^* (G)$ and $\mc S (G)$. Recall from \eqref{eq:1.9}, Definition
\ref{def:1.SG} and Theorem \ref{thm:1.20} that
\[
\begin{array}{lllll}
C_r^* (G) & = & \lim_{K \in \CO (G)} C_r^* (G,K) & = \lim_{n \to \infty} (G,K_n) ,\\
\mc S (G) & = & \lim_{K \in \CO (G)} \mc S (G,K) & = 
\bigcup_{n=1}^\infty \mc S (G,K_n) .
\end{array}
\]
The algebra $\mc S (G)$ is not Fr\'echet, so its topological K-theory presents new 
challenges. While K-theories have been constructed for wider classes of topological
algebras, it is unclear whether those functors commute with direct limits. To avoid
such problems, we simply define
\[
K_* (\mc S (G)) := \lim_{K \in \CO (G)} K_* (\mc S (G,K)) .
\]
Notice that this makes sense because each $\mc S (G,K)$ is a Fr\'echet algebra,
see Theorem \ref{thm:1.4}.a.

\begin{prop}\label{prop:6.6}
There is a commutative diagram
\[
\begin{array}{ccc}
\lim_{K \in \CO (G)} K_* (C_r^* (G,K)) & \to & K_* (C_r^* (G)) \\
\uparrow &  & \uparrow \\
\lim_{K \in \CO (G)} K_* (\mc S (G,K)) & = & K_* (\mc S (G))
\end{array}
\]
in which all the arrows are natural isomorphisms.
\end{prop}
\begin{proof}
The continuity of K-theory for Banach algebras tells us that upper line is an 
isomorphism, where the map is induced by the inclusions $C_r^* (G,K) \to C_r^* (G)$.
By Theorems \ref{thm:1.4} and \ref{thm:6.2} 
\begin{equation}\label{eq:6.8}
\text{the inclusion } \mc S(G,K) \to C_r^* (G,K) 
\text{ induces an isomorphism on K-theory.}
\end{equation}
Hence the left column of the diagram is a natural isomorphism. We define the 
map in the right column of the diagram as the composition of the other maps.
Then it is a natural isomorphism, and it is induced by the inclusion 
$\mc S (G) \to C_r^* (G)$.
\end{proof}

Like in \eqref{eq:5.2}, one deduces from Propositions \ref{prop:1.22} and
\ref{prop:6.6} and Theorems \ref{thm:6.1} and \ref{thm:6.2} that there 
are natural isomorphisms
\begin{equation}\label{eq:6.9}
\begin{array}{lllll}
K_* (C_r^* (G)) & \cong & \bigoplus_{\mf s \in \mf B (G)} 
K_* (C_r^* (G,K_{\mf s})^{\mf s}) & \cong & \bigoplus_{\mf s \in \mf B (G)} 
K_* (C_r^* (G)^{\mf s}) \\
K_* (\mc S (G)) & \cong & \bigoplus_{\mf s \in \mf B (G)} 
K_* (\mc S (G,K_{\mf s})^{\mf s})
\end{array}.
\end{equation}

\begin{ex}\label{ex:6.7}
Consider $G = SL_2 (F)$ and $\mf s = [T,\mr{triv}]_G$. From Example \ref{ex:1.B}
and Theorem \ref{thm:1.17} we see that the $C_r* (G)^{\mf s}$ is Morita 
equivalent to 
\[
\big( C(S^1) \otimes M_2 (\C) \big)^{S_2} \oplus \C_{St} := A \oplus \C_{St} .
\]
As $K_* (\C_{St}) = K_0 (\C) = \Z$, we focus on $A$. By a suitable choice of
coordinates, we can achieve that $S_2 = \{1, s_\alpha\}$ acts by
\[
(s_\alpha a)(z) = \matje{1}{0}{0}{z} a (z^{-1}) \matje{1}{0}{0}{z^{-1}} 
\qquad a \in C(S^1) \otimes M_2 (\C), z \in S^1 .
\] 
This shows that every $S_2$-orbit in $S^1 \setminus \{-1\}$ supports a unique
irreducible $A$-re\-pre\-sentation, while there are precisely two inequivalent
irreducible $A$-representations with $C(S^1)^{S_2}$-character -1. 
The upper half circle is a fundamental domain for $S_2$ acting on $S^1$, and it
is homeomorphic to $[-1,1]$ by taking real parts. Evaluation of $A$ at -1
yields a short exact sequence of $C^*$-algebras
\[
0 \to C_0 ( (-1,1]) \otimes M_2 (\C) \to A \to \C^2 \to 0 .
\]
The associated six-term exact sequence is
\[
\begin{array}{ccccc}
K^0 \big( (-1,1] \big) = 0 & \to & K_0 (A) & \to & K_0 (\C^2) = \Z^2\\
\uparrow & & & & \downarrow \\
K_1 (\C^2) = 0 & \leftarrow & K_1 (A) & \leftarrow & K^1 \big( (-1,1] \big) = 0
\end{array}.
\]
Here $K_* ((-1,1]) = 0$ because the algebra $C_0 ((-1,1])$ is homotopy equivalent
to 0. We find that $K_* (A) \cong K_0 (A) \cong \Z^2$.
\end{ex}

With methods like in Example \ref{ex:6.7} one can compute $K_* (C_r^* (G)^{\mf s})$
for many Bernstein blocks $\Rep (G)^{\mf s}$, but the computations quickly become
cumbersome.

To analyse $K_* (\mc S (G,K_{\mf s})^{\mf s})$ and $K_* (C_r^* (G,K_{\mf s})^{\mf s})$
in general, we relate them to the previous section. Theorem \ref{thm:6.5} can be 
applied to the Fr\'echet algebras appearing in the Plancherel isomorphism for
$\mc S (G)$ (Theorem \ref{thm:1.16}).

\begin{thm}\label{thm:6.8}
\textup{\cite[Theorem 10 and Corollary 11]{SolCh}} \\
For any $K \in \CO (G)$, the Chern character induces an isomorphism
\[
K_* (\mc S (G,K)) \otimes_\Z \C \to HP_* (\mc S (G,K), \hat \otimes) .
\]
\end{thm}

A combination of the results in this and the previous paragraph yields a 
description of $K_* (C_r^* (G))$ modulo torsion:

\begin{cor}\label{cor:6.9}
\enuma{
\item Fix $\mf s = [L,\sigma]_G \in \mf B (G)$. There are isomorphisms 
\begin{align*}
K_* (C_r^* (G)^{\mf s}) \otimes_\Z \C & \cong K_* (C_r^* (G,K_{\mf s})^{\mf s}) 
\otimes_\Z \C \cong K_* (\mc S (G,K_{\mf s})^{\mf s}) \otimes_\Z \C \\
& \cong HP_* (\mc S (G,K_{\mf s})^{\mf s}, \hat \otimes) \cong HP_* \big( C^\infty 
(X_\nr^u (L)) \rtimes \C [W_{\mf s}^e, \natural_{\mf s}^{-1}], \hat \otimes \big) \\
& \cong K_* \big( C^\infty (X_\nr^u (L)) \rtimes \C [W_{\mf s}^e, \natural_{\mf s}^{-1}] \big) 
\otimes_\Z \C \\
& \cong K_* \big( C (X_\nr^u (L)) \rtimes \C [W_{\mf s}^e, \natural_{\mf s}^{-1}] \big) 
\otimes_\Z \C \cong K_{W_{\mf s}^e,\natural_{\mf s}^{-1}} (X_\nr^u (L)) \otimes_\Z \C .
\end{align*}
\item There are isomorphisms
\[
K_* (C_r^* (G)) \otimes_\Z \C \cong K_* (\mc S (G)) \otimes_\Z \C \cong
\bigoplus\nolimits_{\mf s \in \mf B (G)} K_{W_{\mf s}^e,\natural_{\mf s}^{-1}} 
(X_\nr^u (L)) \otimes_\Z \C .
\]
}
\end{cor}
\begin{proof}
(a) The first isomorphism is \eqref{eq:6.9} and the second comes from \eqref{eq:6.8}.
Then we use Theorems \ref{thm:6.8} and \ref{thm:5.13}. The fifth isomorphism is an
instance of Theorem \ref{thm:6.5}, the sixth comes from Theorem \ref{thm:6.2}
and the last step is Definition \ref{def:6.6}.\\
(b) This follows from Proposition \ref{prop:6.6}, \eqref{eq:6.9} and part (a).
\end{proof}

Corollary \ref{cor:6.9} gives a description of the group $K_* (C_r^* (G))$ modulo 
torsion. In some cases that determines $K_* (C_r^* (G))$ up to isomorphism, because 
it does not have torsion elements:

\begin{thm}\label{thm:6.24} \textup{\cite[Theorem 5.3]{SolComp}} \\
Let $\mc G$ be an inner form of $GL_n$, or a symplectic group, or a special 
orthogonal group (not necessarily $F$-split). Then $K_* (C_r^* (G))$ is
a free abelian group.
\end{thm}

\subsection{Relation with the Baum--Connes conjecture} \

K-theory of group-$C^*$-algebras figures prominently in the Baum--Connes
conjecture \cite{BCH,Val}. For any Hausdorff space $Y$ with a proper $G$-action,
Kasparov's equivariant KK-theory provides a notion of the $G$-equivariant
K-homology $K_*^G (Y)$. There is an assembly map
\[
\mu_{G,Y} : K_*^G (Y) \to K_* (C_r^* (G)) ,
\]
which can be defined in several (analytic) ways \cite[\S 6]{Val}. The
Baum--Connes conjecture asserts that $\mu_{G,Y}$ is an isomorphism when
$Y$ is a classifying space for proper $G$-actions, as in \cite[\S 1]{BCH}.

For our reductive $p$-adic group $G$, we can take as $Y$ the (extended)
Bruhat--Tits building $\mc B G$, see \cite[\S 6]{BCH}. In this case the
Baum--Connes conjecture is known to hold, a celebrated result of V. Lafforgue
\cite{Laf}. 

There is a more algebraic version of equivariant K-homology, called (equivariant)
chamber homology, see \cite[\S 3]{HiNi} and \cite[\S 2]{BHP}. It is available
for any totally disconnected locally compact group $G'$ acting on a 
polysimplicial complex $X'$, and can be defined as follows. Let $X'_n$ be
the set of $n$-dimensional polysimplices of $X'$. For every $\sigma \in X'_n$,
the stabilizer group $G'_\sigma$ is compact, and we form its (complex-valued)
Hecke algebra $\mc H (G'_\sigma)$. The usual boundary operator for polysimplices
induces a boundary operator 
\[
\partial_* \quad \text{on} \quad \bigoplus\nolimits_{n \geq 0} 
\bigoplus\nolimits_{\sigma \in X'_n} \mc H (G'_\sigma) ,
\]
which decreases the degrees by one. Taking $G$-coinvariants, we form the
differential complex with terms $C_ n = \big( \bigoplus\nolimits_{\sigma \in X'_n} 
\mc H (G'_\sigma) \big)_{G'}$. The homology of $(C_*,\partial_*)$ is
denoted
\[
CH_*^{G'} (X') = \bigoplus\nolimits_{n \geq 0} CH_n^{G'} (X') .
\]
This is similar to $HH_* (\mc H (G'))$, in the sense that it looks somewhat
like a resolution which could be used in the definition of $HH_n$ as a
derived functor. Higson and Nistor made that precise:

\begin{thm}\label{thm:6.21} 
\textup{\cite[Theorem 4.2]{HiNi}} \\
Assume that, for every $K' \in \CO (G')$, the fixed point set of $K'$ in
$X'$ is nonempty and contractible. Then there are natural isomorphisms
\[
CH_n^{G'}(X') \cong HH_n (\mc H (G'))_\cpt ,
\]
where \textup{cpt} means "supported on compact elements" as in Theorem \ref{thm:5.15}.
\end{thm}

The definition of $CH_*^{G'}(X')$ can be reformulated with 
$\bigoplus_{\sigma \in X'_n / G'} R(G'_\sigma) \otimes_\Z \C$ instead of $C_n$,
see \cite[p. 215]{BHP}. In that way one can regard equivariant chamber homology as
an algebraic combination of ordinary homology and virtual representations of 
compact subgroups. This is reminiscent of how the algebraic K-theory of $\mc H (G)$
was computed in \cite{BaLu}.

Similar to \eqref{eq:6.40}, Voigt constructed an equivariant Chern character
\[
\mr{Ch}^{G'} : K_i^{G'} \to \bigoplus\nolimits_{m \geq 0} CH_{i + 2m}^{G'}.
\]

\begin{thm}\label{thm:6.22} 
\textup{\cite[\S 6]{Voi}} \\
Suppose that $X'$ is a finite dimensional locally finite polysimplicial complex,
and that $X' / G'$ is compact. Then 
\[
\mr{Ch}^{G'} \otimes \mr{id} : K_i^{G'}(X') \otimes_\Z \C \to 
\bigoplus\nolimits_{m \geq 0} CH_{i + 2m}^{G'}(X') \quad \text{is an isomorphism.}
\]
\end{thm}

For $G$ acting on $\mc B G$, the conditions in Theorem \ref{thm:6.21} are fulfilled
by the CAT(0)-property of $\mc B G$ and the Bruhat--Tits fixed point theorem
\cite[\S 2.3]{Tit}. The conditions on $X'$ in Theorem \ref{thm:6.22} hold for
$\mc B G$ by \cite[\S 2.2 and \S 2.5]{Tit}.

The various homology theories associated to $G$, its group algebras and $\mc B G$
can be combined in the following diagram, which is an extended version of 
\cite[Proposition 9.4]{BHP}:
\begin{equation}\label{eq:6.41}
\xymatrix{
K_i^G (\mc B G) \ar[d]^{\mr{Ch}^G} \ar[r]^{\mu_{G,\mc B G}} & K_i (C_r^* (G)) &
K_i (\mc S (G)) \ar[l] \ar[d]^{\mr{Ch}} \\
\bigoplus\limits_{m \geq 0} CH_{i + 2m}^G (\mc B G) \ar@2{-}[r]^{\hspace{-16mm} \sim} & 
\bigoplus\limits_{m \geq 0} HH_{i + 2m}(\mc H (G))_\cpt \cong HP_i (\mc H (G)) \ar[r] &
HP_i (\mc S (G), \hat \otimes_b) 
}
\end{equation}
Here the Chern character $\mr{Ch} : K_* (\mc S (G)) \to HP_* (\mc S (G), \hat \otimes_b)$ 
is defined as the direct limit over $K \in \CO (G)$ of 
$\mr{Ch} : K_* (\mc S (G,K)) \to HP_* (\mc S (G,K), \hat \otimes)$, which makes sense by
Theorem \ref{thm:5.11} and our definition of $K_* (\mc S (G))$.

\begin{thm}\label{thm:6.23}
\enuma{
\item All the horizontal maps in the diagram \eqref{eq:6.41} are natural isomorphisms.
\item The vertical maps in \eqref{eq:6.41} become isomorphisms if we apply $\otimes_\Z \C$
to the K-groups.
\item If we extend \eqref{eq:6.41} with inverses of all the arrows that are isomorphisms,
then the diagram commutes.
}
\end{thm}
\begin{proof}
(a) For the upper line we refer to \cite{Laf} and Proposition \ref{prop:6.6}.
The first isomorphism on the lower line is Theorem \ref{thm:6.21}, the second is 
Theorem \ref{thm:5.15} and the third is Corollary \ref{cor:5.14}. \\
(b) This follows from Theorems \ref{thm:6.22} and \ref{thm:6.8}.\\
(c) See \cite[Theorem 3.7 and Lemma 3.9]{SolHP}.
\end{proof}

All the terms in \eqref{eq:6.41} that do not involve $\mc B G$ admit a natural
Bernstein decomposition, see Lemma \ref{lem:5.12} and \eqref{eq:6.9}. 
We computed the summands associated to an arbitrary $\mf s \in \mf B (G)$ explicitly
in Theorem \ref{thm:5.13} and Corollary \ref{cor:6.9}. Thus the isomorphisms
$\mu_{G,\mc B G}$ and Theorem \ref{thm:6.21} provide natural Bernstein decompositions
of $K_*^G (\mc B G)$ and $CH_*^G (\mc B G)$. 

Unfortunately it remains inclear how to describe these decompositions in terms of the
action of $G$ on its Bruhat--Tits building $\mc B G$. This reflects the difficulty
of recovering the Bernstein decomposition via restrictions of $G$-representations to
compact open subgroups of $G$.

\subsection{A progenerator of Mod$(C_r^* (G)^{\mf s})$ } \
\label{par:6.3}

These and the next paragraphs contain some new material, which aims to compute
$K_* (C_r^* (G))$ including torsion elements.
Fix $\mf s = [L,\sigma]_G \in \mf B (G)$, with a unitary 
$\sigma \in \Irr_\cusp (L)$. Recall from Proposition \ref{prop:3.4} that 
$\Rep (G)^{\mf s}$ has a canonical progenerator
\[
\Pi_{\mf s} = I_P^G (\ind_{L^1}^L (\sigma)) .
\]
For the remainder of the section, we pick $K \in \CO (G)$ such that
\begin{equation}\label{eq:6.39}
\mc H (G)^{\mf s} \quad \text{and} \quad \mc H(G,K)^{\mf s} 
\text{ are Morita equivalent,}
\end{equation}
for instance $K_{\mf s}$ from \eqref{eq:5.2}. Then $\langle K \rangle \Pi_{\mf s}$ is a 
progenerator of $\Mod (\mc H (G,K)^{\mf s})$. Furthermore Proposition \ref{prop:1.22}
shows that 
\begin{equation}\label{eq:6.42}
C_r^* (G)^{\mf s} \quad \text{and} \quad C_r^* (G,K)^{\mf s} \text{ are Morita equivalent.}
\end{equation}
We define
\[
\begin{array}{lll}
i_{\mf s}^c & := & C_r (G)^{\mf s} \otimes_{\mc H (G)^{\mf s}} \Pi_{\mf s} ,\\
i_{\mf s,K}^c & := & C_r (G,K)^{\mf s} \otimes_{\mc H (G,K)^{\mf s}}
\langle K \rangle \Pi_{\mf s} .
\end{array}
\]

\begin{lem}\label{lem:6.10}
\enuma{
\item \hspace{-2mm} The $C_r^* (G,K)^{\mf s}$-module
$\Pi_{\mf s,K}^c$ is a progenerator of $\Mod (C_r^* (G,K)^{\mf s})$.
\item The $C_r^* (G)^{\mf s}$-module $\Pi_{\mf s}^c$ is a progenerator of 
$\Mod (C_r^* (G)^{\mf s})$ and $\Pi_{\mf s,K}^c = \langle K \rangle \Pi_{\mf s}^c$.
}
\end{lem}
\begin{proof}
(a) The algebra $\mc H (G,K)^{\mf s}$ is unital and its module $\langle K \rangle 
\Pi_{\mf s}$ is finitely generated projective. So it is a direct summand of a free module
of finite rank, say $\mc H (G,K)^n$. Then $\Pi_{\mf s,K}^c$ is a direct summand
of $(C_r^* (G,K)^{\mf s})^n$, so it is finitely generated and projective.
For any nonzero $M \in \Mod (C_r^* (G,K)^{\mf s})$ we have
\[
\Hom_{C_r^* (G,K)^{\mf s}} ( \Pi_{\mf s,K}^c, M) \cong
\Hom_{\mc H (G,K)^{\mf s}} (\langle K \rangle \Pi_{\mf s},M) \neq 0 .
\] 
Therefore $\Pi_{\mf s,K}^c$ generates $\Mod (C_r^* (G,K)^{\mf s})$.\\
(b) By part (a) and \eqref{eq:6.42}, $C_r^* (G)^{\mf s} \langle K \rangle 
\otimes_{C_r^* (G,K)^{\mf s}} \Pi_{\mf s,K}^c$ is a progenerator of\\ 
$\Mod (C_r^* (G)^{\mf s})$. It can be rewritten as 
\begin{multline}\label{eq:6.43}
C_r^* (G)^{\mf s} \langle K \rangle 
\otimes_{C_r^* (G,K)^{\mf s}} C_r^* (G,K)^{\mf s} \otimes_{\mc H (G,K)^{\mf s}}
\langle K \rangle \mc H (G)^{\mf s} \otimes_{\mc H (G)^{\mf s}} \Pi_{\mf s} = \\
C_r^* (G)^{\mf s} \langle K \rangle \otimes_{\mc H (G,K)^{\mf s}} 
\langle K \rangle \mc H (G)^{\mf s} \otimes_{\mc H (G)^{\mf s}} \Pi_{\mf s} .
\end{multline}
By \eqref{eq:6.39} this reduces to $C_r (G)^{\mf s} \otimes_{\mc H (G)^{\mf s}} 
\Pi_{\mf s} = \Pi_{\mf s}^c$, which is therefore also a progenerator. It follows
that
\[
\langle K \rangle \Pi_{\mf s}^c = \langle K \rangle C_r^* (G)^{\mf s} \langle K 
\rangle \otimes_{C_r^* (G,K)^{\mf s}} \Pi_{\mf s,K}^c = \Pi_{\mf s,K}^c . \qedhere
\]
\end{proof}

It is easier to work with $\Pi_{\mf s,K}^c$ than with $\Pi_{\mf s}^c$ because the 
algebra $C_r^* (G,K)^{\mf s}$ is unital and its irreducible modules have finite
dimension. The upcoming results
have versions for $\Pi_{\mf s}^c$, which can be derived with arguments like in
the proof of Lemma \ref{lem:6.10}.b.
By Lemma \ref{lem:6.10}.a there exists $n_{\mf s}^c \in \N$ and an idempotent
$e_{\mf s}^c \in M_{n_{\mf s}^c} (C_r^* (G,K)^{\mf s})$ such that
\[
\Pi_{\mf s,K}^c \cong (C_r^* (G,K)^{\mf s})^{n_{\mf s}^c} e_{\mf s}^c .
\]
Then $\End_{C_r^* (G,K)^{\mf s}} (\Pi_{\mf s,K}^c)^{op}$, which by definition 
acts from the right on $\Pi_{\mf s,K}^c$, is isomorphic to 
$e_{\mf s}^c M_{n_{\mf s}^c} (C_r^* (G,K)^{\mf s}) e_{\mf s}^c$. In particular
$\End_{C_r^* (G,K)^{\mf s}} (\Pi_{\mf s,K}^c)^{op}$ is isomorphic to a corner 
in the Banach algebra $M_{n_{\mf s}^c} (C_r^* (G,K)^{\mf s})$, so 
\begin{equation}\label{eq:6.28}
\End_{C_r^* (G,K)^{\mf s}}(\Pi_{\mf s,K}^c )^{op} \quad \text{and} \quad
\End_{C_r^* (G,K)^{\mf s}} (\Pi_{\mf s,K}^c ) \text{ are Banach algebras.}
\end{equation}
As an instance of Proposition \ref{prop:3.1}, we find that
\begin{equation}\label{eq:6.29} 
C_r^* (G,K)^{\mf s} \quad \text{and} \quad \End_{C_r^* (G,K)^{\mf s}}
(\Pi_{\mf s,K}^c)^{op} \text{ are Morita equivalent.}
\end{equation}
This works both as abstract rings and as Banach algebras, with the Morita
bimodules
\[
\Pi_{\mf s,K}^c \quad \text{and} \quad e_{\mf s}^c (C_r^* (G,K)^{\mf s})^{n_{\mf s}^c} 
\cong \Hom_{C_r^* (G,K)^{\mf s}}(\Pi_{\mf s,K}^c, C_r^* (G,K)^{\mf s}) .
\]
Recall the subset $\Delta (G,\mf s) \subset \Delta (G)$ from \eqref{eq:1.35}.

\begin{lem}\label{lem:6.11}
There exist positive integers $n_{\mf d}$ for $\mf d = [M,\delta] \in 
\Delta (G,\mf s)$, such that 
\[
\Pi_{\mf s,K}^c \cong \bigoplus_{[M,\delta] \in \Delta (G,\mf s)}
\langle K \rangle I_{M U_Q}^Q \big( \delta \otimes C(X_\nr^u (M)) \big)^{n_{[M,\delta]}}
\quad \text{as } C_r^* (G,K)^{\mf s} \text{-modules.}
\]
\end{lem}
\begin{proof}
The $\mc H (G,K)^{\mf s}$-module $\langle K \rangle \Pi_{\mf s}$ is a direct integral
(in an algebraic sense) of the representations $I_P^G (\sigma \otimes \chi)$
with $\chi \in X_\nr (L)$. Tensoring with $C_r (G,K)^{\mf s}$ kills all
nontempered irreducible subquotients. By Theorem \ref{thm:1.9} every 
irreducible tempered subquotient of $\Pi_{\mf s}$ arises as a direct summand
of $I_Q^G (\delta)$, where $Q \subset G$ is a parabolic subgroup containing
$P$, $M$ is a Levi factor of $Q$ containing $L$ and $\delta \in \Irr (M)$
is square-integrable modulo centre. Then $\delta$ is a subquotient of 
$I_{M \cap P}^M (\sigma \otimes \chi)$ for some $\chi \in X_\nr (L)$.
Sometimes the copy of $I_Q^G (\delta)$ in $I_P^G (\sigma \otimes \chi)$ is
annihilated by applying $C_r (G,K)^{\mf s} \otimes_{\mc H (G,K)^{\mf s}}$,
because it belongs to a subrepresentation of $I_P^G (\sigma \otimes \chi)$
generated by elements from nontempered representations.

For each $\chi_M \in X_\nr^u (M)$ the discrete series representation
\begin{equation}\label{eq:6.10}
\delta \otimes \chi_M \quad \text{is a subquotient of} \quad 
I_{M \cap P}^M (\sigma \otimes \chi_M |_L \chi ).
\end{equation}
Here $(L,\sigma \otimes \chi)$ represents the cuspidal support of $\delta$
and, while $\chi$ is not unique, there are only finitely many possibilities
because Sc$(\delta)$ has only finitely many re\-pre\-sentatives in $\Irr (L)$. 
It follows that $\Pi_{\mf s,K}^c$ is a direct integral of the representations
$I_Q^G (\delta \otimes \chi_M)$ with $[M,\delta] \in \Delta (G,\mf s)$, and
the number of times such a representations appears as a subquotient of 
$\Pi_{\mf s,K}^c$ depends only on $[M,\delta]$. Each $n_{\mf d}$ is
nonzero because $\Pi_{\mf s,K}^c$ is a progenerator (Lemma \ref{lem:6.10}.a).

We know the structure of $C_r^* (G,K)^{\mf s}$ from Theorem \ref{thm:1.17}
and \eqref{eq:1.35}. More concretely, by \eqref{eq:1.33} $C_r^* (G,K)^{\mf s}$
has direct summands 
\[
C_r^* (G,K)_{[M,\delta]} = \Big( C(X_\nr^u ( M)) \otimes \End_\C 
\big( \ind_{K_0 \cap Q}^{K_0} (V_\delta)^K \big) \Big)^{W_{[M,\delta]}} .
\]
The projectivity of $\Pi_{\mf s,K}^c$ as $C_r^* (G,K)^{\mf s}$-module
(Lemma \ref{lem:6.10}.a) implies that the direct integral of the representations
$I_Q^G (\delta \otimes \chi_M)$ from \eqref{eq:6.10} occurs in the form
\[
\langle K \rangle I_Q^G \big( C(X_\nr^u (M)) \otimes \delta \big). \qedhere
\]
\end{proof}

\subsection{Construction of an action on the progenerator} \

We want to define an action of $C(X_\nr^u (L)) \rtimes \C [W_{\mf s}^e,
\natural_{\mf s}]$ on $\Pi_{\mf s,K}^c$ by $C_r^* (G,K)^{\mf s}$-intertwiners, 
which for finite length tempered representations recovers Theorems \ref{thm:3.11} 
and \ref{thm:3.13}. Unfortunately this will be a rather technical affair, because 
we have to involve a lot of arguments from \cite{SolEnd}. The easy part is to 
let $C(X_\nr^u (L))$ act canonically on $\Pi_{\mf s}^c$ and $\Pi_{\mf s,K}^c$, 
as follows. 

The set of tempered representations in $X_\nr (L) \sigma$ is 
$X_\nr^u (L) \sigma$, and by Lemma \ref{lem:1.13} every $\sigma \otimes \chi$
can be written uniquely as $\sigma \otimes \chi \, |\chi|^{-1} \otimes |\chi|$,
where $\sigma \otimes \chi \, |\chi|^{-1}$ is unitary and $\chi \in X_\nr^+ (L)$.
We let $f \in C(X_\nr^u (L))$ act on $I_P^G (\sigma \otimes \chi)$ as
multiplication by $f(\chi \, |\chi|^{-1})$. This can be restricted to
tempered subquotients and in particular to any $I_Q^G (\delta \otimes \chi_M)$
occurring in $I_P^G (\sigma \otimes \chi)$. When we vary $\chi_M$ in
$X_\nr^u (M)$, we see that this integrates to an action of $f$ on
$I_Q^G (C(X_\nr^u (M)) \otimes \delta)$, namely
\begin{equation}\label{eq:6.11}
I_Q^G \text{ applied to multiplication by } [\chi_M \mapsto f(\chi_M |_L 
\chi \, |\chi|^{-1})] \text{ on } C(X_\nr^u (M)) \otimes \delta.
\end{equation}
Via Lemma \ref{lem:6.10} this gives rise to a canonical action of
$C(X_\nr^u (L))$ on $\Pi_{\mf s}^c$ by $G$-intertwiners, which restricts
to an action of $C(X_\nr^u (L))$ on $\Pi_{\mf s,K}^c$
by $C_r^* (G,K)^{\mf s}$-intertwiners.

Recall from Theorem \ref{thm:3.6} or \cite{SolEnd} that 
\begin{equation}\label{eq:6.12}
\End_G (\Pi_{\mf s}) \otimes_{\mc O (X_\nr (L))} \C (X_\nr (L)) \cong
\C (X_\nr (L)) \rtimes \C [W_{\mf s}^e, \natural_{\mf s}] .
\end{equation}
Here $W_{\mf s}^e$ appears as a set of intertwining operators $\mc T_w$
that multiply as in $\C [W_{\mf s}^e,\natural_{\mf s}]$. More precisely,
the specialization 
\[
\mc T_w : I_P^G (\sigma \otimes \chi) \to I_P^G (\sigma \otimes w(\chi))
\]
is rational as a function of $\chi \in X_\nr (L)$. On the other hand, from
Theorem \ref{thm:3.10} or the underlying \cite[Proposition 7.3]{SolEnd} we 
obtain
\[
\C [W_{\mf s}^e,\natural_{\mf s}] \subset \mh H (\mf t, W_{\mf s,\sigma 
\otimes \chi}, k_{\sigma \otimes \chi}, \natural_{\sigma \otimes \chi})
\subset \End_G (\Pi_{\mf s})_{U_{\sigma \otimes \chi}}^{an} .
\]
Here $U_{\sigma \otimes \chi}$ is a neighborhood of $X_\nr^+ (L) \chi$
in $X_\nr (L)$, a superscript an stands for complex analytic functions
and $\End_G (\Pi_{\mf s})_{U_{\sigma \otimes \chi}}^{an}$ consists of
endomorphisms of \\
$\Pi_{\mf s} \otimes_{\mc O (X_\nr (L))} C^{an}(U_{\sigma \otimes \chi})$.
To construct an action of $\C[W_{\mf s}^e,\natural_{\mf s}]$ on 
$\Pi_{\mf s}^c$ and on $\Pi_{\mf s,K}^c$, we will need to combine both 
pictures of $\C [W_{\mf s}^e,\natural_{\mf s}]$.

\begin{ex}\label{ex:6.12}
Consider $G = SL_2 (F)$ and $\mf s = [T,\mr{triv}]$. Lemma \ref{lem:6.11}
works out to
\[
\Pi_{\mf s,K}^c = \langle K \rangle I_B^G \big( C(X_\nr^u (T)) \big) 
\oplus \langle K \rangle \mr{St}_G .
\]
By \cite[(5.20)]{SolEnd} the operator $\mc T_{s_\alpha}$ has only one
singularity, at $\chi (h_\alpha^\vee) = q_F$, or equi\-va\-lently for 
$\chi = \delta_B^{1/2}$. In particular $\mc T_{s_\alpha}$ is regular on
$X_\nr^u (T)$, which means that the action of $\mc T_{s_\alpha}$ on
\[
\Pi_{\mf s} \otimes_{\mc O (X_\nr (T))} \C (X_\nr (T))
\]
from \eqref{eq:6.12} extends naturally to an action on
\[
\Pi_{\mf s} \otimes_{\mc O (X_\nr (T))} C (X_\nr^u (T)) = 
I_B^G (C(X_\nr^u (T))
\]
and on $\langle K \rangle I_B^G (C(X_\nr^u (T))$. Let us check that the specialization 
of this action at any $\chi \in X_\nr^u (T)$ recovers Theorem \ref{thm:3.11}.

When $s_\alpha (\chi) \neq \chi$, that follows from \cite[Lemma 8.2]{SolEnd}.
When $s_\alpha (\chi) = \chi$, we have $\chi = \chi_-$ or $\chi = \mr{triv}$.
For $\chi_-$, the relevant Hecke algebra is $\mc O (\mf t) \rtimes \C 
[W_{\mf s,\chi_-}]$ with $W_{\mf s,\chi_-} = S_2$. Via \cite[Lemma 7.1]{SolEnd},
$\mc T_{s_\alpha}$ specialized at $\chi_-$ is identified with a scalar
multiple of $N_{s_\alpha}$. The same identification is used in the 
constructions for Theorem \ref{thm:3.11}.
For $\mr{triv}_T$ the relevant Hecke algebra is $\mh H (\mf t,S_2,\log (q_F))$,
and \cite[Lemma 7.1]{SolEnd} identifies the specialization of $\mc T_{s_\alpha}$
at $\mr{triv}_T$ with 
\[
\tau_{s_\alpha} = -1 + (N_{s_\alpha} + 1) \frac{h_\alpha^\vee}{\log (q_F) +
h_\alpha^\vee} 
\]
from Theorem \ref{thm:2.4}. Now the group $W_{\mf s}$ has two incarnations
$\{1,N_{s_\alpha}\}$ and $\{ 1, \tau_{s_\alpha} \}$, both of which act on
the specialization of $\Pi_{\mf s}$ at $\chi = \mr{triv}_T$. By direct
calculations one checks that the two actions of $W_{\mf s}$ on 
$I_B^G (\mr{triv}_T)$ are both equivalent to the regular representation.

We move on to the Steinberg representation. There are short exact sequences
\begin{equation}\label{eq:6.13}
\begin{array}{ccccccccc}
0 & \to & \mr{St}_G & \to & I_B^G (\delta_B^{1/2}) & \to & \mr{triv}_G & \to & 0, \\
0 & \to & \mr{triv}_G & \to & I_B^G (\delta_B^{-1/2}) & \to & \mr{St}_G & \to & 0 .
\end{array}
\end{equation}
The standard intertwining operator 
\[
J(s_\alpha,B,\mr{triv}_T,\delta_B^{1/2}) : I_B^G (\delta_B^{1/2}) \to
I_B^G (\delta_B^{-1/2})
\]
annihilates $\mr{St}_G$ and sends $\mr{triv}_G \cong I_B^G (\delta_B^{1/2})
/ \mr{St}_G$ bijectively to $\mr{triv}_G \subset I_B^G (\delta_B^{-1/2})$.
Similarly
\[
J(s_\alpha,B,\mr{triv}_T,\delta_B^{-1/2}) : I_B^G (\delta_B^{-1/2}) \to
I_B^G (\delta_B^{1/2})
\]
induces an isomorphism $I_B^G (\delta_B^{-1/2}) / \mr{triv}_G \to \mr{St}_G$.
The specialization of $\mc T_{s_\alpha}$ at $I_B^G (\delta_B^{1/2})$ equals
$J(s_\alpha,B,\mr{triv}_T,\delta_B^{-1/2})$ times a function of $\chi \in 
X_\nr (L)$ which has a pole at $\delta_B^{1/2}$. The specialization of
$\mc T_{s_\alpha}$ at $I_B^G (\delta_B^{-1/2})$ equals 
$J(s_\alpha,B,\mr{triv}_T,\delta_B^{1/2})$ times a nonzero scalar. 
In both cases $\mc T_{s_\alpha}$ induces a "singular map" from $\mr{St}_G$
to $\mr{St}_G$, 0 in one direction and $\infty$ in the other direction.
When we apply $C_r^* (G,K)^{\mf s} \otimes_{\mc H (G,K)^{\mf s}}$ to the
$K$-invariant vectors in \eqref{eq:6.13}, $\langle K \rangle I_B^G (\delta_B^{-1/2})$
becomes $\langle K \rangle \mr{St}_G$ because $\mr{triv}_G$ is killed. By this tensoring
the entire representation $I_B^G (\delta_B^{1/2})$ is annihilated, because
it is generated by vectors from a copy of the nontempered representation
$\mr{triv}_G$. Therefore \eqref{eq:6.13} gives rise to only one copy of
$\langle K \rangle \mr{St}_G$ in $\Pi_{\mf s,K}^c$.
The $\End_G (\Pi_{\mf s})$-module 
\[
\Hom_G (\Pi_{\mf s},\mr{St}_G) \cong \Hom_G (I_B^G (\delta_B^{-1/2}), 
\mr{St}_G) \cong \C 
\]
admits an action of the graded Hecke algebra $\mh H (\mf t,S_2,\log (q_F))$.
The element $N_{s_\alpha}$ of that algebra acts as -1 on 
$\Hom_G (\Pi_{\mf s},\mr{St}_G)$, because it corresponds to the Steinberg 
representation $\mr{St}_{\mh H}$ of $\mh H (\mf t,S_2,\log (q_F))$.
In \cite[(7.7) and Proposition 7.3]{SolEnd}, 
\begin{equation}\label{eq:6.14}
\mc T_{s_\alpha} \text{ is mapped to } \tau_{s_\alpha} .
\end{equation}
There is still a singularity involved, because $\mr{St}_{\mh H}(h_\alpha^\vee)
= -\log (q_F)$, see Example \ref{ex:2.A}. But in the end we want to relate to
Theorems \ref{thm:3.11} and \ref{thm:3.13}, which run via Theorem \ref{thm:2.8}.
In Theorem \ref{thm:2.8}, $\zeta_0$ sends tempered $\mh H (\mf t,S_2,\log (q_F))
$-modules to tempered $\mh H (\mf t,S_2,0)$-modules. This means that we actually
have to map $\mc T_{s_\alpha}$ to an element of $\mh H (\mf t,S_2,0)$, namely
\begin{equation}\label{eq:6.15}
-1 + (N_{s_\alpha} + 1) \frac{h_\alpha^\vee}{0 + h_\alpha^\vee} = N_{s_\alpha} .
\end{equation}
In this way the singularity disappears. Since $N_{s_\alpha}$ acts as -1 on
$\mr{St}_{\mh H}$ and on $\zeta_0 (\mr{St}_{\mh H}) = \C_0 \otimes \mr{sign}$,
$\mc T_{s_\alpha}$ acts as -1 on $\Hom_G (\Pi_{\mf s}, \mr{St}_G)$ in
Theorem \ref{thm:3.11}.a. This forces us to define that $\mc T_{s_\alpha}$
acts as -1 on $\mr{St}_G$ and on
\[
\langle K \rangle \mr{St}_G \cong C_r^* (G,K)^{\mf s} \otimes_{\mc H (G,K)^{\mf s}} 
\langle K \rangle I_B^G (\delta_B^{-1/2}).
\]
Summarizing: we constructed a group action of $\{1,\mc T_{s_\alpha}\}$ on
$\Pi_{\mf s}^c$ and on $\Pi_{\mf s,K}^c$, such that the induced action on 
\[
\Hom_G (\Pi_{\mf s}^c, I_B^G (\chi)) \cong
\Hom_{C_r^* (G,K)^{\mf s}} (\Pi_{\mf s,K}^c, \langle K \rangle I_B^G (\chi))
\]
is equivalent to the action of $\{1,N_{s_\alpha}\}$ on that
space (as in Theorem \ref{thm:3.11}). 
\end{ex}

Example \ref{ex:6.12} and \cite[\S 8]{Lus-Gr} guide us to the desired action
on $\Pi_{\mf s,K}^c$.

\begin{thm}\label{thm:6.12}
There exists an action of $\C [W_{\mf s}^e, \natural_{\mf s}]$ on 
$\Pi_{\mf s,K}^c$ by $C_r^* (G,K)^{\mf s}$-inter\-twi\-ners, such that:
\enuma{
\item It combines with \eqref{eq:6.11} to an action of 
$C(X_\nr^u (L)) \rtimes \C [W_{\mf s}^e, \natural_{\mf s}]$ on $\Pi_{\mf s,K}^c$.
\item Part (a) provides a homomorphism of Banach algebras
\[
\phi_{\mf s,K} : C(X_\nr^u (L)) \rtimes \C [W_{\mf s}^e, \natural_{\mf s}] \to
\End_{C_r^* (G,K)^{\mf s}} (\Pi_{\mf s,K}^c) .
\]
}
\end{thm}
\begin{proof}
We start by defining the action locally.
Consider $\sigma' := \sigma \otimes \chi'$ for some $\chi' \in X_\nr^u (L)$.
Let $U_{\chi'}$ be a small neighborhood  of $X_\nr^+ (L) \chi'$ in $X_\nr (L)$, as in
\cite[\S 7]{SolEnd} and Theorem \ref{thm:3.11}. We use the analytic localization 
$\Pi_{\mf s} \otimes_{\mc O (X_\nr (L))} C^{an}(U_{\chi'})$ of $\Pi_{\mf s}$ at 
$U_{\chi'}$, like in \eqref{eq:3.22}--\eqref{eq:3.15}. This can also be done
with $\Pi_{\mf s,K}^c$ instead of $\Pi_{\mf s}$.

By \eqref{eq:3.21}, any element of
$W_{\mf s}^e$ can be written uniquely as $\gamma w$, where $\gamma$ sends
$\Phi (G,Z^\circ (L))_{\sigma'}$ to $\Phi (G,Z^\circ (L))_{\gamma \sigma'}$ and 
$w \in W(\Phi (G,Z^\circ (L))_{\sigma'})$. Then $\mc T_\gamma$ has no singularities on 
$U_{\chi'}$ \cite[(6.7)]{SolEnd}. We define the action of $\gamma$ to be that of
$\mc T_\gamma$, at least locally over $U_{\chi'}$. That intertwines the $G$-action and is 
consistent with \cite[\S 8]{SolEnd}, although the latter is not explicit in \cite{SolEnd}.  

The Weyl group $W(\Phi (G,Z^\circ (L))_{\sigma'})$ is generated by simple reflections 
$s_\alpha$, so we look at one of those. By \cite[(5.20)]{SolEnd}, the only poles of
$\mc T_{s_\alpha}$ on $U_{\chi'}$ are at $\chi (h_\alpha^\vee) = q_{\sigma',\alpha}$ 
(for $\mc T_{s_\alpha}$ considered from the left) and at $\chi (h_\alpha^\vee) =
q_{\sigma',\alpha}^{-1}$ (considered from the right). By \cite[(6.11)]{SolEnd}, there
is a group isomorphism $W_{\mf s,\sigma'} \cong (W_{\mf s}^e)_{\chi'}$. Next 
\cite[Lemma 7.2]{SolEnd} 
\begin{equation}\label{eq:6.18}
\text{identifies the specialization of } \mc T_{s_\alpha} \text{ at } U_{\chi'} 
\text{ with } f_{s_\alpha}^{\chi'} \mc T_{s_\alpha}^{\chi'},
\end{equation} where $f_{s_\alpha}^{\chi'} \in \C (X_\nr (L))$ without
zeros or poles on $X_\nr^u (L)$, and $\mc T_{s_\alpha}^{\chi'}$ is a version of 
$\mc T_{s_\alpha}$ defined with respect to $\sigma \otimes \chi'$ as basepoint. The
action of $\mc T_{s_\alpha}$ specialized at any point of $\chi_1 X_\nr^+ (L)$ with
$\chi_1 \in X_\nr^u (L) \cap U_{\chi'}$ will be defined as that of $f_{s_\alpha}^{\chi'}
\mc T_{s_\alpha}^{\chi'}$. 

As $f_{s_\alpha}^{\chi'}$ restricts to a continuous function on $X_\nr^u (L)$ and we
already defined how $C(X_\nr^u (L))$ acts on $\Pi_{\mf s,K}^c$, the task at hand is
to define how $\mc T_{s_\alpha}^{\chi'}$ acts. In \cite[Proposition 7.3]{SolEnd},
$\mc T_{s_\alpha}^{\chi'}$ is mapped to 
\begin{equation}\label{eq:6.20}
\tau_{s_\alpha} = -1 + (N_{s_\alpha} + 1) \frac{h_\alpha^\vee}{\log 
(q_{\sigma',\alpha}) + h_\alpha^\vee} \in 
\mh H (\mf t, W_{\mf s,\sigma'}, k_{\sigma'}, \natural_{\sigma'}) 
\otimes_{\mc O (\mf t)^{W_{\mf s,\sigma'}}} \C (\mf t)^{W_{\mf s,\sigma'}}  
\end{equation}
from \eqref{eq:2.20}. Recall from Theorem \ref{thm:2.4}.a that both
\begin{equation}\label{eq:6.16}
\{ \tau_w : w \in W_{\mf s,\sigma'} \} \quad \text{and} \quad 
\{ N_w \tau_r : w \in W(\Phi (G,Z^\circ (L))_{\sigma'}), r \in \Gamma_{\sigma'} \}
\end{equation}
satisfy the multiplication rules for the standard basis elements of 
$\C [W_{\mf s}^e, \natural_{\mf s}]$.

Consider $Q = M U_Q \subset P$ and $\delta \in \Irr_\disc (M)$ which arises as a 
subquotient of  $I_{L (M \cap P)}^M (\sigma \otimes \chi)$ for some $\chi \in U_{\chi'}$,
and such that $\langle K \rangle I_Q^G (\delta)$ is a quotient of $\Pi_{\mf s,K}^c$. 
If $\chi (h_\alpha^\vee) \notin \{ q_{\sigma',\alpha}, q_{\sigma',\alpha}^{-1} \}$, then
we let $\mc T_{s_\alpha}^{\chi'}$ acts on this copy of $\langle K \rangle I_Q^G (\delta)$ as 
$\tau_{s_\alpha}$ (which is regular there). If $\chi (h_\alpha^\vee) \in \{ q_{\sigma',\alpha},
q_{\sigma',\alpha}^{-1} \}$, then $\tau_{s_\alpha}$ has a singularity at 
$\sigma \otimes \chi$, and we let $\mc T_{s_\alpha}^{\chi'}$ act as $N_{s_\alpha}$ instead.
Notice that in both cases we act by a $G$-intertwiner.

The dichotomy $\tau_{s_\alpha} / N_{s_\alpha}$ is analogous to \cite[\S 8.8]{Lus-Gr}. 
The arguments from \cite[\S 8]{Lus-Gr} show that the above leads to an action of the group 
\begin{equation}\label{eq:6.17}
\{ \mc T_w^{\chi'} : w \in W(\Phi (G,Z^\circ (L))_{\sigma'}) \}
\end{equation}
on the sum of the copies
of $I_Q^G (\delta)$ that arise in the above way. The main idea in \cite[\S 8]{Lus-Gr}
is a more precise version of \eqref{eq:6.16}, with $\tau_{s_\alpha}$ or $N_{s_\alpha}$
depending on one $W_{\mf s,\sigma'}$-orbit in $U_{\chi'}$. 
Next we can vary $\delta$, which yields an action of \eqref{eq:6.17} on the localization
of $\Pi_{\mf s,K}^c$ at $U_{\chi'}$.

From Theorem \ref{thm:3.6} and the multiplication rules in 
$\C [W_{\mf s}^e, \natural_{\mf s}]$ we see that $\mc T_\gamma \mc T_{s_\alpha} 
\mc T_\gamma^{-1}$ equals $\mc T_{\gamma s_\alpha \gamma^{-1}}$
for all simple roots $\alpha \in \Phi (G,Z^\circ (L))_{\sigma'}$, and hence
\[
\mc T_\gamma \mc T_w \mc T_\gamma^{-1} = \mc T_{\gamma w \gamma^{-1}} 
\quad \text{for all } w \in W(\Phi (G,Z^\circ (L))_{\sigma'}) .
\]
Then \eqref{eq:6.18} says that 
\begin{equation}\label{eq:6.19} 
\mc T_\gamma f_{s_\alpha}^{\chi'} \mc T_{s_\alpha}^{\chi'} \mc T_\gamma^{-1} =
(f_{s_\alpha}^{\chi'} \circ \gamma^{-1}) \mc T_\gamma \mc T_w^{\chi'} \mc T_\gamma^{-1} =
f_{\gamma s_\alpha \gamma^{-1}}^{\gamma (\chi')} \mc T_{\gamma s_\alpha 
\gamma^{-1}}^{\gamma (\chi')} .
\end{equation}
The relations between $\mc T_{s_\alpha}^{\chi'}, \tau_{s_\alpha}^{\chi'}$ and
$N_{s_\alpha} = N_{s_\alpha}^{\chi'}$ look the same after replacing $\chi'$ by
$\gamma (\chi')$, so \eqref{eq:6.19} implies that 
\[
\mc T_\gamma f_{s_\alpha}^{\chi'} \tau_{s_\alpha}^{\chi'} \mc T_\gamma^{-1} =
f_{\gamma s_\alpha \gamma^{-1}}^{\gamma (\chi')} \tau_{\gamma s_\alpha 
\gamma^{-1}}^{\gamma (\chi')} \quad \text{and} \quad
\mc T_\gamma f_{s_\alpha}^{\chi'} N_{s_\alpha}^{\chi'} \mc T_\gamma^{-1} =
f_{\gamma s_\alpha \gamma^{-1}}^{\gamma (\chi')} N_{\gamma s_\alpha 
\gamma^{-1}}^{\gamma (\chi')} .
\]
It follows that, as operators on a localization of $\Pi_{\mf s,K}^c$ at
$U_{\gamma (\chi')}$:
\[
\mc T_\gamma \circ \mc T_{s_\alpha} \circ \mc T_\gamma^{-1} = 
\mc T_{\gamma s_\alpha \gamma^{-1}} .
\]
Combining instances of this relation leads to
\[
\mc T_\gamma \circ \mc T_w \circ \mc T_\gamma^{-1} = \mc T_{\gamma w 
\gamma^{-1}} \quad \text{and} \quad \mc T_\gamma \circ \mc T_w \circ = 
\mc T_{\gamma w \gamma^{-1}} \circ \mc T_w 
\]
as operators on localized versions of $\Pi_{\mf s,K}^c$. Consider now
$w_2, \gamma_2$ like $w,\gamma$, only with respect to $\gamma (\chi')$
instead of $\chi$. We compute 
\begin{align*}
\mc T_{\gamma_2} \circ \mc T_{w_2} \circ \mc T_\gamma \circ \mc T_w & =
\mc T_{\gamma_2} \circ \mc T_\gamma \circ \mc T_{\gamma^{-1} w_2 \gamma}
\circ \mc T_w = \natural_{\mf s}(\gamma_2, \gamma) \mc T_{\gamma_2 \gamma}
\circ \mc T_{\gamma^{-1} w_2 \gamma w} \\
& = \natural (\gamma_2, \gamma) \mc T_{\gamma_2 w_2 \gamma w} =
\natural (\gamma_2 w_2, \gamma w) \mc T_{\gamma_2 w_2 \gamma w} 
\end{align*}
as operators from 
\[
\Pi_{\mf s,K}^c \otimes_{\mc O (X_\nr (L))} C^{an}(U_{\chi'}) \quad \text{to}
\quad \Pi_{\mf s,K}^c \otimes_{\mc O (X_\nr (L))} C^{an}(\gamma_2 w_2 
\gamma w U_{\chi'}).
\]
This shows that our actions defined locally near $X_\nr^+ (L) \chi'$ extend
to an action of $\C [W_{\mf s}^e,\natural_{\mf s}]$ defined locally near
$W_{\mf s}^e X_\nr^+ (L) \chi'$.

Next we check that the above constructions do not depend on the choice of
the neighborhoods $U_{\chi'}$ (as long as they have the shape from
\cite[\S 7]{SolEnd}). Suppose that we have another set $U_{\chi''}$ like 
$U_{\chi'}$, and that $U_{\chi''} \cap U_{\chi'}$ is nonempty. We need to
verify that our locally defined actions coincide on 
\[
\Pi_{\mf s,K}^c \otimes_{\mc O (X_\nr (L))} C^{an}(U_{\chi'} \cap U_{\chi''}).
\]
We may pick $\chi_2 \in U_{\chi'} \cap U_{\chi''}$ with a neighborhood
$U_{\chi_2}$ of $X_\nr^+ (L) \chi_2$ contained in $U_{\chi'} \cap U_{\chi''}$.
If we can compare the actions relative to $U_{\chi'}$ and $U_{\chi_2}$ and the
actions relative to $U_{\chi''}$ and $U_{\chi_2}$, then a comparison between
$U_{\chi'}$ and $U_{\chi''}$ follows. Therefore it suffices to consider the
case where $U_{\chi''} \subset U_{\chi'}$.

The conditions on $U_{\chi''}$ from \cite[\S 7]{SolEnd} ensure that 
\[
\Phi (G,Z^\circ (L))_{\sigma'} \supset \Phi (G,Z^\circ (L))_{\sigma'} .
\]
We consider $w,\gamma,s_\alpha$ as above. For $\gamma$ the actions with
respect to $U_{\chi'}$ and $U_{\chi''}$ are defined in exactly the same way,
so they agree where they are both defined. 

Suppose that the action of $\mc T_{s_\alpha}$ (from $U_{\chi'}$) is defined
via that of $\tau_{s_\alpha}$ in \eqref{eq:6.20}. Then $\mc T_{s_\alpha}$
acts on $I_Q^G (\delta)$ in the same way as $\mc T_{s_\alpha}$ on 
$\Pi_{\mf s} \otimes_{\mc O (X_\nr (L))} \C (X_\nr (L))$. In that case the
action with respect to $U_{\chi''}$ is defined just like that.

Suppose now that $\mc T_{s_\alpha}$ acts on $I_Q^G (\delta)$ via 
$N_{s_\alpha}$ in \eqref{eq:6.20}. Then $\chi (h_\alpha^\vee) \in 
\{ q_{\sigma',\alpha}, q_{\sigma',\alpha}^{-1} \}$ with $\chi \in 
U_{\chi'} \cap U_{\chi''}$. In that case \eqref{eq:6.20} holds also with
respect to $\sigma \otimes \chi''$ instead of $\sigma'$. Further 
$f_{s_\alpha}^{\chi''} N_{s_\alpha}^{\chi''} = f_{s_\alpha}^{\chi'}
N_{s_\alpha}^{\chi'}$ because 
\[
f_{s_\alpha}^{\chi''} \tau_{s_\alpha}^{\chi''} = f_{s_\alpha}^{\chi'}
\tau_{s_\alpha}^{\chi'} \text{ in } \mh H (\mf t, \langle s_\alpha \rangle,
k_{\sigma'}, \natural_{\sigma'}) \otimes_{\mc O (\mf t)} \C (\mf t) .
\]
The action of $\mc T_{s_\alpha}$ on $I_Q^G (\delta)$ with respect to 
$U_{\chi''}$ is defined to be that of $f_{s_\alpha}^{\chi''}
N_{s_\alpha}^{\chi''} = f_{s_\alpha}^{\chi'} N_{s_\alpha}^{\chi'}$. 
This concludes the definition of an action of 
$\C [W_{\mf s}^e, \natural_{\mf s}]$ on $\Pi_{\mf s,K}^c$.\\
(a) By Theorem \ref{thm:2.4}.b and the multiplication rules in
$\C (X_\nr (L)) \rtimes \C [W_{\mf s}^e, \natural_{\mf s}]$, we have,
as operators on $\Pi_{\mf s,K}^c$:
\begin{equation}\label{eq:6.21}
\mc T_w \circ f = (f \circ w^{-1}) \circ \mc T_w \qquad
\text{for all } w \in W_{\mf s}^e \text{ and } f \in \mc O (X_\nr (L)) .
\end{equation}
The action of $\mc T_w$ is defined via its action on the specializations
of $\Pi_{\mf s,K}^c$ at the various $\chi \in X_\nr (L)$, so \eqref{eq:6.21}
extends to $f \in C(X_\nr (L))$. Then it holds also for $f \in C(X_\nr^u (L))$,
regarded as $X_\nr^+ (L)$-invariant function on $X_\nr (L)$. Therefore the 
actions of $\C [W_{\mf s}^e, \natural_{\mf s}]$ and $C(X_\nr^u (L))$ on
$\Pi_{\mf s,K}^c$ combine in the required way.\\
(b) Part (a) provides an algebra homomorphism $\phi_{\mf s,K}$, and from 
\eqref{eq:6.28} we know that both source and target are Banach algebras. It remains 
to check the continuity of $\phi_{\mf s,K}$.
On the finite dimensional algebra $\C [W_{\mf s}^e, \natural_{\mf s}]$ continuity
is automatic, so we may focus on $C(X_\nr^u (L))$. From \eqref{eq:6.11} we see that
its action on $\Pi_{\mf s,K}^c$ is a direct sum of actions of the following kind.
First a projection $C(X_\nr^u (L)) \to C(X)$ for some closed subspace $X \subset 
X_\nr^u (L)$, then multiplication by $f|_X$ on some vector bundle over $X$. In
terms of the isomorphism 
\[
\End_{C_r^* (G,K)^{\mf s}} (\Pi_{\mf s,K}^c)^{op} \cong
e_{\mf s}^c M_{n_{\mf s}^c} (C_r^* (G,K)^{\mf s}) e_{\mf s}^c
\] 
and the description of $C_r^* (G,K)^{\mf s}$ in Theorem \ref{thm:1.17}, this means
that $\phi_{\mf s,K} |_{C(X_\nr^u (L))}$ is a finite direct sum of homomorphisms of the form 
$C(X_\nr (L)) \to C(X)$. Those are continuous, so $\phi_{\mf s,K}$ is continuous.
\end{proof}

Theorem \ref{thm:6.12} gives rise to a version of Theorem \ref{thm:3.13} 
for $C_r^* (G,K)^{\mf s}$.

\begin{lem}\label{lem:6.14}
Let $\pi \in \Irr_\temp (G) \cap \Rep_\fl (G)^{\mf s}_{X_\nr^+ (L)
W_{\mf s} (\sigma \otimes \chi)}$ with $\chi \in X_\nr^u (L)$, and recall from
Theorem \ref{thm:3.13} that 
\[
\zeta^\vee (\pi) = \ind_{\mc O (X_\nr (L)) \rtimes \C [W_{\mf s,\chi}^e, 
\natural_{\mf s}^{-1}]}^{\mc O (X_\nr (L)) \rtimes \C [W_{\mf s}^e, 
\natural_{\mf s}^{-1}]} (\C_\chi \otimes \pi_\chi) .
\]
Then, as left modules over $\big( C(X_\nr^u (L)) \rtimes \C [W_{\mf s}^e, \natural_{\mf s}]
\big)^{op} \cong C(X_\nr^u (L)) \rtimes \C [W_{\mf s}^e, \natural_{\mf s}^{-1}]$
via Theorem \ref{thm:6.12}:
\[
\Hom_{C_r^* (G,K)^{\mf s}} (\Pi_{\mf s,K}^c, \langle K \rangle \pi) \cong 
\ind_{C (X_\nr^u (L)) \rtimes \C [W_{\mf s,\chi}^e, 
\natural_{\mf s}^{-1}]}^{C (X_\nr^u (L)) \rtimes \C [W_{\mf s}^e, 
\natural_{\mf s}^{-1}]} (\C_\chi \otimes \pi_\chi) .
\]
\end{lem}
\begin{proof}
Both $\zeta^\vee (\pi)$ and $\Hom_{C_r^* (G,K)^{\mf s}} 
(\Pi_{\mf s,K}^c, \langle K \rangle \pi)$ are determined by their specializations at 
$\chi \in X_\nr^u (L)$, and for both the operation to go from there to
the full module is $\ind_{\C[W_{\mf s,\chi}^e,\natural_{\mf s}^{-1}]}^{\C
[W_{\mf s}^e,\natural_{\mf s}^{-1}]}$. Therefore it suffices show that
the specialization of\\ $\Hom_{C_r^* (G,K)^{\mf s}} (\Pi_{\mf s,K}^c, \langle K \rangle \pi)$
at $U_\chi$ is isomorphic to $\C_\chi \otimes \pi_\chi$, as left module of\\
$C(X_\nr^u (L)) \rtimes \C [W_{\mf s}^e, \natural_{\mf s}^{-1}]$.
By definition $C(X_\nr^u (L))$ acts on
\[
\Hom_{C_r^* (G,K)^{\mf s}} (\Pi_{\mf s,K}^c, \langle K \rangle \pi) \otimes_{\mc O 
(X_\nr (L))} C^{an}(U_\chi)
\]
by evaluation at $X_\nr^+ (L) \chi$, or equivalently evaluation at $\chi$.
Thus we only need to show that $\Hom_{C_r^* (G,K)^{\mf s}} (\Pi_{\mf s,K}^c, 
\langle K \rangle \pi)$ is isomorphic to $\pi_\chi$ as right 
$\C [W_{\mf s}^e, \natural_{\mf s}]$-module, where $\Hom_G (\Pi_{\mf s},
\pi)$ is isomorphic to $\pi_\chi$ as right module for
\[
\C [W_{\mf s}^e, \natural_{\mf s}] \subset \mh H (\mf t, W_{\mf s,\sigma 
\otimes \chi},k_{\sigma \otimes \chi}, \natural_{\sigma \otimes \chi}).
\]
In the way we set it up, $\Hom_{C_r^* (G,K)^{\mf s}} (\Pi_{\mf s,K}^c, \langle K \rangle \pi)$
is acted upon by a copy of $\C [W_{\mf s}^e, \natural_{\mf s}]$ which is
generated by three kinds of elements:
\begin{enumerate}[(i)]
\item $\mc T_\gamma$ for $\gamma \in \Gamma_{\sigma \otimes \chi}$,
\item $\mc T_{s_\alpha}^\chi$ for simple roots $\alpha \in 
\Phi (G,Z^\circ (L))_{\sigma \otimes \chi}$ such that 
$\chi (h_\alpha^\vee) \notin \{ q_{\sigma \otimes \chi,\alpha},
q_{\sigma \otimes \chi,\alpha}^{-1} \}$,
\item $N_{s_\alpha}^\chi$ for simple roots $\alpha \in 
\Phi (G,Z^\circ (L))_{\sigma \otimes \chi}$ such that 
$\chi (h_\alpha^\vee) \in \{ q_{\sigma \otimes \chi,\alpha},
q_{\sigma \otimes \chi,\alpha}^{-1} \}$.
\end{enumerate}
We can regard $\Hom_{C_r^* (G,K)^{\mf s}} (\Pi_{\mf s,K}^c, \langle K \rangle \pi) \cong
\Hom_G (\Pi_{\mf s}, \pi)$ as a module over
\[
\mh H (\mf t, W_{\mf s,\sigma \otimes \chi},k_{\sigma \otimes \chi},
\natural_{\sigma \otimes \chi}) \otimes_{\mc O (\mf t / W_{\mf s, \sigma
\otimes \chi})} C^{an}(U_\chi )^{W_{\mf s, \sigma \otimes \chi})},
\]
as in \eqref{eq:3.15}. In that algebra we continuously scale the
parameters $k_{\sigma \otimes \chi,\alpha}$ to 0 (or equivalently deform 
$q_{\sigma \otimes \chi,\alpha}$ to 1), like in \eqref{eq:2.19} and
Paragraph \ref{par:2.4}. Such continuous deformations do not change
representations of the finite dimensional semisimple algebra
$\C [W_{\mf s}^e, \natural_{\mf s}]$. In the limit case 
$k_{\sigma \otimes \chi,\alpha} = 0, q_{\sigma \otimes \chi,\alpha} = 1$
the difference between $\mc T_{s_\alpha}^\chi$ and $N_{s_\alpha}^\chi$
disappears, see \eqref{eq:6.15}. Hence 
$\Hom_{C_r^* (G,K)^{\mf s}} (\Pi_{\mf s,K}^c, \langle K \rangle \pi)$ as a representation
of $\C [W_{\mf s}^e, \natural_{\mf s}]$ as generated by elements of the
above kinds (i), (ii) and (iii) is equivalent to $\Hom_G (\Pi_{\mf s},\pi)$
as representation of $\C [W_{\mf s}^e, \natural_{\mf s}]$ as generated by
the $\mc T_\gamma$ with $\gamma \in \Gamma_{\sigma \otimes \chi}$ and 
the $N_{s_\alpha}^\chi$ for simple roots $\alpha \in 
\Phi (G,Z^\circ (L))_{\sigma \otimes \chi}$. The latter representation
recovers exactly how $\pi_\chi$ is constructed in \cite[\S 7]{SolEnd}.
\end{proof}

\subsection{$K_* (C_r^* (G)^{\mf s})$ via the progenerator} \

The goal of this paragraph is to compute $K_* (C_r^* (G)^{\mf s})$ in
representation theoretic terms. The improvement on Paragraph \ref{par:6.2}
is that we do not need to tensor with $\C$ over $\Z$, so that we can also
detect torsion elements in $K_* (C_r^* (G)^{\mf s})$. Our main tools are
the progenerator from Paragraph \ref{par:6.3} and the below general
technique to compute K-groups.

Let $X$ be a compact Hausdorff space which has the structure of a finite
simplicial complex. Let $A$ and $B$ be $C(X)$-Banach algebras, by which
we mean that $C(X)$ acts on $A$ via a Banach algebra morphism from $C(X)$ 
to the centre of the multiplier algebra of $A$ (and similarly for $B$).
  
\begin{lem} \label{lem:6.13}
\textup{\cite[Lemma 5.1.3]{SolAHA}} \\
Suppose that $\phi : A \to B$ is a homomorphism of $C(X)$-Banach
algebras, such that for each simplex $\sigma$ of $X$ the specialization
\[
\phi_{\sigma,\partial \sigma} : C_0 (X,\partial \sigma) A / C_0 (X,\sigma) A
\to C_0 (X,\partial \sigma) B / C_0 (X,\sigma) B
\]
induces an isomorphism on K-theory. Then $K_* (\phi) : K_* (A) \to
K_* (B)$ is an isomorphism.
\end{lem}

Recall that $K \in \CO (G)$ has been chosen in \eqref{eq:6.39}.
In this paragraph we abbreviate
\[
A = C(X_\nr^u (L)) \rtimes \C [W_{\mf s}^e, \natural_{\mf s}^{-1}],\quad
B = \End_{C_r^* (G,K)^{\mf s}}(\Pi_{\mf s,K}^c)^{op}, 
\quad X = X_\nr^u (L) / W_{\mf s}^e .
\]
Clearly $A$ is a $C(X)$-Banach algebra, and via $\phi_{\mf s,K}: A \to B$
we make $B$ into a $C(X)$-Banach algebra. We may regard $C(X)$ as the 
algebra of $X_\nr^+ (L)$-invariant continuous functions
on $X_\nr (L) / W_{\mf s}^e \cong X_\nr (L) \sigma / W_{\mf s}$. The latter
algebra acts naturally on the family of representations $I_P^G (\sigma 
\otimes \chi)$ with $\chi \in X_\nr (L)$, namely $f$ acts as multiplication by 
\begin{equation}\label{eq:6.30}
f (\sigma \otimes \chi) = f (\sigma \otimes \chi \, |\chi|^{-1}) =
f (\chi \, |\chi|^{-1}) . 
\end{equation}
From Theorem \ref{thm:3.5} we see that $C(X)$ almost acts by elements of the
Bernstein centre of $\Rep (G)^{\mf s}$, but not entirely because it consists
of continuous rather than regular (algebraic) functions. Taking one step 
back, we can say that 
\begin{equation}\label{eq:6.31}
C(X) \text{ acts on finite length $G$-representations via }
Z (\Rep_\fl (G)^{\mf s}). 
\end{equation}
From the Fourier transform of $C_r^* (G,K)^{\mf s}$ (Theorem \ref{thm:1.17}
and \eqref{eq:1.35}), we see that the same recipe yields an action of
$C(X)$ on $C_r^* (G,K)^{\mf s}$. Moreover \eqref{eq:6.31} implies that $C(X)$
acts via a map to $Z(C_r^* (G,K)^{\mf s})$, which makes $C_r^* (G,K)^{\mf s}$
into a $C(X)$-Banach algebra. 
Comparing \eqref{eq:6.11} and \eqref{eq:6.30}, we deduce that the Morita
equivalence \eqref{eq:6.29} is $C(X)$-linear.

It is known from \cite{Ill} that the manifold $X_\nr^u (L)$ admits a smooth
$W_{\mf s}^e$-equivariant triangulation. This means that $X_\nr^u (L)$ can
be made into a finite simplicial complex, such that all points in the 
interior of one simplex have the same stabilizer in $W_{\mf s}^e$. Then
we can divide by the action of $W_{\mf s}^e$, which produces a convenient
triangulation of $X = X_\nr (L)^u / W_{\mf s}^e$. 
Let $\tau'$ be a simplex in $X_\nr^u (L)$ (from the chosen triangulation) and 
write $\tau = W_{\mf s}^e \tau' / W_{\mf s}^e$, which is a simplex in $X$.

\begin{lem}\label{lem:6.15}
The $C^*$-algebra $C_0 (X,\partial \tau) A / C_0 (X,\tau) A$ is Morita
equivalent to\\ $C_0 (\tau, \partial \tau) \otimes \C [W_{\mf s,\tau'}^e,
\natural_{\mf s}^{-1}]$, and there is an isomorphism
\[
K_* \big( C_0 (X,\partial \tau) A / C_0 (X,\tau) A \big) \cong
K_* (C_0 (\tau,\partial \tau)) \otimes_\Z 
R \big( \C [W_{\mf s,\tau'}^e, \natural_{\mf s}^{-1}] \big) .
\]
\end{lem}
\begin{proof}
We can write
\[
\begin{aligned}
C_0 (X,\partial \tau) A / C_0 (X,\tau) A & \cong 
C_0 (W_{\mf s}^e \tau', W_{\mf s}^e \partial \tau') \rtimes 
\C [W_{\mf s}^e, \natural_{\mf s}^{-1}] \\
& = \big( \bigoplus\nolimits_{w \in W_{\mf s}^e / W_{\mf s,\tau'}^e} 
C_0 (w \tau',w \partial \tau') \big) \rtimes 
\C [W_{\mf s}^e, \natural_{\mf s}^{-1}].
\end{aligned}
\]
This is Morita equivalent to $C_0 (\tau',\partial \tau') \rtimes 
\C [W_{\mf s, \tau}^e, \natural_{\mf s}^{-1}]$ via the bimodules
\[
1_{\tau'} C_0 (W_{\mf s}^e \tau', W_{\mf s}^e \partial \tau') \rtimes 
\C [W_{\mf s}^e, \natural_{\mf s}^{-1}] \quad \text{and} \quad
C_0 (W_{\mf s}^e \tau', W_{\mf s}^e \partial \tau') \rtimes 
\C [W_{\mf s}^e, \natural_{\mf s}^{-1}] 1_{\tau'} .
\]
Since all points of $\tau' \setminus \partial \tau'$ have the same
stabilizer in $W_{\mf s}^e$, that algebra equals 
\begin{equation}\label{eq:6.32}
C_0 (\tau',\partial \tau') \otimes \C [W_{\mf s, \tau'}^e, \natural_{\mf s}^{-1}] 
\cong C_0 (\tau,\partial \tau) \otimes \C [W_{\mf s, \tau'}^e, \natural_{\mf s}^{-1}] .
\end{equation} 
The K-theory of \eqref{eq:6.32} is easy to compute, because 
$\C [W_{\mf s, \tau}^e, \natural_{\mf s}^{-1}]$ is a direct sum
of matrix algebras $M_{n_i} (\C)$. Every such summand contributes a copy 
of $K_* (C_0 (\tau,\partial \tau))$ to the K-theory of \eqref{eq:6.32}.
On the other hand $\oplus_i K_0 (M_{n_i}(\C)) = \oplus_i \Z$ is naturally
isomorphic to the representation ring of $\C [W_{\mf s, \tau'}^e, 
\natural_{\mf s}^{-1}]$. We deduce that 
\[
\begin{aligned}
K_* \big( C_0 (X,\partial \tau) A / C_0 (X,\tau) A \big) & \cong
K_* \big( C_0 (\tau,\partial \tau) \otimes 
\C [W_{\mf s,\tau'}^e, \natural_{\mf s}^{-1}] \big) \\
& \cong K_* (C_0 (\tau,\partial \tau)) \otimes_\Z 
R \big( \C [W_{\mf s,\tau'}^e, \natural_{\mf s}^{-1}] \big) . \qedhere
\end{aligned}
\]
\end{proof}

The analogue of Lemma \ref{lem:6.15} for $C_r^* (G,K)^{\mf s}$ is 
more involved.

\begin{prop}\label{prop:6.16}
\enuma{
\item The Banach algebra
\[
C_0 (X,\partial \tau) C_r^* (G,K)^{\mf s} / C_0 (X,\tau) C_r^* (G,K)^{\mf s}
\]
is homotopy equivalent to the tensor product of $C_0 (\tau,\partial \tau)$
and a finite dimensional semisimple algebra.
\item Its K-theory is
\[
K_* (C_0 (\tau,\partial \tau)) \otimes R \big( \Mod_\fl (C_r^* (G,K)^{\mf s}
)_{X_\nr^+ (L) W_{\mf s}(\sigma \otimes \chi)} \big) .
\] 
\item The same holds for the Banach algebra $C_0 (X,\partial \tau) B / C_0 (X,\tau) B$.
}
\end{prop}
\begin{proof}
(a) From Theorem \ref{thm:1.17} we see that 
\[
C_0 (X,\partial \tau) C_r^* (G,K)^{\mf s} / C_0 (X,\tau) C_r^* (G,K)^{\mf s}
\]
is a direct sum of algebras of the form
\begin{equation}\label{eq:6.33}
\big( C_0 (X_\nr^u (M) \text{cc}(\delta) \cap X_\nr^+ (L) \tilde \tau, 
X_\nr^u (M) \text{cc}(\delta) \cap X_\nr^+ (L) \partial \tilde \tau )
\otimes \End_\C \big( I_Q^G (V_\delta)^K \big) \big)^{W_{[M,\delta]}} ,
\end{equation}
where $\tilde \tau = W_{\mf s}^e \tau' \subset X_\nr^u (L)$,
$X_\nr^u (M)$ is identified with its restriction to $L$ and cc$(\delta) \in 
X_\nr (L)$ is chosen such that $\delta$ is a quotient of $I_{L (M \cap P)}^M
(\sigma \otimes \text{cc}(\delta))$. The set of summands \eqref{eq:6.33} whose
specialization at a point $W_{\mf s}^e \chi \in \tau$ is nonzero depends only
on which discrete series representations appear in the parabolic inductions (for
the various Levi subgroups of $G$ containing $L$) of
$\sigma \otimes \chi'$ with $\chi' \in X_\nr^+ (L) \chi$. 

We claim that, when $\chi$ varies continuously over $X_\nr^u (L)$, this set
of discrete series can only jump when $W_{\mf s,\chi}^e$ changes. Clearly this
is an issue that can be studied locally on $X_\nr^u (L)$, and that can be 
done like in Theorem \ref{thm:3.10}. In \cite[\S 7]{SolEnd}, Theorem 
\ref{thm:3.10} is in fact proven in larger generality, not for
$\Rep_\fl (G)^{\mf s}_{X_\nr^+ (L) W_{\mf s} (\sigma \otimes \chi)}$ but for
$\Rep_\fl (G)^{\mf s}_{W_{\mf s} U_{\sigma \otimes \chi}}$ where $U_{\sigma 
\otimes \chi}$ is a small neighborhood of $X_\nr^+ (L) (\sigma \otimes \chi)$.
This reduces our claim to the analogous claim for 
\[
\mh H (\mf t, W_{\mf s,\sigma \otimes \chi}, k_{\sigma \otimes \chi},
\natural_{\sigma \otimes \chi}) - \Mod_{\fl,U_\R} ,
\] 
where $U_\R$ is a small neighborhood of $\mf t_\R$ in $\mf t$, such that
$U_\R / \mf t_\R$ is a $W_{\mf s,\sigma \otimes \chi}$-invariant ball around
0 in $i \mf t_\R$. Now we have to consider discrete series representations 
in the parabolic inductions of $\C_{\lambda + \nu} \in \Irr (\mc O (\mf t))$,
where $\nu \in \mf t_\R$ and $\lambda$ varies in $U_\R \cap i \mf t_\R$.
This version of our claim is proven in Lemma \ref{lem:2.10}.\\
One nice feature of the chosen triangulation of $X_\nr^u (L)$ is that all points
of $\tau' \setminus \partial \tau'$ have the same stabilizer in $W_{\mf s}^e$.
Together with the above claim, that implies that every summand \eqref{eq:6.33}
has a nonzero specialization at every point of $\tau \setminus \partial \tau
\subset X_\nr^u (L) / W_{\mf s^\vee}$. For a given summand \eqref{eq:6.33},
we can choose $w_\delta \in W_{\mf s^e}$ such that $w_\delta \tau'$ is 
contained in $X_\nr^u (M) X_\nr^+ (L) \text{cc}(\delta)$. Then \eqref{eq:6.33}
reduces to 
\begin{equation}\label{eq:6.34}
\big( C_0 (w_\delta \tau', w_\delta \partial \tau' ) \otimes \End_\C \big( 
I_Q^G (V_\delta)^K \big) \big)^{W_{[M,\delta], \tau'}} . 
\end{equation}
The group $W_{[M,\delta]}$ is a subquotient of $W_{\mf s}^e$ (Lemma \ref{lem:1.23}),
so by the equivariance of the triangulation of $X_\nr^u (L)$ every element of
$W_{[M,\delta], \tau'}$ acts trivially on $\tau'$. Furthermore 
$w_\delta \tau', w_\delta \partial \tau'$ is $W_{[M,\delta], \tau'}$-equivarianty 
contractible, so can apply \cite[Lemma 7]{SolCh}. It says that, if we replace 
$C_0 (w_\delta \tau', w_\delta \partial \tau' )$ by $C^\infty (w \tau')$ in 
\eqref{eq:6.34}, then the algebra becomes homotopy equivalent to its specialization 
at any point of $w_\delta \tau' \setminus w_\delta \partial \tau'$. For 
\eqref{eq:6.34} itself the proof of
\cite[Lemma 7]{SolCh} can be followed up to \cite[(20)]{SolCh}, and then it
says that \eqref{eq:6.34} is homotopy equivalent to
\begin{equation}\label{eq:6.35}
C_0 (w_\delta \tau', w_\delta \partial \tau' ) \otimes \End_\C \big( I_Q^G 
(V_\delta)^K \big)^{W_{[M,\delta], \chi}} \qquad \text{for any } 
\chi \in w_\delta \tau' \setminus w_\delta \partial \tau' .
\end{equation}
We note that $\End_\C \big( I_Q^G (V_\delta)^K \big)^{W_{[M,\delta], \chi}}$ is
a finite dimensional semisimple algebra. Now we consider the direct sum over 
all summands \eqref{eq:6.33} of 
\[
C_0 (X,\partial \tau) C_r^* (G,K)^{\mf s} / C_0 (X,\tau) C_r^* (G,K)^{\mf s}.
\]
We obtain a tensor product of 
\[
C_0 (\tau,\partial \tau) \cong C_0 (w_\delta \tau', w_\delta \partial \tau' ) 
\]
with a finite dimensional semisimple algebra.\\
(b) The algebra \eqref{eq:6.35} is of the same kind as \eqref{eq:6.32}. The 
arguments in the proof of Lemma \ref{lem:6.15} show that
\begin{equation}\label{eq:6.36}
K_* \eqref{eq:6.33} \cong K_* \eqref{eq:6.35} \cong K_* (C_0 (w_\delta \tau', 
w_\delta \partial \tau' )) \otimes_\Z R \Big( \End_\C \big( I_Q^G (V_\delta)^K 
\big)^{W_{[M,\delta], \chi}} \Big) . \hspace{-2mm}
\end{equation}
Identifying $\chi$ with $W_{\mf s}^e \chi \in \tau \setminus \partial \tau$, 
we can take the direct sum of the groups \\
$R \big( \End_\C \big( I_Q^G (V_\delta)^K \big)^{W_{[M,\delta], \chi}} \big)$,
over all involved instances of \eqref{eq:6.33}.
By Theorem \ref{thm:1.17} that gives precisely 
\[
R \big( \Mod_\fl (C_r^* (G,K)^{\mf s})_{X_\nr^+ (L) W_{\mf s}(\sigma \otimes 
\chi)} \big).
\]
That and \eqref{eq:6.36} provide the required description of
\[
K_* \big( C_0 (X,\partial \tau) C_r^* (G,K)^{\mf s} / C_0 (X,\tau) 
C_r^* (G,K)^{\mf s} \big).
\]
(c) From part (a) and the $C(X)$-linear Morita equivalence between $C_r^* (G,K)^{\mf s}$ 
and $B \cong e_{\mf s}^c M_{n_{\mf s}^c} (C_r^* (G,K)^{\mf s}) e_{\mf s}^c$, we deduce 
that (a) holds for $C_0 (X,\partial \tau) B / C_0 (X,\tau) B$.
The involved finite dimensional semisimple algebra may not be the same as before, 
but it is Morita equivalent to that from part (a). 

The calculation of $K_* (C_0 (X,\partial \tau) B / C_0 (X,\tau) B)$ is the same as 
in part (b), and also follows from Morita equivalence.
\end{proof}

The following result shows that the K-groups in Lemma \ref{lem:6.15} and Proposition 
\ref{prop:6.16} are isomorphic.

\begin{lem}\label{lem:6.17}
\enuma{
\item The map 
\[
R\big( \C[W_{\mf s,\chi}^e,\natural_{\mf s}^{-1}] \big) \to R \big( \Mod_\fl 
(C_r^* (G,K)^{\mf s})_{X_\nr^+ (L) W_{\mf s}(\sigma \otimes \chi)} \big)
\]
induced by $\otimes_{C(X_\nr^u (L)) \rtimes \C[W_{\mf s,\chi}^e,\natural_{\mf s}]} 
\Pi_{\mf s,K}^c$ is a group isomorphism. 
\item Let $\phi_{\mf s,K,\tau,\partial \tau} : C_0 (X,\partial \tau) A /
C_0 (X,\tau) A \to C_0 (X,\partial \tau) B / C_0 (X,\tau) B$ be the homomorphism
of Banach algebras induced by $\phi_{\mf s,K}$. Then $K_* (\phi_{\mf s,K,\tau,
\partial \tau})$ is an isomorphism. 
}
\end{lem}
\begin{proof}
(a) For 
\[
\pi \in \Irr_\temp (G)^{\mf s}_{X_\nr^+ (L) W_{\mf s} (\sigma \otimes \chi)}
\cong \Irr (C_r^* (G,K)^{\mf s})_{X_\nr^+ (L) W_{\mf s}(\sigma \otimes \chi)},
\]
Theorem \ref{thm:6.12} shows that $\zeta^\vee (\pi)$ is equivalent to
$\Hom_{C_r^* (G,K)^{\mf s}} (\Pi_{\mf s,K}^c, \langle K \rangle \pi)$, in the sense that both
have the same central character $W_{\mf s}^e \chi$ and the same 
$\C[W_{\mf s,\chi}^e,\natural_{\mf s}^{-1}]$-representation $\pi_\chi$. 
In combination with Theorem \ref{thm:3.11}.a we find that
\begin{multline}\label{eq:6.37}
\phi_{\mf s,K}^* \circ \Hom_{C_r^* (G,K)^{\mf s}} (\Pi_{\mf s,K}^c,?) : \\
R \big( \Mod_\fl (C_r^* (G,K)^{\mf s})_{X_\nr^+ (L) W_{\mf s}(\sigma \otimes \chi)} 
\big)  \to R\big( \C[W_{\mf s,\chi}^e,\natural_{\mf s}^{-1}] \big) \cong
R ( \Mod (A)_{W_{\mf s}^e \chi} )
\end{multline}
is a group isomorphism. The inverse of this map is $\otimes_{C(X_\nr^u (L)) \rtimes 
\C[W_{\mf s,\chi}^e,\natural_{\mf s}]} \Pi_{\mf s,K}^c$.\\
(b) From Lemma \ref{lem:6.15}, Proposition \ref{prop:6.16} and \eqref{eq:6.37} 
we see that $K_* (\phi_{\mf s,K,\tau,\partial \tau})$ is the tensor product of the
identity on $K_* (C_0 (\tau,\partial \tau))$ and the isomorphism from part (a). 
\end{proof}

Our preparations to apply Lemma \ref{lem:6.13} are complete.

\begin{thm}\label{thm:6.18}
Let $\mf s \in \mf B (G)$ and choose $K \in \CO (G)$ as in \eqref{eq:6.39}.
There are isomorphisms of $\Z / 2 \Z$-graded groups
\[
K_* \big( C(X_\nr^u (L)) \rtimes \C [W_{\mf s}^e, \natural_{\mf s}^{-1}] \big)
\xrightarrow{K_* (\phi_{\mf s,K})} K_* \big( \End_{C_r^* (G,K)^{\mf s}}
(\Pi_{\mf s,K}^c)^{op}\big) \cong K_* (C_r^* (G,K)^{\mf s}) .
\]
\end{thm}
\begin{proof}
Lemma \ref{lem:6.17} shows that the conditions of Lemma \ref{lem:6.13} are
fulfilled. Then Lemma \ref{lem:6.13} tells us that $K_* (\phi_{\mf s,K})$ is 
an isomorphism. The second isomorphism follows from the Morita equivalence 
\eqref{eq:6.29} and Theorem \ref{thm:6.19}.
\end{proof}

Recall from Definition \ref{def:6.6} that we interpreted $K_* \big( 
C(X_\nr^u (L)) \rtimes \C [W_{\mf s}^e, \natural_{\mf s}^{-1}] \big)$ as
the K-theory of $\natural_{\mf s}^{-1}$-twisted $W_{\mf s}^e$-equivariant
vector bundles on $X_\nr^u (L)$. From \eqref{eq:6.9}, Theorem \ref{thm:6.18}
and Definition \ref{def:6.6} we conclude:

\begin{cor}\label{cor:6.20}
There are isomorphisms
\[
K_* (C_r^* (G)) \cong \bigoplus\nolimits_{\mf s \in \mf B (G)} 
K_* (C_r^* (G,K_{\mf s})^{\mf s}) \cong 
\bigoplus\nolimits_{\mf s \in \mf B (G)} K^*_{W_{\mf s}^e,\natural_{\mf s}^{-1}}
(X_\nr^u (L)) .
\]
\end{cor}

We note that Theorem \ref{thm:6.18} and Corollary \ref{cor:6.20} prove
\cite[Conjecture 5]{ABPS2}, a version of the ABPS conjecture in topological
K-theory.


\begin{thebibliography}{99}

\bibitem[Afg]{Afg} A. Afgoustidis,
``Progr\`es r\'ecents sur les repr\'esentations supercuspidales",
S\'eminaire Bourbaki {\bf 76} (2023), 1211

\bibitem[AfAu]{AfAu} A. Afghoustidis, A.-M. Aubert,
``C*-blocks and crossed products for classical $p$-adic groups",
IMRN {\bf 2022.22} (2021), 17849--17908

\bibitem[Ati]{Ati} M.F. Atiyah,
\emph{K-theory}, Mathematics Lecture Note Series,
W.A. Benjamin, New York NJ, 1967

\bibitem[ABP]{ABP} A.-M. Aubert, P.F. Baum, R.J. Plymen,
``The Hecke algebra of a reductive $p$-adic group: a view from noncommutative geometry",
pp. 1--34 in: \emph{Noncommutative geometry and number theory}, 
Aspects of Mathematics {\bf E37}, Vieweg Verlag, Wiesbaden, 2006

\bibitem[ABPS1]{ABPS1} A.-M. Aubert, P.F. Baum, R.J. Plymen, M. Solleveld,
``On the local Langlands correspondence for non-tempered representations",
M\"unster J. Math. {\bf 7} (2014) 27--50

\bibitem[ABPS2]{ABPS2} A.-M. Aubert, P.F. Baum, R.J. Plymen, M. Solleveld,
``Conjectures about $p$-adic groups and their noncommutative geometry",
Contemp. Math. {\bf 691} (2017), 15--51

\bibitem[AMS1]{AMS1} A.-M. Aubert, A. Moussaoui, M. Solleveld,
``Generalizations of the Springer correspondence and cuspidal Langlands parameters",
Manus. Math. {\bf 157} (2018), 121--192

\bibitem[AMS2]{AMS3} A.-M. Aubert, A. Moussaoui, M. Solleveld, 
``Affine Hecke algebras for Langlands parameters",
arXiv:1701.03593v4, 2022

\bibitem[BaMo1]{BaMo1} D. Barbasch, A. Moy, 
``Reduction to real infinitesimal character in affine Hecke algebras",
J. Amer. Math. Soc. {\bf 6.3} (1993), 611--635

\bibitem[BaMo2]{BaMo2} D. Barbasch, A. Moy, 
``Unitary spherical spectrum for $p$-adic classical groups", 
Acta Appl. Math. {\bf 44} (1996), 3--37

\bibitem[BaL\"u]{BaLu} A. Bartels, W. L\"uck,
``Algebraic K-theory of reductive $p$-adic groups",
arXiv:2306.03452, 2023

\bibitem[BaCo]{BaCo} P.F. Baum, A. Connes,	
``Chern character for discrete groups", pp. 163--232 in:
\emph{A f\^ete of topology}, Academic Press, Boston MA, 1988

\bibitem[BCH]{BCH} P.F. Baum, A. Connes, N. Higson,
``Classifying space for proper actions and K-theory of group C*-algebras",
pp. 240--291 in \emph{C*-algebras: 1943-1993. A fifty year celebration}, 
Contemp. Math. {\bf 167}, American Mathematical Society, Providence RI, 1994

\bibitem[BHP]{BHP} P.F. Baum, N. Higson, R.J. Plymen,
``Representation theory of $p$-adic groups: a view from operator algebras",
Proc. Sympos. Pure. Math. {\bf 68} (2000), 111--149

\bibitem[Ber]{Ber} J. Bernstein,	
``All reductive $p$-adic groups are tame",
Funct. Anal. Applic. {\bf 8.2} (1974), 91--93

\bibitem[BeDe]{BeDe} J. Bernstein, P. Deligne,	
``Le "centre" de Bernstein", pp. 1--32 in:
\emph{Repr\'esentations des groupes r\'eductifs sur un corps local}, 
Travaux en cours, Hermann, Paris, 1984

\bibitem[BDK]{BDK} J. Bernstein, P. Deligne, D. Kazhdan,
``Trace Paley-Wiener theorem for reductive $p$-adic groups",
J. Analyse Math. {\bf 47} (1986), 180--192

\bibitem[Bla]{Bla} B. Blackadar,
\emph{K-theory for operator algebras 2nd ed.},
Mathematical Sciences Research Institute Publications {\bf 5},
Cambridge University Press, Cambridge, 1998

\bibitem[Blo]{Blo} C. Blondel,	 
``Quelques propri\'et\'es des paires couvrantes",
Math. Ann. {\bf 331} (2005), 243--257 	

\bibitem[Bor]{Bor} A. Borel,
``Admissible representations of a semi-simple group over a local field
with vectors fixed under an Iwahori subgroup",
Inv. Math. {\bf 35} (1976), 233--259

\bibitem[Bos]{Bos} J.-B. Bost,
``Principe d'Oka, K-th\'eorie et syst\`emes dynamiques non commutatifs",
Invent. Math. {\bf 101} (1990), 261--333

\bibitem[BrPl]{BrPl} J. Brodzki, R.J. Plymen,
``Periodic cyclic homology of certain nuclear algebras",
C.R. Acad. Sci. Paris {\bf 329} (1999), 671--676

\bibitem[Bry]{Bry} J.L. Brylinski,
``Cyclic homology and equivariant theories",
Ann. Inst. Fourier {\bf 37.4} (1987), 15--28

\bibitem[BuKu]{BuKu} C.J. Bushnell, P.C. Kutzko,
``Smooth representations of reductive $p$-adic groups: structure theory via types",
Proc. London Math. Soc. {\bf 77.3} (1998), 582--634

\bibitem[Cas]{Cas} W. Casselman, 
``The Steinberg character as a true character",
pp. 413--417 in: \emph{Harmonic ana\-ly\-sis on homogeneous spaces},
Proc. Sympos. Pure Math., Vol. {\bf XXVI},
American Mathematical Society, Providence RI, 1973

\bibitem[Con]{Con} A. Connes,
``Noncommutative differential geometry",
Publ. Math. Inst. Hautes \'Etudes Sci. {\bf 62} (1985), 41--144

\bibitem[Cun1]{Cun} J.R. Cuntz,
``Bivariante K-Theorie f\"ur lokalkonvexe Algebren und der Chern-Connes-Charakter",
Doc. Math {\bf 2} (1997), 139--182

\bibitem[Cun2]{Cun2} J.R. Cuntz,
``Morita invariance in cyclic homology for nonunital algebras",
K-Theory {\bf 15} (1998), 301--305

\bibitem[CuQu]{CuQu} J.R. Cuntz, D. Quillen,
``Excision in bivariant periodic cyclic cohomology",
Inv. Math. {\bf 127} (1997), 67--98

\bibitem[CuRe]{CuRe} C.W. Curtis, I. Reiner,
\emph{Representation theory of finite groups and associative algebras},
Pure and Applied Mathematics {\bf 11},
John Wiley \& Sons, New York--London, 1962

\bibitem[DiSc]{DiSc} J. Dieudonn\'e, L. Schwartz, 
``La dualit\'e dans les espaces $(\mc F)$ et $(\mc{LF})$",
Ann. Inst. Fourier {\bf 1} (1949), 61--101

\bibitem[DiFi]{DiFi} I.K. Dimitrov, R. Fioresi, 
``Generalized root systems", 
Trans. Amer. Math. Soc. Ser. B {\bf 11} (2024), 1462--1508

\bibitem[Dix]{Dix} J. Dixmier,
\emph{Les C*-alg\`ebres et leurs representations},
Cahiers Scientifiques {\bf 29}, Gauthier-Villars \'Editeur, Paris, 1969

\bibitem[Eve]{Eve} S. Evens,
``The Langlands classification for graded Hecke algebras",
Proc. Amer. Math. Soc. {\bf 124.4} (1996), 1285--1290

\bibitem[FGV]{FGV} H. Figueroa, J.M. Gracia-Bondia, J.C. Varilly,
\emph{Elements of noncommutative geometry},
Birkh\"auser Advanced Texts: Basler Lehrb\"ucher, Birkh\"auser, Boston MA, 2001

\bibitem[Fin]{Fin} J. Fintzen,
``Supercuspidal representations: construction, classification, and characters",
arXiv:2510.12883, 2025 (to appear in Proc. Symp. Pure Math.)

\bibitem[FlSo]{FlSo} J. Flikkema, M. Solleveld,
``Intertwining operators for representations of covering groups of reductive $p$-adic groups",
arXiv:2502.18128v2, 2025

\bibitem[Gou]{Gou} F.Q. Gouv\^ea, \emph{P-adic numbers},
Universitext, Springer-Verlag, Berlin, 1993

\bibitem[HC]{HC} Harish-Chandra,			
``Harmonic analysis on reductive $p$-adic groups",
pp. 167--192 in: \emph{Harmonic analysis on homogeneous spaces}, 
Proc. Sympos. Pure Math. {\bf 26}, American Mathematical Society, Providence RI, 1973

\bibitem[Hei1]{Hei1} V. Heiermann,	
``Une formule de Plancherel pour l'alg\`ebre de Hecke d'un groupe r\'eductif $p$-adique",
Comment. Math. Helv. {\bf 76} (2001), 388--415

\bibitem[Hei2]{Hei2} V. Heiermann,
``Op\'erateurs d'entrelacement et alg\`ebres de Hecke avec param\`etres d'un groupe 
r\'eductif $p$-adique - le cas des groupes classiques",
Selecta Math. {\bf 17.3} (2011), 713--756

\bibitem[HiNi]{HiNi} N. Higson, V. Nistor,
``Cyclic homology of totally disconnected groups acting on buildings",
J. Funct. Anal. {\bf 141.2} (1996), 466--495

\bibitem[Hoc]{Hoc} G. Hochschild, 
``On the cohomology theory for associative algebras",
Ann. of Math. {\bf 47} (1946), 568--579

\bibitem[HKR]{HKR} G. Hochschild, B. Kostant, A. Rosenberg,
``Differential forms on regular affine algebras",
Trans. Amer. Math. Soc. {\bf 102.3} (1962), 383--408

\bibitem[Ill]{Ill} S. Illman,
``Smooth equivariant triangulations of G-manifolds for G a finite group"
Math. Ann. {\bf 233} (1978), 199--220

\bibitem[IwMa]{IwMa} N. Iwahori, H. Matsumoto,
``On some Bruhat decomposition and the structure of the Hecke rings of the $p$-adic 
Chevalley groups",
Inst. Hautes \'Etudes Sci. Publ. Math {\bf 25} (1965), 5--48

\bibitem[Jac]{Jac} H. Jacquet,
``Représentations des groupes linéaires $p$-adiques", pp. 119--220 in: 
\emph{C.I.M.E., II Ciclo Montecatini 1970, Theory Group Represent. Fourier Analysis}, 1971 

\bibitem[Joh]{Joh} B.E. Johnson,
\emph{Cohomology in Banach algebras},
Mem. Amer. Math. Soc. {\bf 127},
American Mathematical Society, Providence RI, 1972

\bibitem[Jul]{Jul} P. Julg,
``K-th\'eorie \'equivariante et produits crois\'es",
C.R. Acad. Sci. Paris {\bf 292} (1981), 629--632

\bibitem[KaPr]{KaPr} T. Kaletha, G. Prasad,
\emph{Bruhat–Tits Theory: A New Approach},
New Mathematical Monographs {\bf 44},
Cambridge University Press, 2023

\bibitem[KNS]{KNS} D. Kazhdan, V. Nistor, P. Schneider,
``Hochschild and cyclic homology of finite type algebras",
Sel. Math. New Ser. {\bf 4.2} (1998), 321--359

\bibitem[KaSo]{KaSo} D. Kazhdan, M. Solleveld,
``A comparison of Hochschild homology in algebraic and smooth settings",
Bull. London Math. Soc. {\bf 57.4} (2025), 1249--1269

\bibitem[Ker]{Ker} I. Kersten, 
\emph{Brauergruppen}, Universit\"atsverlag G\"ottingen, 2007

\bibitem[KrRa]{KrRa} C. Kriloff, A. Ram, 
``Representations of graded Hecke algebras", 
Represent. Th. {\bf 6} (2002), 31--69

\bibitem[Laf]{Laf} V. Lafforgue,
``K-th\'eorie bivariante pour les alg\`ebres de Banach et conjecture de Baum--Connes",
Invent. Math. {\bf 149.1} (2002), 1--95

\bibitem[Lod]{Lod} J.-L. Loday,
\emph{Cyclic homology 2nd ed.},
Mathematischen Wissenschaften {\bf 301}, Springer-Verlag, Berlin, 1997

\bibitem[Lus]{Lus-Gr} G. Lusztig,
``Affine Hecke algebras and their graded version",
J. Amer. Math. Soc {\bf 2.3} (1989), 599--635

\bibitem[Mey1]{Mey1} R. Meyer,
``Homological algebra for Schwartz algebras of reductive $p$-adic groups",
pp. 263--300 in: \emph{Noncommutative geometry and number theory},
Aspects of Mathematics {\bf E37}, Vieweg Verlag, Wiesbaden, 2006

\bibitem[Mey2]{Mey} R. Meyer,
\emph{Local and analytic cyclic homology},
Tracts in Mathematics {\bf 3},
European Mathematical Society Publishing House, 2007

\bibitem[MoTa]{MoTa} A. Moy, M. Tadi\'c,
``Some algebras of essentially compact distributions of a reductive $p$-adic group",
pp. 247--266 in: \emph{Harmonic analysis, group representations, automorphic
forms and invariant theory}, 
Lect. Notes Ser. Inst. Math. Sci. Natl. Univ. Singapore {\bf 12} (2007)

\bibitem[Nis]{Nis} V. Nistor,
``A non-commutative geometry approach to the representation theory of reductive
$p$-adic groups: Homology of Hecke algebras, a survey and some new results", 
pp. 301--323 in: \emph{Noncommutative geometry and number theory},	 
Aspects of Mathematics {\bf E37}, Vieweg Verlag, Wiesbaden, 2006

\bibitem[Oha]{Oha} K. Ohara,
``On parameters of Hecke algebras for $p$-adic groups",
arXiv:2505.16040, 2025 		 

\bibitem[Phi1]{Phi1} N.C. Phillips,
``Equivariant K-theory and freeness of group actions on C*-algebras",
Lecture Notes in Mathematics {\bf 1274},
Springer-Verlag, Berlin, 1987

\bibitem[Phi2]{Phi2} N.C. Phillips,
``K-theory for Fr\'echet algebras",
Int. J. Math. {\bf 2.1} (1991), 77--129

\bibitem[Ply]{Ply} R.J. Plymen,
``Reduced C*-algebra for reductive $p$-adic groups",
J. Funct. Anal. {\bf 88.2} (1990), 251--266

\bibitem[RaRa]{RaRa} A. Ram, J. Rammage,
``Affine Hecke algebras, cyclotomic Hecke algebras and Clifford theory",
pp. 428--466 in: \emph{A tribute to C.S. Seshadri (Chennai 2002)}, 
Trends in Mathematics, Birkh\"auser, 2003

\bibitem[Ren]{Ren} D. Renard,
\emph{Repr\'esentations des groupes r\'eductifs p-adiques},
Cours sp\'ecialis\'es {\bf 17}, Soci\'et\'e Math\'ematique de France, 2010

\bibitem[Roc]{Roc} A. Roche,
``The Bernstein decomposition and the Bernstein centre", 
pp. 3--52 in: \emph{Ottawa lectures on admissible representations of reductive $p$-adic groups},
Fields Inst. Monogr. {\bf 26}, American Mathematical Society, Providence RI, 2009

\bibitem[Sau]{Sau} F. Sauvageot,
``Principe de densit\'e pour les groupes r\'eductifs",
Compos. Math. {\bf 108} (1997), 151--184

\bibitem[Sch]{Schn} P. Schneider,
``The cyclic homology of $p$-adic reductive groups",
J. f\"ur reine angew. Math. {\bf 475} (1996), 39--54

\bibitem[SSZ]{SSZ} P. Schneider, E.-W. Zink,
``K-types for the tempered components of a $p$-adic general linear group. 
With an appendix by Schneider and U. Stuhler",
J. reine angew. Math. {\bf 517} (1999), 161--208

\bibitem[Sil]{Sil} A.J. Silberger,
\emph{Introduction to harmonic analysis on reductive p-adic groups},
Mathematical Notes {\bf 23}, Princeton University Press, Princeton NJ, 1979

\bibitem[Slo]{Slo} K. Slooten,
``Generalized Springer correspondence and Green functions for type B/C graded Hecke algebras",
Advances in Math. {\bf 203.1} (2006), 34--108

\bibitem[Sol1]{SolCh} M. Solleveld,
``Some Fr\'echet algebras for which the Chern character is an isomorphism",
K-theory {\bf 36} (2005), 275--290 and correction in arXiv:math/0505282 (2008)

\bibitem[Sol2]{SolHP} M. Solleveld,
``Periodic cyclic homology of reductive $p$-adic groups",
J. Noncommutative Geometry {\bf 3.4} (2009), 501--558

\bibitem[Sol3]{SolHomGHA} M. Solleveld,
``Homology of graded Hecke algebras",
J. Algebra {\bf 323} (2010), 1622--1648

\bibitem[Sol4]{SolGHA} M. Solleveld,
``Parabolically induced representations of graded Hecke algebras",
Algebras Represent. Th. {\bf 15.2} (2012), 233--271

\bibitem[Sol5]{SolAHA} M. Solleveld,
``On the classification of irreducible representations of affine Hecke algebras 
with unequal parameters",
Represent. Th. {\bf 16} (2012), 1--87

\bibitem[Sol6]{SolQ} M. Solleveld,
``Pseudo-reductive and quasi-reductive groups over non-archimedean local fields",
J. Algebra {\bf 510} (2018), 331--392

\bibitem[Sol7]{SolComp} M. Solleveld,
``On completions of Hecke algebras",
pp. 207--262 in: \emph{Representations of Reductive p-adic Groups, 
A.-M. Aubert, M. Mishra, A. Roche, S. Spallone (eds.)},
Progress in Mathematics {\bf 328}, Birkh\"auser, 2019

\bibitem[Sol8]{SolSurv} M. Solleveld,
``Affine Hecke algebras and their representations",
Indag. Math. {\bf 32.5} (2021), 1005--1082 

\bibitem[Sol9]{SolEnd} M. Solleveld,
``Endomorphism algebras and Hecke algebras for reductive $p$-adic groups",
J. Algebra {\bf 606} (2022), 371--470, and correction in arXiv:2005.07899v3 (2023)

\bibitem[Sol10]{SolTwist} M. Solleveld,
``Hochschild homology of twisted crossed products and twisted graded Hecke algebras",
Ann. K-theory {\bf 8.1} (2023), 81--126 

\bibitem[Sol11]{SolHH} M. Solleveld,
``Hochschild homology of reductive $p$-adic groups",
J. Noncommut. Geom. {\bf 18.4} (2024), 1349--1413

\bibitem[Sol12]{SolParam} M. Solleveld,
``Parameters of Hecke algebras for Bernstein components of $p$-adic groups",
Indag. Math. {\bf 36.1} (2025), 124--170 

\bibitem[Tay]{Tay} J.L. Taylor,
``Homology and cohomology for topological algebras",
Adv. Math. {\bf 9} (1972), 137--182

\bibitem[Tel]{Tel} N. Teleman,
``Microlocalisation de l'homologie de Hochschild",
C.R. Acad. Sci. Paris {\bf 326} (1988), 1261--1264

\bibitem[Tit]{Tit} J. Tits,	``Reductive groups over local fields",
pp. 29--69 in: \emph{Automorphic forms, representations and L-functions Part I},
Proc. Sympos. Pure Math. {\bf 33}, American Mathematical Society, Providence RI, 1979

\bibitem[Val]{Val} A. Valette,
\emph{Introduction to the Baum-Connes conjecture},
Lectures in Mathematics ETH Z\"urich, Birkh\"auser Verlag, Basel, 2002

\bibitem[Vig]{Vig} M.-F. Vign\'eras,
``On formal dimensions for reductive $p$-adic groups"
pp. 225--266 in: \emph{Festschrift in honor of I.I.\! Piatetski-Shapiro on the occasion of his 
sixtieth birthday, Part I}, Israel Math. Conf. Proc. {\bf 2}, Weizmann, Jerusalem, 1990

\bibitem[Voi]{Voi} C. Voigt,
``Chern character for totally disconnected groups",
Math. Ann. {\bf 343.3} (2009), 507--540

\bibitem[Wal]{Wal} J.-L. Waldspurger, ``La formule de Plancherel pour les groupes 
$p$-adiques (d'apr\`es Harish-Chandra)",
J. Inst. Math. Jussieu {\bf 2.2} (2003), 235--333


\end{thebibliography}
\end{document}